\documentclass[12pt, reqno]{article}
\usepackage{amssymb, mathrsfs}
\usepackage{latexsym}
\usepackage{amsmath}
\usepackage{amsthm}
\usepackage{amsfonts}
\usepackage{dsfont}
\usepackage{mathtools}
\usepackage{epsfig}
\usepackage{verbatim}
\usepackage[all,cmtip]{xy}
\usepackage[left=1.2in,right=1.2in,top=1in,bottom=1in]{geometry}
\usepackage{enumerate}
\usepackage[shortlabels, inline]{enumitem}

\usepackage[colorlinks=true,linkcolor=blue,citecolor=blue]{hyperref}
\usepackage[title]{appendix}
\usepackage{float}

\usepackage{tikz}
\usepackage{tikz-cd}
\usetikzlibrary{matrix, patterns, calc, intersections}

\usepackage{graphicx}

\usepackage{etoolbox}

\makeatletter
\newcounter{author}
\renewcommand*\author[1]{%
	\stepcounter{author}%
	\ifnum\c@author=1
	\gdef\@author{#1}%
	\else
	\xdef\@author{\unexpanded\expandafter{\@author\and#1}}%
	\fi
	\csgdef{author@\the\c@author}{#1}}
\newcommand*\email[1]{%
	\csgdef{email@\the\c@author}{#1}}
\newcommand*\address[1]{%
	\csgdef{address@\the\c@author}{#1}}
\AtEndDocument{%
	\xdef\author@count{\the\c@author}%
	\c@author=1
	\print@authors}
\newcommand*\print@authors{%
	\ifnum\c@author>\author@count
	\else
	\print@author{\the\c@author}%
	\advance\c@author by 1
	\expandafter\print@authors
	\fi}
\newcommand*\print@author[1]{%
	\par\medskip
	\begin{tabular}{@{}l@{}}%
		\csuse{address@#1}\\
		\textit{Email address}:
		\csuse{email@#1}
\end{tabular}}
\makeatother

\let\emptyset\varnothing

\newtheorem*{lemma*}{Lemma}

\newtheorem{theorem}{Theorem}[section]
\newtheorem{proposition}[theorem]{Proposition}
\newtheorem{lemma}[theorem]{Lemma}

\newtheorem{corollary}[theorem]{Corollary}

\theoremstyle{definition}
\newtheorem{claim}[theorem]{Claim}
\newtheorem*{claim*}{Claim}
\newtheorem{definition}[theorem]{Definition}
\newtheorem{example}[theorem]{Example}

\theoremstyle{remark}
\newtheorem{remark}[theorem]{Remark}

\newcommand{\ind}{\textup{Ind}}
\newcommand{\supp}{\textup{Supp}}
\newcommand{\ev}{\textup{ev}}

\newcommand{\topo}{{TOP}}
\newcommand{\id}{\textup{Id}}
\newcommand{\sub}{\textup{Sub}}

\newcommand{\ord}{\textup{ord}}
\newcommand{\fin}{\textup{fin}}
\newcommand{\homt}{\textup{Aut}_h}

\newcommand{\rel}{\mathrm{\,rel\,}}

\newcommand{\GP}{\mathcal{GP}}

\mathchardef\mhyphen="2D

\begin{document}

\title{Additivity of higher rho invariants and nonrigidity of topological manifolds}

\author{Shmuel Weinberger\thanks{partially supported by NSF 1510178.}}
\address{(Shmuel Weinberger) Department of Mathematics, University of Chicago }
\email{shmuel@math.uchicago.edu}

\author{Zhizhang Xie\thanks{partially supported by NSF 1500823, NSF 1800737}}
\address{(Zhizhang Xie) Department of Mathematics, Texas A\&M University }
\email{xie@math.tamu.edu}

\author{Guoliang Yu\thanks{partially supported by NSF 1700021, NSF 1564398 and Simons Fellows Program.}}
\address{(Guoliang Yu) Department of
Mathematics, Texas A\&M University}
\email{guoliangyu@math.tamu.edu}

\date{}

\maketitle

{\centering\footnotesize \textit{In memory of  Ron Douglas, Andrew Ranicki and John Roe with a deep sense of loss.} \par}

\begin{abstract}
Let $X$ be a closed oriented connected topological manifold of dimension $n\geq 5$. The structure group $\mathcal S^\topo(X)$ is the abelian group of equivalence classes of all pairs $(f, M)$ such that $M$ is a closed oriented manifold and $f\colon M \to X$ is an orientation-preserving homotopy equivalence.  The main purpose of this article is to prove that a higher rho invariant defines a group homomorphism from the structure group $\mathcal S^\topo(X)$ of $X$ to the analytic structure group $K_n(C^\ast_{L,0}(\widetilde X)^\Gamma)$ of $X$. Here $\widetilde X$ is the universal cover of $X$, $\Gamma = \pi_1 X$ is the fundamental group of $X$, and $C^\ast_{L,0}(\widetilde X)^\Gamma$ is a certain $C^\ast$-algebra. In fact, we introduce a higher rho invariant map on the homology manifold structure group of a closed oriented connected \emph{topological} manifold, and prove its additivity. This higher rho invariant map restricts to the higher rho invariant map on the topological structure group. More generally, the same techniques developed in this paper can be applied to define a higher rho invariant map on the homology manifold structure group of a closed oriented connected \emph{homology} manifold. As an application, we use the additivity of the higher rho invariant map to study non-rigidity of topological manifolds. More precisely,  we give a lower bound for the free rank of the \emph{algebraically reduced} structure group of $X$ by the number of torsion elements in $\pi_1 X$. Here the algebraic reduced structure group of $X$ is the quotient of $\mathcal S^\topo(X)$ modulo a certain action of self-homotopy equivalences of $X$. We also introduce a notion of homological higher rho invariant, which can be used to detect many elements in the structure group of a  closed oriented topological manifold, even when the fundamental group  of the manifold is torsion free. In particular, we apply this homological higher rho invariant to show that the structure group is not finitely generated for a class of manifolds.  
\end{abstract}

\section{Introduction}

Let $D$ be an elliptic operator on a closed manifold $M$ of dimension $n$. Suppose $\widetilde M$ is the universal cover of $M$, and $\widetilde D$ is the lift of $D$ onto $\widetilde M$. Then $\widetilde D$ defines a higher index class in $K_n(C^\ast_r(\pi_1 M))$, where $\pi_1 M$ is the fundamental group of $M$ and $K_n(C^\ast_r(\pi_1 M))$ is the $K$-theory of the reduced group $C^\ast$-algebra $C^\ast_r(\pi_1 M)$. This higher index class is an obstruction to the invertibility of $\widetilde D$. It is a far-reaching generalization of the classical Fredholm index and plays a fundamental role in the studies of many problems in geometry and topology such as the Novikov conjecture and the Gromov-Lawson-Rosenberg conjecture. Higher index classes are often referred to as primary invariants. When the higher index class of an elliptic operator is trivial and given a specific trivialization, a secondary index theoretic invariant naturally arises.   This secondary invariant is  called the higher rho invariant in acknowledgement of the discussion in \cite[Chapter 14E]{MR1687388} of invariants for odd dimensional manifolds with finite fundamental group and the discussion in \cite{A-P-S75b} of invariants of odd dimensional manifolds with flat bundles, and its connection to index theory for manifolds with boundary. It serves as an obstruction of the locality of the inverse of an invertible elliptic operator.   In our current context,   if there is  an orientation-preserving homotopy equivalence between two oriented closed manifolds, then the higher index of the signature operator on the disjoint union of the two manifolds (one of them with opposite orientation) is trivial with a trivialization given by the homotopy equivalence. Hence such a homotopy equivalence naturally defines a higher rho invariant. More generally, the notion of higher rho invariants can be defined for homotopy equivalences between topological manifolds, and it is a powerful tool to detect whether a homotopy equivalence can be ``deformed" into a homeomorphism.  The main purpose of this paper is to prove that the higher rho invariant defines a group homomorphism on the structure group of a \emph{topological} manifold. As an application, we use the higher rho invariant to detect elements in reduced structure groups of topological manifolds. 

Let $X$ be a closed oriented connected \emph{topological} manifold of dimension $n$. The structure group $\mathcal S^\topo(X)$ is the abelian group of equivalence classes of all pairs $(f, M)$ such that $M$ is a closed oriented manifold and $f\colon M \to X$ is an orientation-preserving homotopy equivalence. In this version, the zero element corresponds to the identity map $\id\colon X \to X$. Of course replacing $X$ by a homotopy equivalent manifold $Y$ means that the new zero element is $\id\colon Y \to Y$. Thus, although it transpires that $\mathcal S^\topo(X)$ is a functor, its identification with the set of manifolds homotopy equivalent to $X$ is that of collapsing a torsor to a group by choosing a base element. The main result of this article is to prove that the higher rho invariant defines a group homomorphism from $\mathcal S^\topo(X)$ to $K_n(C^\ast_{L,0}(\widetilde X)^\Gamma)$, where $\Gamma = \pi_1 X$ is the fundamental group of $X$, $\widetilde X$ is the universal cover of $X$ and $C^\ast_{L,0}(\widetilde X)^\Gamma$ is a certain geometric $C^\ast$-algebra. The definition of $C^\ast_{L,0}(\widetilde X)^\Gamma$ is reviewed in Section $\ref{sec:prem}$, and the precise definition of the higher rho invariant is given in Section $\ref{sec:highrho}$. 

Perhaps the simplest interpretation of the abelian group structure on $\mathcal S^\topo(X)$ can be described through a periodicity map\footnote{The periodicity map was first given by Siebemann \cite[Appendix C to Essay V]{MR0645390}, with a correction by Nicas \cite{MR668807} (cf. the discussion after Proposition $\ref{prop:strid}$). We denote this Siebenmann periodicity map by $\mathcal{SP}$. A geometric construction of a periodicity map, denoted by $\GP$,  was given by Cappell and the first author \cite{MR873283}, cf. the discussion after Proposition $\ref{prop:strid}$ below. At the time when Cappell and the first author wrote their paper, it was not known whether the map $\GP$ coincides with the Siebenmann periodicity map $\mathcal{SP}$. Later, Crowley and Macko showed that a quaternionic (resp. octonionic) version of $\GP$ coincides with $\mathcal{SP}^2$ (resp. $\mathcal{SP}^4$) \cite{MR2826928}. In this paper, we show that $\GP$ indeed coincides with $\mathcal{SP}$ by applying the device of periodicity spaces from the work of the first author and Yan \cite{MR1823952,MR2154831} (cf. Proposition $\ref{prop:strhom}$ and Theorem $\ref{thm:gp=sp}$).}, which is an injection from $\mathcal S^\topo(X)$ to $\mathcal S^\topo_\partial(X\times D^4)$, where $D^4$ is the $4$-dim Euclidean unit ball and $\mathcal S^\topo_\partial(X\times D^4)$ is the $\rel \partial$ version of structure set of $X\times D^4$ (cf. Definition $\ref{def:relbdry}$ below).   The set $\mathcal S^\topo_\partial(X\times D^4)$  carries a natural abelian group structure by stacking (cf. Definition $\ref{def:stack}$), hence induces an abelian group structure on $\mathcal S^\topo(X)$.   Both $\mathcal S^\topo(X)$ and   $\mathcal S^\topo_\partial(X\times D^4)$ carry a higher rho invariant map. It is not difficult to verify that the higher rho invariant map on $\mathcal S^\topo_\partial(X\times D^4)$ is additive, i.e.,  a homomorphism between abelian groups. One possible approach to show the additivity of the higher rho invariant map on $\mathcal S^\topo(X)$ is to prove  the compatibility of higher rho invariant maps on $\mathcal S^\topo(X)$ and   $\mathcal S^\topo_\partial(X\times D^4)$. However, there are some essential analytical difficulties to \emph{directly} prove such a compatibility,  due to the subtleties of the periodicity map.   A main novelty of this paper is to give a new description  of the topological structure group in terms of smooth manifolds with boundary (see Definition $\ref{def:newstruc}$ and the discussion in Section $\ref{sec:smstr}$ below). This new description uses a broader class of objects than closed manifolds and an equivalence relation broader than $h$-cobordism, which allows us to replace topological manifolds in the usual definition of $\mathcal S^\topo(X)$ by smooth manifolds with boundary. Such a description   leads to a transparent group structure given by disjoint union. The main body of the paper is devoted to proving that the new description coincides with the classical description of the topological structure group; and to developing the theory of higher rho invariants in this new setting, in which higher rho invariants are easily seen to be additive.    As a consequence, the higher rho invariant maps on   $\mathcal S^\topo(X)$ and   $\mathcal S^\topo_\partial(X\times D^4)$ are indeed compatible (see Proposition $\ref{prop:comp}$ below). 

We point out that, in the odd dimensional case,  the higher rho invariant defined in this paper is a refinement of  the higher rho invariant for signature operators  in the literature (cf. \cite[Section 3]{MR2220524} \cite[Remark 4.6]{MR3514938} \cite{MR3622237}). More precisely, in the odd dimensional case, the higher rho invariant in the literature  is twice of the  higher rho invariant of this paper (cf. Remark $\ref{rmk:rhocompare}$ and Theorem $\ref{thm:tworho}$ below). 

The higher rho invariant map we construct is actually defined on the homology manifold structure group $\mathcal S^{HTOP}(X)$ of $X$, where $\mathcal S^{HTOP}(X)$  is the abelian group of equivalence classes of all pairs $(f, M)$ such that $M$ is a closed oriented ANR homology manifold and $f\colon M \to X$ is an orientation-preserving homotopy equivalence.  This higher rho invariant map coincides with the original higher rho invariant map when restricted to the topological structure group, cf. the discussion after Theorem $\ref{thm:struciso}$ below.  More generally, our method can also be applied to define a higher rho invariant map on the homology manifold structure group of a closed oriented connected \emph{homology} manifold, cf. Remark $\ref{rk:hommfld}$.     All results of the paper can be easily extended to the case where $X$ has multiple connected components, by studying each component separately. For simplicity, we shall only consider the case where $X$ is connected.

As an application, we use our main theorem to estimate the sizes of reduced structure groups of topological manifolds.  There are two different ways for a self-homotopy $h$ of $X$ to act on $\mathcal S^\topo(X)$ (cf. Section $\ref{sec:nonrig}$). One action induces a group isomorphism of $\mathcal S^\topo(X)$, and is compatible with the actions of $h$ on other terms in the topological surgery long exact sequence of $X$. We denote this action by $\alpha_h$.  The other action, denoted by $\beta_h$, only induces a set-theoretic bijection  of $\mathcal S^\topo(X)$, and in general is \emph{not} compatible with the actions of $h$ on other terms in the topological surgery long exact sequence of $X$.  Let $\widetilde{\mathcal S}^\topo_{alg}(X)$ be the quotient group of $\mathcal S^\topo(X)$ modulo the subgroup generated by elements of the form $\theta - \alpha_h(\theta)$ for all $\theta\in \mathcal S^\topo(X)$ and all orientation-preserving self-homotopy equivalences $h$ of $X$ (see Definition $\ref{def:rdstr}$). We call $\widetilde{\mathcal S}^\topo_{alg}(X)$ the \emph{algebraically reduced structure group} of $X$. Similarly, we can define a version of reduced structure group for the other action, which will be denoted by $\widetilde{ \mathcal S}^\topo_{geom}(X)$  and called the \emph{geometrically reduced structure group} of $X$ from now on.  We apply our main theorem, combining with the work in \cite{MR3590536},   to give a lower bound of the free rank of $\widetilde{\mathcal S}^\topo_{alg}(X)$. There is strong evidence suggesting that an analogue holds for  $\widetilde{ \mathcal S}^\topo_{geom}(X)$ as well.

When the strong Novikov conjecture holds for $\pi_1 X$, we introduce a homological higher rho invariant. We call this invariant the Novikov rho invariant for obvious reasons to be explained in Section $\ref{sec:novrho}$. The Novikov rho invariant can be used to detect many elements in $\mathcal S^\topo(X)$, even when $\pi_1 X$ is torsion free. In particular, we apply this Novikov rho invariant to show that the structure group is not finitely generated for certain manifolds. 

If the Baum-Connes conjecture holds for $\pi_1 X$, then the Novikov rho invariant is equivalent to the higher rho invariant.  The first author also studied a different homological higher rho invariant for manifolds satisfying certain   vanishing conditions on homology in the middle dimension \cite{MR1707352}.  The higher rho invariant in the current paper can also be generalized to those settings. For example, if $M$ is a manifold and $\mathbb Z/2$ acts freely and homologically trivially in the sense of \cite{MR963636}, then this involution defines a homological higher rho invariant \cite[Remark 0.8(a)]{MR1707352}, which away from the prime number $2$ is realized by an element in the  topological structure group of $M$. This plays an important role in the cobordism of homologically trivial actions \cite{MR963636}.

The higher rho invariant map on the structure set of a smooth manifold was first introduced by Higson and Roe \cite{MR2220524}.  Zenobi extended the higher rho invariant map (as a map of sets) to topological manifolds \cite{MR3622237}. In the cyclic cohomology setting, Lott studied the higher eta invariant (a close relative of the higher rho invariant) under certain conditions \cite{JLott92}.

Our approach to the higher rho invariant is very much inspired by the work of Higson and Roe on the analytic surgery long exact sequence for smooth manifolds and structure sets of smooth manifolds \cite{MR2220522,MR2220523,MR2220524}. In their work, Higson and Roe proved that in the smooth setting, the higher rho invariant establishes a set  theoretic commutative diagram between the smooth surgery sequence and  the analytic surgery sequence. Our main result implies that in the topological setting, the higher rho invariant can be used to construct a commutative diagram of abelian groups between the topological surgery sequence  and the analytic surgery sequence. In fact, the same method also establishes a commutative diagram between the homology-manifold surgery sequence and the analytic surgery (cf. Proposition $\ref{prop:homocomm}$).  Furthermore, this method can also be used to show that the Cheeger-Gromov rho invariant \cite{MR806699} defines a homomorphism on the structure group. 

There are other equivalent ways of studying the topological surgery sequence. Our approach in the current paper is closer to those of Wall \cite{MR1687388} and Quinn \cite{MR0276980}, and is more geometric in nature. If we were to take a more algebraic approach by using Ranicki's algebraic surgery long exact sequence \cite{MR1211640}, then many of the discussions in Section $\ref{sec:strgrp}$ can be avoided. In particular, if we use Ranicki's algebraic surgery long exact sequence, then the techniques from \cite{MR2220522,MR2220523,MR2220524}  can be adapted more directly to the topological setting. On the other hand, our geometric approach appears to be more intuitive and directly implicates elliptic operators in the discussion. We remark that the rational additivity of the higher rho invariant for finite fundamental groups (more generally, the rational additivity after mapping the fundamental group to a finite group) was proved by Crowley and Macko \cite[Theorem 1.1]{MR2826928}. 

The paper is organized as follows. In Section $\ref{sec:prem}$, we recall some standard definitions of geometric $C^\ast$-algebras. In Section $\ref{sec:strgrp}$, we introduce a new definition of structure groups of topological manifolds based on ideas of Wall and ideas from controlled topology. This new definition leads to a transparent group structure of the topological structure group, which is given by disjoint union. We prove that the new definition of the structure group  is naturally isomorphic to the classical structure group.  In Section $\ref{sec:hirhoadd}$ and Section $\ref{sec:bord}$, we define the higher rho invariant map, and prove that it is well-defined and additive. In Section $\ref{sec:toana}$, we compare the topological surgery long exact sequence to the analytic surgery long exact sequence. In particular, the topological  surgery long exact sequence maps naturally to the analytic surgery long exact sequence, and they fit into a commutative diagram of exact sequences of abelian groups (cf. Diagram $\eqref{diag:surgery}$). In Section $\ref{sec:novrho}$, we introduce the Novikov rho invariant, which is a homological version of the higher rho invariant. We use the Novikov rho invariant to show that the structure group is not finitely generated for a class of manifolds.   In Section $\ref{sec:nonrig}$, we give a lower bound of the free-rank of the algebraically reduced structure groups of a topological manifold, under certain mild conditions. In Section $\ref{sec:lip}$, we outline how to adapt the methods in this paper to handle signature operators arising from Lipschitz structures on topological manifolds. We also show that the higher rho invariant map defined using Lipschitz structures is compatible with the Siebenmann periodicity map.

The first author is partially supported by NSF 1510178. The second author is partially supported by NSF 1500823 and NSF 1800737. The third author is partially supported by NSF 1700021, NSF 1564398 and Simons Fellows Program.

The authors wish to thank the referees for carefully reading the article and making many constructive suggestions which have significantly clarified the discussion throughout the article. 
The authors also would like to thank Paolo Piazza, Thomas Schick and Vito Felice Zenobi for helpful comments.

Part of the paper was written while the second and third authors were staying at the Shanghai Center for Mathematical Sciences.   The second and third authors would like to thank Shanghai Center for Mathematical Sciences for its hospitality during their stays, supported by the international exchange grant NSFC 11420101001.

\section{Preliminaries}\label{sec:prem}

In this section, we briefly recall some standard definitions of geometric $C^\ast$-algebras. We refer the reader to \cite{MR2431253, MR1147350, MR1451759} for more details.

Let $X$ be a proper metric space, i.e., every closed metric ball in $X$ is compact. An $X$-module is a separable Hilbert space equipped with a	$\ast$-representation of $C_0(X)$, the algebra of all continuous functions on $X$ which vanish at infinity. An	$X$-module is called nondegenerate if the $\ast$-representation of $C_0(X)$ is nondegenerate. An $X$-module is said to be standard if no nonzero function in $C_0(X)$ acts as a compact operator. 
\begin{definition}
	Let $H_X$ be a $X$-module and $T$ a bounded linear operator acting on $H_X$. 
	\begin{enumerate}[(i)]
		\item The propagation of $T$ is defined to be the nonnegative real number 
		\[ \sup\{ d(x, y)\mid (x, y)\in \supp(T)\},\] where $\supp(T)$ is  the complement (in $X\times X$) of the set of points $(x, y)\in X\times X$ for which there exist $f, g\in C_0(X)$ such that $gTf= 0$ and $f(x)\neq 0$, $g(y) \neq 0$;
		\item $T$ is said to be locally compact if $fT$ and $Tf$ are compact for all $f\in C_0(X)$; 
		\item $T$ is said to be pseudo-local if $[T, f]$ is compact for all $f\in C_0(X)$.  
	\end{enumerate}
\end{definition}

\begin{definition}\label{def:localg}
	Let $H_X$ be a standard nondegenerate $X$-module and $\mathcal B(H_X)$ the set of all bounded linear operators on $H_X$.  
	\begin{enumerate}[(i)]
		\item The Roe algebra of $X$, denoted by $C^\ast(X)$, is the $C^\ast$-algebra generated by all locally compact operators in $\mathcal B(H_X)$ with finite propagation.
		\item $C_L^\ast(X)$  is the $C^\ast$-algebra generated by all bounded and uniformly norm-continuous functions $f: [0, \infty) \to C^\ast(X)$   such that 
		\[ \textup{propagation of $f(t) \to 0 $, as $t\to \infty$. }\]  
		\item $C_{L, 0}^\ast(X)$ is the kernel of the evaluation map 
		\[  \ev : C_L^\ast(X) \to C^\ast(X),  \quad   \ev (f) = f(0).\]
		In particular, $C_{L, 0}^\ast(X)$ is an  ideal of $C_L^\ast(X)$.
		\item If $Y$ is a subspace of $X$, then the $C^\ast$-algebra $C_L^\ast(Y; X)$ (resp. $C_{L,0}^\ast(Y;X)$)  is defined to be the closed subalgebra of $C_L^\ast(X)$ (resp. $C_{L,0}^\ast(X)$) generated by all elements $f$ such that there exist $c_t>0$ satisfying $ \lim_{t\to \infty} c_t = 0$ and \[ \hspace{1cm}\supp(f(t)) \subset \{ (x, y) \in X\times X \mid d((x,y), Y\times Y) \leq c_t\} \]
		for all $t$.	
	\end{enumerate}
\end{definition}

\begin{remark}
	Similarly, we can also define  maximal versions of $C_L^\ast(X)$,  $C_{L,0}^\ast( X)$,  $C^\ast_{L}( Y;X)$ and  $C^\ast_{L,0}( Y;X)$, cf.  \cite{MR2431253}. We point out that all the above $C^\ast$-algebras are  nonunital. 
\end{remark}

Now in addition we assume that a discrete group $\Gamma$ acts properly on  $X$  by isometries. Let $H_X$ be a $X$-module equipped with a covariant unitary representation of $\Gamma$. If we denote the representation of $C_0(X)$ by $\varphi$ and the representation of $\Gamma$ by $\pi$, this means 
\[  \pi(\gamma) (\varphi(f) v )  =  \varphi(f^\gamma) (\pi(\gamma) v),\] 
where $f\in C_0(X)$, $\gamma\in \Gamma$, $v\in H_X$ and $f^\gamma (x) = f (\gamma^{-1} x)$. In this case, we call $(H_X, \Gamma, \varphi)$ a covariant system.  

\begin{definition}[\cite{MR2732068}]
	A covariant system $(H_X, \Gamma, \varphi)$ is called admissible if 
	\begin{enumerate}[(1)]
		\item the $\Gamma$-action on $X$ is proper and cocompact;
		\item $H_X$ is a nondegenerate standard $X$-module;
		\item for each $x\in X$, the stabilizer group $\Gamma_x$ acts on $H_X$ regularly in the sense that the action is isomorphic to the action of $\Gamma_x$ on $l^2(\Gamma_x)\otimes H$ for some infinite dimensional Hilbert space $H$. Here $\Gamma_x$ acts on $l^2(\Gamma_x)$ by translations and acts on $H$ trivially. 
	\end{enumerate}
\end{definition}
We remark that for each locally compact metric space $X$ with a proper and cocompact isometric action of $\Gamma$, there exists an admissible covariant system $(H_X, \Gamma, \varphi)$. Also, we point out that the condition $(3)$ above is automatically satisfied if $\Gamma$ acts freely on $X$. If no confusion arises, we will denote an admissible covariant system $(H_X, \Gamma, \varphi)$ by $H_X$ and call it an admissible $(X, \Gamma)$-module.

\begin{definition}
	Let $X$ be a locally compact metric space $X$ with a proper and cocompact isometric action of $\Gamma$. If $H_X$ is an admissible $(X, \Gamma)$-module, we denote by $\mathbb C[X]^\Gamma$ the $\ast$-algebra of all $\Gamma$-invariant locally compact operators with finite propagations in $\mathcal B(H_{X})$.  We define $C^\ast(X)^\Gamma$ to be the completion of $\mathbb C[X]^\Gamma$ in $\mathcal B(H_{ X})$.
\end{definition}

Similarly, we can also define   $C_L^\ast( X)^\Gamma, C_{L,0}^\ast( X)^\Gamma,C^\ast_{L}( Y;  X)^\Gamma$ and  $C^\ast_{L,0}( Y;  X)^\Gamma$. 

\begin{remark}
	Up to isomorphism, $C^\ast(X) = C^\ast(X, H_X)$ does not depend on the choice of the standard nondegenerate $X$-module $H_X$. The same holds for  $C_L^\ast(X)$,  $C_{L,0}^\ast( X)$,  $C^\ast_{L}( Y;X)$, $C^\ast_{L,0}( Y;X)$ and their $\Gamma$-equivariant versions. 
\end{remark}

\begin{remark}
	Note that we can also define maximal versions of all $\Gamma$-equivariant $C^\ast$-algebras above. For example, we define the maximal $\Gamma$-invariant Roe algebra $C^\ast_{\max}(X)^\Gamma$ to be the completion of $\mathbb C[X]^\Gamma$ under the maximal norm:
	\[  \|a\|_{\max} = \sup_{\phi} \ \big\{\|\phi(a)\| \mid \phi: \mathbb C[X]^\Gamma \to \mathcal B(H') \ \textup{a $\ast$-representation} \big\}. \]
	Similarly, we can define the maximal versions of  $C_L^\ast( X)^\Gamma, C_{L,0}^\ast( X)^\Gamma,$ $  C^\ast_{L}( Y;  X)^\Gamma$ and $C^\ast_{L,0}( Y;  X)^\Gamma$.  	See for example \cite{MR3590536} for more details.
\end{remark}

\section{Structure groups of topological manifolds}\label{sec:strgrp}

In this section we introduce a definition of structure groups for topological manifolds,  based on ideas of Wall and ideas from controlled topology. We shall prove that our new definition is naturally isomorphic to  the classical  structure group. This new definition is similar in spirit to the algebraic definition given by Ranicki in \cite{MR1211640}, but has some advantages for our analytic purposes. One feature of our definition is that  elements of the structure group can be represented by smooth manifolds with boundary, which is crucial for  construction of our higher rho invariant.  
Another feature is that the group structure of the structure group becomes transparent. Indeed, the group structure is given by  disjoint union.  

Given an oriented closed connected topological manifold $X$, the structure set $\mathcal S^\topo(X)$ is defined to be the set of equivalence classes of orientation-preserving homotopy equivalences $f\colon M \to X$. Here $M$ is an oriented closed connected topological manifold. Two orientation-preserving homotopy equivalences $f\colon M \to X$ and $g\colon N \to X$ are equivalent if there exists an  h-cobordism\footnote{To be precise, $M$ and $N$ are identified with the corresponding boundary component of $W$ by some orientation preserving homeomorphisms, which are part of the data of an $h$-cobordism. Following the usual convention, we shall omit these homeomorphisms from the notation.  } $(W; M, N)$ with an orientation-preserving homotopy equivalence 
\[ F\colon (W; M, N) \to (X\times I; X\times\{0\}, X\times\{1\})\] such that $F|_M = f$ and $F|_N = g$.  
It is known that $\mathcal S^\topo(X)$ has an abelian group structure, cf. \cite{MR873283, MR0645390, MR1211640}. 

More generally, let $X$ be a (not necessarily oriented) closed topological manifold  Let  $w\colon \pi_1(X)\to \mathbb Z/2$ be its orientation character. We can similarly define a group $\mathcal S^\topo(X, w)$ consisting of equivalence classes of orientation character preserving homotopy equivalences $f\colon (M, w_M) \to (X, w)$, where $(M, w_M)$ is a closed topological manifold $M$ with its orientation character $w_M \colon \pi_1 M \to \mathbb Z/2$. Here a continuous map $f\colon (M, w_M) \to (X, w)$ is called orientation character preserving if the map $\pi_1 M \xrightarrow{f_\ast} \pi_1 X \to \mathbb Z/2$ agrees with $w_M$.  The equivalence relation is defined similarly.

In the above definition of $\mathcal S^\topo(X, w)$, if we replace manifolds by ANR homology manifolds everywhere, then we obtain the  homology-manifold-structure group $\mathcal S^{HTOP}(X, w)$ of $(X, w)$,  cf. \cite{MR1183997}.  

\subsection{A new definition of the structure group}

In this subsection, we give a new definition of the structure group of a topological manifold.   Let $X$ be a closed topological manifold. Fix a metric on $X$ that agrees with the topology of $X$. Note that such a metric always exists.

\begin{definition}
	Let $Y$ be a topological space. We call a continuous map  $\varphi\colon Y \to X$ a control map of $Y$. 
\end{definition}

\begin{definition}
	Let $Y$ and $Z$ be two compact Hausdorff spaces equipped with continuous control maps $\psi\colon Y \to X$ and $\varphi\colon  Z\to X$. 	A continuous map $f\colon Y \to Z$ is said to be an infinitesimally controlled homotopy equivalence over $X$, if there exist proper
	continuous maps 
	\[ \Phi\colon Z\times [1, \infty) \to X\times [1, \infty), \] 
	\[\Psi\colon Y\times [1,\infty)  \to X\times [1, \infty),\]   
	\[ F\colon Y\times [1,\infty)  \to Z\times [1,\infty)\]
	and
	\[ G\colon Z\times [1,\infty)   \to Y\times [1,\infty)   \] satisfying the following conditions:
	\begin{enumerate}[(1)]
		\item $\Phi \circ F = \Psi$;
		\item $F|_{Y\times \{1\}} = f$, $\Phi|_{Z\times \{1\}} = \varphi$, and  $\Psi|_{Y\times \{1\} } = \psi$;
		\item there is a proper continuous homotopy $\{H_s\}_{0\leq s \leq 1}$ between 
		\[ H_0 = F\circ G \textup{ and }  H_1 = \id\colon Z\times[1, \infty) \to Z \times [1, \infty) \] such that the diameter of the set 
		\[  \Phi(H(z, t)) = \{\Phi(H_s(z, t))\mid 0\leq s \leq 1 \}\] goes uniformly (i.e. independent of $z\in Z$) to zero, as $t\to \infty$;
		\item there is a proper continuous homotopy $\{R_s\}_{0\leq s \leq 1}$ between 
		\[ R_0 = G\circ F \textup{ and } R_1 = \id\colon Y\times [1, \infty) \to Y \times [1, \infty)\]  such that the diameter of the set 
		\[  \Psi(R(y, t)) = \{\Psi(R_s(y, t))\mid 0\leq s \leq 1 \} \] goes uniformly  to zero, as $t \to \infty$. 
		
	\end{enumerate}
	
\end{definition}

We will also need the following notion of restrictions of homotopy equivalences gaining infinitesimal control on parts of spaces.   Suppose $M$ is a topological manifold with boundary $\partial M$. We define the space of $M$ attached with a cylinder by  
\[ CM = M\cup_{\partial M} (\partial M\times [1, \infty)). \]
Suppose $(M, \partial M, \varphi)$ and $(N, \partial N, \psi)$ are two manifold pairs equipped with continuous maps $\varphi\colon M \to X$ and $\psi\colon N \to X$. Let  \[ f\colon (N, \partial N)  \to (M, \partial M)\] be a homotopy equivalence with $\varphi \circ f = \psi$.  Suppose 
\[ g\colon (M, \partial M) \to (N, \partial N)  \] is a homotopy inverse of $f$. Note that $\psi\circ g \neq  \varphi$ in general. Let  $\{h_s\}_{0\leq s\leq 1}$ be a homotopy between $f\circ g$ and 
\[ \id\colon (M, \partial M) \to (M, \partial M).\] Similarly,  let $\{r_s\}_{0\leq s\leq 1}$ be a homotopy  between $g\circ f$ and \[ \id\colon (N, \partial N) \to (N, \partial N).\]

\begin{definition}\label{def:infcon}
	With the above notation, we say that  on the boundary $f$ restricts to an \emph{infinitesimally controlled} homotopy equivalence $f|_{\partial N} \colon \partial N\to \partial M$ over $X$, if there exist proper continuous maps 
	\[ \Phi\colon CM \to X\times [1, \infty) \textup{ and } \Psi\colon CN \to X\times [1, \infty),\]   
	\[ F\colon CN  \to CM \textup{ and } G\colon CM \to CN,  \]
	a proper continuous homotopy $\{H_s\}_{0\leq s \leq 1}$ between 
	\[ H_0 = F\circ G \textup{ and } H_1 = \id\colon CM\to CM\] and 
	a proper continuous homotopy $\{R_s\}_{0\leq s \leq 1}$ between 
	\[  R_0 = G\circ F \textup{ and } R_1 = \id\colon CN\to CN \]
	satisfying the following conditions:
	\begin{enumerate}[(1)]
		\item $\Phi|_{M} = \varphi$, $\Psi|_{N} = \psi$, $F|_N = f$, $G|_{M} = g$, $H_s|_{M} = h_s$, and $R_s|_{N} = r_s$;
		\item $\Phi \circ F = \Psi$;
		\item the diameter of the set 
		\[  \Phi(H(a, t)) = \{\Phi(H_s(a, t))\mid 0\leq s \leq 1 \}\] goes uniformly to zero, for all $(a, t) \in \partial M\times [1, \infty)$,  as $t\to \infty$;
		\item the diameter of the set 
		\[  \Psi(R(b, t)) = \{\Psi(R_s(b, t))\mid 0\leq s \leq 1 \}\] goes uniformly  to zero, for all $(b, t)\in \partial N\times [1, \infty)$, as $t\to \infty$. 
		
	\end{enumerate}
	
\end{definition}

In the following, we adopt the notion of manifold $k$-ads from Wall's book \cite[Chapter 0]{MR1687388} to encode a total space with some number of distinguished subsets all of whose intersections are required to be ``good". For example, a manifold $1$-ad is a manifold with boundary. 

In order to make our definition of structure groups functorial, we shall follow Farrell and Hsiang's modifications of Wall's definition of $L$-groups \cite[Section 3]{MR0448372}. Although Farrell and Hsiang's modifications are for $L$-groups, the same idea also applies to structure groups.   Let $w$ be a $\mathbb Z/2$-principal bundle over $X$, and if no confusion is likely to rise, denote the corresponding morphism induced on $\pi_1(X)$ (after choosing a base point in $X$) also by $w\colon \pi_1(X)\to \mathbb Z/2$. Moreover, we denote the local $\mathbb Z$-coefficients on $X$ associated to $w$ by $w\mathbb Z$. 

We define
$\mathcal S_n(X, w)$ to be the set of equivalence classes of the following objects.

\begin{definition}\label{def:newstruc}
	An object of  $\mathcal S_n(X, w)$ consists of the following data:
	\begin{figure}[h]
		\centering
		\begin{tikzpicture}[scale=1, every node/.style={transform shape}]
		\coordinate (ref) at (0,0); 
		\draw [ultra thick] ($(ref) + (0, 2.5)$) to [out = 0, in = 180] ($(ref) + (2, 2)$)  arc (-90:90:2 and 1) ($(ref) + (2, 4)$) to [out = 180, in = 0] ($(ref) + (0, 3.5)$);
		\draw[red, very thick]  ($(ref) + (0, 3.5)$) arc (90:270:0.25 and 0.5);
		\draw[red, very thick]  ($(ref) + (0, 2.5)$) arc (270:450:0.25 and 0.5);
		\draw [ ultra thick] ($(ref) + (1.5, 3)$) to [out = 20, in = 160] ($(ref) + (2.7, 3)$) to [out = -160, in = -20] ($(ref) + (1.5, 3)$);
		\draw [ ultra thick] ($(ref) + (1.5, 3)$) to [out = 160, in = 175] ($(ref) + (1.43, 3.05)$);
		\draw [ ultra thick] ($(ref) + (2.7, 3)$) to [out = 20, in = 5] ($(ref) + (2.77, 3.05)$);

		\draw [->, red, very thick] (-0.4, 7.9) to [out = 225, in = 135 ] (-0.4, 3.1); 
		
		\draw [->, thick] (4.2, 7.5) -- (7, 5.7); 
		
		\node at (7.3, 5.5) {$X$}; 
		\node at (4.4, 8) {$N$}; 
		\node at (4.4, 3) {$M$}; 
		\node at (-1, 8.5) {$\partial N$}; 
		
		\draw [dashed, ->] (-1, 8.3) to [out = -90, in = 180] (-0.4, 8.1);

		\node at (-1, 2.4) {$\partial M$}; 
		
		\draw [dashed, ->] (-1, 2.6) to [out = 90, in = 180] (-0.4, 2.9);
		
		\node at (2.2, 5.5) {$f$}; 
		\node at (5.5, 4) {$\varphi$}; 
		\node at (5.5, 7) {$\psi$}; 
		\node [red] at (-2, 5.5) {$f|_{\partial N}$};
		
		\draw [->, thick] (4.2, 3.5) -- (7, 5.3); 
		
		\draw [->, ultra thick] (2, 6.7) -- (2, 4.3);
		\begin{scope}[shift ={(0, 5) }]
		\coordinate (ref) at (0,0); 
		\draw [ultra thick] ($(ref) + (0, 2.5)$) to [out = 0, in = 180] ($(ref) + (2, 2)$)  arc (-90:90:2 and 1) ($(ref) + (2, 4)$) to [out = 180, in = 0]  ($(ref) + (0, 3.5)$);
		\draw[red, very thick]  ($(ref) + (0, 3.5)$) arc (90:270:0.25 and 0.5);
		\draw[red, very thick]  ($(ref) + (0, 2.5)$) arc (270:450:0.25 and 0.5);
		\draw [ ultra thick] ($(ref) + (1.5, 3)$) to [out = 20, in = 160] ($(ref) + (2.7, 3)$) to [out = -160, in = -20] ($(ref) + (1.5, 3)$);
		\draw [ ultra thick] ($(ref) + (1.5, 3)$) to [out = 160, in = 175] ($(ref) + (1.43, 3.05)$);
		\draw [ ultra thick] ($(ref) + (2.7, 3)$) to [out = 20, in = 5] ($(ref) + (2.77, 3.05)$);
		\end{scope}
		\end{tikzpicture}
		\caption{An object  $\theta = (M, \partial M, \varphi, N, \partial N, \psi, f)$ of  $\mathcal  S_n(X, w)$, where $f|_{\partial N}$ is an infinitesimally controlled homotopy equivalence and $f$ is a homotopy equivalence.}
	\end{figure}
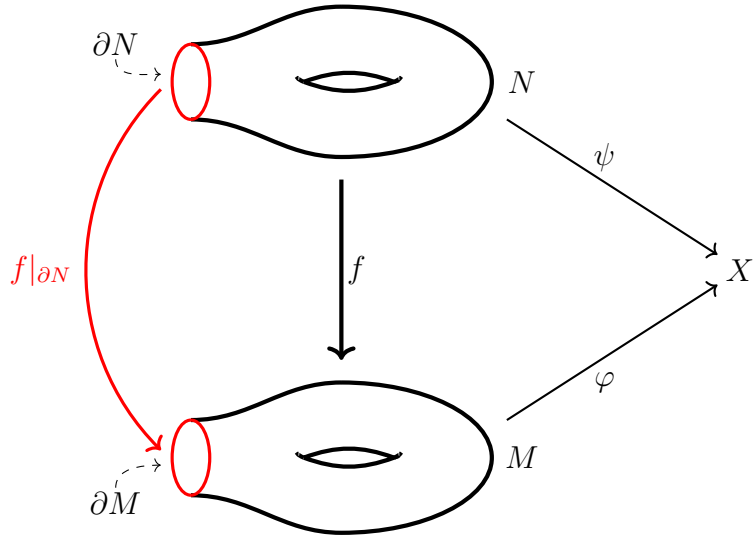
	
	\begin{enumerate}[(1)]
		\item two manifold $1$-ads $(M, \partial M)$ and  $(N, \partial N)$ with $\dim M = \dim N = n$, where $\partial M$ (resp. $\partial N$) is the boundary of $M$ (resp. $ N$);  
		\item continuous maps $\varphi\colon M \to X$ and $\psi\colon N \to X$ so that the pullback $\mathbb Z/2$-principal bundle  $\varphi^\ast(w)$ (resp. $\psi^\ast(w)$) is the orientation covering of $M$ (resp. $N$); 
		\item a homotopy equivalence of manifold $1$-ads \[ f\colon (N, \partial N) \to (M, \partial M)\] such that $f_\ast([M]) = [N]$ and  $\varphi \circ f = \psi$, where $[M]$ (resp. $[N]$)  is the fundamental class of $M$  (resp. $N$), an element in $H_n(M, \partial M; \varphi^\ast(w)\mathbb Z)$ (resp. $H_n(N, \partial N; \psi^\ast(w)\mathbb Z)$). Moreover, on the boundary $f$ restricts to an infinitesimally controlled homotopy equivalence $f|_{\partial N} \colon \partial N \to \partial M$  over $X$. 
	\end{enumerate} 
\end{definition}

\begin{remark}
	We shall prove below that, if $X$ is a closed connected topological manifold of dimension $\geq 5$ and $w$ is the orientation covering of $X$, then $\mathcal S_n(X,  w) $ is naturally isomorphic to  the  structure group $\mathcal S^{\topo}(X, w)$ of $(X, w)$, where the latter group is described at the beginning of this section with $w$ being the associated orientation character of $X$.   
\end{remark}


If $\theta = (M, \partial M, \varphi, N, \partial N, \psi, f)$ is an object, then we denote by $-\theta$ to be the same object except that the fundamental classes of $M$ and $N$ switch sign. For two objects $\theta_1$ and $\theta_2$, we write $\theta_1 +\theta_2$ to be the disjoint union of $\theta_1$ and $\theta_2$. This sum operation is clearly commutative and associative,  and admits a zero element: the object with $M$ (hence $N$) empty. We denote the zero element by $0$.

\begin{definition}\label{def:strequiv}	
	The equivalence relation for defining $\mathcal  S_n(X, w)$ is given as follows.  Let \[ \theta = (M, \partial M, \varphi, N, \partial N, \psi, f) \] be an object from Definition $\ref{def:newstruc}$. We write $\theta \sim 0$ if the following conditions are satisfied. 
	\begin{figure}[H]
		\centering
		\begin{tikzpicture}[scale=1, every node/.style={transform shape}]
		\coordinate (ref) at (0,0); 
		\shade[shading = ball, ball color = gray!30!white, opacity=1,  draw = black] 
		($(ref) + (-3, 4)$)  arc (90:270:1) ($(ref) + (-3, 2)$)  to [out = 0, in = 180] ($(ref) + (-2, 2.5)$) to  [out = 0, in = 180]   ($(ref) + (-1, 2)$) to [out = 0, in = 180]  ($(ref) + (0, 2.5)$) to [out = 0, in = 180] ($(ref) + (2, 2)$)  arc (-90:90:2 and 1) ($(ref) + (2, 4)$) to [out = 180, in = 0] ($(ref) + (0, 3.5)$) to [out = 180, in = 0]  ($(ref) + (-1, 4)$) to [out = 180, in = 0 ]  ($(ref) + (-2, 3.5)$) to [out = 180, in = 0]  ($(ref) + (-3, 4)$)  arc (90:270:1) ($(ref) + (-3, 2)$);

		%
		\draw [red, very thick] ($(ref) + (0, 3.5)$) to [out = 180, in = 0]  ($(ref) + (-1, 4)$) to [out = 180, in = 0 ]  ($(ref) + (-2, 3.5)$) to [out = 180, in = 0]  ($(ref) + (-3, 4)$)  arc (90:270:1) ($(ref) + (-3, 2)$)  to [out = 0, in = 180] ($(ref) + (-2, 2.5)$) to  [out = 0, in = 180]   ($(ref) + (-1, 2)$) to [out = 0, in = 180]  ($(ref) + (0, 2.5)$);
		
		\draw [ultra thick] ($(ref) + (0, 2.5)$) to [out = 0, in = 180] ($(ref) + (2, 2)$)  arc (-90:90:2 and 1) ($(ref) + (2, 4)$) to [out = 180, in = 0] ($(ref) + (0, 3.5)$);
		\draw[red, very thick]  ($(ref) + (0, 3.5)$) arc (90:270:0.25 and 0.5);
		\draw[dashed, red, very thick]  ($(ref) + (0, 2.5)$) arc (270:450:0.25 and 0.5);
		
		\draw [fill =white,  ultra thick] ($(ref) + (1.5, 3)$) to [out = 20, in = 160] ($(ref) + (2.7, 3)$) to [out = -160, in = -20] ($(ref) + (1.5, 3)$);
		\draw [ ultra thick] ($(ref) + (1.5, 3)$) to [out = 160, in = -10] ($(ref) + (1.4, 3.03)$);
		\draw [ ultra thick] ($(ref) + (2.7, 3)$) to [out = 20, in = -170] ($(ref) + (2.8, 3.03)$);
		
		\draw [red, fill =white,  very thick] ($(ref) + (-3.4, 3)$) to [out = 60, in = 120] ($(ref) + (-2.6, 3)$) to [out = -120, in = -60] ($(ref) + (-3.4, 3)$);	
		\draw [ red, very thick] ($(ref) + (-3.4, 3)$) to [out = 120, in = -75] ($(ref) + (-3.46, 3.15)$);
		\draw [ red, very thick] ($(ref) + (-2.6, 3)$) to [out = 60, in = -115] ($(ref) + (-2.54, 3.15)$);
		
		\begin{scope}[shift= {(2, 0)}]
		\coordinate (ref) at (0,0);
		\draw [red, fill =white,  very thick] ($(ref) + (-3.4, 3)$) to [out = 60, in = 120] ($(ref) + (-2.6, 3)$) to [out = -120, in = -60] ($(ref) + (-3.4, 3)$);	
		\draw [red, very thick] ($(ref) + (-3.4, 3)$) to [out = 120, in = -75] ($(ref) + (-3.46, 3.15)$);
		\draw [ red, very thick] ($(ref) + (-2.6, 3)$) to [out = 60, in = -115] ($(ref) + (-2.54, 3.15)$);
		
		\end{scope}

		
		\draw [->, red, very thick] (-2, 6.8) to [out = 225, in = 135 ] (-2, 3.8); 
		
		\draw [->, thick] (4.2, 7.5) -- (7, 5.7); 
		
		\node at (7.3, 5.5) {$X$}; 
		\node at (2.5, 6.6) {$N$}; 
		\node at (2.5, 1.6) {$M$}; 
		
		
		\node at (-2.7, 6.6) {$\partial_2 V$}; 
		
		\node at (-2.7, 1.6) {$\partial_2 W$}; 
		
		\node at (0.75, 9.5) {\scriptsize $\partial N = \partial \partial_2 V$};

		\draw [dashed, ->] (0.55, 9.3) to [out = -90, in = 90] (0, 8.6); 
		
		\node at (0.75, 1.5) {\scriptsize $\partial M = \partial \partial_2 W$}; 
		
		\draw [dashed, ->] (0.55, 1.7) to [out = 90, in = -90] (0, 2.4);

		\node at (5.5, 4) {$\Phi$}; 
		\node at (5.5, 7) {$\Psi$}; 
		\node [red] at (-2, 5.5) {$F|_{\partial_2 V}$};
		
		\draw [->, thick] (4.2, 3.5) -- (7, 5.3);

		
		\begin{scope}[shift ={(0, 5) }]
		\coordinate (ref) at (0,0);
		\shade[shading = ball, ball color = gray!30!white,  draw = black] 
		($(ref) + (-3, 4)$)  arc (90:270:1) ($(ref) + (-3, 2)$)  to [out = 0, in = 180] ($(ref) + (-2, 2.5)$) to  [out = 0, in = 180]   ($(ref) + (-1, 2)$) to [out = 0, in = 180]  ($(ref) + (0, 2.5)$) to [out = 0, in = 180] ($(ref) + (2, 2)$)  arc (-90:90:2 and 1) ($(ref) + (2, 4)$) to [out = 180, in = 0] ($(ref) + (0, 3.5)$) to [out = 180, in = 0]  ($(ref) + (-1, 4)$) to [out = 180, in = 0 ]  ($(ref) + (-2, 3.5)$) to [out = 180, in = 0]  ($(ref) + (-3, 4)$)  arc (90:270:1) ($(ref) + (-3, 2)$);	
		
		\draw [red, very thick] ($(ref) + (0, 3.5)$) to [out = 180, in = 0]  ($(ref) + (-1, 4)$) to [out = 180, in = 0 ]  ($(ref) + (-2, 3.5)$) to [out = 180, in = 0]  ($(ref) + (-3, 4)$)  arc (90:270:1) ($(ref) + (-3, 2)$)  to [out = 0, in = 180] ($(ref) + (-2, 2.5)$) to  [out = 0, in = 180]   ($(ref) + (-1, 2)$) to [out = 0, in = 180]  ($(ref) + (0, 2.5)$);

		\draw [ultra thick] ($(ref) + (0, 2.5)$) to [out = 0, in = 180] ($(ref) + (2, 2)$)  arc (-90:90:2 and 1) ($(ref) + (2, 4)$) to [out = 180, in = 0] ($(ref) + (0, 3.5)$);
		\draw[red, very thick]  ($(ref) + (0, 3.5)$) arc (90:270:0.25 and 0.5);
		\draw[dashed, red, very thick]  ($(ref) + (0, 2.5)$) arc (270:450:0.25 and 0.5);

		\draw [fill =white,  ultra thick] ($(ref) + (2.5, 3.4)$) to [out = 205, in = 155] ($(ref) + (2.5, 2.6)$) to [out = 20, in = -20] ($(ref) + (2.5, 3.4)$);
		\draw [ultra thick] ($(ref) + (2.5, 3.4)$) to [out = 25, in = -165] ($(ref) + (2.65, 3.45)$);
		\draw [ultra thick] ($(ref) + (2.5, 2.6)$) to [out = -25, in = 165] ($(ref) + (2.65, 2.55)$);
		
		\draw [red, fill =white,  very thick] ($(ref) + (-3.4, 3)$) to [out = 60, in = 120] ($(ref) + (-2.6, 3)$) to [out = -120, in = -60] ($(ref) + (-3.4, 3)$);	
		\draw [ red, very thick] ($(ref) + (-3.4, 3)$) to [out = 120, in = -75] ($(ref) + (-3.46, 3.15)$);
		\draw [ red, very thick] ($(ref) + (-2.6, 3)$) to [out = 60, in = -115] ($(ref) + (-2.54, 3.15)$);
		
		\begin{scope}[shift= {(2, 0)}]
		\coordinate (ref) at (0,0);
		\draw [red, fill =white,  very thick] ($(ref) + (-3.4, 3)$) to [out = 60, in = 120] ($(ref) + (-2.6, 3)$) to [out = -120, in = -60] ($(ref) + (-3.4, 3)$);	
		\draw [red, very thick] ($(ref) + (-3.4, 3)$) to [out = 120, in = -75] ($(ref) + (-3.46, 3.15)$);
		\draw [ red, very thick] ($(ref) + (-2.6, 3)$) to [out = 60, in = -115] ($(ref) + (-2.54, 3.15)$);
		
		\end{scope}
		
		\end{scope}
		
		
		\node at (0.6, 8) {$V$}; 
		\node at (0.6, 3) {$W$}; 
		
		\node at (0.75, 5.5) {$F$}; 
		\draw [->, thick] (0.5, 6.8) -- (0.5, 4.2);

		\end{tikzpicture}
		\vspace{3mm}
		\caption{Equivalence relation $\theta\sim 0$ in the definition of  $\mathcal S_n(X, w)$, where $F|_{\partial_2 V}$ is an infinitesimally controlled homotopy equivalence and $F|_{N} = f$ is a homotopy equivalence. In this picture,  $V$ (resp. $W$) should be viewed as a solid  with boundary $\partial V$ (resp. $\partial W$). }
	\end{figure}
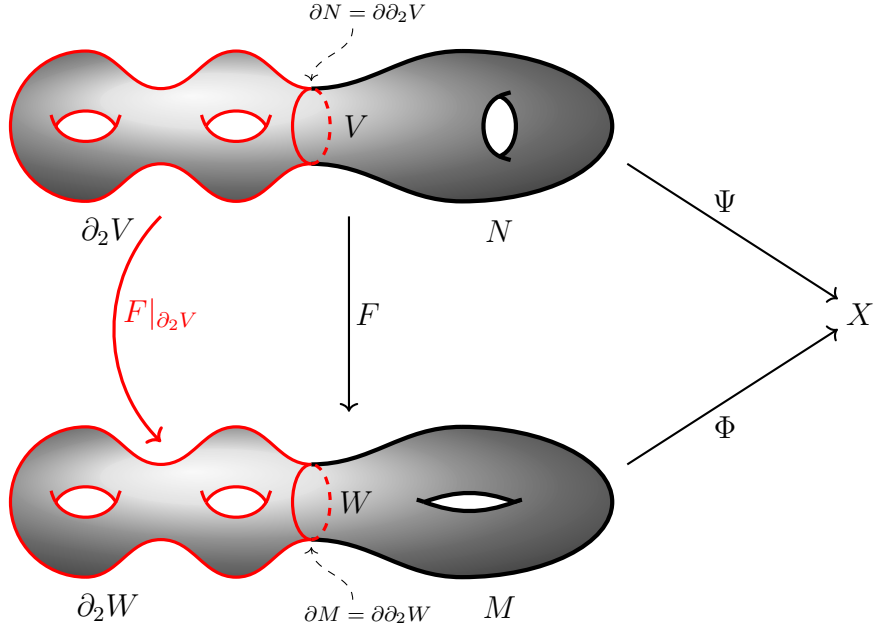
	
	\begin{enumerate}[(1)]
		\item There exists a manifold $2$-ad $(W, \partial W)$ of dimension $(n+1)$ with a continuous map $\Phi\colon W\to X$ so that the pullback $\mathbb Z/2$-principal bundle $\Phi^\ast(w)$ is the orientation covering of $W$. Here $\partial W = M \cup_{\partial M} \partial_2 W $, and in particular $\partial M = \partial (\partial_2 W)$. In other words, $W$ is a manifold with corners, and its boundary is the union of  $M$ and $\partial_2 W$ (two manifolds with boundary) which are glued together along their common boundary $\partial M = \partial (\partial_2 W)$. We also assume $\Phi|_M = \varphi$.
		\item Similarly, we have a manifold $2$-ad $(V, \partial V)$ of dimension $(n+1)$ with a continuous map $\Psi: V\to X$ so that  the pullback $\mathbb Z/2$-principal bundle $\Psi^\ast(w)$ is the orientation covering of $V$. Moreover,    $\partial V = N \cup_{\partial N} \partial_2 V $ and $\Psi|_N = \psi$.  
		\item There is a homotopy equivalence of $2$-ads \[ F\colon (V, \partial V) \to (W, \partial W)\] such that $F_\ast([V]) = [W]$ and $\Psi\circ F  = \Phi$, where  $[V]$ (resp. $[W]$)  is the fundamental class of $V$ (resp. $W$), an element of  $H_n(V, \partial V; \Psi^\ast(w)\mathbb Z)$ (resp. $H_n(W, \partial W; \Phi^\ast(w)\mathbb Z)$).  Moreover, $F$ restricts to $f$ on $N$, and  $F$ restricts\footnote{Here we are using an obvious generalization of Definition $\ref{def:infcon}$ to the case of manifold $2$-ads or manifold $n$-ads.} to  an infinitesimally controlled homotopy equivalence \[ F|_{\partial_2 V} \colon \partial_2V \to \partial_2 W\] over $X$.  
	\end{enumerate}
\end{definition}

We further write $\theta_1\sim \theta_2$ if $\theta_1 +(-\theta_2) \sim 0$. It is not difficult to check that $\sim$ is an equivalence relation. 	We denote the set of equivalence classes by  $\mathcal  S_n(X, w)$. Note that $\mathcal  S_n(X, w)$ is an abelian group with the addition given by disjoint union. 

\begin{remark}
	Let $w$ and $v$ be $\mathbb Z/2$-principal bundles over $X$ and $Y$ respectively. Then each bundle map $\Delta\colon w \to v$  induces a group homomorphism $\Delta_\ast\colon \mathcal  S_n(X, w) \to \mathcal S_n(Y, v)$ by essentially composing\footnote{For example, for the element $\theta = (M, \partial M, \varphi, N, \partial N, \psi, f)\in \mathcal S(X, w)$, $\Delta$ induces a canonical local coefficient isomorphism $\varphi^\ast(w) \to (\Delta\varphi)^\ast(v)$, which in turn induces a map from $H_n(M, \partial M; \varphi^\ast(w)\mathbb Z)$ to $H_n\big(M, \partial M; (\Delta\varphi)^\ast(v)\mathbb Z\big)$. Then the corresponding fundamental class for the manifold pair $(M, \partial M)$  as part of the data of $\Delta(\theta) = (M, \partial M, \Delta\circ \varphi, N, \partial N, \Delta\circ \varphi, f) $ is  the image of $[M]$ under this map $H_n(M, \partial M; \varphi^\ast(w)\mathbb Z) \to H_n\big(M, \partial M; (\Delta\varphi)^\ast(v)\mathbb Z\big)$. See also \cite[Section 3, page 103]{MR0448372}. } the data of an element $\theta\in \mathcal S_n(X, w)$ with the map  $\Delta$. This makes the definition of $\mathcal  S_n(X, w)$ functorial.  
\end{remark}

\subsection{Surgery long exact sequences}
In this subsection, we give a description of  surgery long exact sequences based on ideas of Wall. This will be used later to naturally identify $\mathcal S_n(X, w)$ with the  structure group $\mathcal S^\topo(X, w)$   . 

First, let us review the definition of normal maps. Let both $(M, \partial N)$ and $(N, \partial N)$ be $n$-dimensional manifolds with boundary. Let $w_M$ (resp. $w_N$) be the orientation covering of $M$ (resp. $N$). Equivalently, we think of $w_M$ and $w_N$ as elements in $H^1(M; \mathbb Z/2)$ and $H^1(N; \mathbb Z/2)$ respectively. 
\begin{definition}\label{def:normal}
	Let $\nu$ be a $k$-dimensional vector bundle over $N$ such that the first Stiefel-Whitney class of $\nu$ agrees with $w_N\in H^1(N; \mathbb Z/2)$. A map 
	\[ f\colon (M, \partial M) \to (N, \partial N)\] is called a normal map if the following conditions are satisfied: 
	\begin{enumerate}[(1)]
		\item $f$ preserves orientation characters, that is, $w_M = f^\ast (w_N)$;
		\item there exists an embedding $M \hookrightarrow \mathbb R^{n+k}$ with its normal bundle denoted by $\nu_M$ such that there is a bundle map 	
		\[ \bar f\colon \nu_M\to \nu\] covering $f$ and is an isomorphism on each fiber.

	\end{enumerate}
	
	If in addition $f$ maps the fundamental class $[M]$, an element in $ H_n(M, \partial M; w_M\mathbb Z)$, to the fundamental class  $ [N]$, an element in $H_n(N, \partial N; w_N\mathbb Z)$, then we say $f$ has degree one. 
\end{definition}

Now we review the following geometric definition of $L$-groups due to Wall \cite[Chapter 9]{MR1687388}.

\begin{definition}\label{def:Lgrp}
	An object of ${L}_n(\pi_1 X, w)$  consists of the following data: 
	
	\begin{enumerate}[(1)]
		\item two manifold $1$-ads $(M, \partial M)$ and  $(N, \partial N)$ with $\dim M = \dim N = n$, where $\partial M$ (resp. $\partial N$) is the boundary of $M$ (resp. $\partial N$);
		\item continuous maps $\varphi\colon M \to X$ and $\psi\colon N \to X$  so that the pullback $\mathbb Z/2$-principal bundles $\varphi^\ast(w)$ (resp. $\psi^\ast(w)$) is the orientation covering of $M$ (resp. $N$); 
		\item a degree one normal map of the $1$-ads 
		\[ f\colon (N, \partial N) \to (M, \partial M)\] such that $\varphi \circ f = \psi$. Moreover, on the boundary 
		$f|_{\partial N} \colon \partial N\to \partial M$ is a homotopy equivalence.
	\end{enumerate} 
	\begin{figure}[H]
		\centering
		\begin{tikzpicture}[scale=1, every node/.style={transform shape}]
		\coordinate (ref) at (0,0); 
		\draw [blue,  thick] ($(ref) + (0, 2.5)$) to [out = 0, in = 180] ($(ref) + (2, 2)$)  arc (-90:90:2 and 1) ($(ref) + (2, 4)$) to [out = 180, in = 0] ($(ref) + (0, 3.5)$);
		\draw[ultra thick]  ($(ref) + (0, 3.5)$) arc (90:270:0.25 and 0.5);
		\draw[ultra thick]  ($(ref) + (0, 2.5)$) arc (270:450:0.25 and 0.5);
		\draw [ blue,  thick] ($(ref) + (1.5, 3)$) to [out = 20, in = 160] ($(ref) + (2.7, 3)$) to [out = -160, in = -20] ($(ref) + (1.5, 3)$);
		\draw [ blue, thick] ($(ref) + (1.5, 3)$) to [out = 160, in = 175] ($(ref) + (1.43, 3.05)$);
		\draw [ blue, thick] ($(ref) + (2.7, 3)$) to [out = 20, in = 5] ($(ref) + (2.77, 3.05)$);

		\draw [->, ultra thick] (-0.4, 7.9) to [out = 225, in = 135 ] (-0.4, 3.1); 
		
		\draw [->, thick] (4.2, 7.5) -- (7, 5.7); 
		
		\node at (7.3, 5.5) {$X$}; 
		\node at (4.4, 8) {$N$}; 
		\node at (4.4, 3) {$M$}; 
		\node at (-1, 8.5) {$\partial N$}; 
		
		\draw [dashed, ->] (-1, 8.3) to [out = -90, in = 180] (-0.4, 8.1);

		\node at (-1, 2.4) {$\partial M$}; 
		
		\draw [dashed, ->] (-1, 2.6) to [out = 90, in = 180] (-0.4, 2.9);
		
		\node [blue] at (2.2, 5.5) {$f$}; 
		\node at (5.5, 4) {$\varphi$}; 
		\node at (5.5, 7) {$\psi$}; 
		\node  at (-2, 5.5) {$f|_{\partial N}$};
		
		\draw [->, thick] (4.2, 3.5) -- (7, 5.3); 
		
		\draw [->, blue,  thick] (2, 6.7) -- (2, 4.3);
		
		\begin{scope}[shift ={(0, 5) }]
		\coordinate (ref) at (0,0); 
		\draw [blue,  thick] ($(ref) + (0, 2.5)$) to [out = 0, in = 180] ($(ref) + (2, 2)$)  arc (-90:90:2 and 1) ($(ref) + (2, 4)$) to [out = 180, in = 0] ($(ref) + (0, 3.5)$);
		\draw[ultra thick]  ($(ref) + (0, 3.5)$) arc (90:270:0.25 and 0.5);
		\draw[ultra thick]  ($(ref) + (0, 2.5)$) arc (270:450:0.25 and 0.5);
		
		\draw [blue, fill =white, thick] ($(ref) + (2.5, 3.4)$) to [out = 205, in = 155] ($(ref) + (2.5, 2.6)$) to [out = 20, in = -20] ($(ref) + (2.5, 3.4)$);
		\draw [blue, thick] ($(ref) + (2.5, 3.4)$) to [out = 25, in = -165] ($(ref) + (2.65, 3.45)$);
		\draw [blue, thick] ($(ref) + (2.5, 2.6)$) to [out = -25, in = 165] ($(ref) + (2.65, 2.55)$);
		
		\begin{scope}[shift = {(-1, 0)}]
		\coordinate (ref) at (0,0);
		\draw [blue, fill =white, thick] ($(ref) + (2.5, 3.4)$) to [out = 205, in = 155] ($(ref) + (2.5, 2.6)$) to [out = 20, in = -20] ($(ref) + (2.5, 3.4)$);
		\draw [blue, thick] ($(ref) + (2.5, 3.4)$) to [out = 25, in = -165] ($(ref) + (2.65, 3.45)$);
		\draw [blue, thick] ($(ref) + (2.5, 2.6)$) to [out = -25, in = 165] ($(ref) + (2.65, 2.55)$);
		
		\end{scope}
		\end{scope}
		\end{tikzpicture}
		\vspace{3mm}
		\caption{An object  $\theta = (M, \partial M, \varphi, N, \partial N, \psi, f)$ in  $ L_n(\pi_1 X,  w)$, where $f$ is a degree one normal map and $f|_{\partial N}$ is a homotopy equivalence.}
	\end{figure}
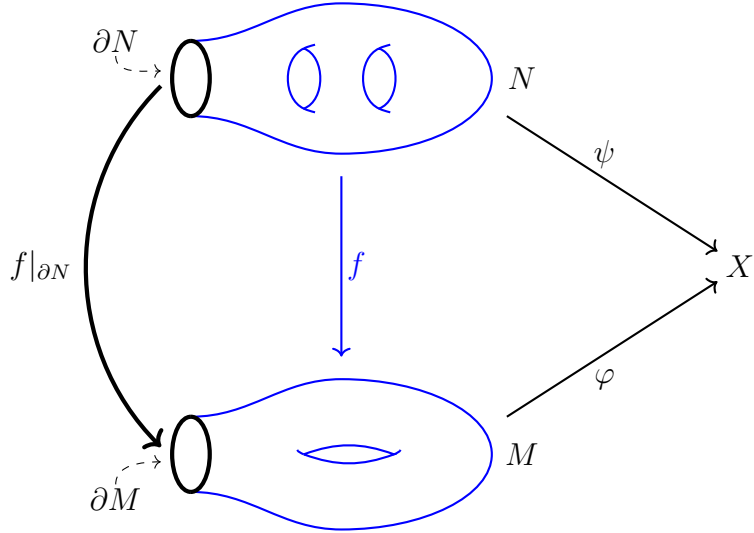
\end{definition}

\begin{definition}\label{def:equivforL} 
	The equivalence relation for defining $L_n(\pi_1 X, w)$ is given as follows. Let \[ \theta = (M, \partial M, \varphi, N, \partial N, \psi, f) \] be an object from Definition $\ref{def:Lgrp}$ above. We write $\theta \sim 0$ if the following conditions are satisfied. 
	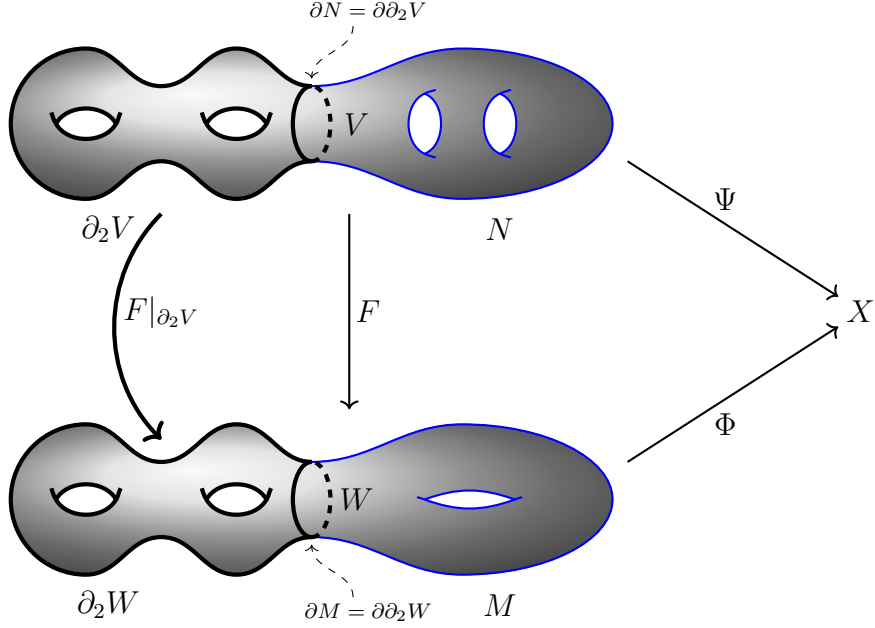
\begin{figure}[h]
		\centering
		\begin{tikzpicture}[scale=1, every node/.style={transform shape}]
		\coordinate (ref) at (0,0); 
		
		\shade[shading = ball, ball color = gray!30!white,  draw = black] 
		($(ref) + (-3, 4)$)  arc (90:270:1) ($(ref) + (-3, 2)$)  to [out = 0, in = 180] ($(ref) + (-2, 2.5)$) to  [out = 0, in = 180]   ($(ref) + (-1, 2)$) to [out = 0, in = 180]  ($(ref) + (0, 2.5)$) to [out = 0, in = 180] ($(ref) + (2, 2)$)  arc (-90:90:2 and 1) ($(ref) + (2, 4)$) to [out = 180, in = 0] ($(ref) + (0, 3.5)$) to [out = 180, in = 0]  ($(ref) + (-1, 4)$) to [out = 180, in = 0 ]  ($(ref) + (-2, 3.5)$) to [out = 180, in = 0]  ($(ref) + (-3, 4)$)  arc (90:270:1) ($(ref) + (-3, 2)$);	
		
		\draw [ultra thick] ($(ref) + (0, 3.5)$) to [out = 180, in = 0]  ($(ref) + (-1, 4)$) to [out = 180, in = 0 ]  ($(ref) + (-2, 3.5)$) to [out = 180, in = 0]  ($(ref) + (-3, 4)$)  arc (90:270:1) ($(ref) + (-3, 2)$)  to [out = 0, in = 180] ($(ref) + (-2, 2.5)$) to  [out = 0, in = 180]   ($(ref) + (-1, 2)$) to [out = 0, in = 180]  ($(ref) + (0, 2.5)$);
		
		\draw [blue, thick] ($(ref) + (0, 2.5)$) to [out = 0, in = 180] ($(ref) + (2, 2)$)  arc (-90:90:2 and 1) ($(ref) + (2, 4)$) to [out = 180, in = 0] ($(ref) + (0, 3.5)$);
		\draw[ultra thick]  ($(ref) + (0, 3.5)$) arc (90:270:0.25 and 0.5);
		\draw[dashed, ultra thick]  ($(ref) + (0, 2.5)$) arc (270:450:0.25 and 0.5);
		
		\draw [blue, fill =white,   thick] ($(ref) + (1.5, 3)$) to [out = 20, in = 160] ($(ref) + (2.7, 3)$) to [out = -160, in = -20] ($(ref) + (1.5, 3)$);
		\draw [ blue, thick] ($(ref) + (1.5, 3)$) to [out = 160, in = -10] ($(ref) + (1.4, 3.03)$);
		\draw [ blue,  thick] ($(ref) + (2.7, 3)$) to [out = 20, in = -170] ($(ref) + (2.8, 3.03)$);
		
		\draw [fill =white,  ultra thick] ($(ref) + (-3.4, 3)$) to [out = 60, in = 120] ($(ref) + (-2.6, 3)$) to [out = -120, in = -60] ($(ref) + (-3.4, 3)$);	
		\draw [ ultra thick] ($(ref) + (-3.4, 3)$) to [out = 120, in = -75] ($(ref) + (-3.46, 3.15)$);
		\draw [ ultra thick] ($(ref) + (-2.6, 3)$) to [out = 60, in = -115] ($(ref) + (-2.54, 3.15)$);
		
		\begin{scope}[shift= {(2, 0)}]
		\coordinate (ref) at (0,0);
		\draw [fill =white,  ultra thick] ($(ref) + (-3.4, 3)$) to [out = 60, in = 120] ($(ref) + (-2.6, 3)$) to [out = -120, in = -60] ($(ref) + (-3.4, 3)$);	
		\draw [ ultra thick] ($(ref) + (-3.4, 3)$) to [out = 120, in = -75] ($(ref) + (-3.46, 3.15)$);
		\draw [ ultra thick] ($(ref) + (-2.6, 3)$) to [out = 60, in = -115] ($(ref) + (-2.54, 3.15)$);
		
		\end{scope}

		
		\draw [->, ultra thick] (-2, 6.8) to [out = 225, in = 135 ] (-2, 3.8); 
		
		\draw [->, thick] (4.2, 7.5) -- (7, 5.7); 
		
		\node at (7.3, 5.5) {$X$}; 
		\node at (2.5, 6.6) {$N$}; 
		\node at (2.5, 1.6) {$M$}; 
		
		
		\node at (-2.7, 6.6) {$\partial_2 V$}; 
		
		\node at (-2.7, 1.6) {$\partial_2 W$}; 
		
		\node at (0.75, 9.5) {\scriptsize $\partial N = \partial \partial_2 V$};

		\draw [dashed, ->] (0.55, 9.3) to [out = -90, in = 90] (0, 8.6); 
		
		\node at (0.75, 1.5) {\scriptsize $\partial M = \partial \partial_2 W$}; 
		
		\draw [dashed, ->] (0.55, 1.7) to [out = 90, in = -90] (0, 2.4);

		\node at (5.5, 4) {$\Phi$}; 
		\node at (5.5, 7) {$\Psi$}; 
		\node at (-2, 5.5) {$F|_{\partial_2 V}$};
		
		\draw [->, thick] (4.2, 3.5) -- (7, 5.3);

		
		\begin{scope}[shift ={(0, 5) }]
		\coordinate (ref) at (0,0);
		\shade[shading = ball, ball color = gray!30!white] 
		($(ref) + (-3, 4)$)  arc (90:270:1) ($(ref) + (-3, 2)$)  to [out = 0, in = 180] ($(ref) + (-2, 2.5)$) to  [out = 0, in = 180]   ($(ref) + (-1, 2)$) to [out = 0, in = 180]  ($(ref) + (0, 2.5)$) to [out = 0, in = 180] ($(ref) + (2, 2)$)  arc (-90:90:2 and 1) ($(ref) + (2, 4)$) to [out = 180, in = 0] ($(ref) + (0, 3.5)$) to [out = 180, in = 0]  ($(ref) + (-1, 4)$) to [out = 180, in = 0 ]  ($(ref) + (-2, 3.5)$) to [out = 180, in = 0]  ($(ref) + (-3, 4)$)  arc (90:270:1) ($(ref) + (-3, 2)$);	
		
		\draw [ultra thick] ($(ref) + (0, 3.5)$) to [out = 180, in = 0]  ($(ref) + (-1, 4)$) to [out = 180, in = 0 ]  ($(ref) + (-2, 3.5)$) to [out = 180, in = 0]  ($(ref) + (-3, 4)$)  arc (90:270:1) ($(ref) + (-3, 2)$)  to [out = 0, in = 180] ($(ref) + (-2, 2.5)$) to  [out = 0, in = 180]   ($(ref) + (-1, 2)$) to [out = 0, in = 180]  ($(ref) + (0, 2.5)$);

		\draw [blue, thick] ($(ref) + (0, 2.5)$) to [out = 0, in = 180] ($(ref) + (2, 2)$)  arc (-90:90:2 and 1) ($(ref) + (2, 4)$) to [out = 180, in = 0] ($(ref) + (0, 3.5)$);
		\draw[ultra thick]  ($(ref) + (0, 3.5)$) arc (90:270:0.25 and 0.5);
		\draw[dashed, ultra thick]  ($(ref) + (0, 2.5)$) arc (270:450:0.25 and 0.5);

		\draw [blue, fill =white, thick] ($(ref) + (2.5, 3.4)$) to [out = 205, in = 155] ($(ref) + (2.5, 2.6)$) to [out = 20, in = -20] ($(ref) + (2.5, 3.4)$);
		\draw [blue, thick] ($(ref) + (2.5, 3.4)$) to [out = 25, in = -165] ($(ref) + (2.65, 3.45)$);
		\draw [blue, thick] ($(ref) + (2.5, 2.6)$) to [out = -25, in = 165] ($(ref) + (2.65, 2.55)$);
		
		\begin{scope}[shift = {(-1, 0)}]
		\coordinate (ref) at (0,0);
		\draw [blue, fill =white, thick] ($(ref) + (2.5, 3.4)$) to [out = 205, in = 155] ($(ref) + (2.5, 2.6)$) to [out = 20, in = -20] ($(ref) + (2.5, 3.4)$);
		\draw [blue, thick] ($(ref) + (2.5, 3.4)$) to [out = 25, in = -165] ($(ref) + (2.65, 3.45)$);
		\draw [blue, thick] ($(ref) + (2.5, 2.6)$) to [out = -25, in = 165] ($(ref) + (2.65, 2.55)$);
		
		\end{scope}
		
		\coordinate (ref) at (0,0);
		\draw [fill =white,  ultra thick] ($(ref) + (-3.4, 3)$) to [out = 60, in = 120] ($(ref) + (-2.6, 3)$) to [out = -120, in = -60] ($(ref) + (-3.4, 3)$);	
		\draw [ ultra thick] ($(ref) + (-3.4, 3)$) to [out = 120, in = -75] ($(ref) + (-3.46, 3.15)$);
		\draw [ultra thick] ($(ref) + (-2.6, 3)$) to [out = 60, in = -115] ($(ref) + (-2.54, 3.15)$);
		
		\begin{scope}[shift= {(2, 0)}]
		\coordinate (ref) at (0,0);
		\draw [fill =white,  ultra thick] ($(ref) + (-3.4, 3)$) to [out = 60, in = 120] ($(ref) + (-2.6, 3)$) to [out = -120, in = -60] ($(ref) + (-3.4, 3)$);	
		\draw [ ultra thick] ($(ref) + (-3.4, 3)$) to [out = 120, in = -75] ($(ref) + (-3.46, 3.15)$);
		\draw [ultra thick] ($(ref) + (-2.6, 3)$) to [out = 60, in = -115] ($(ref) + (-2.54, 3.15)$);
		\end{scope}
		
		\end{scope}


		\node at (0.6, 8) {$V$}; 
		\node at (0.6, 3) {$W$}; 
		
		\node at (0.75, 5.5) {$F$}; 
		\draw [->, thick] (0.5, 6.8) -- (0.5, 4.2);

		\end{tikzpicture}
		\caption{Equivalence relation $\theta\sim 0$ for the definition of $
			L_n(\pi_1X,  w)$, where $F|_{\partial_2 V}$ is a homotopy equivalence and $F|_{N} = f$ is a degree one normal map.}
	\end{figure}
	\begin{enumerate}[(1)]
		\item There exists a manifold $2$-ad $(W, \partial W)$ of dimension $(n+1)$ with a continuous map $\Phi\colon W\to X$ so that the pullback $\mathbb Z/2$-principal bundle $\Phi^\ast(w)$ is the orientation covering of $W$. Here $\partial W = M \cup_{\partial M} \partial_2 W $ and $\Phi|_M = \varphi$.   
		\item Similarly, we have a manifold $2$-ad $(V, \partial V)$ of dimension $(n+1)$ with a continuous map $\Psi: V\to X$ so that the pullback $\mathbb Z/2$-principal bundle $\Psi^\ast(w)$ is the orientation covering of $V$. Moreover,    $\partial V = N \cup_{\partial N} \partial_2 V $ and $\Psi|_N = \psi$.  
		\item There is a degree one normal map of manifold $2$-ads 
		\[ F\colon (V, \partial V) \to (W, \partial W)\] such that $\Psi\circ F  = \Phi$. Moreover, $F$ restricts to $f$ on $N$,  and  $ F|_{\partial_2 V} \colon \partial_2V \to \partial_2 W $ is a homotopy equivalence over $X$.  
	\end{enumerate}
\end{definition}

We denote the set of equivalence classes by $L_n(\pi_1 X,  w)$. Note that $L_n(\pi_1 X,  w)$ is an abelian group with the addition given by  disjoint union.

It is a theorem of Wall that the above definition of $L$-groups is equivalent to the algebraic definition of $L$-groups $L^h_{n}(\Gamma,  w)$ when $n\geq 5$ \cite[Chapter 9]{MR1687388} \cite{MR0246310}. More precisely, Wall's chapter $9$ \cite[Chapter 9]{MR1687388} dealt with the surgery theory for simple homotopy equivalences, and the algebraic $L$-groups that appeared in that chapter are usually denoted by $L_\ast^s(\Gamma, w)$, where $s$ stands for simple and $w$ is a group homomorphism $\Gamma\to \mathbb Z/2$. In the current paper, we deal with homotopy equivalences instead of simple homotopy equivalences, hence the groups $L_\ast^h(\Gamma, w)$ instead of $L_\ast^s(\Gamma, w)$. When $n$ is even,  $L_n^s(\Gamma, w)$ is defined as the abelian group of equivalence classes of quadratic forms of the ring\footnote{Here $\mathbb Z\Gamma$ is considered as a ring with involution, where the involution is induced by $w$ that maps $\gamma \to w(\gamma) g^{-1}$.} $\mathbb Z\Gamma$  \cite[Chapter 5, page 49]{MR1687388}. When $n$ is odd,   $L_n^s(\Gamma, w)$ is defined as the abelian group of equivalence classes of automorphisms on hyperbolic forms of the ring $\mathbb Z\Gamma$ \cite[Chapter 6, page 68]{MR1687388}.  The definition of $L_\ast^h(\Gamma, w)$ is essentially the same,  once we drop the simplicity condition.  In general, the groups $L_\ast^h(\Gamma, w)$ and $L_\ast^s(\Gamma, w)$ are different. The same argument in Wall's chapter 9 \cite[Chapter 9]{MR1687388} proves the following theorem.

\begin{theorem}[{\cite[Chapter 9]{MR1687388}}] \label{thm:idenL} Let $\Gamma = \pi_1 X$. For all $n\geq 5$,  $L_n(\pi_1 X, w)$ is naturally isomorphic to the algebraic definition of $L_n^h(\Gamma,  w)$.
\end{theorem}

The dimension restriction ($n\geq 5$) in the above theorem is necessary. If $n < 5$, although this cobordism theoretic definition of $ L_n(\pi_1 X, w)$ still gives an abelian group, it is not clear what it really describes because of well known problems of low dimensional surgery. Moreover, these groups (in low dimensions) could well be dependent on the category. On the other hand, as long as $n\geq 5$, not only  $L_n(\pi_1 X, w)$ is naturally isomorphic to the algebraic definition of $L_n^h(\pi_1 X,  w)$, but also the map $\times \mathbb{CP}^2$ (i.e. taking the direct product with $\mathbb{CP}^2$) induces an isomorphism 
\[  L_n(\pi_1 X, w) \xrightarrow{\, \cong\, }  L_{n+4}(\pi_1 X, w). \]  This motivates us to give the following definition, which makes Wall's geometric definition of $L$-groups into a $4$-periodic theory in all dimensions.  

\begin{definition}\label{def:periodL}
	For each $n\in \mathbb Z$, we define $ \mathfrak L_{n}(\pi_1 X, w)$ to be the direct limit of 
	\[
	\scalebox{1}{	$\cdots  \to  L_{k}(\pi_1 X, w) \xrightarrow{\times \mathbb{CP}^2}   L_{k+4}(\pi_1 X, w)
		\xrightarrow{\times \mathbb{CP}^2} L_{k+8}(\pi_1 X, w) \to \cdots,$}
	\]
	where $k \equiv n \pmod{4}$. 
\end{definition} 

Now we shall also introduce a controlled version of Wall's $L$-group definition, which will be identified with $H_\ast(X; \mathbb L_\bullet)$. Here $\mathbb L_\bullet$ is an $\Omega$-spectrum of simplicial sets of quadratic forms and formations over $\mathbb Z$ such that $\mathbb L_0\simeq G/TOP$, cf. \cite[Section 3]{MR884801}. 

\begin{definition}\label{def:norm}
	An object of $\mathcal N_n(X,  w)$ consists of the following data: 
	\begin{figure}[H]
		\centering
		\begin{tikzpicture}[scale=1, every node/.style={transform shape}]
		\coordinate (ref) at (0,0); 
		\draw [blue,  thick] ($(ref) + (0, 2.5)$) to [out = 0, in = 180] ($(ref) + (2, 2)$)  arc (-90:90:2 and 1) ($(ref) + (2, 4)$) to [out = 180, in = 0] ($(ref) + (0, 3.5)$);
		\draw[red, very thick]  ($(ref) + (0, 3.5)$) arc (90:270:0.25 and 0.5);
		\draw[red, very thick]  ($(ref) + (0, 2.5)$) arc (270:450:0.25 and 0.5);
		\draw [ blue,  thick] ($(ref) + (1.5, 3)$) to [out = 20, in = 160] ($(ref) + (2.7, 3)$) to [out = -160, in = -20] ($(ref) + (1.5, 3)$);
		\draw [ blue, thick] ($(ref) + (1.5, 3)$) to [out = 160, in = 175] ($(ref) + (1.43, 3.05)$);
		\draw [ blue, thick] ($(ref) + (2.7, 3)$) to [out = 20, in = 5] ($(ref) + (2.77, 3.05)$);

		\draw [->, red, very thick] (-0.4, 7.9) to [out = 225, in = 135 ] (-0.4, 3.1); 
		
		\draw [->, thick] (4.2, 7.5) -- (7, 5.7); 
		
		\node at (7.3, 5.5) {$X$}; 
		\node at (4.4, 8) {$N$}; 
		\node at (4.4, 3) {$M$}; 
		\node at (-1, 8.5) {$\partial N$}; 
		
		\draw [dashed, ->] (-1, 8.3) to [out = -90, in = 180] (-0.4, 8.1);

		\node at (-1, 2.4) {$\partial M$}; 
		
		\draw [dashed, ->] (-1, 2.6) to [out = 90, in = 180] (-0.4, 2.9);
		
		\node [blue] at (2.2, 5.5) {$f$}; 
		\node at (5.5, 4) {$\varphi$}; 
		\node at (5.5, 7) {$\psi$}; 
		\node [red] at (-2, 5.5) {$f|_{\partial N}$};
		
		\draw [->, thick] (4.2, 3.5) -- (7, 5.3); 
		
		\draw [->, blue,  thick] (2, 6.7) -- (2, 4.3);
		
		\begin{scope}[shift ={(0, 5) }]
		\coordinate (ref) at (0,0); 
		\draw [blue,  thick] ($(ref) + (0, 2.5)$) to [out = 0, in = 180] ($(ref) + (2, 2)$)  arc (-90:90:2 and 1) ($(ref) + (2, 4)$) to [out = 180, in = 0] ($(ref) + (0, 3.5)$);
		\draw[red, very thick]  ($(ref) + (0, 3.5)$) arc (90:270:0.25 and 0.5);
		\draw[red, very thick]  ($(ref) + (0, 2.5)$) arc (270:450:0.25 and 0.5);
		
		\draw [blue, fill =white, thick] ($(ref) + (2.5, 3.4)$) to [out = 205, in = 155] ($(ref) + (2.5, 2.6)$) to [out = 20, in = -20] ($(ref) + (2.5, 3.4)$);
		\draw [blue, thick] ($(ref) + (2.5, 3.4)$) to [out = 25, in = -165] ($(ref) + (2.65, 3.45)$);
		\draw [blue, thick] ($(ref) + (2.5, 2.6)$) to [out = -25, in = 165] ($(ref) + (2.65, 2.55)$);
		
		\begin{scope}[shift = {(-1, 0)}]
		\coordinate (ref) at (0,0);
		\draw [blue, fill =white, thick] ($(ref) + (2.5, 3.4)$) to [out = 205, in = 155] ($(ref) + (2.5, 2.6)$) to [out = 20, in = -20] ($(ref) + (2.5, 3.4)$);
		\draw [blue, thick] ($(ref) + (2.5, 3.4)$) to [out = 25, in = -165] ($(ref) + (2.65, 3.45)$);
		\draw [blue, thick] ($(ref) + (2.5, 2.6)$) to [out = -25, in = 165] ($(ref) + (2.65, 2.55)$);
		
		\end{scope}
		\end{scope}
		\end{tikzpicture}
		\vspace{3mm}
		\caption{An object  $\theta = (M, \partial M, \varphi, N, \partial N, \psi, f)$  of  $\mathcal N_n(X,   w)$, where $f$ is a degree one normal map and  $f|_{\partial N}$ is an infinitesimally controlled homotopy equivalence.}
	\end{figure}
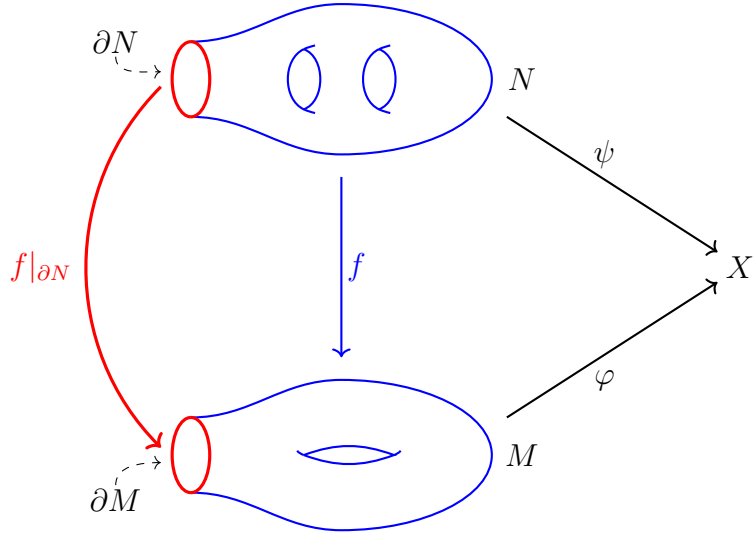
	\begin{enumerate}[(1)]
		\item two manifold $1$-ads $(M, \partial M)$ and  $(N, \partial N)$ with $\dim M = \dim N = n$, where $\partial M$ (resp. $\partial N$) is the boundary of $M$ (resp. $\partial N$);
		\item continuous maps $\varphi\colon M \to X$ and $\psi\colon N \to X$ so that the pullback $\mathbb Z/2$-principal bundles $\varphi^\ast(w)$ (resp. $\psi^\ast(w)$) is the orientation covering of $M$ (resp. $N$) respectively; 
		\item a degree one normal map of manifold $1$-ads \[ f\colon (N, \partial N) \to (M, \partial M)\] such that $\varphi \circ f = \psi$. Moreover, on the boundary $ f|_{\partial N} \colon \partial N\to \partial M$ is an \emph{infinitesimally controlled} homotopy equivalence over  $X$.
	\end{enumerate} 
\end{definition}

\begin{definition}
	The equivalence relation for defining $\mathcal N_n(X, w)$ is given as follows.  Let \[ \xi = (M, \partial M, \varphi, N, \partial N, \psi, f) \] be an object from Definition $\ref{def:norm}$ above.  We write $\xi \sim 0$ if the following conditions are satisfied. 
	\begin{figure}[h]
		\centering
		\begin{tikzpicture}[scale=1, every node/.style={transform shape}]
		\coordinate (ref) at (0,0); 
		
		\shade[shading = ball, ball color = gray!30!white,  draw = black] 
		($(ref) + (-3, 4)$)  arc (90:270:1) ($(ref) + (-3, 2)$)  to [out = 0, in = 180] ($(ref) + (-2, 2.5)$) to  [out = 0, in = 180]   ($(ref) + (-1, 2)$) to [out = 0, in = 180]  ($(ref) + (0, 2.5)$) to [out = 0, in = 180] ($(ref) + (2, 2)$)  arc (-90:90:2 and 1) ($(ref) + (2, 4)$) to [out = 180, in = 0] ($(ref) + (0, 3.5)$) to [out = 180, in = 0]  ($(ref) + (-1, 4)$) to [out = 180, in = 0 ]  ($(ref) + (-2, 3.5)$) to [out = 180, in = 0]  ($(ref) + (-3, 4)$)  arc (90:270:1) ($(ref) + (-3, 2)$);	
		
		\draw [red, very thick] ($(ref) + (0, 3.5)$) to [out = 180, in = 0]  ($(ref) + (-1, 4)$) to [out = 180, in = 0 ]  ($(ref) + (-2, 3.5)$) to [out = 180, in = 0]  ($(ref) + (-3, 4)$)  arc (90:270:1) ($(ref) + (-3, 2)$)  to [out = 0, in = 180] ($(ref) + (-2, 2.5)$) to  [out = 0, in = 180]   ($(ref) + (-1, 2)$) to [out = 0, in = 180]  ($(ref) + (0, 2.5)$);
		
		\draw [blue, thick] ($(ref) + (0, 2.5)$) to [out = 0, in = 180] ($(ref) + (2, 2)$)  arc (-90:90:2 and 1) ($(ref) + (2, 4)$) to [out = 180, in = 0] ($(ref) + (0, 3.5)$);
		\draw[red, very thick]  ($(ref) + (0, 3.5)$) arc (90:270:0.25 and 0.5);
		\draw[dashed, red, very thick]  ($(ref) + (0, 2.5)$) arc (270:450:0.25 and 0.5);
		
		\draw [blue, fill =white,   thick] ($(ref) + (1.5, 3)$) to [out = 20, in = 160] ($(ref) + (2.7, 3)$) to [out = -160, in = -20] ($(ref) + (1.5, 3)$);
		\draw [ blue, thick] ($(ref) + (1.5, 3)$) to [out = 160, in = -10] ($(ref) + (1.4, 3.03)$);
		\draw [ blue,  thick] ($(ref) + (2.7, 3)$) to [out = 20, in = -170] ($(ref) + (2.8, 3.03)$);
		
		\draw [red, fill =white,  very thick] ($(ref) + (-3.4, 3)$) to [out = 60, in = 120] ($(ref) + (-2.6, 3)$) to [out = -120, in = -60] ($(ref) + (-3.4, 3)$);	
		\draw [  red, very thick] ($(ref) + (-3.4, 3)$) to [out = 120, in = -75] ($(ref) + (-3.46, 3.15)$);
		\draw [ red, very thick] ($(ref) + (-2.6, 3)$) to [out = 60, in = -115] ($(ref) + (-2.54, 3.15)$);
		
		\begin{scope}[shift= {(2, 0)}]
		\coordinate (ref) at (0,0);
		
		\draw [red, fill =white,  very thick] ($(ref) + (-3.4, 3)$) to [out = 60, in = 120] ($(ref) + (-2.6, 3)$) to [out = -120, in = -60] ($(ref) + (-3.4, 3)$);	
		\draw [  red, very thick] ($(ref) + (-3.4, 3)$) to [out = 120, in = -75] ($(ref) + (-3.46, 3.15)$);
		\draw [ red, very thick] ($(ref) + (-2.6, 3)$) to [out = 60, in = -115] ($(ref) + (-2.54, 3.15)$);

		\end{scope}

		
		\draw [->, red, very thick] (-2, 6.8) to [out = 225, in = 135 ] (-2, 3.8); 
		
		\draw [->, thick] (4.2, 7.5) -- (7, 5.7); 
		
		\node at (7.3, 5.5) {$X$}; 
		\node at (2.5, 6.6) {$N$}; 
		\node at (2.5, 1.6) {$M$}; 
		
		
		\node at (-2.7, 6.6) {$\partial_2 V$}; 
		
		\node at (-2.7, 1.6) {$\partial_2 W$}; 
		
		\node at (0.75, 9.5) {\scriptsize $\partial N = \partial \partial_2 V$};

		\draw [dashed, ->] (0.55, 9.3) to [out = -90, in = 90] (0, 8.6); 
		
		\node at (0.75, 1.5) {\scriptsize $\partial M = \partial \partial_2 W$}; 
		
		\draw [dashed, ->] (0.55, 1.7) to [out = 90, in = -90] (0, 2.4);

		\node at (5.5, 4) {$\Phi$}; 
		\node at (5.5, 7) {$\Psi$}; 
		\node [red] at (-2, 5.5) {$F|_{\partial_2 V}$};
		
		\draw [->, thick] (4.2, 3.5) -- (7, 5.3);

		
		\begin{scope}[shift ={(0, 5) }]
		\coordinate (ref) at (0,0);
		\shade[shading = ball, ball color = gray!30!white] 
		($(ref) + (-3, 4)$)  arc (90:270:1) ($(ref) + (-3, 2)$)  to [out = 0, in = 180] ($(ref) + (-2, 2.5)$) to  [out = 0, in = 180]   ($(ref) + (-1, 2)$) to [out = 0, in = 180]  ($(ref) + (0, 2.5)$) to [out = 0, in = 180] ($(ref) + (2, 2)$)  arc (-90:90:2 and 1) ($(ref) + (2, 4)$) to [out = 180, in = 0] ($(ref) + (0, 3.5)$) to [out = 180, in = 0]  ($(ref) + (-1, 4)$) to [out = 180, in = 0 ]  ($(ref) + (-2, 3.5)$) to [out = 180, in = 0]  ($(ref) + (-3, 4)$)  arc (90:270:1) ($(ref) + (-3, 2)$);	
		
		\draw [red, very thick] ($(ref) + (0, 3.5)$) to [out = 180, in = 0]  ($(ref) + (-1, 4)$) to [out = 180, in = 0 ]  ($(ref) + (-2, 3.5)$) to [out = 180, in = 0]  ($(ref) + (-3, 4)$)  arc (90:270:1) ($(ref) + (-3, 2)$)  to [out = 0, in = 180] ($(ref) + (-2, 2.5)$) to  [out = 0, in = 180]   ($(ref) + (-1, 2)$) to [out = 0, in = 180]  ($(ref) + (0, 2.5)$);

		\draw [blue, thick] ($(ref) + (0, 2.5)$) to [out = 0, in = 180] ($(ref) + (2, 2)$)  arc (-90:90:2 and 1) ($(ref) + (2, 4)$) to [out = 180, in = 0] ($(ref) + (0, 3.5)$);
		\draw[red, very thick]  ($(ref) + (0, 3.5)$) arc (90:270:0.25 and 0.5);
		\draw[dashed, red, very thick]  ($(ref) + (0, 2.5)$) arc (270:450:0.25 and 0.5);

		\draw [blue, fill =white, thick] ($(ref) + (2.5, 3.4)$) to [out = 205, in = 155] ($(ref) + (2.5, 2.6)$) to [out = 20, in = -20] ($(ref) + (2.5, 3.4)$);
		\draw [blue, thick] ($(ref) + (2.5, 3.4)$) to [out = 25, in = -165] ($(ref) + (2.65, 3.45)$);
		\draw [blue, thick] ($(ref) + (2.5, 2.6)$) to [out = -25, in = 165] ($(ref) + (2.65, 2.55)$);
		
		\begin{scope}[shift = {(-1, 0)}]
		\coordinate (ref) at (0,0);
		\draw [blue, fill =white, thick] ($(ref) + (2.5, 3.4)$) to [out = 205, in = 155] ($(ref) + (2.5, 2.6)$) to [out = 20, in = -20] ($(ref) + (2.5, 3.4)$);
		\draw [blue, thick] ($(ref) + (2.5, 3.4)$) to [out = 25, in = -165] ($(ref) + (2.65, 3.45)$);
		\draw [blue, thick] ($(ref) + (2.5, 2.6)$) to [out = -25, in = 165] ($(ref) + (2.65, 2.55)$);
		
		\end{scope}
		
		\coordinate (ref) at (0,0);
		
		\draw [red, fill =white,  very thick] ($(ref) + (-3.4, 3)$) to [out = 60, in = 120] ($(ref) + (-2.6, 3)$) to [out = -120, in = -60] ($(ref) + (-3.4, 3)$);	
		\draw [  red, very thick] ($(ref) + (-3.4, 3)$) to [out = 120, in = -75] ($(ref) + (-3.46, 3.15)$);
		\draw [ red, very thick] ($(ref) + (-2.6, 3)$) to [out = 60, in = -115] ($(ref) + (-2.54, 3.15)$);
		
		\begin{scope}[shift= {(2, 0)}]
		\coordinate (ref) at (0,0);
		
		\draw [red, fill =white,  very thick] ($(ref) + (-3.4, 3)$) to [out = 60, in = 120] ($(ref) + (-2.6, 3)$) to [out = -120, in = -60] ($(ref) + (-3.4, 3)$);	
		\draw [  red, very thick] ($(ref) + (-3.4, 3)$) to [out = 120, in = -75] ($(ref) + (-3.46, 3.15)$);
		\draw [ red, very thick] ($(ref) + (-2.6, 3)$) to [out = 60, in = -115] ($(ref) + (-2.54, 3.15)$);

		\end{scope}
		
		\end{scope}
		

		\node at (0.6, 8) {$V$}; 
		\node at (0.6, 3) {$W$}; 
		
		\node at (0.75, 5.5) {$F$}; 
		\draw [->, thick] (0.5, 6.8) -- (0.5, 4.2);
		
		\end{tikzpicture}
		\caption{Equivalence relation $\theta\sim 0$ for the definition of  $\mathcal N_n(X, w)$, where $F|_{\partial_2 V}$ is an infinitesimally controlled homotopy equivalence and $F|_N = f$ is a degree one normal map. }
	\end{figure}
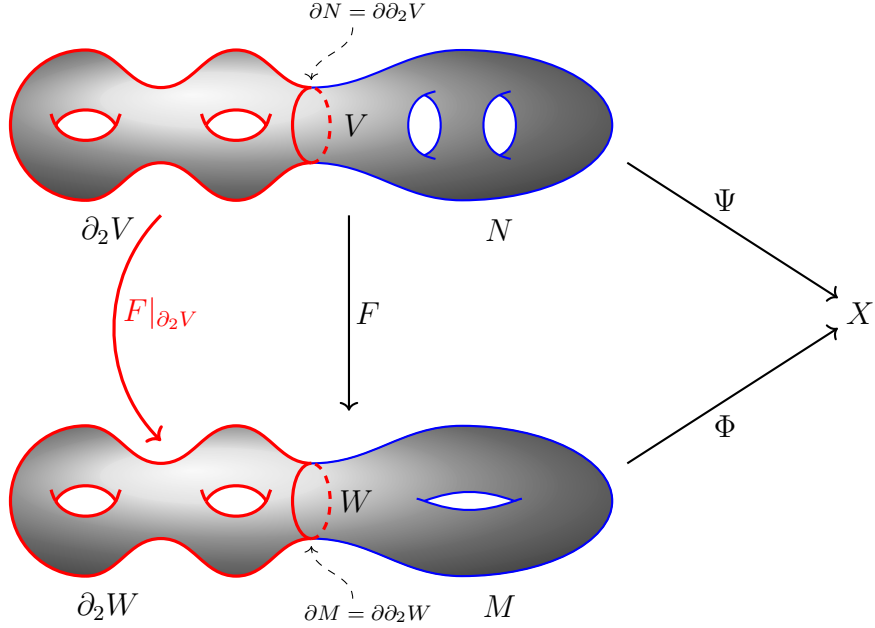			
	
	\begin{enumerate}[(1)]
		\item There exists a manifold $2$-ad $(W, \partial W)$ of dimension $(n+1)$ with a continuous map $\Phi\colon W\to X$ so that the pullback $\mathbb Z/2$-principal bundle $\Phi^\ast(w)$ is the orientation covering of $W$. Moreover, $\partial W = M \cup_{\partial M} \partial_2 W $ and $\Phi|_M = \varphi$. 
		\item Similarly, we have a manifold $2$-ad $(V, \partial V)$ of dimension $(n+1)$ with a continuous map $\Psi: V\to X$ so that the pullback $\mathbb Z/2$-principal bundle $\Psi^\ast(w)$ is the orientation covering of $V$. Moreover,    $\partial V = N \cup_{\partial N} \partial_2 V $ and $\Psi|_N = \psi$.    
		\item There is a degree one normal map of $2$-ads \[ F\colon (V, \partial V) \to (W, \partial W)\] such that $\Psi\circ F  = \Phi$. Moreover, $F$ restricts to $f$ on $N$, and  
		\[ F|_{\partial_2 V} \colon \partial_2V \to \partial_2 W\] is an \emph{infinitesimally controlled} homotopy equivalence over $X$.  
	\end{enumerate}
\end{definition}

We denote  by $\mathcal N_n(X,  w)$ the set of equivalence classes from Definition $\ref{def:norm}$. Note that  $\mathcal   N_n(X, w)$ is an abelian group with the addition given by disjoint union.

Following the same strategy from Definition $\ref{def:periodL}$, we shall turn $\mathcal N_n(X, w)$  and $\mathcal  S_n(X, w)$ into $4$-periodic theories.  
\begin{definition}\label{def:period}
	For each $n\in \mathbb Z$, we define $ \mathfrak N_{n}(X, w)$ to be the direct limit 
	\[  \cdots \to  \mathcal N_{k}(X, w) \xrightarrow{\times \mathbb{CP}^2}  \mathcal N_{k+4}(X, w) \xrightarrow{\times \mathbb{CP}^2} \mathcal N_{k+8}(X, w) \to \cdots, \]
	where $k \equiv n \pmod{4}$. Similarly, we define $ \mathfrak S_{n}(X, w)$ to be the direct limit 
	\[  \cdots \to  \mathcal S_{k}(X, w) \xrightarrow{\times \mathbb{CP}^2}  \mathcal S_{k+4}(X, w) \xrightarrow{\times \mathbb{CP}^2} \mathcal S_{k+8}(X, w) \to \cdots, \]
	where $k \equiv n \pmod{4}$.
\end{definition}

Let us discuss a key point of our definitions of $\mathcal N_n(X, w)$ and $\mathcal S_n(X, w)$. For simplicity\footnote{The argument for the general case is the same, once we use twisted $\mathbb L$-spectra \cite[Appendix A]{MR1211640} instead of the usual $\mathbb L$-spectrum.}, let us focus on the oriented case, that is, $w$ is the trivial $\mathbb Z/2$-principal covering of $X$, and will be dropped from our notation during this discussion. The $L$-groups (either $L^{h}$ or $L^s$), besides having an algebraic definition, also have a cobordism theoretic definition, according to Wall \cite[Chapter 9]{MR1687388}.  Now by ideas from controlled topology (cf. \cite{MR1388311} \cite[Section 3]{MR884801}),  if we impose extra infinitesimal control on objects in Wall's cobordism theoretic definition of $L$-groups, we obtain a generalized homology theory. Indeed, a cobordism type construction such as $\mathcal N_n(X)$ defines a generalized homology theory if it satisfies an appropriate transversality condition (see for example \cite{MR1388303}). This transversality condition is indeed satisfied by $\mathcal N_n(X)$ due to infinitesimal control. See for example Definition $\ref{def:pltrans}$ for how this works in the PL (i.e. piecewise-linear) category. 

Let us recast the above discussion in the language of spectra.  In fact,  based on Wall's cobordism theoretic definition,  Quinn constructed geometric surgery spectra of $\Delta$-sets which realized $L$-groups as their homotopy groups \cite{MR0282375}. From this perspective, our construction $\mathcal N_\ast(X)$ above can be viewed as a cobordism theoretic representation of the generalized homology theory determined by such a spectrum. See \cite{MR1388303} for a more thorough discussion of some closely related topics.   On the other hand, based on the algebraic definition of $L$-groups, Ranicki constructed quadratic $L$-theory spectra which also realize $L$-groups as their homotopy groups \cite[Chapter 13]{MR1211640}. Consequently, Quinn's geometric surgery spectra and Ranicki's quadratic $L$-theory spectra give arise to (homotopy) equivalent spectra. It follows that  $\mathcal N_n(X)$ is naturally isomorphic to $H_n(X; \mathbb L_\bullet)$, where $\mathbb L_\bullet$ is the quadratic $L$-theory spectrum for the trivial group $e$ --- an $\Omega$-spectrum of simplicial sets of quadratic forms and formations over $\mathbb Z$ --- such that $\mathbb L_0\simeq  G/TOP$. Moreover, the natural morphism
\[ i_\ast \colon \mathcal N_n(X) \to  L_n(\pi_1 X),   \]
which is defined by forgetting infinitesimal control,  can also be viewed as  induced by a map $\mu$ of spectra. Now the groups $\mathcal S_\ast(X)$ are just the (stable) homotopy groups of the homotopy fiber of this map $\mu$.  

Let us prove that $\mathcal N_n(X)$ is naturally isomorphic to $H_n(X; \mathbb L_\bullet)$ for all $n\geq 0$. This isomorphism will be used later in the surgery exact sequence to identify $\mathcal S_n(X)$ with $\mathcal S^\topo(X)$ when $X$ is a closed connected oriented topological manifold of dimension $\geq 5$.   By a standard fact from algebraic topology, it suffices to show that  there is a natural morphism between the two homolgy theories that induces an isomorphism when $X$ is one point. In our current case, there is a natural morphism from $\mathcal N_n(X)$ to $H_n(X; \mathbb L_\bullet)$ by mapping an element of $\mathcal N_n(X)$ to its corresponding (local) algebraic Poincar\'{e} complex. It remains to identify the groups $\mathcal N_n(\{pt\})$ and $H_n(\{pt\}; \mathbb L_\bullet)$. In fact, in order to make this identification, we shall slightly modify our definitions of $L_n(\pi_1 X)$, $\mathcal N_n(X)$ and $\mathcal S_n(X)$. It will become clear that the modified version coincides with our original version above (Definition  $\ref{def:newstruc}$, $\ref{def:Lgrp}$ and $\ref{def:norm}$) when $n\neq 4$. In other words, the only essential change happens in dimension $ n = 4$.

\begin{definition}\label{modLgrp}
	An object of ${L}^{new}_n(\pi_1 X)$  consists of the following data: 
	
	\begin{enumerate}[(1)]
		\item two \emph{oriented} manifold $1$-ads $(M, \partial M)$ and  $(N, \partial N)$ with $\dim M = \dim N = n$, where $\partial M$ (resp. $\partial N$) is the boundary of $M$ (resp. $\partial N$);
		\item continuous maps $\varphi\colon M \to X$ and $\psi\colon N \to X$; 
		\item a degree one normal map of the $1$-ads 
		\[ f\colon (N, \partial N) \to (M, \partial M)\] such that $\varphi \circ f = \psi$. Moreover on the boundary,  \[ f|_{\partial N} \colon \partial N\to \partial M\]  is a $\mathbb Z\pi$-homology equivalence.
	\end{enumerate} 
\end{definition}
Here   $f|_{\partial N}$ is a $\mathbb Z\pi$-homology equivalence means that it induces an isomorphism $H_\ast(\partial N; \mathbb Z\pi) \xrightarrow{ \cong } H_\ast(\partial M; \mathbb Z\pi)$ on homology with local coefficients in $\mathbb Z\pi$, where $\pi = \pi_1(\partial M)$. Equivalently, let $\widetilde{\partial M}$ be the universal covering space of $\partial M$ and  $(\partial N)_\pi$ the covering space of $\partial N$ which is the pullback of $\widetilde{\partial M}$ along the map $f|_{\partial N}$. We say $f|_{\partial N}$ is a $\mathbb Z\pi$-homology equivalence if the lifted map $\widetilde f\colon (\partial N)_\pi \to \widetilde{\partial M}$ induces an isomorphism on homology. 

The equivalence relation for the objects in ${L}^{new}_n(\pi_1 X)$ is defined the same as in Definition $\ref{def:equivforL}$ except that we replace homotopy equivalences by $\mathbb Z\pi$-homology equivalences everywhere. Similarly, we can define $\mathcal N_n^{new}(X)$ and $\mathcal S_n^{new}(X)$ using (infinitesimally controlled) $\mathbb Z\pi$-homology equivalences instead of (infinitesimally controlled) homotopy equivalences. 

In fact, in dimension $n \geq 5$, the modified version $L^{new}_n(\pi_1 X)$ is naturally isomorphic to the original version $L_n(\pi_1 X)$. See for example \cite{MR0423377}. Following the discussion above, from the viewpoint of spectra,  we also see that $\mathcal N_n^{new}(X)$ and $\mathcal S_n^{new}(X)$ are naturally isomorphic to $\mathcal N_n(X)$ and $\mathcal S_n(X)$ respectively when $n \geq 5$.

With the same notation from above, observe that 
\[ f|_{\partial N}\colon \partial N \to \partial M\] is  a degree one map, since $f\colon (N, \partial N) \to (M, \partial M)$ is. When $n=0$ or $1$, $f|_{\partial N}$ is automatically a homeomorphism. When $n=2$,  $f|_{\partial N}$ is a degree one map between circles, which is necessarily a homotopy equivalence and thus homotopic to a homeomorphism. When $n=3$, $f|_{\partial N}$ is a degree one map between oriented surfaces, in particular,  it induces a surjection $(f|_{\partial N})_\ast\colon \pi_1(\partial N) \twoheadrightarrow \pi_1(\partial M)$ between the fundamental groups. Let $G \subset \pi_1(\partial N)$ be the kernel of this surjection. If $G$ is trivial, then $(f|_{\partial N})_\ast$ is an isomorphism, which together with $\mathbb Z\pi$-homology equivalence implies that $f$ is a homotopy equivalence.
If $G$ is nontrivial, we claim that $\pi_1(\partial N)/G$ has to be finite. Indeed, if $\pi_1(\partial N)/G$ is infinite,  then  $\widetilde{\partial N}/G$ is a noncompact surface, where  $\widetilde {\partial N}$ is the universal covering space of $\partial N$. It follows that $G$ is a free group. Note that the pullback covering space $(\partial N)_\pi$ over $\partial N$ is isomorphic to the covering space $\widetilde{\partial N}/G$ over $\partial N$. However, the first homology group of $\widetilde{\partial N}/G$ is nontrivial, whereas the first homology group of $\widetilde{\partial M}$ is always trivial. This is a contradiction to the assumption that $f|_{\partial N}$ is a $\mathbb Z\pi$-homology equivalence. 
Therefore, if $G$ is nontrivial, then    $\pi_1(\partial N)/G$ is finite. It follows that $\pi_1(\partial N)/G = \pi_1(\partial M)$ has to be trivial, that is, $\partial M$ is the $2$-sphere. In this case, $(\partial N)_\pi = \partial N$, and a comparison of homology groups shows that $f|_{\partial N}$ is a  homotopy equivalence. Note that a homotopy equivalence between oriented surfaces is homotopic to a homeomorphism.

To summarize, in dimension $n=0, 1, 2$ and $3$, the new definitions $L^{new}_n(\pi_1 X)$, $\mathcal N_n^{new}(X)$ and $\mathcal S_n^{new}(X)$ coincide with their original versions. Moreover, in these dimensions, we can assume that our objects have no
boundary. Indeed, for each \[ \theta = (M, \partial M, \varphi, N, \partial N, \psi, f)\in L^{new}_n(\pi_1 X), \]  we can first assume $f|_{\partial N}$ is a homeomorphism by the above discussion, then glue the element $ (N, \partial N, \psi, N, \partial N, \psi, \textup{Id})$ to $\theta$, that is, glue $(N, \partial N)$ to $(M, \partial M)$ along the boundary and $(N, \partial N)$ to $(N, \partial N)$ along the boundary. Note that the element 
\[(N, \partial N, \psi, N, \partial N, \psi, \textup{Id})\] is equivalent to $0$ in $L^{new}_n(\pi_1 X)$, thus such a gluing does not change the class of $\theta$ in  $L^{new}_n(\pi_1 X)$. In other words, for the trivial group $e$, the group $L_n^{new}(e) = L_n(e)$ is precisely the conventional
manifold bordism group $\Omega_n(G/TOP)$, when $n=0, 1, 2$ or $3$. It follows from the Atiyah-Hirzebruch spectral sequence and the homotopy groups of $G/TOP$ that $L_0(e) = L_1(e) = L_3(e) = 0$, and $L_2(e) = \mathbb Z/2$ (given by the Arf invariant).

The case where $n=4$ is more subtle. Ideally, we would like to realize all intersection forms in the algebraic definition of $L^h_4(e)$ by elements in the geometric definition of $L_4(e)$ (Definition $\ref{def:Lgrp}$). One standard method is to apply the Wall realization (cf. \cite[Chapter 10]{MR1687388}). For any orientation-preserving homotopy equivalence between two closed $3$-manifolds $g\colon  A \to B$,  the Wall realization process applied to $A$ will produce a cobordism $W$ between $A$ and another $3$-manifold $C$ together with a degree one normal map $G\colon W\to B\times I$ such that  $G|_C\colon C\to B = B\times \{1\}$ is a $\mathbb Z\pi$-homology equivalence\footnote{In dimensions $\geq 4$, the Wall realization will actually produce a homotopy equivalence on the other end. However, for $3$-manifolds, we only get a $\mathbb Z\pi$-homology equivalence in general.}. This is the reason that we need the modified version $L^{new}_4(e)$. With this modification, every element of $L^h_4(e)$ can be realized by an element of $L^{new}_4(e)$. More explicitly, 
recall that the Poincar\'e homology sphere  bounds a manifold $E$ with the $E_8$-quadratic form. Now consider a map  $f\colon (E, \partial E) \to (D^4, S^3)$ that induces a homology equivalence from $\partial E$ to $S^3$. This defines a generator of the group $L^h_4(e)$. In particular, all elements of  $L^h_4(e)$ can be realized by using boundary connected sums of $f\colon (E, \partial E) \to (D^4, S^3)$. 

\begin{remark}
	For $n \leq 4$, we have only showed that  $L_n^{new}(e)$ is isomorphic to $L_4^h(e)$ in the case of the trivial group $e$. This is all we need to conclude Theorem $\ref{thm:normal}$ below. We do \emph{not} claim that $L_4^{new}(\pi_1 X)$ is isomorphic $L_4^h(\pi_1 X)$ for a general fundamental group. 
	
\end{remark}

\begin{theorem}\label{thm:normal} We have a natural isomorphism
	\[\mathcal N_n^{new}(X) \cong H_n(X; \mathbb L_\bullet)  \textup{ for all $n\in \mathbb Z$}. \]
\end{theorem}
\begin{proof}
	There is a natural morphism from $\mathcal N^{new}_n(X)$ to $H_n(X; \mathbb L_\bullet)$ by mapping an element of $\mathcal N^{new}_n(X)$ to its corresponding (local) algebraic Poincar\'{e} complex\footnote{Since an object in $\mathcal N_n(X)$ consists of two manifolds with boundary with their boundaries related by an infinitesimal $\mathbb Z\pi$-homology equivalence, the (local) algebraic Poincar\'{e} complex is obtained by gluing the two local relative Poincar\'e complexes along the boundary by this infinitesimal $\mathbb Z\pi$-homology equivalence.}. This is a natural morphism of two generalized homology theories. To show this natural morphism is an isomorphism, it suffices to show that it induces an isomorphism when $X$ is a point. Now if $X$ is a point, then any $\mathbb Z\pi$-homology equivalence automatically has  infinitesimal control. If $n\geq 5$, it follows from  \cite[Chapter 9]{MR1687388} of Wall that 
	\[ \mathcal N_n^{new}(\{pt\}) = L_n^{new}(e) \to H_n(X; \mathbb L_\bullet) = L^h_n(e)\] is an isomorphism. Now by the discussion before the theorem, we have 
	\[ \mathcal N_n^{new}(\{pt\}) =  L_n^{new}(e) = \begin{cases}
	0 & \textup{ if } n=0,\\
	L_n^h(e) & \textup{ if } 1 \leq n\leq 4,
	\end{cases} \] which coincides with  $H_n(\{pt\}; \mathbb L_0)$.  This finishes the proof. 
\end{proof}

Note that the new versions $L_n^{new}(\pi_1 X)$, $\mathcal N_n^{new}(X)$ and $\mathcal S^{new}_n(X)$ were only needed to make sure that these groups indeed give us the topological surgery exact sequence. As we have seen above, when $n\neq 4$, the new version coincides with the original version, and in fact  we shall  exclusively be interested in the case where $n\geq 5$. From now on, if no confusion is likely to arise, we will continue to write $L_n(\pi_1 X)$, $\mathcal N_n(X)$ and $\mathcal S_n(X)$ instead of $L_n^{new}(\pi_1 X)$, $\mathcal N_n^{new}(X)$ and $\mathcal S^{new}_n(X)$.

Now to form the surgery long exact sequence, let us introduce the following relative $L$-groups.

\begin{definition}\label{def:relL}
	An object   
	\[ \theta = (M, \partial_\pm M, \varphi, N, \partial_\pm N,  \psi, f)\] of $L_n(\pi_1X, X, w)$ consists of the following data (see Figure $\ref{fig:relL}$): 
	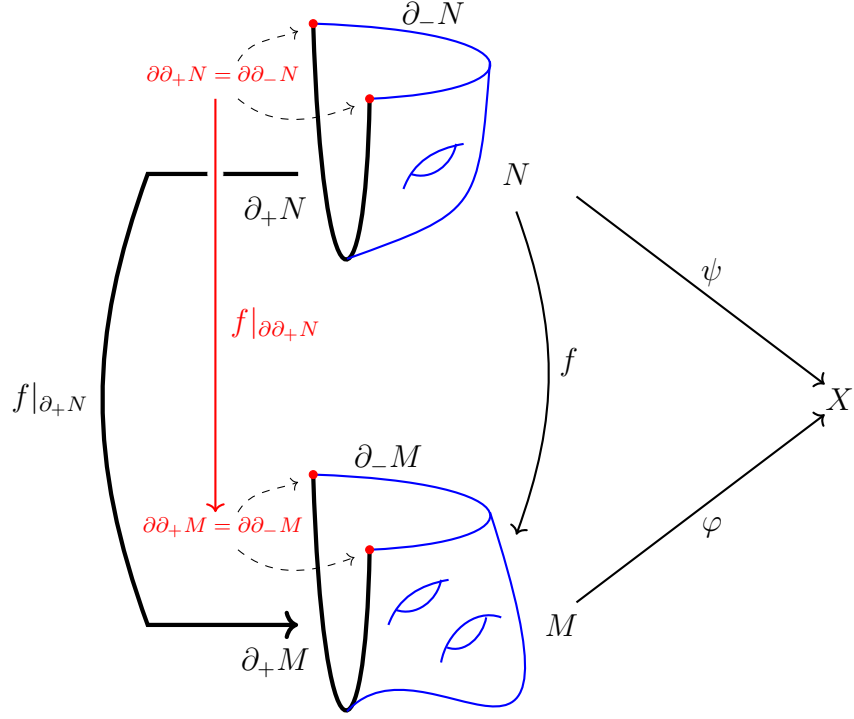
\begin{figure}[H]
		\centering
		\begin{tikzpicture}[scale=1, every node/.style={transform shape}]
		\clip (-5.5,1.5) rectangle (6.5,11.5); 	
		\coordinate (center) at (0, 0);
		\draw[blue, thick] ($(center) +(-1,5)$) .. controls ($(center) + (2,4.9)$) and ($(center) +(2, 4.1)$) .. ($(center) +(-0.25, 4)$);
		
		\draw[ultra  thick] ($(center) +(-1,5)$) .. controls ($(center) +(-0.875, 1)$) and ($(center) +(-0.3, 1)$).. ($(center) +(-0.25, 4)$);
		
		\filldraw [red] ($(center) + (-1,5)$) circle (1.5pt); 			
		\filldraw [red] ($(center) +(-0.25,4)$) circle (1.5pt); 
		
		\draw[blue, thick] ($(center) +(-0.55, 1.85)$) .. controls   ($(center) +(0.5, 3)$) and ($(center) +(2.8, 0)$) ..  ($(center) +(1.35, 4.5)$); 
		\draw[blue, thick] ($(center) +(0, 3)$) to [out=70, in=190] ($(center) +(0.8,3.6)$); 
		\draw[blue, thick] ($(center) +(0.1, 3.2)$) to [out=-20, in=-100] ($(center) +(0.7,3.6)$); 		
		\draw[blue, thick] ($(center) +(0.7, 2.5)$) to [out=80, in=160] ($(center) +(1.5, 3.1)$); 
		\draw[blue, thick] ($(center) +(0.78, 2.7)$) to [out=-20, in=-100] ($(center) +(1.3,3.1)$);
		
		\node [red] at (-2.2, 4.3) {\scriptsize $\partial \partial_+ M = \partial \partial_- M$};
		\draw[->, dashed] (-2, 4.5) to [out=60, in=200] (-1.2, 4.9);
		\draw[->, dashed] (-2, 4) to [out=-40, in=200] (-0.45, 3.9);
		\node at (0, 5.2) {$\partial_- M$};
		\node at (-1.5, 2.5) {$\partial_+ M$};
		\node at (2.3, 3) {$M$};
		\draw [->, thick]	(1.7, 8.5) to [out = -70, in = 70] (1.7, 4.2);
		\node at (2.4, 6.5) {$f$};
		
		\draw [->, ultra thick, name path = a]	(-1.2, 9) -- (-3.2, 9) to [out = -110, in = 110] (-3.2, 3) -- (-1.2, 3);
		\node at (-4.5, 6) {$f|_{\partial_+ N}$};  
		
		\path [name path = b] (-2.3, 10) -- (-2.3, 4.5);
		
		\path [name intersections={of = a and b}];
		\coordinate (S)  at (intersection-1);
		\draw [fill = white, white] (S) circle (1.mm);

		\draw [->, red, thick] (-2.3, 10) -- (-2.3, 4.5);
		\node [red] at (-1.5, 7) {$f|_{\partial \partial_+ N}$} ;

		\node at (6, 6) {$X$}; 	
		\draw[->, thick] (2.5, 8.7) -- (5.8, 6.2);
		\draw[->, thick] (2.5, 3.3) -- (5.8, 5.8);
		\node at (4.3, 4.3) {$\varphi$}; 	
		\node at (4.3, 7.7) {$\psi$}; 	
		
		
		\begin{scope}[shift = {(0, 6)}]
		\coordinate (center) at (0, 0);
		\draw[blue, thick] ($(center) +(-1,5)$) .. controls ($(center) + (2,4.9)$) and ($(center) +(2, 4.1)$) .. ($(center) +(-0.25, 4)$);
		
		\draw[ultra thick] ($(center) +(-1,5)$) .. controls ($(center) +(-0.875, 1)$) and ($(center) +(-0.3, 1)$).. ($(center) +(-0.25, 4)$);
		
		\draw[blue, thick] ($(center) + (-0.55, 1.87)$) .. controls ($(center) +(1.2, 2.5)$) ..  ($(center) +(1.35, 4.5)$); 
		
		\filldraw [red] ($(center) + (-1,5)$) circle (1.5pt); 			
		\filldraw [red] ($(center) +(-0.25,4)$) circle (1.5pt); 
		
		\draw[blue,thick] ($(center) +(0.2, 2.8)$) to [out=70, in=190] ($(center) +(1,3.4)$); 
		\draw[blue,thick] ($(center) +(0.3, 3)$) to [out=-20, in=-100] ($(center) +(0.9,3.4)$); 		
		\node at ($(center) + (0.6, 5.1)$) {$\partial_- N$};
		\node at ($(center) +  (-1.5, 2.5)$) {$\partial_+ N$};
		
		\node [red] at ($(center) + (-2.2, 4.3)$) {\scriptsize $\partial \partial_+ N = \partial \partial_- N$};
		\draw[->, dashed]  ($(center) + (-2, 4.5)$) to [out=60, in=200]  ($(center) +(-1.2, 4.9)$);
		\draw[->, dashed]  ($(center) +(-2, 4)$) to [out=-40, in=200]  ($(center) +(-0.45, 3.9)$);
		\node at  ($(center) + (1.7, 3)$) {$N$};
		\end{scope}
		\end{tikzpicture}
		\vspace{2mm}
		\caption{An object $\theta = (M, \partial_\pm M, \varphi, N, \partial_\pm N,  \psi, f)$ of $L_n(\pi_1X, X, w)$,  where $f|_{\partial \partial_+ N}$ is an infinitesimally controlled homotopy equivalence, $f$ is  a degree one normal map, and $f|_{\partial_+ N}$ is a homotopy equivalence.}
		\label{fig:relL}
	\end{figure}

	\begin{enumerate}[(1)]
		\item two manifold $2$-ads $(M, \partial_\pm M) $ and $(N, \partial_\pm N)$ of $\dim M = \dim N = n$, with $\partial M = \partial_+ M\cup \partial_- M$ (resp. $\partial N = \partial_+ N\cup \partial_- N$) the boundary of $M$ (resp. $\partial N$),  in particular, $\partial (\partial_+ M) = \partial (\partial_- M)$ and $\partial (\partial_+ N) = \partial (\partial_- N)$;
		\item continuous maps $\varphi\colon M \to X$ and $\psi\colon N \to X$ so that the pullback $\mathbb Z/2$-principal bundles $\varphi^\ast(w)$ (resp. $\psi^\ast(w)$) is the orientation covering of $M$ (resp. $N$) respectively; 
		\item a degree one normal map of manifold $2$-ads \[ f\colon (N, \partial N) \to (M, \partial M)\] such that $\varphi \circ f = \psi$; 
		\item the restriction $f|_{\partial_+ N}\colon \partial_+ N \to \partial_+ M $ is a  homotopy equivalence over $X$ such that $f_\ast[\partial_+ N] = [\partial_+ M]$; 
		\item the restriction $f|_{\partial_- N}\colon \partial_- N \to \partial_- M $ is a degree one normal map over $X$; 
		\item the homotopy equivalence $f|_{\partial_+ N}$ restricts to an \emph{infinitesimally controlled} homotopy equivalence 
		$ f|_{\partial(\partial_\pm N)}\colon \partial (\partial_\pm N) \to \partial (\partial_\pm M)$ over $X$.
	\end{enumerate} 
\end{definition}

\begin{definition}\label{def:equirel}
	The equivalence relation for defining  $L_n(\pi_1X, X, w)$ is given as follows. 
	Let \[ \theta = (M, \partial_\pm M, \varphi, N, \partial_\pm N,  \psi, f) \]
	be an object from Definition $\ref{def:relL}$ above. We write $\theta \sim 0$ if the following conditions are satisfied. 
	
	\begin{figure}[h]
		\begin{tikzpicture}[scale=0.85, every node/.style={transform shape}]
		\coordinate (center) at (4, 5);
		\shade[shading = ball, ball color=gray!60!white, opacity=1] (center) circle [radius =2];

		\begin{scope}[rotate around = {-20:(center) }]
		\coordinate (h1) at (4.4, 6);
		\coordinate (h2) at (5, 6.4);
		\draw[fill=white, ultra thick] (h1) to [out=70, in=190] (h2) to [out = -100, in = -20] (h1) ; 
		\draw [ultra thick] ($(h1)-(0.05, 0.2)$) to [out=90, in=-110] (h1);
		\draw [ultra thick] (h2) to [out=10, in= 185] ($(h2)+(0.2, 0.03)$); 
		\end{scope}   
		\draw[ red, thick, rotate around={-70:(center)}, name path = c1] ($(center)-(2,0)$) arc (180:281:2 and 0.5);
		\draw[red,  thick, dashed, rotate around={-70:(center)}, name path = dc1] ($(center)-(2,0)$) arc (180:97:2 and 0.5);
		
		\draw[blue, thick, rotate around={70:(center)}] ($(center)-(2,0)$) arc (180:106:2 and 0.17); 
		\draw[blue, dashed, thick, rotate around={-110:(center)}] ($(center)+(2,0)$) arc (0:110:2 and 0.17); 
		
		\draw[ultra thick, name path = c2] ($(center)+(2,0)$) arc (360:259:2 and 0.5);
		\draw[ultra thick, dashed, name path = dc2]($(center)+(2,0)$) arc (0:83: 2 and 0.5);

		\fill[red, name intersections={of=c1 and c2, by = i1}]
		(i1) circle (1.5pt); 
		\fill[red, name intersections={of=dc1 and dc2, by = i2}]
		(i2) circle (1.5pt);

		\begin{scope}[rotate around = {100:(center) }]
		\coordinate (h1) at (4.4, 6);
		\coordinate (h2) at (5, 6.4);
		\draw[blue, fill=white, thick] (h1) to [out=70, in=190] (h2) to [out=-100, in=-20] (h1); 
		\draw [blue, thick] ($(h1)-(0.05, 0.2)$) to [out=90, in=-110] (h1);
		\draw [blue, thick] (h2) to [out=10, in= 185] ($(h2)+(0.2, 0.03)$); 
		\end{scope} 
		\begin{scope}[shift ={(-2, -1) }]
		\begin{scope}[rotate around = {55:(center) }]
		\coordinate (h1) at (4.4, 6);
		\coordinate (h2) at (5, 6.4);
		\draw[blue, fill=white, thick] (h1) to [out=70, in=190] (h2) to [out = -100, in = -20] (h1); 
		\draw [blue, thick] ($(h1)-(0.05, 0.2)$) to [out=90, in=-110] (h1);
		\draw [blue, thick] (h2) to [out=10, in= 185] ($(h2)+(0.2, 0.03)$); 
		\end{scope}
		\end{scope}

		\begin{scope}[rotate around = {250:(center) }]
		\coordinate (h1) at (4.4, 6);
		\coordinate (h2) at (5, 6.4);
		\draw[blue, fill=white, thick] (h1) to [out=70, in=190] (h2) to [out = -100, in = -20] (h1); 
		\draw [blue, thick] ($(h1)-(0.05, 0.2)$) to [out=90, in=-110] (h1);
		\draw [blue, thick] (h2) to [out=10, in= 185] ($(h2)+(0.2, 0.03)$); 
		\end{scope}

		\begin{scope}[shift = {(4, 5)}, rotate around = {230:(center)} ]
		\begin{scope}[scale = 0.8]
		\begin{scope}[shift = {(-4, -5)}]
		\coordinate (h1) at (4.4, 6);
		\coordinate (h2) at (5, 6.4);
		\draw[blue, fill=white, thick] (h1) to [out=70, in=190] (h2) to [out = -100, in = -20] (h1); 
		\draw [blue, thick] ($(h1)-(0.05, 0.2)$) to [out=90, in=-110] (h1);
		\draw [blue, thick] (h2) to [out=10, in= 185] ($(h2)+(0.2, 0.03)$); 
		
		\end{scope}   
		\end{scope}
		\end{scope}
		
		
		\node at (5.7, 3.3) {$N$};
		\node at (5.6, 6.8) {$\partial_2 V$};
		\node at (7.2, 4.3) {\scriptsize $\partial_+ N = \partial \partial_{2, +}V$};
		\draw [->, dashed] (6.8, 4.6) to [out = 100, in = -20] (6.1, 5);
		\node at (3, 2.2) {\scriptsize $\partial_- N = \partial \partial_{3, -}V$};
		\draw [->, dashed] (3, 2.4) to [out = 120, in = -150] (3.25, 3.1);
		\node at (1.8, 3.8) {$\partial_{3}V$};
		\node [draw, red] at (1.7, 7.5) {\scriptsize $\partial \partial_{2, -}V = \partial\partial_{3,+} V$};
		\draw [->, dashed] (1.7, 7.2) to [out = -100, in = 150] (3.3, 6.95); 
		
		\node at (3.9, 5) {$V$};

		
		\begin{scope}[shift = {(7, 0 )}] 
		\coordinate (center) at (4, 5);
		\draw[thick]  (center) circle [radius = 2];
		\shade[shading = ball, ball color=gray!60!white,opacity=1] (center) circle [radius =2]; 
		
		\begin{scope}[rotate around = {90:(center)}, scale = 0.9, shift = {(0.8, -1.8)}]
		\coordinate (h1) at (4.4, 6);
		\coordinate (h2) at (5, 6.4);
		\draw[fill=white, ultra thick] (h1) to [out=70, in=190] (h2) to [out=-100, in=-20] (h1); 
		\draw [ultra thick] ($(h1)-(0.05, 0.2)$) to [out=90, in=-110] (h1);
		\draw [ultra thick] (h2) to [out=10, in= 185] ($(h2)+(0.2, 0.03)$);  
		\end{scope}   
		\draw[red, thick, rotate around={-70:(center)}, name path = c1] ($(center)-(2,0)$) arc (180:281:2 and 0.5);
		\draw[ red, thick, dashed, rotate around={-70:(center)}, name path = dc1] ($(center)-(2,0)$) arc (180:97:2 and 0.5);
		\draw[blue, dashed, thick, rotate around={-110:(center)}] ($(center)+(2,0)$) arc (0:110:2 and 0.17); 
		\draw[blue, thick, rotate around={70:(center)}] ($(center)-(2,0)$) arc (180:106:2 and 0.17); 
		
		\draw[ultra thick, name path = c2] ($(center)+(2,0)$) arc (360:259:2 and 0.5);
		\draw[ultra thick, dashed, name path = dc2]($(center)+(2,0)$) arc (0:83: 2 and 0.5);

		\fill[red, name intersections={of=c1 and c2, by = i1}]
		(i1) circle (1.5pt); 
		\fill[red, name intersections={of=dc1 and dc2, by = i2}]
		(i2) circle (1.5pt);
		
		%
		\begin{scope}[rotate around = {120:(center) }]
		\coordinate (h1) at (4.4, 6);
		\coordinate (h2) at (5, 6.4);
		\draw[blue, fill=white, thick] (h1) to [out=70, in=190] (h2) to [out=-100, in=-20] (h1); 
		\draw [blue, thick] ($(h1)-(0.05, 0.2)$) to [out=90, in=-110] (h1);
		\draw [blue, thick] (h2) to [out=10, in= 185] ($(h2)+(0.1, 0.015)$); 
		\end{scope}

		\begin{scope}[shift = {(4, 5)}, rotate around = {230:(center)} ]
		\begin{scope}[scale = 0.8]
		\begin{scope}[shift = {(-4, -5)}]
		\coordinate (h1) at (4.4, 6);
		\coordinate (h2) at (5, 6.4);
		\draw[blue, fill=white, thick] (h1) to [out=70, in=190] (h2)  to [out=-100, in=-20] (h1); 
		\draw [blue, thick] ($(h1)-(0.05, 0.2)$) to [out=90, in=-110] (h1);
		\draw [blue, thick] (h2) to [out=10, in= 185] ($(h2)+(0.2, 0.03)$);  
		\end{scope}   
		\end{scope}
		\end{scope}
		
		\node at (5.7, 3.3) {$M$};
		\node at (5.6, 6.8) {$\partial_2 W$};
		\node at (7.2, 4.3) {\scriptsize $\partial_+ M = \partial \partial_{2, +}W$};
		\draw [->, dashed] (6.8, 4.6) to [out = 100, in = -20] (6.1, 5);
		\node at (3.7, 2.2) {\scriptsize $\partial_- M = \partial \partial_{3, -}W$};
		\draw [->, dashed] (3.5, 2.4) to [out = 150, in = -130] (3.25, 3.1);
		\node at (1.8, 3.8) {$\partial_{3}W$};
		\node [draw, red] at (1.7, 7.5) {\scriptsize $\partial \partial_{2, -}W = \partial\partial_{3,+} W$};
		\draw [->, dashed] (1.7, 7.2) to [out = -100, in = 150] (3.3, 6.95); 
		
		\end{scope}
		
		\draw[->, thick] (6.5, 5.5) -- (8.5, 5.5);	
		\node at (7.5, 5.8) {$F$};

		\draw[name path = a, ->, very thick]  (4.5, 7.1) to [out = 45, in = 135] (11.5, 7.1);  
		\node at (8.5, 8.8) {$F|_{\partial_2 V}$ is a homotopy equivalence};
		
		\node at (10.9, 5) {$W$};

		\path [name path = b] (3.3, 7.5) -- (7, 7.5);
		
		\path [name intersections={of = a and b}];
		\coordinate (S)  at (intersection-1);
		\draw [fill = white, white] (S) circle (1.mm); 
		
		\draw[->, red, thick]  (3.3, 7.5) -- (7, 7.5); 
		
		
		\draw [->, thick] (4.7, 2.9) -- (7.3, 0.7);
		\draw [->, thick] (9.5, 3.3) -- (7.7, 0.7);
		\node at (7.5, 0.5) {$X$};
		\node at (5.7, 1.5) {$\Psi$};
		\node at (8.6, 1.5) {$\Phi$};
		\end{tikzpicture}
		\caption{Equivalence relation  $\theta \sim 0$ for the definition of $L_n(\pi_1X, X,  w)$, where $F|_{\partial\partial_{2,-}V}$ is an infinitesimally controlled homotopy equivalence and  $F|_{\partial_2 V}$ is a homotopy equivalence.} 	
	\end{figure}
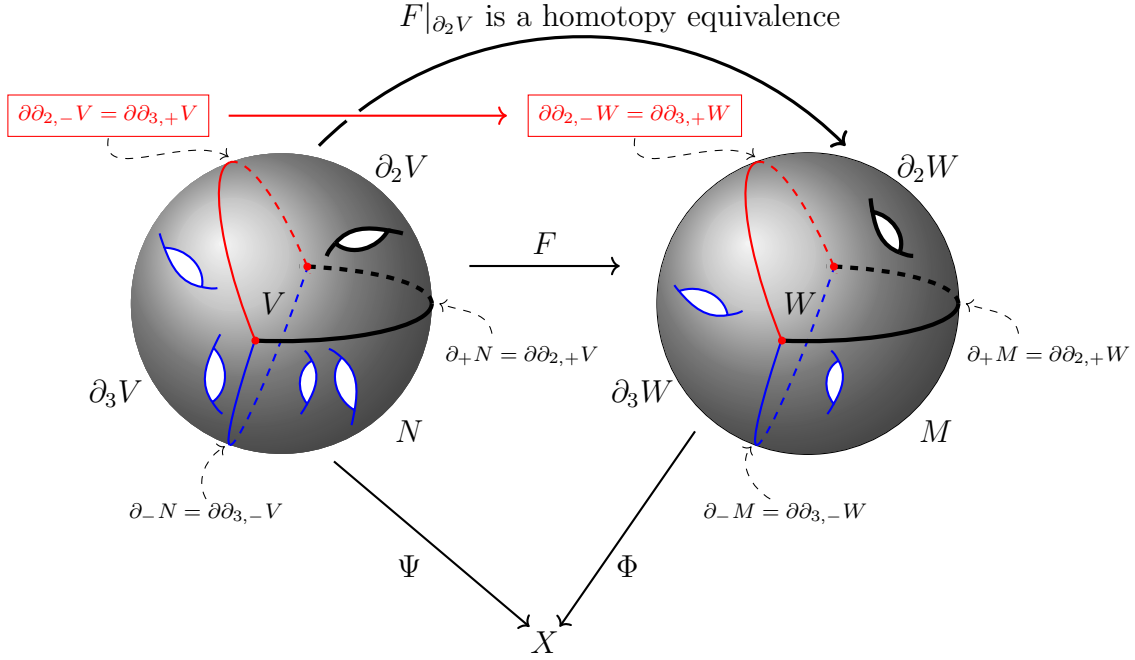

	\begin{enumerate}[(1)]
		\item There exists a manifold $3$-ad $(W, \partial W)$ of dimension $(n+1)$ with a continuous map $\Phi\colon W\to X$ so that the pullback $\mathbb Z/2$-principal bundle $\Phi^\ast(w)$ is the orientation covering of $W$. Here $\partial W$ is the union of $M$, $\partial_2 W$ and $\partial_3 W$, and $\Phi|_M = \varphi$.  Moreover, we have decompositions   
		\[ \partial M =\partial_+ M \cup \partial_- M, \] \[ \partial(\partial_2 W) = \partial\partial_{2,+} W \cup \partial\partial_{2,-} W\]    \[\textup{ and } \partial(\partial_3 W) = \partial\partial_{3,+} W \cup \partial\partial_{3, -} W\] such that 
		$ \partial_+ M = \partial\partial_{2,+} W,  \partial_- M = \partial\partial_{3, -} W $ and $ \partial\partial_{2, -}W = \partial\partial_{3, +}W.$
		Furthermore, we have 
		\begin{align*}
		\partial_+ M \cap  \partial_- M  = \partial\partial_{2, +} W\cap \partial\partial_{2, -} W  = \partial\partial_{3, +} W\cap \partial\partial_{3, -} W. 
		\end{align*}  
		
		\item Similarly, we have a manifold $3$-ad $(V, \partial V)$ of dimension $(n+1)$ with a continuous map $\Psi: V\to X$ so that the pullback $\mathbb Z/2$-principal bundle $\Psi^\ast(w)$ is the orientation covering of $V$. Moreover,  $\partial V = N \cup \partial_2 V \cup \partial_3 V$ satisfies similar conditions as $W$. 
		\item There is a degree one normal map of manifold $3$-ads 
		\[ F\colon (V, \partial V) \to (W, \partial W)\] such that $\Phi\circ F  = \Psi$ and  $F|_N = f$.
		\item  the map $F|_{\partial_2 V}\colon \partial_2 V \to \partial_2 W$ is a homotopy equivalence over $X$ such that $F_\ast[\partial_2 V] = [\partial_2 W]$.
		\item $F$ restricts to an \emph{infinitesimally controlled} homotopy equivalence 
		\[ F|_{\partial\partial_{2,-}V }\colon  \partial\partial_{2,-}V \to \partial\partial_{2,-} W\]
		over $X$.
		
	\end{enumerate}
\end{definition}

We denote by $L_n(\pi_1X, X; w)$ the set of equivalence classes from Definition $\ref{def:relL}$. Note that $L_n(\pi_1X, X; w)$ is an abelian group with the addition given by disjoint union. We also can make this relative $L$-groups into a $4$-periodic theory. 

\begin{definition}\label{def:periodrelL}
	For each $n\in \mathbb Z$, we define $ \mathfrak L_{n}(\pi_1 X, X, w)$ to be the direct limit of 
	\[
	\scalebox{1}{$ \cdots \to  L_{k}(\pi_1 X, X, w) \xrightarrow{\times \mathbb{CP}^2}   L_{k+4}(\pi_1 X, X, w) \xrightarrow{\times \mathbb{CP}^2} L_{k+8}(\pi_1 X, X, w) \to \cdots,$}
	\]
	where $k \equiv n \pmod{4}$. 
\end{definition}

Note that every element 
\[ \theta = (M, \partial M, \varphi, N, \partial N, \psi, f) \in L_n(\pi_1 X, w)\]
naturally defines an element in $L_n(\pi_1 X, X,  w)$ by letting $\partial_- M = \emptyset$. We denote this map by 
\[ j_\ast\colon L_n(\pi_1 X, w) \to  L_n(\pi_1 X, X,  w).  \]
For every element 
\[ \theta = (M, \partial_\pm M, \varphi, N, \partial_\pm N,  \psi, f)\in L_n(\pi_1 X, X, w), \] we see that 
\[ \theta_- \coloneqq  \{ \partial_- M, \partial(\partial_- M), \varphi, \partial_- N, \partial(\partial_- N), \psi, f\}  \]
defines an element in $\mathcal N_{n-1}(X,  w)$. We call $\theta_-$ the $(-)$-boundary of $\theta$.

Various groups defined above fit into the following long exact sequence. The proof is essentially identical to that of Theorem $9.6$ of Wall's Chapter 9 \cite{MR1687388}.

\begin{theorem}\label{thm:les}
	We have the following long exact sequence: 
	\begin{align*}
	\cdots & \to \mathcal N_n(X, w) \xrightarrow{i_\ast} L_n(\pi_1 X, w) \xrightarrow{j_\ast} L_n(\pi_1 X, X, w) \\
	& \xrightarrow{\partial_\ast} \mathcal N_{n-1}(X, w) \to L_{n-1}(\pi_1 X, w) \to \cdots  
	\end{align*}
	where  $i_\ast \colon \mathcal N_n(X, w) \to  L_n(\pi_1 X,  w)$ is the natural map defined 
	by forgetting infinitesimal control, and the map $\partial_\ast$ maps each element of $L_n(\pi_1 X, X, w)$ to its $(-)$-boundary, that is, if 
	$ \theta = \{M, \partial_\pm M, \varphi, N, \partial_\pm N,  \psi, f\}\in L_n(\pi_1 X, X, w),$ then
	\[  \partial_\ast(\theta)  = \theta_- \coloneqq  \{ \partial_- M, \partial(\partial_- M), \varphi, \partial_- N, \partial(\partial_- N), \psi, f\}. \]
	Consequently, we have the following $4$-periodic (for all $n\in \mathbb Z$)  long exact sequence 
	\begin{equation}\label{perilong}
	\begin{gathered}
	\cdots \to \mathfrak N_n(X, w) \xrightarrow{i_\ast} \mathfrak L_n(\pi_1 X, w) \xrightarrow{j_\ast} \mathfrak L_n(\pi_1 X, X, w) \\
	\xrightarrow{\partial_\ast} \mathfrak N_{n-1}(X, w) \to \mathfrak L_{n-1}(\pi_1 X, w) \to  \cdots 
	\end{gathered} 
	\end{equation}
\end{theorem}
\begin{proof}
	It is easy to see that the map $\partial_\ast$ is well-defined. It remains to prove the exactness of the sequence. 
	
	An element in $\mathcal N_{n-1}(X, w)$ maps to zero in $L_{n-1}(\pi_1 X, w)$ if and only if its image  is cobordant to the empty set in $L_{n-1}(\pi_1 X, w)$.  However, such a cobordism  defines an element in $L_n(\pi_1 X, X,  w)$. This proves the exactness at $\mathcal N_{n-1}(X)$. 
	
	Note that $\partial_\ast j_\ast = 0$ by definition. On the other hand, given an element $\theta$ in $ L_n(\pi_1 X, X,  w)$, if $\theta$ maps to zero in $\mathcal N_{n-1}(X,  w)$, then we can take a cobordism of $\partial_\ast(\theta)$ to the empty set, and glue it to $\theta$ along $\partial_\ast(\theta)$. The resulting new element is cobordant to $\theta$, and its $(-)$-boundary is empty, so it lies in the image of the map $j_\ast$. This proves the exactness at $L_n(\pi_1 X, X,  w)$. 
	
	Finally, let us prove the exactness at $L_n(\pi_1 X, w)$. Suppose  
	\[ \xi = \{M, \partial M, \varphi, N, \partial N, \psi, f\}\] is an element of $\mathcal N_n(X,  w)$, then we shall show that it is cobordant to zero in $L_n(\pi_1 X, X, w)$. Indeed, a cobordism of $\xi$ to the empty set is provided by $\xi \times I$ where $I$ is the unit interval. More precisely, $\xi\times I$ consists of the following data.
	\begin{enumerate}[(i)]
		\item 	$W = M\times I $ with a continuous map 
		\[ \Phi = \varphi\colon W\xrightarrow{p_1} M \to X, \] where $p_1\colon W \to M$ is the projection map onto $M$. Moreover,  
		\[ \partial W = M \cup \partial_2 W \cup \partial_3 W\] with $\partial_2 W = \partial M \times I$ and  $\partial_3 W = M$. 
		\item Similarly,  $V = N\times I$ with a continuous map  
		\[ \Psi = \psi\colon V\xrightarrow{q_1} N \to X, \] where $q_1 \colon V \to N$ is the projection map onto $N$.  Moreover, 
		\[ \partial V = N \cup \partial_2 V \cup \partial_3 V, \] where $\partial_2 V= \partial N \times I $ and $\partial_3 V = N$. 
		\item  A degree one normal map 
		\[ F = f\times \id\colon (V, \partial V) \to (W, \partial W)\]
		such that $\Phi\circ F = \Psi$ and $F|_N = f$.
		\item $F|_{\partial_{2} V}\colon \partial_{2} V = \partial N \times I \to \partial_2 W = \partial M \times I$ is a homotopy equivalence such that $F_\ast[\partial_2 V] = [\partial_2 W]$. This is because $f\colon \partial N \to \partial M$ is an infinitesimally controlled homotopy equivalence, thus in particular a homotopy equivalence. 
		\item $F$ restricts to  an \emph{infinitesimally controlled} homotopy equivalence  \[ F|_{\partial\partial_{2,-}V }\colon  \partial\partial_{2,-}V = \partial N  \to \partial\partial_{2,-} W = \partial M\] over $X$. 
	\end{enumerate} 
	This proves that $j_\ast i_\ast = 0$. Conversely, suppose an element 
	\[ \theta = \{M, \partial M, \varphi, N, \partial N, \psi, f\} \in L_n(\pi_1 X, w)\] maps to zero in $L_n(\pi_1 X, X, w)$, that is, $\theta$ is cobordant to $0$ in $L_n(\pi_1 X, X, w)$. Let us use the same notation as in Definition $\ref{def:equirel}$. In our current case,   we have $\partial M =\partial_+ M$ with $\partial_- M = \emptyset$, $\partial\partial_2 W = \partial\partial_{2,+} W \cup \partial\partial_{2,-} W$, and   $\partial\partial_3 W = \partial\partial_{3,+} W $ with $ \partial\partial_{3, -} W = \emptyset$ such that 
	\[ \partial_+ M = \partial\partial_{2,+} W \textup{ and } \partial_{2, -}W = \partial\partial_{3, +}W.  \]
	Moreover, we have $\partial \partial_{2, +} W\cap \partial\partial_{2, -} W = \emptyset.$ Similar conditions also apply to $V$.  It follows that  $F\colon (V, \partial V) \to (W, \partial W)$ is a cobordism between $\theta$ and the element
	\[ \eta = (\partial_3 W, \partial \partial_3 W, \Phi|_{\partial_3 W}, \partial_3 V, \partial \partial_3 V, \Psi|_{\partial_3 V}, F|_{\partial_3 W}).\] Note that $\eta$ is an element of $\mathcal N_n(X; w)$. This finishes the proof. 
\end{proof}

\subsection{Relations between various structure groups}\label{sec:iden}

In this subsection, we prove  $\mathcal S_{n}(X; w)$ is naturally isomorphic to $L_{n+1}(\pi_1 X, X, w)$. Moreover,  when $X$ is a closed topological manifold of dimension $n\geq 5$, we show that $\mathcal S^\topo(X, w)$ is naturally isomorphic to  $\mathcal S_{n}(X; w)$.  

Consider the natural homomorphism 
\[  c_\ast\colon  \mathcal S_{n}(X, w) \to L_{n+1}(\pi_1 X, X, w) \]
given by 
\[ \theta = \{M, \partial M, \varphi, N, \partial N, \psi, f\} \mapsto \theta\times I, \] 
where $\theta \times I$ consists of the following data:
\begin{enumerate}[(1)]
	\item a manifold $2$-ad $(M \times I, \partial_\pm (M\times I)) $ with 
	\[  \partial_+ (M\times I) = M = \partial_- (M\times I);\] in particular, we have $\partial \partial_+(M\times I) = \partial M= \partial \partial_-(M\times I)$;
	\item similarly, another manifold $2$-ad $(N \times I, \partial_\pm (N\times I)) $ with 
	\[  \partial_+ (N\times I) = N = \partial_- (N\times I); \]
	\item a continuous map $\widetilde \varphi \coloneqq  \varphi\circ p\colon M\times I \xrightarrow{p_1} M \xrightarrow{\varphi} X$  such that the pullback $\mathbb Z/2$-principal bundle  $(\varphi\circ p)^\ast(w)$ $M\times I $ is the orientation covering of $M\times I$, where $p$ is the projection map from $M\times I$ to $M$; similarly, a continuous map 
	\[ \widetilde \psi \coloneqq  \psi\circ q\colon N\times I \xrightarrow{q} N \xrightarrow{\psi} X\]  such that the pullback $\mathbb Z/2$-principal bundle $(\psi\circ q)^\ast(w)$ is the orientation covering of $N\times I $, where $q$ is the projection map from $N\times I$ to $N$; 
	\item a degree one normal map of manifold $2$-ads  \[\hspace{1cm} \widetilde f \coloneqq f\times \id \colon (N\times I, \partial_\pm (N\times I) ) \to (M, \partial_\pm (M\times I) )\] such that $\widetilde \varphi \circ \widetilde f = \widetilde \psi$;
	\item the restriction 
	\[ \widetilde f|_{\partial_+ (N\times I)}\colon \partial_+ (N\times I) = N \to \partial_+ (M\times I) = M \] is a  homotopy equivalence over $X$ such that $\widetilde f_\ast [N] = [M]$;
	\item the restriction $\widetilde f|_{\partial_- (N\times I)}\colon N \to M $ is a degree one normal map over $X$; here we recall that every homotopy equivalence naturally defines a degree one normal map; 
	\item $\widetilde f$ restricts to an \emph{infinitesimally controlled} homotopy equivalence 
	\[ \widetilde f|_{\partial \partial_\pm (N\times I)}\colon \partial \partial_\pm (N\times I) = \partial N \to \partial \partial_\pm (M\times I) = \partial M \] over $X$.
\end{enumerate} 

There is also  a natural homomorphism 
\[  r_\ast\colon  L_{n+1}(\pi_1 X, X, w) \to  \mathcal S_{n}(X,  w)\] 	
by taking the $(+)$-boundary of an element, that is, 
\[\theta = \{M, \partial_\pm M, \varphi, N, \partial_\pm N, \psi, f\}\in L_{n+1}(\pi_1 X, X, w)\]
maps to 
\[ \theta_+ = \{\partial_+ M, \partial \partial_+ M, \varphi, \partial_+ N, \partial\partial_+ N, \psi, f\}\in \mathcal S_{n}(X,  w).  \] 
It is easy to see that $c_\ast$ and $r_\ast$ are well-defined. 
\begin{proposition}\label{prop:isorel} 
	The homomorphisms $c_\ast$ and $r_\ast$ are inverses of each other. In particular, we have $ \mathcal S_{n}(X, w) \cong L_{n+1}(\pi_1 X, X, w)$. 
\end{proposition}
\begin{proof}
	Clearly, we have $r_\ast \circ c_\ast = 1$. Conversely, if 
	\[ \theta = \{M, \partial_\pm M, \varphi, N, \partial_\pm N, \psi, f\}\] is an element in $L_{n+1}(\pi_1 X, X,  w)$, then  $c_\ast\circ r_\ast(\theta)$ is cobordant to $\theta$ in $L_{n+1}(\pi_1 X, X, w)$. Indeed, consider the element 
	\[ (\theta\times I)  \bigcup_{(\theta_+\times I)\times \{0\}\, \subset\, \theta\times \{1\} } (\theta_+\times I\times I)  \]
	where $(\theta_+\times I)\times \{0\}$ is glued to the subset $(\theta_+\times I) \subset \theta$ in $\theta\times \{1\}$. This 
	produces a cobordism between $c_\ast\circ r_\ast(\theta) = (\theta_+ \times I)\times \{1\}$ and $\theta \times \{0\}$. In other words, we have $c_\ast\circ r_\ast = 1$. This finishes the proof. 
\end{proof}

Now we shall use the surgery long exact sequence to  identify  $\mathcal S_{n}(X; w)$ with $\mathcal S^\topo(X; w)$. Consider the  natural map
\[  \iota_\ast\colon \mathcal S^\topo(X; w) \to \mathcal S_{n}(X; w)  \]
by 
\[ [\varphi \colon M \to X] \mapsto \theta = \{M, \partial M =\emptyset, \varphi, X, \partial X = \emptyset, \id, f = \varphi\}. \]
It is easy to see that $\iota_\ast$ is a well-defined map of sets.

For notational simplicity, we will work with the case where $X$ is oriented and the $\mathbb Z/2$-principal  bundle  $w$ on $X$ is trivial. The same argument also works for the general case. Recall that, for $n = \dim X \geq 5$, we have the following geometric surgery long exact sequence 
\begin{equation}\label{eq:sles}
\scalebox{0.95}{$\begin{aligned}
	\cdots \to & L^h_{n+i+1}(\pi_1 X) \to \mathcal S^\topo_\partial (X\times D^i)  \\
	\to &  \mathcal N^\topo_{\partial}(X\times D^i) \to L^h_{n+i}(\pi_1 X)  \to   \\
	\cdots \to&   L^h_{n+1}(\pi_1 X) \dashrightarrow \mathcal S^\topo(X)\to   \mathcal N^\topo(X) \to L^h_{n}(\pi_1 X)
	\end{aligned}$}
\end{equation}
where 
\begin{enumerate}[(a)]
	\item $D^i$ is the $i$-dimensional Euclidean unit ball;
	\item $\mathcal S^\topo_\partial$ is the $\rel\partial$ version\footnote{$\rel\partial$ means ``relative to boundary".} of structure set, whose definition we shall review below;
	\item $\mathcal N^\topo$ is the set of normal invariants, and $\mathcal N^\topo_\partial$ is its $\rel\partial$ version;
	\item  and the map $L^h_{n+1}(\pi_1 X) \dashrightarrow \mathcal S^\topo(X)$ is a natural action of $L^h_{n+1}(\pi_1 X)$ on $\mathcal S^\topo(X)$.
\end{enumerate}  
Moreover, all terms starting from $L^h_{n+1}(\pi_1X)$ to the left, in the sequence  $\eqref{eq:sles}$,  are abelian groups,  and all arrows to the left of $L^h_{n+1}(\pi_1X)$ are group homomorphisms. 

\begin{definition}\label{def:relbdry}
	Suppose $Y$ is an oriented compact manifold with boundary $\partial Y$. We define $ S^\topo_\partial(Y)$ to be the set of equivalence classes of orientation-preserving homotopy equivalences 
	\[ f\colon (M, \partial M) \to (Y, \partial Y)\] from compact manifolds with boundary such that $f\colon \partial M \to \partial Y$ is a homeomorphism. Two orientation-preserving homotopy equivalences $f\colon (M, \partial M) \to (Y, \partial Y)$ and $g\colon (N, \partial N) \to (Y, \partial Y)$ are equivalent if there exists a $\rel\partial$-h-cobordism\footnote{In particular, $ \partial W = M\cup_{\partial M}\partial Y\times I\cup_{\partial N} N$. } $ (W; M, N)$ with  an orientation-preserving homotopy equivalence 
	\[ F\colon (W; M,  N) \to (Y\times I; Y\times\{0\}, Y\times\{1\})\]
	such that $F|_M = f$, $F|_N = g$ and $F|_{\partial Y \times I}$ is a homeomorphism, where $\partial Y\times I$ is the part of $\partial W$ that sits between $\partial M$ and $\partial N$.
\end{definition}

Similar to the definition of $\iota_\ast\colon \mathcal S^\topo(X)  \to \mathcal S_{n}(X)$,   there is a natural map 
\[  \beta_\ast\colon \mathcal S^\topo_\partial (X\times D^i)  \to \mathcal S_{n+i}(X) \]
for all $i\geq 1$, by mapping 
\[ \{\varphi \colon (M, \partial M)  \to (X\times D^i, X\times S^{i-1})  \in \mathcal S^\topo_\partial (X\times D^i)\] to 
\[ \theta = \{M, \partial M, p\circ \varphi, X\times D^i, X\times S^{i-1}, p, f = \varphi\} \in \mathcal S_{n+i}(X) \]
where $p\colon X\times D^i \to X$ is the projection of $X\times D^i$ onto $X$. When $i\geq 1$, there is a natural group structure on $ \mathcal S^\topo_\partial (X\times D^i) $,   which is geometrically defined by stacking. Let us review the definition of stacking (cf. \cite{MR873283}, \cite[Definition 2.4]{MR2826928}). 
For $i\geq 1$, let us denote $ S^{i-1}_{\pm} \coloneqq  \{ (x_1, \cdots, x_i)\in S^{i-1} = \partial D^i \mid \pm x_1\geq 0\},$
and 
\[  D^{i}_{\pm} \coloneqq \{ (x_1, \cdots, x_i)\in D^i \mid \pm x_1\geq 0\}.  \]
Fix suitable homeomorphisms 
\[ (D^i, S_+^{i-1}, S_-^{i-1}) \cong (D_+^i, S^{i-1}_+, D^{k-1})\] and 
\[ (D^i, S_+^{i-1}, S_-^{i-1}) \cong (D_-^i, S^{i-1}_-, D^{k-1}).\] We see that $D^i = D^i_+ \cup_{D^{i-1}} D^i_-$. 
\begin{definition}\label{def:stack}
	Suppose 
	\[ f_1\colon (M_1, \partial M_1) \to (X\times D^i, X\times S^{i-1})\] 
	and 
	\[ f_2\colon (M_2, \partial M_2) \to (X\times D^i, X\times S^{i-1})\]  are two elements in $\mathcal S^\topo_{\partial}(X\times D^4)$. We define the sum of $f_1$ and $f_2$ to be the following element: 
	\[  f = f_1\cup f_2\colon M \to  X\times D^i = 
	X\times D_+^i\cup X\times D_-^i \]
	where $M$ is obtained from gluing $M_1$ and $M_2$ along \[ f_1^{-1}(X\times D^{i-1}) 
	\cong f_2^{-1}(X\times D^{i-1}). \]  
\end{definition}

With respect to this group structure, the map $\beta_\ast$ is in fact a group homomorphism. 

\begin{lemma}\label{lm:grphom}
	For $i\geq 1$, the map 
	\[ \beta_\ast\colon \mathcal S^\topo_\partial (X\times D^i)  \to \mathcal S_{n+i}(X) \] is a group homomorphism. 
\end{lemma}
\begin{proof}
	It is easy to see that the stacking procedure gives a cobordism between the resulting new element and the disjoint union of the two elements that we started with. This finishes the proof. 
\end{proof}

For $i\geq 0$, we have a natural homomorphism
\[  \alpha_\ast\colon \mathcal N^\topo_\partial (X\times D^i)  \to \mathcal N_{n+i}(X) \] 
which is defined similarly as the map $\beta_\ast$ above . 

\begin{lemma}\label{lm:norminj}
	If $X$ is a closed oriented connected topological manifold of dimension $n\geq 5$, then the map
	\[ \alpha_\ast\colon \mathcal N^\topo (X)  \to \mathcal N_{n}(X)\]
	is  an isomorphism.   
\end{lemma}
\begin{proof}
	For each $i\geq 0$, there is a commutative diagram 
	\[  \xymatrix{ \mathcal N^\topo_\partial (X\times D^i)  \ar[r] \ar[d]^\cong & \mathcal N_{n+i}(X) \ar[d]^\cong \\
		H_{n+i}(X; \mathbb L_\bullet) \ar@{=}[r] & H_{n+i}(X; \mathbb L_\bullet)   } \]
	where the vertical isomorphisms are the corresponding  algebraic normal invariant maps. This finishes the proof. 
\end{proof}

\begin{remark}
	We remark that all explicit descriptions of normal invariants are torsorial. While a Poincar\'e complex can have a normal invariant, and this is a well defined homotopical notion, and the description of normal invariants as lifts of the Spivak fibration is also canonical, the description of the aggregate of these lifts in terms of the classifying space $G/\mathrm{CAT}$ involves first picking a preferred lift.
\end{remark}

Now by construction, the geometric surgery long exact sequence $\eqref{eq:sles}$ and the long exact sequence from Theorem $\ref{thm:les}$,  together with the maps $\alpha_\ast$, $\beta_\ast$ and $\iota_\ast$, fit into the following commutative diagram:
\begin{equation}\label{diag:topsurg}
\begin{split}
\scalebox{1}{\xymatrixcolsep{1pc}\xymatrix{ 
		\mathcal N^{\topo}_\partial (X\times I)  \ar[r] \ar[d]_{\alpha_\ast}^{\cong} & L_{n+1}(\pi_1 X) \ar@{-->}[r] \ar@{=}[d]& \mathcal S^\topo (X) \ar[r] \ar[d]_{\iota_\ast} &  \mathcal N^{\topo}(X) \ar[r] \ar[d]_{\alpha_\ast}^{\cong} & L_n(\pi_1 X)  \ar@{=}[d] \\
		\mathcal N_{n+1}(X)\ar[r] & L_{n+1}(\pi_1 X) \ar[r] & \mathcal S_n(X) \ar[r] & 	  \mathcal N_n(X)  \ar[r] & L_n(\pi_1 X). 
	}   
}
\end{split}
\end{equation}
By using the action of  $L_{n+1}(\pi_1X)$ on $\mathcal S^\topo(X)$ and the proof of the standard five lemma, we obtain the following proposition. 

\begin{proposition}\label{prop:strid} If $ \dim X = n\geq 5$, then 
	the map 
	\[ \iota_\ast\colon \mathcal S^\topo(X) \to \mathcal S_{n}(X) \] is a bijection of sets.  Moreover, for $i\geq 1$, the map \[ \beta_\ast\colon \mathcal S^\topo_\partial (X\times D^i)  \to \mathcal S_{n+i}(X) \]  is a group isomorphism. 
\end{proposition}

For any oriented closed topological manifold $X$ with $\dim X\geq 5$, $\mathcal S^\topo(X)$ carries an abelian group structure by Siebenmann's periodicity theorem \cite[Appendix C to Essay V]{MR0645390}, which makes the geometric surgery long exact sequence $\eqref{eq:sles}$ into an exact sequence of abelian groups everywhere. More precisely, Siebenmann's periodicity theorem states that there is an injection  
\[  \mathcal S^\topo(X) \hookrightarrow \mathcal S^\topo_\partial(X\times D^4). \] To see that the abelian group structure on $\mathcal S^\topo_\partial(X\times D^4)$ induces an abelian group structure on $\mathcal S^\topo(X)$, one needs the fact that the map $\sigma\colon \mathcal S^\topo_\partial (X\times D^4) \to \mathbb Z$ arising in Nicas' correction \cite{MR668807} to Siebenmann's periodicity theorem is a group homomorphism. The map $\sigma$ can be described as follows. Suppose   
\[ f\colon (Y, \partial Y) \to
(X\times D^4, X\times S^3)\] is an element of $\mathcal S^\topo_\partial (X\times D^4)$. Let $V$ be the transverse inverse image of $\{pt\}\times D^4$, where $pt$ is a point of $X$. Then $V$ is a compact oriented manifold whose boundary is $S^3$.  Now we glue a $D^4$ to $V$ along $S^3$, and obtained a  closed oriented manifold $V'$ of dimension $4$. The map $\sigma$ takes $f$ to $1/8$ of the signature of $V'$ (since this signature is automatically a multiple of $8$). It is easy to see that $\sigma$ is a group homomorphism, since the signature of a connected sum is the sum of the signatures. The Siebenmann's periodicity theorem (with the above correction by Nicas) can now be stated in terms of  the following exact sequence:  
\[ 0 \to \mathcal S^\topo(X) \to \mathcal S^\topo_\partial(X\times D^4) \xrightarrow{\sigma } \mathbb Z. \] As a consequence, $\mathcal S^\topo(X)$,  viewed as the kernel of the group homomorphism $\sigma$, carries an abelian group structure. 

In order to see how the map $\iota_\ast \colon \mathcal S^\topo(X) \to \mathcal S_n(X)$ behaves respect to the abelian group structure on $\mathcal S^\topo(X)$ induced by the Siebenmann periodicity theorem, we shall in fact use a geometric periodicity map, denoted by $\GP$,  due to Cappell and the first author \cite{MR873283}. A priori,  it is not clear whether the map $\GP$ coincides with the Siebenmann periodicity map $\mathcal{SP}$. Crowley and Macko showed that a quaternionic (resp. octonionic) version of $\GP$ coincides with $\mathcal{SP}^2\colon S^\topo(X) \to \mathcal S^\topo_\partial(X\times D^8)$ (resp. $\mathcal{SP}^4\colon S^\topo(X) \to \mathcal S^\topo_\partial(X\times D^{16})$) \cite{MR2826928}. Our discussion below will show that $\GP$ coincides with $\mathcal{SP}$.

Now let us briefly review the construction of the map $\GP$.   It is a fundamental fact that any homotopy equivalence $f\colon M \to X$ in $\mathcal S^\topo(X)$ has a unique (up to isotopy)  associated embedding $M \to X\times D^3$. This is due to Browder-Casson-Haefiger-Sulivan-Wall, which is
explained in  \cite{MR1687388} in the analogous PL case, with the topological case following from \cite{MR0645390}. Quinn's end theorem \cite{MR549490} gives a canonical mapping cylinder structure on the neighborhood of $M$ in $X\times D^3$. The same reasoning gives a mapping cylinder approximate fibration $X\times D^3$ over $M$. The map   
\[ \GP: \mathcal S^\topo(X) \to \mathcal S^\topo_{\partial}(X\times D^4)\] is defined by taking the Hopf fibration $S^3\to S^2$ over the mapping cylinder (away from $M$) and  gluing $M$ back.

\begin{proposition}\label{prop:strhom}
	If $ n = \dim X\geq 5$, 	the map 
	\[ \iota_\ast\colon \mathcal S^\topo(X) \to \mathcal S_{n}(X) \] is a group homomorphism. 
\end{proposition}
\begin{proof}
	It suffices to prove the following diagram commutes: 
	\begin{equation}\label{diag:perio}
	\begin{split}
	\xymatrixcolsep{4pc}\xymatrix{ \mathcal S^\topo (X) \ar@{^{(}->}[r]^-{\GP} \ar[d]_{\iota_\ast} &  \mathcal S^\topo_\partial (X\times D^4) \ar[d]^{\beta_\ast} \\
		\mathcal S_{n}(X)  \ar@{^{(}->}[r]^-{\times \mathbb{CP}^2}  &  \mathcal S_{n+4 }(X) 
	}   
	\end{split}
	\end{equation} 
	where the map $\GP$ is the geometric Siebenmann periodicity map described above, and $\times \mathbb {CP}^2$ is the map induced by taking product with $\mathbb {CP}^2$.  Here is the intuitive reason why this diagram commutes. For an element $f\colon M \to X$ in $\mathcal S^\topo(X)$, the construction of $\GP$ first looks at the boundary of the neighborhood $M$ in $X\times D^3$, which is homeomorphic to $X\times S^2$ and  approximately fibers over $M$ with fiber $S^2$. Now taking the Hopf disk bundle over  $S^2$ (i.e. the unit disk bundle of the tautological line bundle over $S^2 = \mathbb{CP}^1$) and fiberwise coning its boundary is somehow invoking  the $\times \mathbb{CP}^2$ isomorphism in surgery theory. Indeed,  $\mathbb{CP}^2$ can be obtained by gluing the Hopf disk bundle over $S^2$ and $D^4$ along their boundaries, both of which are $S^3$.
	
	Now let us give a rigorous argument using the device of ``periodicity spaces"  \cite{MR1823952,MR2154831}. One convenient setting for our argument below is the stratified surgery theory of Browder and Quinn \cite{MR0375348}. The relevant periodicity space for our discussion is the stratified space $P = \mathbb{CP}^2\cup D^3$, where $D^3$ is glued to $\mathbb{CP}^2$ along $S^2 = \mathbb{CP}^1$. And $P$ is a stratified space with three strata: $S^2$, $\mathring{D}^3$ and $\mathring{D}^4 = \mathbb{CP}^2 - S^2$. 
	
	In our definition (Definition $\ref{def:Lgrp}$)  of $L_n(\pi_1 X)$, an object consists of manifolds with boundary and some extra data. In the following, we shall enlarge the set of objects by allowing stratified (singular) manifolds  over $P$. Here we call $Y$ a stratified manifold (with boundary)  over $P$ if there is a map  $f\colon Y \to P$ such that the strata of $Y$ are the pullbacks of the strata of $P$ and the map is transversal to each stratum of $P$. 
	
	The group $L^{BQ}_{n+4}(\pi_1 X; P)$ is defined to be the cobordism of surgery problems\footnote{Here we follow the usual convention to call an element in Wall's geometric definition (cf. Definition $\ref{def:Lgrp}$) of $L$-groups a \emph{surgery problem}.} over $X$,  except that we allow the objects to be $(n+4)$-dimensional stratified manifolds over $P$. Equivalently, each element of $L^{BQ}_{n+4}(\pi_1 X; P)$ can be thought of as a surgery problem over $X\times S^2$  with extensions over $X\times D^3$ and $X\times \mathbb{CP}^2$, where in the latter case, we insist that the surgery problem over $X\times \mathbb{CP}^2$ restricts transversally to the given one on $X\times S^2$. We point out that the reference space here is still $X$; and $P$ is only used as a model space to produce a stratification structure for these objects, hence the notation  $L^{BQ}_{n+4}(\pi_1 X; P)$ to distinguish the role of $P$ from that of $X$. Similarly, we define $L_{n+4}^{BQ}(\pi_1 X; D^4 \rel \partial)$, where the objects are stratified over $D^4$ and restrict to a trivial surgery problem over $S^3 \subset D^4$. 
	
	We make the following key observations. 
	\begin{enumerate}[(1)]
		\item Some straightforward calculations within Browder-Quinn's stratified surgery  theory show that the inclusion map $D^4 \rel \partial$ into $P$ (that is, $D^4$ is identified with the complement of a tubular neighborhood of $S^2$ in  $\mathbb{CP}^2$) induces an isomorphism $L_{n+4}^{BQ}(\pi_1 X; D^4\rel \partial) \xrightarrow{\ \cong \ } L^{BQ}_{n+4}(\pi_1 X; P)$.  The restriction map that takes a surgery problem over $P$ to its restriction over $\mathbb{CP}^2$ also induces an isomorphism \[ L^{BQ}_{n+4}(\pi_1 X; P) \xrightarrow{\ \cong \ } L_{n+4}(\pi_1 X). \] In terms of spectra, these maps induce homotopy equivalences of the corresponding spectra. 
		
		\item Similar to Definitions $\ref{def:newstruc}$  and $\ref{def:norm}$,  we can also define groups  $\mathcal S_\ast^{BQ}(X; P)$ and $\mathcal N_\ast^{BQ}(X; P)$, and they fit into an exact sequence:
		\begin{align*}
		\cdots \to L_{n+5}^{BQ}(\pi_1 X; P)   \to  \mathcal S_{n+4}^{BQ}(X; P)  
		\to  	\mathcal N_{n+4}^{BQ}(X; P) \\
		\to L_{n+4}^{BQ}(\pi_1 X;  P)\to  \mathcal S_{n+3}^{BQ}(X; P)  \to  \cdots 
		\end{align*}
		\[	\scalebox{0.85}{$ $} \]
	\end{enumerate}
	
	\begin{remark}
		Let $L_{n+4}(\pi_1 X; D^3, S^2)$ be the relative $L$-group for the pair of spaces $(X\times D^3, X\times S^2)$, which is trivial by Wall's $\pi\mhyphen\pi$ theorem \cite[Theorem 3.3]{MR1687388}. The composite map
		\[  \mathcal{S}^{\topo}(X) \xrightarrow{\times P} \mathcal{S}^{BQ}_{n+4}(X; P) \to L_{n+4}(\pi_1 X; D^3, S^2)\] is trivial because it ends in the trivial group, where the first map is taking direct product with $P$ and the second map is given by the restriction from $P$ to $D^3$.  It also has the interpretation, in the topological category, that after crossing a structure over $X$ with $D^3$, it can be (approximately) fibered over $X$.  This is essentially the content of the embedding theorem of Browder-Casson-Haefiger-Sulivan-Wall, which is the core of the construction of the geometric periodicity map $\mathcal{GP}$.  		
	\end{remark}
	The inclusion map $D^4 \rel\partial$ into $P$ and the restriction map from $P$ to $\mathbb{CP}^2$ induce the following commutative diagram: 
	\[ 
	\scalebox{0.95}{\xymatrixcolsep{1pc}\xymatrix{ 
			L_{n+5}^{BQ}(\pi_1 X;  D^4 \rel \partial) \ar[r] \ar[d]&  \mathcal S_{n+4}^{BQ}(X; D^4\rel \partial) \ar[r]  \ar[d]_{i_\ast} &  \mathcal N_{n+4}^{BQ}(X; D^4\rel \partial) \ar[r] \ar[d] &  L_{n+4}^{BQ}(\pi_1 X; D^4\rel \partial) \ar[d]   \\		
			L_{n+5}^{BQ}(\pi_1 X; P)   \ar[r] \ar[d] &  \mathcal S_{n+4}^{BQ}(X; P)  \ar[r] \ar[d]_{r_\ast} &  \mathcal N_{n+4}^{BQ}(X; P) \ar[r] \ar[d] &  L_{n+4}^{BQ}(\pi_1 X; P)  \ar[d]    \\		
			L_{n+5}(\pi_1 X)   \ar[r] &  \mathcal S_{n+4}(X)  \ar[r] &  \mathcal N_{n+4}(X) \ar[r] &  L_{n+4}(\pi_1 X)    
		}
	}
	\]
	where  all vertical maps are isomorphisms. Moreover, the same argument from Proposition $\ref{prop:strid}$ shows that the natural map 
	\[ \mathcal S^\topo_\partial (X\times D^4) \to \mathcal S_{n+4}^{BQ}(X; D^4\rel\partial)\] is an isomorphism. 
	
	Now it follows by construction that the map $\GP$ coincides with the following composition\footnote{Here the map  $\times P\colon \mathcal S_n(X) \to \mathcal S_{n+4}^{BQ}(X; P)$ is  given by taking the direct product of an element with $P$. It is an isomorphism, hence $P$ is called a periodicity space. The fact that $\times P$ is an isomorphism is essentially a consequence of Wall's $\pi\mhyphen\pi$ theorem \cite[Theorem 3.3]{MR1687388} (see \cite{MR1823952,MR2154831}).}
	\[ \mathcal S^\topo(X) \xrightarrow{\iota_\ast} \mathcal S_n(X)\xrightarrow{\times P} \mathcal S_{n+4}^{BQ}(X; P) \xrightarrow{i_\ast^{-1}} \mathcal S_{n+4}^{BQ}(X; D^4\rel\partial) \cong \mathcal S^\topo_\partial (X\times D^4).   \]
	The map 
	\[ \times \mathbb{CP}^2\colon  \mathcal S_{n}(X)\to \mathcal S_{n+4}(X)\] coincides with the following composition 
	\[   \mathcal S_n(X)\xrightarrow{\times P} \mathcal S_{n+4}^{BQ}(X; P) \xrightarrow{r_\ast} \mathcal S_{n+4}(X) \]
	and the map $\beta_\ast\colon \mathcal S^\topo_\partial (X\times D^4)  \to \mathcal S_{n+4}(X) $  coincides with the composition:
	\[   \mathcal S^\topo_\partial (X\times D^4)\cong \mathcal S_{n+4}^{BQ}(X; D^4\rel\partial) \xrightarrow{i_\ast} \mathcal S_{n+4}^{BQ}(X; P) \xrightarrow{r_\ast} \mathcal S_{n+4}(X). \]
	Therefore, we have the following commutative diagram. 
	\[ 
	\xymatrixcolsep{4pc}\xymatrix{ \mathcal S^\topo(X) \ar@{^{(}->}[r]^-{\GP} \ar[d]_{\iota_\ast} &  \mathcal S^\topo_\partial (X\times D^4) \ar[d]^{\beta_\ast}\\
		\mathcal S_{n}(X)  \ar@{^{(}->}[r]^-{\times \mathbb{CP}^2} &  \mathcal S_{n+4}(X) 
	}   
	\]
	This finishes the proof. 	
\end{proof}

For reference, we mention explicitly that our discussion readily implies:

\begin{theorem}\label{thm:gp=sp}
	The geometric periodicity map coincides with the Siebenmann periodicity map, that is, $\mathcal{GP} = \mathcal{SP}$.
\end{theorem} 
\begin{proof}
	The former is induced by crossing with the stratified space $P$ and the latter by crossing with $\mathbb{CP}^2$, as explained above. Hence the theorem follows from the discussion above. 
\end{proof}

To summarize, combining Proposition $\ref{prop:strid}$ and Proposition $\ref{prop:strhom}$, we have the following theorem. 
\begin{theorem}\label{thm:struciso}
	If $X$ is a closed oriented connected topological manifold with $\dim X = n \geq 5$, then the map
	\[ \iota_\ast\colon  \mathcal S^\topo(X) \to \mathcal S_n(X)\]
	is an isomorphism.
\end{theorem}

\begin{remark}
	More generally, if $X$ is a closed connected topological manifold of dimension $\geq 5$ and $w$ is the orientation covering of $X$, then the map
	\[ \iota_\ast\colon  \mathcal S^\topo(X, w) \to \mathcal S_n(X, w)\] is an isomorphism.
\end{remark}

We conclude this subsection with the following brief discussion of the $4$-periodic surgery exact sequence from line $\eqref{perilong}$ in Theorem $\ref{thm:les}$ and homology manifold structure groups. 

For this discussion, let us assume that $X$ is a closed connected \emph{topological} manifold of dimension $\geq 6$. If we invert $\mathbb{CP}^2$ and consider the $4$-periodic theory as in Definition $\ref{def:periodL}$ and $\ref{def:period}$, then $\mathfrak S_n(X, w)$ is naturally isomorphic to the homology manifold structure group $\mathcal S^{HTOP}(X, w)$ of $X$. Indeed, according to \cite[Corollary, Page 438]{MR1183997},  $\mathcal S^{HTOP}(X) $ is isomorphic to  $\mathcal S^\topo_\partial (X\times D^4)$, where the latter group consists entirely of structures with manifold representatives, due to the rel $\partial$ condition\footnote{Quinn showed that an ANR homology manifold whose boundary is a manifold can be resolved rel boundary \cite{MR700771,MR848688}.}. The $4$-periodic surgery long exact sequence from line $\eqref{perilong}$ in Theorem $\ref{thm:les}$ gives the surgery long exact sequence for homology manifold structures, cf. \cite[Main Theorem]{MR1183997}. In fact, by construction, our higher rho invariant map (Definition $\ref{def:highrho}$) is naturally defined on the homology manifold structure group. Also see the discussion after Proposition $\ref{prop:plplus}$.  All the main results, in particular Theorem $\ref{thm:main}$ and Theorem $\ref{thm:structure}$, of this paper hold for both the manifold structure group and the homology manifold structure group. 

\begin{remark}
	For the $4$-periodic theory, the analogue of Theorem $\ref{thm:normal}$ is 
	\[  \mathfrak N_n(X) \cong  H_n(X; \mathbb L_\bullet\langle 0\rangle) \textup{ for all } n \in \mathbb Z, \] 
	where $\mathbb L_\bullet\langle0\rangle$ is an $\Omega$-spectrum of simplicial sets of quadratic forms and formations over $\mathbb Z$ such that $\mathbb L_0\langle0\rangle\simeq \mathbb Z\times G/TOP$. In fact, the proof for this $4$-periodic analogue is easier, and the extra discussion  surrounding Theorem $\ref{thm:normal}$   (such as the modifications in dimension $4$) is not needed, since after applying the periodicity map all calculations can be done in a sufficiently high dimension.  
\end{remark}

\begin{remark}\label{rk:hommfld}
	More generally, the same method from above can be applied to the case where $X$ is a closed oriented connected \emph{ANR homology} manifold of dimension $\geq 6$. The analogue of the $4$-periodic  exact sequence from line $\eqref{perilong}$ in Theorem $\ref{thm:les}$  in this case is precisely the homology manifold surgery exact sequence of \cite[Main Theorem, Page 439]{MR1183997}. Here the only essential difference from the topological manifold case is that we do not have special low dimensional features to correct for the lack of realization of $L_0(e)$ by manifolds.
\end{remark}

\subsection{Structure group by smooth or PL representatives} \label{sec:smstr}
In this subsection, we observe that the elements in our definition of the structure group always have smooth representatives. Perhaps this is the main philosophical point of this approach to surgery -- the most naturally funtorial version of structures is independent of the category, and boils down to the topological category. 

Let $X$ be a closed connected topological	manifold. Consider the smooth and PL (piecewise-linear) versions of the long exact sequence from Theorem $\ref{thm:les}$, and denote them by 
\begin{equation}\label{eq:smles}
\begin{aligned}
\cdots \to \mathcal N^{C^\infty}_{n+1}(X; w) \xrightarrow{i_\ast} L^{C^\infty}_{n+1}(\pi_1 X; w) \xrightarrow{j_\ast}
\mathcal S^{C^\infty}_{n}( X; w) \xrightarrow{\partial_\ast} \mathcal N^{C^\infty}_{n}(X; w) \to \cdots 
\end{aligned}
\end{equation}  
and 
\begin{equation}\label{eq:plles}
\begin{aligned}
\cdots  \to \mathcal N^{PL}_{n+1}(X; w) \xrightarrow{i_\ast} L^{PL}_{n+1}(\pi_1 X; w) \xrightarrow{j_\ast} 
\mathcal S^{PL}_{n}( X; w) \xrightarrow{\partial_\ast} \mathcal N^{PL}_{n}(X; w) \to\cdots 
\end{aligned}
\end{equation}  
respectively, where the various groups are defined as follows. 
\begin{definition}\label{def:smrep}
	An element 
	\[ \theta =  (M, \partial M, \varphi, N, \partial N, f) \in \mathcal S^{C^\infty}_{n}(X; w) \textup{ (resp. $ \mathcal S^{PL}_{n}(X; w)$)}\]   consists of the following data: 
	\begin{enumerate}[(1)]
		\item $\theta$ is an element of $\mathcal S_{n}(X; w)$ (cf. Definition $\ref{def:newstruc}$);
		\item $M $ and $ N$ are smooth (resp. PL) manifolds with boundary, and the map $f\colon N \to M$ is smooth (resp. PL).
	\end{enumerate} 
\end{definition}	
We point out that $X$ is only a topological manifold, and the control maps $\varphi$ and $\psi$ are only assumed to be continuous.
The groups $\mathcal N^{C^\infty}_n(X; w)$, $ L^{C^\infty}_n(\pi_1 X; w)$, $\mathcal N^{PL}_n(X; w)$ and $ L^{PL}_n(\pi_1 X; w)$  are defined similarly. There is an obvious map from the smooth version to the PL version, which in turn can be mapped to the topological version.

Recall that, for $n \geq 5$, the group $L_n(\pi_1 X, w)$ remains the same in all smooth, PL, and topological categories (see for example \cite[Chapter 9]{MR1687388} where the proof for identifying $L_n(\pi_1 X, w)$ with the algebraic definition $L^h_n(\pi_1 X, w)$  works equally for all three categories). 

We shall make the same modification as in Definition $\ref{modLgrp}$ for dimensions $\leq 4$ so that both $\mathcal N^{C^\infty}_n(X; w)$ and $\mathcal N^{PL}_n(X; w)$ define the same homology theory as $\mathcal N_n(X; w)$. To be more precise,  the same  discussion before Theorem $\ref{thm:normal}$ works for the PL category. While surgery
is usually impossible in the smooth category, thanks to the work of Cappell and
Shaneson \cite{MR0301750}, it still works stably, that is, after taking connected sums with sufficiently many copies of $S^2\times S^2$, where $S^2$ is the standard $2$-sphere. As a consequence,  it follows that $\mathcal N^{C^\infty}_n(X; w)$ is naturally isomorphic to $\mathcal N^{PL}_n(X; w)$ and $\mathcal N_n(X; w)$ for all $n\geq 0$. Now the following proposition is an immediate consequence of the five lemma.

\begin{proposition}\label{prop:smooth} For $n\geq 5$, we have natural isomorphisms
	\[ \mathcal S^{C^\infty}_{n}(X; w) \cong \mathcal S^{PL}_{n}(X; w) \cong \mathcal S_{n}(X; w). \] In particular, every element in $\mathcal S_{n}(X; w)$ has a smooth \textup{(}resp. PL\textup{)} representative.  
\end{proposition}

\begin{remark}
	The reader should not confuse $\mathcal S^{C^\infty}_{n}(X; w)$ with the smooth structure set of a smooth manifold. The group  $\mathcal S^{C^\infty}_{n}(X; w)$ still characterizes topological manifold structures on $X$. The novelty here is that we allow manifolds with boundary and corners in the definition of $\mathcal S^{C^\infty}_{n}(X; w)$. Similar remarks apply to $\mathcal S^{PL}_{n}(X; w) $ as well. 
	
	This type of argument would fail if we restrict ourselves to only closed manifolds as in the classical definition of $\mathcal S^\topo(X)$.
	Moreover, we point out that our definition $\mathcal S_{n}(X)$ continues to make sense even if $X$ is not a manifold. One essential point here is that with our new definition, we are forcing structure groups to be \emph{functorial}, which they do not  seem to be in the case of smooth manifold structures. In fact, the smooth manifold structure set does \emph{not} carry an abelian group structure which makes the smooth surgery exact sequence into an exact sequence of abelian groups.

\end{remark}

\subsection{Piecewise linear control}\label{sec:PL}

In this subsection, we will give another definition of the structure group using a different type of control that, we will see, can be used to obtain infinitesimal control.

Note that our definition of $\mathcal S_n(X)$ only depends on the homotopy type of $X$. In other words,  $\mathcal S_n(X)$ is isomorphic to $\mathcal S_n(X')$ for every pair of homotopy equivalent spaces $X$ and $X'$.   Recall that every closed topological manifold is homotopy equivalent to a finite $CW$-complex (cf. \cite{MR0242166}), and therefore  homotopy equivalent\footnote{In fact, any manifold of dimension $\neq 4$ is homeomorphic to a CW complex.}  to a finite simplicial complex. On the other hand, every finite simplicial complex is homotopy equivalent to a smooth manifold with boundary. Indeed, after being embedded into a Euclidean space, a finite simplicial complex is homotopy equivalent to a regular neighborhood, which is a smooth manifold with boundary\footnote{Note that, in general, the dimension of this smooth manifold with boundary is larger than the dimension of the original topological manifold we started with.}. To summarize, every topological manifold of dimension $\geq 5$ is homotopy equivalent to a smooth manifold with boundary. So from now on,   without loss of generality, let us assume $X$ is a smooth manifold with boundary. In particular, let us fix a triangulation of $X$ throughout this subsection. 

\begin{remark}
	In the above discussion, when we homotope a topological manifold $Z$ to a smooth manifold, say $X$, the dimension of $X$ is larger than that of $Z$ in general. However, we point out that the objects in the definition of $\mathcal S_n(Z)$ and $\mathcal S_n(X)$ are still of dimension $n$, regardless of the dimension of $X$ or $Z$.
\end{remark}

In the following, we work in the PL category. In particular, all objects are equipped with a triangulation and all morphisms are assumed to be simplicial. We refer the reader to \cite{MR0226645,MR0226646,MR0232404} for more details on PL transversality. 
\begin{definition}\label{def:pltrans}
	Let $Y$ and $Z$ be a pair of PL manifolds equipped with certain triangulations. A homotopy equivalence $h\colon Y \to Z$ is said to be PL controlled over $X$ via the control map $\varphi\colon Z \to X$ if  the following is satisfied. 
	\begin{enumerate}[(1)]
		\item $\varphi$ is transversal to the triangulation of $X$. That is, the map $\varphi\colon Z\to X$ is transversal to every simplex $\Delta^k$ in the triangulation of $X$.  In particular,  the inverse image of each simplex $\Delta^k$ (in the triangulation of  $X$) is a manifold $k$-ad. 
		\item $h$ restricts to a homotopy equivalence from $(\varphi\circ h)^{-1}(\Delta^k) $ to $\varphi^{-1}(\Delta^k)$ for every simplex $\Delta^k$ of $X$. More precisely, there exists a homotopy inverse $g\colon Z \to Y$  of $h$ such that 
		\begin{enumerate}[(i)]
			\item the homotopy $H\colon h\circ g \simeq \id $ restricts to a homotopy on $\varphi^{-1}(\Delta^k )$;
			\item  the homotopy $H'\colon g\circ h \simeq \id $ restricts to a homotopy on $(\varphi\circ h)^{-1}(\Delta^k )$.
		\end{enumerate}
	\end{enumerate} 
\end{definition}

\begin{remark}
	Note that, in the above definition, for  each simplex $\Delta^k$ in $X$, the homotopy equivalences $h$, $g$ and the homotopies $H$ and $H'$ all respect the  appropriate manifold ad structure on the inverse image $\varphi^{-1}(\Delta^k)$. In particular, near various boundaries of $\varphi^{-1}(\Delta^k)$, the map $h$, $g$, $H$ and $H'$ have appropriate product structures. For example, the inverse image $ K = \varphi^{-1}(\Delta^1)$ of an $1$-simplex $\Delta^1$ is a manifold $1$-ad, that is, a manifold with boundary $\partial K$. In this case, the restrictions of $h$, $g$, $H$ and $H'$ on $\varphi^{-1}(\Delta^1)$ maps $\partial K$ to $\partial K$, and have product structure near $\partial K$. We refer the reader to \cite[Chapter 0]{MR1687388} for more details on the notion of manifold $m$-ads.
\end{remark}

Now similar to Section $\ref{sec:smstr}$, we can define a new surgery long exact sequence by using PL-control instead of infinitesimal control (Definition $\ref{def:infcon}$) :
\begin{equation}\label{eq:ples}
\begin{aligned}
\cdots \to \mathcal N^{PL^+}_{n+1}(X; w) \xrightarrow{i_\ast} L^{PL^+}_{n+1}(\pi_1 X; w) \xrightarrow{j_\ast}  \mathcal S^{PL^+}_{n}( X; w) \xrightarrow{\partial_\ast} \mathcal N^{PL^+}_{n}(X; w)\to \cdots 
\end{aligned}
\end{equation}
where the superscript $PL^+$ stands for PL-representatives with PL-control. More precisely, for example, for elements in $S^{PL}_{n}( X; w)$, we replace  infinitesimal control with PL-control, and denote the new group by  $\mathcal S^{PL^+}_{n}( X; w)$.

\begin{remark}
	The definition of PL-control we gave works well only when $X$ is a PL manifold. We point out that, when $X$ is a  PL manifold with boundary, to define $\mathcal S^{PL^+}_{n}( X; w)$,   every element $\theta = (M, \partial M, \varphi, N, \partial N, \psi, f)$ is assumed to be disjoint from the boundary of $X$. That is, $\varphi(M)\cap \partial X = \emptyset$, $\psi(N) \cap \partial X = \emptyset$,  and all other relevant data do not intersect  $\partial X$. Similar remarks apply to $\mathcal N^{PL^+}_n(X; w)_{PL}$ and $L^{PL^+}_n(\pi_1 X; w)$.  
\end{remark}

\begin{remark}
	Using the ideas of Quinn from \cite{MR1388303}, we can generalize the above construction to the case where $X$ is an arbitrary finite polyhedron.  In that case, to define PL control, the conditions are not on inverse images of simplices, but rather of their dual cones. In this more general setting, the covariant functoriality of  $\mathcal S^{PL^+}_{n}( X; w)$ becomes clear.  
	
\end{remark}

In the following proposition, we prove that the above notion of PL-control in fact implies infinitesimal control, by modifying the control map $\varphi\colon Z\to X$ if necessary. More precisely, we can keep the map $h\colon Y \to Z$ unchanged, and homotope the control map $\varphi$ to another control map $\bar \varphi$ so that $h$ becomes infinitesimally controlled with respect to $\bar \varphi$. 

\begin{proposition}\label{prop:gaincontrol}
	Let $X$ be an $n$-dimensional PL manifold with boundary equipped with a triangulation.  Suppose  $h\colon Y \to Z$ is a PL controlled homotopy equivalence over $X$ via the control map $\varphi \colon Z \to X$.  Then there exists a control map  $\bar{\varphi} \colon Z \to X$ such that 
	\begin{enumerate}[(1)]
		\item $\bar{\varphi}$ is homotopic to $\varphi$;
		\item  $h$ restricts to a proper homotopy equivalence 
		\[  h\colon \bar{\psi}^{-1} U \to \bar{\varphi}^{-1} U  \]
		for all open subsets $U\subset X$, where $\bar{\psi} = \bar{\varphi}\circ h$.
	\end{enumerate}
\end{proposition}
\begin{proof}
	The proof is by induction and uses a ``dual cone" picture as described, for example, in \cite[Section 6]{MR1388303}. 
	
	Suppose $K$ is a simplicial complex. We take the first barycentric subdivision of $K$. For every simplex $\sigma$ in $K$, we define the dual cone $D(\sigma)$ to be the union of all simplices of the subdivision which intersect $\sigma$ in exactly the barycenter of $\sigma$. Now the key idea of the proof is to crush all the nontrivial changes in topology of $Z$ and $Y$, and the homotopy equivalence $h$ to small parts.  More precisely, we have the following induction construction.  Let $X^{(k)}$ be the $k$-skeleton of $X$, that is, $X^{(k)}$ is the union of all simplices in $X$ of dimensions $\leq k$.
	
	\begin{enumerate}[(i)]
		\item \textbf{Initial step}.  First, consider $Z^{(1)} = \varphi^{-1}(X^{(1)})$ the inverse image of $X^{(1)}$. Note that $\varphi^{-1}(X^{(1)})$ has  product structure near $Z^{(0)} = \varphi^{-1}(X^{(0)})$. Let us denote by $Z^{(0)}_\varepsilon$ for such a small open neighborhood (with product structure) of $Z^{(0)}$ in $Z^{(1)}$.   Let $Z^{(1)}_c = Z^{(1)} - Z^{(0)}_\varepsilon $ be the complement of $Z^{(0)}_\varepsilon$ in $Z^{(1)}$.   
		
		We define a new control map 
		\[ \varphi_1\colon Z^{(1)} = \varphi^{-1}(X^{(1)}) \to X\] by mapping each component of $Z^{(1)}_c$ to the barycenter of the corresponding $1$-simplex in $X$, and stretching out $Z^{(0)}_\varepsilon$ (which is of product structure) accordingly. Intuitively, we see that the nontrivial changes of topology from $Z^{(0)}$ to $Z^{(1)}$ are all pushed to the barycenters of $1$-simplices in $X$. In particular, for all open subsets $V\subset X^{(1)}$, $h$ restricts to a proper homotopy equivalence $h\colon \psi_1^{-1}(V)  \to \varphi_1^{-1}(V)$, where $\psi_1 = \varphi_1\circ h$. 
		
		\begin{figure}[H]
			\hspace*{0.5cm}   
			\begin{tikzpicture}[scale=1, every node/.style={transform shape}]
			\path [pattern = crosshatch, pattern color = blue] (0, 5) to [out = 0, in = 180]  (6, 4) -- (6, 3) to [out = 180, in = 0] (0, 1.5) -- (0, 2.5)   arc (-90:90:1.5 and 0.75) (0, 4) -- (0, 5);

			\draw[very thick] 	(0, 5) to [out = 0, in = 180]  (6, 4) ; 
			\draw[very thick]  (6, 3) to [out = 180, in = 0] (0, 1.5); 
			\draw [very thick] (0, 2.5)   arc (-90:90:1.5 and 0.75) ;

			
			\path [fill = white, draw = white] (0, 4) arc (-90:90:0.25 and 0.5) (0, 5)-- (0, 4); 
			\draw [dashed, thick] (0, 5)  arc (90:270:0.25 and 0.5); 
			
			\draw [pattern = horizontal lines, very thick] (-0.7, 5) arc (90:450:0.25 and 0.5);

			\draw [pattern = horizontal lines, very thick] (0, 4) arc (-90:90:0.25 and 0.5) (0, 5) -- (-0.7, 5) arc (90:-90:0.25 and 0.5) (-0.7, 4) -- (0, 4) arc (-90:90:0.25 
			and 0.5) ; 
			
			
			\path [fill = white, draw = white] (0, 1.5) arc (-90:90:0.25 and 0.5) (0, 2.5 )-- (0, 1.5); 
			
			
			\draw [dashed, thick] (0, 2.5)  arc (90:270:0.25 and 0.5); 
			
			\draw [pattern = horizontal lines, very thick] (-0.7, 2.5) arc (90:450:0.25 and 0.5);

			\draw [pattern = horizontal lines, very thick] (0, 1.5) arc (-90:90:0.25 and 0.5) (0, 2.5) -- (-0.7, 2.5) arc (90:-90:0.25 and 0.5) (-0.7, 1.5) -- (0, 1.5) arc (-90:90:0.25 
			and 0.5) ;

			
			\path [fill = white, draw = white] (6, 4) arc (90:270:0.25 and 0.5) (6, 3 )-- (6, 4);

			
			\draw [pattern = horizontal lines, very thick] (6, 4) arc (90:270:0.25 and 0.5) (6, 3 ) -- (6.7, 3) arc (-90:90:0.25 and 0.5) (6.7, 4) -- (6, 4) arc (90:270:0.25 and 0.5); 
			\draw [very thick ] (6.7, 4) arc (90:270:0.25 and 0.5); 
			
			\draw [dashed, thick] (6, 3)  arc (-90:90:0.25 and 0.5);

			\draw  [fill = white, thick] (2.6, 3.5) to [out = -45, in = -135] (4, 3.5) to [out = 135, in = 45] (2.6, 3.5); 
			\draw [thick] (4, 3.5) to [out = 45, in = -135] (4.08, 3.6);
			\draw [thick] (2.6, 3.5) to [out = 135, in = -45] (2.52, 3.6);
			
			
			\draw [very thick]  (-0.7, 0) -- (6.7, 0);
			
			\filldraw (-0.7, 0) circle (1.5pt);	
			\filldraw (6.7, 0) circle (1.5pt);
			\filldraw (3, 0) circle (1.5pt);


			\draw [->,thick] (-0.7, 1.35) -- (-0.7, 0.1); 
			\draw [->, thick] (0, 1.35) -- (0, 0.1);
			
			\draw [->, thick] (1, 1.45) -- (1, 0.1);
			
			\draw [->, thick] (2, 1.7) -- (2, 0.1); 
			\draw [->, thick] (3, 2.05) -- (3, 0.1);
			\draw [->, thick] (4, 2.4) -- (4, 0.1);
			\draw [->, thick] (5, 2.7) -- (5, 0.1);
			\draw [->, thick] (6, 2.85) -- (6, 0.1);
			\draw [->, thick] (6.7, 2.85) -- (6.7, 0.1); 
			
			\node at (7.5, 1.5) {$\varphi$}; 
			\end{tikzpicture}
			\caption{the original map $\varphi$}
		\end{figure}
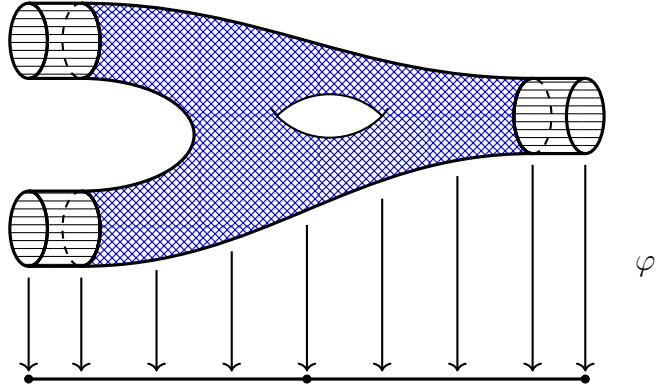
		
		\begin{figure}[H]
			\hspace*{0.5cm}  
			\begin{tikzpicture}[scale=1, every node/.style={transform shape}]
			
			\path [pattern = crosshatch, pattern color = blue] (0, 5) to [out = 0, in = 180]  (6, 4) -- (6, 3) to [out = 180, in = 0] (0, 1.5) -- (0, 2.5)   arc (-90:90:1.5 and 0.75) (0, 4) -- (0, 5);

			\draw[very thick] 	(0, 5) to [out = 0, in = 180]  (6, 4) ; 
			\draw[very thick]  (6, 3) to [out = 180, in = 0] (0, 1.5); 
			\draw [very thick] (0, 2.5)   arc (-90:90:1.5 and 0.75) ;

			
			\path [fill = white, draw = white] (0, 4) arc (-90:90:0.25 and 0.5) (0, 5)-- (0, 4); 
			\draw [dashed, thick] (0, 5)  arc (90:270:0.25 and 0.5); 
			
			\draw [pattern = horizontal lines, very thick] (-0.7, 5) arc (90:450:0.25 and 0.5);

			\draw [pattern = horizontal lines, very thick] (0, 4) arc (-90:90:0.25 and 0.5) (0, 5) -- (-0.7, 5) arc (90:-90:0.25 and 0.5) (-0.7, 4) -- (0, 4) arc (-90:90:0.25 
			and 0.5) ; 
			
			
			\path [fill = white, draw = white] (0, 1.5) arc (-90:90:0.25 and 0.5) (0, 2.5 )-- (0, 1.5); 
			
			
			\draw [dashed, thick] (0, 2.5)  arc (90:270:0.25 and 0.5); 
			
			\draw [pattern = horizontal lines, very thick] (-0.7, 2.5) arc (90:450:0.25 and 0.5);

			\draw [pattern = horizontal lines, very thick] (0, 1.5) arc (-90:90:0.25 and 0.5) (0, 2.5) -- (-0.7, 2.5) arc (90:-90:0.25 and 0.5) (-0.7, 1.5) -- (0, 1.5) arc (-90:90:0.25 
			and 0.5) ;

			
			\path [fill = white, draw = white] (6, 4) arc (90:270:0.25 and 0.5) (6, 3 )-- (6, 4);

			
			\draw [pattern = horizontal lines, very thick] (6, 4) arc (90:270:0.25 and 0.5) (6, 3 ) -- (6.7, 3) arc (-90:90:0.25 and 0.5) (6.7, 4) -- (6, 4) arc (90:270:0.25 and 0.5); 
			\draw [very thick ] (6.7, 4) arc (90:270:0.25 and 0.5); 
			
			\draw [dashed, thick] (6, 3)  arc (-90:90:0.25 and 0.5);

			\draw  [fill = white, thick] (2.6, 3.5) to [out = -45, in = -135] (4, 3.5) to [out = 135, in = 45] (2.6, 3.5); 
			\draw [thick] (4, 3.5) to [out = 45, in = -135] (4.08, 3.6);
			\draw [thick] (2.6, 3.5) to [out = 135, in = -45] (2.52, 3.6);
			
			
			\draw [very thick]  (-0.7, 0) -- (6.7, 0);
			
			\filldraw (-0.7, 0) circle (1.5pt);	
			\filldraw (6.7, 0) circle (1.5pt);
			\filldraw (3, 0) circle (1.5pt);

			\path [pattern = crosshatch, pattern color  = blue!20!white, opacity = 0.5] (0.3, 1.3) to [out = 0, in = 180] (5.8, 2.8) -- (3, 0.1) -- (0.3, 1.3);  
			
			%

			\draw [->,thick] (-0.7, 1.4) -- (-0.7, 0.1); 
			
			\draw [->, thick] (-0.2, 1.4) -- (2, 0.1); 
			\draw [->, thick] (-0.35, 1.4) -- (1.15, 0.1); 	
			\draw [->, thick] (-0.5, 1.4) -- (0.36, 0.1);
			
			\draw [ thick] (0.1, 1.4) -- (3, 0.1); 
			\draw [thick] (1, 1.45) -- (3, 0.1);
			
			\draw [ thick] (2, 1.7) -- (3, 0.1); 
			\draw [ thick] (3, 2.05) -- (3, 0.1);
			\draw [ thick] (4, 2.4) -- (3, 0.1);
			\draw [ thick] (5, 2.7) -- (3, 0.1);

			\draw [thick] (5.9, 2.9) -- (3, 0.1);
			\draw [->, thick] (6.7, 2.9) -- (6.7, 0.1); 
			\draw [->, thick] (6.35, 2.9) -- (4.85, 0.1); 
			\draw [->, thick] (6.2, 2.9) -- (4.05, 0.1); 
			\draw [->, thick] (6.5, 2.9) -- (5.64, 0.1);

			\node at (7.5, 1.5) {$\varphi_1$}; 
			\end{tikzpicture}
			\caption{the new map $\varphi_1$}
		\end{figure}
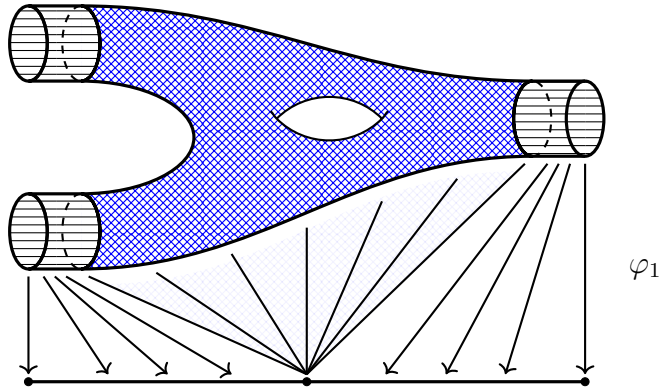
		
		\item \textbf{Induction step}.	Suppose we have defined the control map \[ \varphi_k \colon Z^{(k)} =  \varphi^{-1}(X^{(k)}) \to X^{(k)}. \]   Now let us extend  $\varphi_k$ to a control map \[ \varphi_{k+1}\colon Z^{(k+1)} = \varphi^{-1}(X^{(k+1)}) \to  X^{(k+1)}.\] Intuitively, for each simplex $\Delta^{k+1}$ of $X$, we shall define $\varphi_{k+1}$ so that most of $\varphi^{-1}(\Delta^{k+1})$ is mapped to the barycenter of $\Delta^{k+1}$, and the remaining part of $\varphi^{-1}(\Delta^{k+1})$ (which  again has an appropriate product structure) is stretched out accordingly.  
		
		More precisely, note that $Z^{(k+1)}$ has product structure near $Z^{(k)}$. Let us denote by $Z^{(k)}_\varepsilon$ for such a small open neighborhood (with product structure) of $Z^{(k)}$ in $Z^{(k+1)}$.   Let $Z^{(k+1)}_c = Z^{(k+1)} - Z^{(k)}_\varepsilon $ be the complement of $Z^{(k)}_\varepsilon$ in $Z^{(k+1)}$. We define a new control map 
		\[ \varphi_2\colon Z^{(2)} = \varphi^{-1}(X^{(2)}) \to X\] by mapping each component of $Z^{(2)}_c$ to the barycenter of the corresponding $(k+1)$-simplex in $X$, and stretching out $Z^{(1)}_\varepsilon$ (which has a product structure) accordingly.   It is clear that this  process keeps $\varphi_k$ unchanged on $Z^{(k)}$.  
	\end{enumerate}	
	In the end, we obtain a new control  map $\bar \varphi = \varphi_n \colon Z \to X$, where  $n$ is the dimension of $X$. It is clear from the construction  that $\bar \varphi$ is homotopic to $\varphi$. Moreover, $h$ restricts to a proper homotopy equivalence 
	\[  h\colon \bar{\psi}^{-1} U \to \bar{\varphi}^{-1} U  \]
	for all open subsets $U\subset X$, where $\bar{\psi} = \bar{\varphi}\circ h$.  This finishes  the proof. 
	
\end{proof}

\begin{definition}
	Let $Y$ and $Z$ be a pair of PL manifolds equipped with triangulations. A homotopy equivalence $h\colon Y \to Z$ is said to be PL infinitesimally controlled over $X$ via a control map $\bar\varphi \colon Z \to X$ if $h$ is PL controlled over $X$ via $\bar{\varphi}$ and  $h$ restricts to a proper homotopy equivalence 
	\[  h\colon \bar{\psi}^{-1} U \to \bar{\varphi}^{-1} U  \]
	for all open subsets $U\subset X$, where $\bar{\psi} = \bar{\varphi}\circ h$.
\end{definition}

The following is an immediate corollary of Proposition $\ref{prop:gaincontrol}$. 

\begin{corollary} \label{cor:gaincontr}
	Suppose $\theta = (M, \partial M, \varphi, N, \partial N, \psi, f)$ is an element in $\mathcal S^{PL}_{m}(X)_{PL}$. 
	Then there exists a control map $\bar{\varphi}\colon M \to X$ such that $\bar{\varphi}$ is homotopic to $\varphi$ and on the boundary  $f$ restricts to a PL infinitesimally controlled homotopy equivalence 
	$f|_{\partial N} \colon \partial N \to \partial M$.  
\end{corollary}

In order to more directly apply the discussion above to the geometrically controlled category (Section $\ref{sec:geomhp}$), let us state the PL infinitesimal control in terms of triangulations. 
We borrow the notation from Definition $\ref{def:infcon}$. In our current situation,  we can choose proper simplicial maps \[ \Phi\colon CM \to X\times [1, \infty) \textup{ and } \Psi\colon CN \to X\times [1, \infty),\]   
\[ F\colon CN  \to CM \textup{ and } G\colon CM \to CN,  \]
and a proper simplicial homotopy $\{H_s\}_{0\leq s \leq 1}$ between 
\[ H_0 = F\circ G \textup{ and } H_1 = \id\colon CM\to CM\] and 
a proper simplicial homotopy $\{H'_s\}_{0\leq s \leq 1}$ between 
\[  H'_0 = G\circ F \textup{ and } H'_1 = \id\colon CN\to CN \]
such that the following are satisfied:
\begin{enumerate}[(1)]
	\item $\Phi = \bar \varphi \times \id \colon \partial M\times [1, \infty) \to \partial M\times [1, \infty)$, where $\bar{\varphi}$ is the new  controlled map obtained from $\varphi$ as in Proposition $\ref{prop:gaincontrol}$ above;  
	$\Psi = \bar \psi \times \id \colon \partial N\times [1, \infty) \to \partial N\times [1, \infty)$, where $\bar{\psi} =  \bar{\varphi}\circ f$;
	and  $F = f\times \id\colon \partial N\times [1, \infty) \to \partial M\times [1, \infty) $ with commutative diagram  
	\[  \xymatrix{ \partial N\times [1, \infty)  \ar[dr]_{\bar\psi \times \id}  \ar[rr]^{f\times \id} &  & \partial M \times [1, \infty) \ar[dl]^{\bar{\varphi}\times \id}  \\
		&  X\times[1, \infty) & 
	}\]
	\item $X\times [1, \infty)$ is equipped with a triangulation of \emph{bounded geometry} (cf. Definition $\ref{def:uniformbd}$) such that the sizes of simplices uniformly go to zero, as we approach infinity along the cylindrical direction; for example, this is can be achieved by the standard subdivision in Section $\ref{sec:subdiv}$. 
	\item $f\times \id$ is PL infinitesimally controlled. More precisely, the homotopy 
	\[ H \colon  F\circ G \simeq \id \colon CM\to CM \] restricts   to a homotopy on $\Phi^{-1}(\Delta)$, for  every simplex $\Delta$ in $X\times [1, \infty)$. The homotopy 
	\[ H' \colon  G\circ F \simeq \id \colon CN\to CN \]  restricts to a homotopy on $\Psi^{-1}(\Delta)$, for  every simplex $\Delta$ in $X\times [1, \infty)$.
	\item all maps $f\times \id$, $G$, $H$ and $H'$ are geometrically controlled over the cone $CX$ (see Definition $\ref{def:cone}$) in the sense of Section $\ref{sec:geomhp}$ below. 
\end{enumerate}

Recall that  the surgery long exact sequence built using PL transversality is equivalent to the surgery long exact sequence built using block bundles,  cf. \cite{MR2619602}, also \cite{MR0380841} and \cite{MR668807}. By Proposition $\ref{prop:gaincontrol}$, we have the following commutative diagram:  
\[  
\scalebox{1}{\xymatrixcolsep{1pc}\xymatrix{  
		\ar[r]  &	\mathcal N^{PL^+}_{k+1}(X; w)  \ar[r] \ar[d]_{\alpha_{k+1}}  &   L^{PL^+}_{k+1}(\pi_1 X; w) \ar[r] \ar[d]_{\beta_k}  &  \mathcal S^{PL^+}_{k}( X; w) \ar[r] \ar[d]_{\lambda_k} &  \mathcal N^{PL^+}_{k}(X; w)  \ar[d]_{\alpha_k} \ar[r] &   \\
		\ar[r]  & 	\mathcal N^{PL}_{k+1}(X; w) \ar[r]  &   L^{PL}_{k+1}(\pi_1 X; w) \ar[r] &  \mathcal S^{PL}_{k}( X; w) \ar[r] &  \mathcal N^{PL}_{k}(X; w) \ar[r] & .
}}
\] 
Let us first prove that all vertical maps are isomorphisms, when $k\geq \dim X+5$. We will then show how to handle the general case. 

\begin{proposition}\label{prop:plplus}
	If $k\geq \dim X+ 5$, then the map 
	\[ \lambda_k \colon \mathcal S_k^{PL^+}(X; w) \to \mathcal S_k^{PL}(X; w)\]
	is an isomorphism.
\end{proposition}
\begin{proof}
	Let us first prove that the maps 
	\[ \alpha_k\colon \mathcal N_k^{PL^+}(X; w) \to \mathcal N_k^{PL}(X; w)  \]  
	and 
	\[ \beta_k\colon L_k^{PL^+}(\pi_1 X; w)  \to L_k^{PL}(\pi_1 X; w) \]
	are isomorphisms. This for example can be proven by the same techniques from \cite[Chapter 9]{MR1687388}. The reason for the assumption $k\geq \dim X+ 5$ comes from the fact that,  in order to apply the techniques from \cite[Chapter 9]{MR1687388}, the fibers of the control maps such as $\varphi\colon M \to X$ need to be at least $5$-dimensional.  So when $k\geq \dim X+5$, we have that  $\alpha_k$ and $\beta_k$ are isomorphisms. Now  the proof is finished by applying the five lemma. 
\end{proof}

Now let us consider the case of $\mathcal S_n(X)$, where $ n= \dim X$.  If we apply the periodicity map $\times \mathbb{CP}^2$ twice, then we have 
\[   \mathcal S_n(X) \cong \mathcal S_n^{PL}(X) \hookrightarrow \mathcal S_{n+8}^{PL}(X) \cong \mathcal S_{n+8}^{PL^+}(X).\]
So far in this subsection we have been assuming that $X$ is a PL manifold. Now let us consider that case where $X$ is topological manifold of dimension $n$, which is the case that we are mainly interested in. Let $X'$ be a PL manifold (possibly with dimension greater than $n$) that is homotopic to $X$ (cf. the discussion at the beginning of this subsection). Then after applying the periodicity map $\ell$ times, we have 
\[   \mathcal S_n(X) \cong \mathcal S_n(X') \cong \mathcal S_n^{PL}(X')\hookrightarrow \mathcal S_{n+4\ell}^{PL}(X') \cong \mathcal S_{n+4\ell}^{PL^+}(X'),\]
as long as $n + 4\ell\geq \dim X' + 5.$

As a consequence,  our definition of the higher rho invariant map (see Definition $\ref{def:highrho}$ below)
\[ \rho\colon \mathcal S_n(X) \to  K_n(C_{L, 0}^\ast(\widetilde X)^\Gamma)  \] 
is in fact the composition 
\[  \mathcal S_n(X) \cong  \mathcal S_{n+4\ell}^{PL^+}(X') \xrightarrow{\ \rho\ }   K_{n+4\ell}(C_{L, 0}^\ast(\widetilde X')^\Gamma) \cong K_{n}(C_{L, 0}^\ast(\widetilde X)^\Gamma).  \] 
On the other hand, by the product formula for higher rho invariants, we have 
\[ \rho(\theta \times \mathbb{CP}^2) = \rho(\theta)  \textup{ in } K_{n+4\ell}(C_{L, 0}^\ast(\widetilde X')^\Gamma) \cong K_{n+4\ell + 4}(C_{L, 0}^\ast(\widetilde X')^\Gamma), \]
for any element $\theta \in S_{n+4\ell}^{PL^+}(X')$. This product formula can be proved by a similar argument as in Proposition $\ref{prop:odd}$ and the fact that the signature of $\mathbb{CP}^{2}$ is equal to $1$. Also see Remark $\ref{rmk:sigprod}$ below.  It follows that the higher rho invariant map above is independent of the choices of $X'$ and $\ell$.  From now on, if no confusion is likely to arise, we will write $\mathcal S_n(X)$ in place of  $\mathcal S_{n+4\ell}^{PL^+}(X')$. 
\begin{remark}\label{rmk:sigprod}
	Suppose $Y_1^m$ and $Y_2^n$ are complete Riemannian manifolds of dimension $m$ and $n$.  Let $D_{Y_1}$  , $D_{Y_2}$ and $D_{Y_1\times Y_2}$ be the signature operator on $Y_1$, $Y_2$ and $Y_1\times Y_2$ respectively.  Then  the signature operator $D_{Y_1\times Y_2} = D_{Y_1}\boxtimes D_{Y_2}$, if $m\cdot n$ is even, and  $D_{Y_1\times Y_2} = 2 (D_{Y_1}\boxtimes D_{Y_2})$ if $m\cdot n $ is odd. Here  $D_{Y_1}\boxtimes D_{Y_2}$ is the external product of $D_{Y_1}$ and $D_{Y_2}$. See for example \cite[Lemma 6]{MR2170494}.  
\end{remark}

\begin{remark}
	Everything in this subsection has an obvious equivalent  counterpart in terms of smooth representatives with PL-control. For example, we also have a long exact sequence 
	\begin{equation}
	\begin{aligned}
	\cdots \to \mathcal N^{C^\infty}_{n+1}(X; w)_{PL} \xrightarrow{i_\ast} L^{C^\infty}_{n+1}(\pi_1 X; w)_{PL} \xrightarrow{j_\ast} \\
	\mathcal S^{C^\infty}_{n}( X; w)_{PL} \xrightarrow{\partial_\ast} \mathcal N^{C^\infty}_{n}(X; w)_{PL}\to \cdots 
	\end{aligned}
	\end{equation}
	where the subscript $PL$ stands for $PL$-control. That is, for example, for elements in $\mathcal S^{C^\infty}_{n}( X; w)$, we replace infinitesimal control with PL-control. 
\end{remark}

\section{Additivity of higher rho invariants}\label{sec:hirhoadd}

In this section, we define the higher rho invariant for elements in $\mathcal S_{n}(X)$ using our new description of the structure group, where $X$ is a closed oriented topological manifold of dimension $n$. Furthermore,  we prove that the higher rho invariant defines a group homomorphism from $\mathcal S_{n}(X)$ to $K_{n}(C^\ast_{L,0}(\widetilde X)^\Gamma)$, where $\widetilde X$ is a universal cover of $X$ and $\Gamma = \pi_1 X$.

\subsection{A hybrid $C^\ast$-algebra}
In this subsection, we introduce a certain hybrid $C^\ast$-algebra that is useful for the definition of higher rho invariant.  

Suppose $\Gamma$ is a countable discrete group. Let $Y$ be proper metric space equipped with a proper $\Gamma$-action.

\begin{definition}
	We define $C_{c}^\ast(Y)^\Gamma$ to be the $C^\ast$-subalgebra of $C^\ast(Y)$ generated by elements $\alpha \in C^\ast(Y)$ of the following form:  for any $\varepsilon > 0$, there exists a $\Gamma$-invariant $\Gamma$-cocompact subset $K\subseteq Y$ such that the propagations of $\alpha\chi_{(Y-K)}$ and $\chi_{(Y-K)}\alpha$ are both less than $\varepsilon$.   Here $\chi_{(Y-K)}$ is the characteristic function on $Y-K$. 
\end{definition}

\begin{definition}
	We define $C_{L, 0, c}^\ast(Y)^\Gamma$ to be the $C^\ast$-subalgebra of $C_{L, 0}^\ast(Y)$ generated by elements $\alpha\in C_{L, 0}^\ast(Y)$ of the following form:  for any $\varepsilon > 0$, there exists a $\Gamma$-invariant $\Gamma$-cocompact subset $K\subseteq Y$ such that the propagations of $\alpha(t)\chi_{(Y-K)}$ and $\chi_{(Y-K)}\alpha(t)$ are both less than $\varepsilon$,  for all $t\in [0, \infty)$.    
\end{definition}

Let $X\times [1, \infty)$ be as before. We denote the universal cover of $X$ by $\widetilde X$, and write $\Gamma = \pi_1X$.   It is obvious that  $\mathscr I_r = C^\ast_{L, 0}(\widetilde X\times [1, r]; \widetilde X\times [1, \infty))^\Gamma$ is a two-sided closed ideal of $C^\ast_{L, 0, c}(\widetilde X\times [1, \infty))^\Gamma $, for any $r\geq 1$.

\begin{definition}
	Let $\mathscr I$ be the norm closure of the union 
	\[  \bigcup_{r\geq 1} \mathscr J_r = \bigcup_{r\geq 1} C^\ast_{L, 0}(\widetilde X\times [1, r]; \widetilde X\times [1, \infty))^\Gamma.  \] 
\end{definition}
Note that $\mathscr J$ is also a two-sided closed ideal of   $C^\ast_{L, 0, c}(\widetilde X\times [1, \infty))^\Gamma $. Recall that 
\begin{align*}
K_i(C^\ast_{L, 0}(\widetilde X\times [1, r]; \widetilde X\times [1, \infty))^\Gamma)  = K_i(C^\ast_{L, 0}(\widetilde X\times [1, r])^\Gamma)  =  K_i(C^\ast_{L, 0}(\widetilde X)^\Gamma), 
\end{align*} 
for $i=0, 1$. It follows that $K_i(\mathscr J) = K_i(C^\ast_{L, 0}(\widetilde X)^\Gamma).$
\begin{proposition}\label{prop:coninf}
	The inclusion  $\mathscr J \subset C^\ast_{L, 0, c}(\widetilde X\times [1, \infty))^\Gamma $ induces an isomorphism at the level of $K$-theory. That is, we have 
	\[ K_i(C^\ast_{L, 0, c}(\widetilde X\times [1, \infty))^\Gamma  )\cong  K_i(\mathscr J) =  K_i( C^\ast_{L, 0}(\widetilde X)^\Gamma),  \] 
	for $i=0, 1$. 
\end{proposition}
\begin{proof}
	For notational simplicity, let us write 
	\[ \mathscr A = C^\ast_{L, 0, c}(\widetilde X\times [1, \infty))^\Gamma \]
	We have the following short exact sequence of $C^\ast$-algebras: 
	\[  0 \to \mathscr I \to \mathscr A \xrightarrow{q} \mathscr A/\mathscr I \to 0. \] 
	To prove the proposition, it suffices to show that 
	\[  K_i(\mathscr A/ \mathscr I) = 0 \]
	for $i =0 , 1$. This can be proven by an Eilenberg swindle argument as follows. 
	
	We prove the odd case, that is, $ K_1(\mathscr A/\mathscr I) = 0$; the even case is completely similar.  Suppose $\bar\alpha$ is an invertible element in $\mathscr A^+/\mathscr I$, where $ \mathscr A^+$ is the unitization of $\mathscr A$. Let $\alpha$ be a lift of $\bar{\alpha}$ in $\mathscr A^+$. Without loss of generality, let us assume that $\alpha - 1 \in \mathscr A$.   For each $n\in \mathbb N$, we define an element $\alpha_n\in \mathscr A^+$ as follows:  
	\[ \alpha_n(t) = \begin{cases*} 
	1 & \textup{ if $0\leq t \leq n$},\\
	\alpha(t-n) & \textup{ if $t\geq n$}.
	\end{cases*} \]
	We define 
	\[ \beta(t) = \bigoplus_{n=0}^\infty \big( 1 + \chi_{n} (\alpha_n(t)-1)\chi_n\big),  \]
	where $\chi_n$ is the characteristic function on the set $\widetilde X\times [n, \infty)$. 
	
	We claim that $\beta$ is an element in $\mathscr A^+$. Indeed, recall that, for any $\varepsilon >0$, there exists a positive integer $N$ such that: 
	\begin{enumerate}[(1)]
		\item the propagation of $\alpha(t)$ is $<\varepsilon$, for all $t\geq N$; 
		\item  the propagation of $\alpha(t)|_{\widetilde X\times [N, \infty) }$ is  $<\varepsilon$,  for all $t\geq 0$.	
	\end{enumerate}
	It follows that,   for any $\varepsilon > 0$, we have that 
	\begin{enumerate}[(i)]
		\item the propagation of $\beta(t)$ is $< \varepsilon$ for all $t\geq 2N$;  
		\item the propagation of $\beta(t)|_{ \widetilde X\times [N, \infty) }$ is $< \varepsilon$, for all $t\geq 0$. 
	\end{enumerate} 
	This proves that $\beta \in \mathscr A^+$.

	Let us denote by $\bar\beta$ the image of $\beta$ in $\mathscr A^+/\mathscr I$. We show that $\bar\beta$ is in fact invertible in $\mathscr A^+/\mathscr I$. Let $\omega\in \mathscr A^+$ be the lift of $(\bar{\alpha})^{-1}$. Define
	\[ \mu(t) = \bigoplus_{n=0}^\infty \big( 1 + \chi_{n} (\omega_n(t)-1)\chi_n\big),  \] 
	where $\omega_n$ and $\chi_n$ are defined similarly as above.   Note that the operators 
	\[  1 - \beta(t)\mu(t)   \textup{ and } 1 -\mu(t) \beta(t)  \textup{ are supported in } \widetilde X\times [1, n+1],  \]  
	for all $t\in [0, n]$. It follows that $1 - \beta\mu$ and $1 - \mu \beta $ are in the closure of $\bigcup_{r\geq 1} \mathscr J_r$. In particular,   $1 - \beta\mu$ and $1 - \mu \beta $ are in $\mathscr J$. Therefore, $\bar \beta$ is an invertible element in $\mathscr A^+/\mathscr I$. 
	
	Similarly, let us define 
	\[ \gamma(t) = \bigoplus_{n=1}^\infty \big( 1 + \chi_{n} (\alpha_n(t)-1)\chi_n\big).   \]
	The same argument from above also shows that $\gamma \in \mathscr A^+$. Denote by $\bar{\gamma}$ the image of $\gamma$ in $\mathscr A^+/\mathscr I$. Then  $\bar\gamma$ is also an invertible element in $\mathscr A^+/\mathscr I$. 
	
	Note that $\bar \gamma(t-1) = \bar \beta(t)$. Moreover,   $\bar \gamma$ and $\bar \beta$ are connected by a path of invertible elements  $\bar\gamma_s$, $0\leq s \leq 1$, where  $\bar{\gamma}_s(t) = \bar{\gamma}(t-s)$. Therefore, we have  
	\[   [\bar \gamma] = [ \bar \beta ] \in K_1(\mathscr A/\mathscr I).\] 
	It follows that 
	\[   [\bar \beta] = [\bar\alpha] \oplus [\bar \gamma] = [\bar\alpha]\oplus [\bar \beta] \in K_1(\mathscr A/\mathscr I), \]  
	which implies that $[\bar \alpha] = 0$. This finishes the proof. 
\end{proof}

We also introduce a hybrid version for localization $C^\ast$-algebras. 

\begin{definition}
	We define $C_{L, c}^\ast(Y)^\Gamma$ to be the $C^\ast$-subalgebra of $C_{L}^\ast(Y)$ generated by elements $\alpha\in C_{L}^\ast(Y)$ of the following form:  for any $\varepsilon > 0$, there exists a $\Gamma$-invariant $\Gamma$-cocompact subset $K\subseteq Y$ such that the propagations of $\alpha(t)\chi_{(Y-K)}$ and $\chi_{(Y-K)}\alpha(t)$ are both less than $\varepsilon$ for all $t\in [0, \infty)$.    
\end{definition}

The analogue of Proposition $\ref{prop:coninf}$ does \emph{not} hold for $C_{L, c}^\ast(Y)^\Gamma$. In fact, the following lemma shows that the $K$-theory groups of $C_{L, c}^\ast(Y)^\Gamma$ always vanish. 

\begin{lemma} We have $ K_i(C^\ast_{L, c}(\widetilde X\times [1, \infty))^\Gamma  ) = 0$, 
	for $i=0, 1$. 
\end{lemma}
\begin{proof}
	It is easy to see that 
	\[ K_i(C^\ast_{L, c}(\widetilde X\times [1, \infty))^\Gamma  ) =  K_i(C^\ast_{L}(\widetilde X\times [1, \infty))^\Gamma  ). \]
	The latter is always zero by a standard Eilenberg swindle argument. This finishes the proof. 
\end{proof}

The following corollary is an immediate consequence of the above lemma. 
\begin{corollary}\label{cor:hybrid}
	We have the following isomorphism:
	\[  K_i(C^\ast_{c}(\widetilde X\times [1, \infty))^\Gamma  )\cong  K_{i+1}(C^\ast_{L,  0, c}(\widetilde X\times [1, \infty))^\Gamma ).\]
\end{corollary}
\begin{proof}
	It follows from applying the results above to the $K$-theory long exact sequence of 
	\[ \scalebox{1}{$0 \to C^\ast_{L,0, c}(\widetilde X\times [1, \infty))^\Gamma\to  C^\ast_{L, c}(\widetilde X\times [1, \infty))^\Gamma 
		\to C^\ast_{ c}(\widetilde X\times [1, \infty))^\Gamma \to 0.$} \]
\end{proof}

\subsection{Simplicial complexes and refinements}\label{sec:subdiv}

In this subsection, let us describe a refinement procedure for a given triangulation $M$. This refinement procedure produces a particular subdivision of $M$, denoted by $\sub(M)$,  such that all successive refinements $\sub^n(M) \coloneqq  \sub(\sub^{n-1}(M))$ have uniform bounded geometry, that is, uniform  with respect to $n\in \mathbb N$. There are other treatments of subdivision schemes in the literature which also achieve the uniformity of bounded geometry \cite{MR1799608}\cite{MR3153953}. The following discussion is taken from \cite{higsonxie-witt}.

Let us first recall the notion of typed simplicial complexes.  
\begin{definition}[cf. \cite{MR0327923, MR1115824}]
	Suppose $M$ is a simplicial complex of dimension $n$. Let $M^0$ be the set of vertices of $M$. A \textit{type} on $M$ is a map $\varphi: M^0\to \{0, 1, \cdots, n\}$ such that for any simplex $\omega\in M$, the images by $\varphi$ of the vertices of $\omega$ are pairwise distinct. A simplicial complex equipped  with a type is said to be \textit{typed}. 
\end{definition}

Given any simplicial complex $M$ of dimension $n$, we denote its barycentric subdivision by $Y$. Then $Y$ admits a type. Indeed, $Y$ is the set of totally ordered subsets of $M$, that is, 
\[  Y^{k} = \{(\sigma_0, \cdots, \sigma_k)\mid \sigma_j \in X \textup{ and } \sigma_i \textup{ is a face of } \sigma_{i+1}\}. \]
The dimension function, which maps each barycenter of a simplex of $M$ to the dimension of that simplex, is a type on $Y$.

Now suppose $M$ is a typed simplicial complex of dimension $n$. In particular, this gives a consistent way of ordering the vertices of each simplex in $X$ according to the type map. Therefore, each $k$-simplex of $M$ can be canonically identified with the standard $k$-simplex $\Delta^k$. Now to define our refinement procedure, it suffices to describe certain subdivisions of the standard simplices so that the number of simplices containing any given vertex remains uniformly bounded for all successive subvisions. One way to achieve this is by the so-called standard subdivision \cite[Appendix II.4]{MR0087148}. In the following, we briefly recall the construction of standard subdivision, and refer the reader to \cite[Appendix II.4]{MR0087148} for more details.

Let $\sigma= [v_0, v_1, \cdots, v_k]$ be a standard simplex with its vertices given in the order shown. Set 
\[ v_{ij} = \frac{1}{2}v_i + \frac{1}{2}v_j, \quad i \leq j; \]
in particular, $v_{ii} = v_i$. These are the vertices of the standard subdivision of $\sigma$, denoted $ \sub(\sigma)$. Define a partial ordering on these vertices by setting
\[  v_{ij} \leq v_{kl} \quad \textup{ if }\quad  k \leq i \textup{ and } j \leq l. \]
Now the simplices of $\sub(\sigma)$ are all those formed from the $v_{ij}$ which are in increasing order. Moreover, each simplex in $\sub(\sigma)$ naturally inherits an ordering of vertices from the above partial ordering of $v_{ij}$. It is not difficult to verify that $\sub(\sigma)$ carries a natural type by mapping $v_{ij} \mapsto (j-i)$. 

To summarize, given a typed simplicial complex $M$ of dimension $n$, we apply the above standard subdivision procedure (consistently) to each $n$-simplex of $M$. We call the resulting simplicial complex the standard subdivision of $M$, denoted by $\sub(M)$. Note that $\sub(M)$ is also typed.

\subsection{Hilbert-Poincar\'{e} complexes}\label{sec:hp}
In this subsection, we recall the definition of Hilbert-Poincar\'{e} complexes, which is fundamental for studying higher signatures of topological spaces. We refer to \cite{MR2220522} for more details.

Let $A$ be a unital $C^\ast$-algebra. Consider a chain complex of Hilbert modules over $A$:
\[  E_0 \xleftarrow{b_1} E_{1} \xleftarrow{b_{2}} \cdots \xleftarrow{b_{n}} E_{n} \]
where the differentials $b_j$ are bounded adjointable operators. The $j$-th homology of the complex is the quotient space obtained by dividing the kernel of $b_j$ by the image of $b_{j+1}$. Note that, since the differentials need not to have closed range, the homology spaces are not necessarily Hilbert modules themselves.   

\begin{definition}\label{def:hilpoin}
	An $n$-dimensional Hilbert-Poincar\'{e} complex (over a $C^\ast$-algebra $A$) is a complex of finitely generated Hilbert $A$-modules 
	\[  E_0 \xleftarrow{b_1} E_{1} \xleftarrow{b_{2}} \cdots \xleftarrow{b_{n}} E_{n} \]
	together with adjointable operators $T: E_p \to E_{n-p}$ such that 
	\begin{enumerate}[(1)]
		\item if $v\in E_{p}$, then $T^\ast v = (-1)^{(n-p)p} Tv$;
		\item if $v\in E_{p}$, then $Tb^\ast (v) + (-1)^{p} bT(v) = 0$;
		\item $T$ is a chain homotopy equivalence\footnote{To be precise, by item $(2)$,  we need to impose appropriate signs so that $T$ becomes a genuine chain map. However, we will follow the usual convention and leave it as is, with the understanding that appropriate signs are employed.}  from  the dual complex 
		\[  E_n \xleftarrow{b_n^\ast} E_{n-1} \xleftarrow{b_{n-1}^\ast} \cdots \xleftarrow{b_{1}^\ast} E_{0} \]
		to the complex $(E, b)$. 
	\end{enumerate}
\end{definition}

Now we will associate to each $n$-dimensional Hilbert-Poincar\'{e} complex an index class, called signature, in the $K$-theory group $K_n(A)$.

\begin{definition}\label{def:sop}
	Let $(E, b, T)$ be an $n$-dimensional Hilbert-Poincar\'{e} complex. We denote $l$ to be the integer such that
	\[ n = \begin{cases}
	2l & \textup{if $n$ is even,} \\
	2l+1 & \textup{if $n$ is odd.}
	\end{cases} \]
	Define $S: E\to E$ to be the bounded adjointable operator such that 
	\[ S(v) = i^{p(p-1)+l} T(v)\]
	for $v\in E_p$. Here $i = \sqrt{-1}$. 
\end{definition} 

It is not hard to verify that $S = S^\ast$ and $bS + Sb^\ast = 0$. Moreover, if we define $B = b+ b^\ast$, then the self-adjoint operators $B \pm S: E \to E$ are invertible  \cite[Lemma 3.5]{MR2220522}.

\begin{definition}\label{def:sig} The signature index of a Hilbert-Poincar\'e complex is defined as follows. 
	\begin{enumerate}[(i)]
		\item Let $(E, b, T)$ be an odd-dimensional Hilbert-Poincar\'{e} complex. Its signature is the class in $K_1(A)$ of the invertible operator 
		\[ (B+ S)(B-S)^{-1}: E_{ev} \to E_{ev}\]
		where $E_{ev} = \oplus_{p} E_{2p}$. 
		\item If $(E, b, T)$ is an even-dimensional Hilbert-Poincar\'{e} complex, then its signature is the class in $K_0(A)$ determined by the formal difference $[P_+]- [P_-]$ of the positive projections of $B+S$ and $B-S$. 
	\end{enumerate}	
\end{definition}

\subsection{Geometrically controlled Poincar\'{e} complexes} \label{sec:geomhp}

In this subsection, we  recall the definition of geometrically controlled Poincar\'{e} complexes \cite{MR2220523}. They are  Hilbert-Poincar\'{e} complexes in the geometrically controlled category. 

\begin{definition}\label{def:uniformbd}
	A simplicial complex $M$ is of bounded geometry if there is a positive integer $k$ such that each of the vertices of $M$ lies in at most $k$ different simplices of $M$.
\end{definition}

\begin{definition}\label{def:geommod}
	Let $X$ be a proper metric space. A complex vector space $V$ is geometrically controlled over $X$ if it is provided with a basis $B\subset V$ and a function $c: B\to X$ with the following property: for every $R>0$, there is an $N< \infty$ such that if $S\subset X$ has diameter less than $R$ then $c^{-1}(S)$ has cardinality less than $N$. From now on, we call such $V$ a geometrically controlled $X$-module, and  the function $c$ a labeling of the elements in $B$. In particular, we say $v\in B$ is labeled by $c(v)\in X$. 
\end{definition}

Note that each geometrically controlled vector space $V$ over $X$ is assigned with a basis $B$. There is a natural completion of $V$ into a Hilbert space $\overline V$ in which the basis $B$ of $V$ becomes an orthonormal basis of $\overline V$. 

Let $V^\ast_{f} = \textup{Hom}_f(V, \mathbb C)$ be the vector space of finitely supported linear functions on $V$. Then $V^\ast_f$ is identified with $V$ under the inner product on $\overline V$.

\begin{definition}\label{def:geomap}
	A linear map $T\colon V\to W$ is geometrically controlled over $X$ if 
	\begin{enumerate}[(1)]
		\item $V$ and $W$ are geometrically controlled;
		\item the matrix coefficients of $T$ with respect to the given bases of $V$ and $W$ are uniformly bounded;
		\item and there is a constant $K>0$ such that the $(v, w)$-matrix coefficients is zero whenever $d(c(v), c(w))>K$. The smallest such $K$ is called the propagation of $T$. 
	\end{enumerate}
\end{definition}

It is easy to see that a geometrically controlled linear map $T\colon V \to W$ has a natural dual 
\[ T^\ast\colon W^\ast_f \to V^\ast_f \]
which is canonically identified with a geometrically controlled linear map, still denoted by $T^\ast$,
\[  T^\ast\colon W \to V.  \]

\begin{definition}
	A chain complex 
	\[  E_0 \xleftarrow{b_1} E_{1} \xleftarrow{b_{2}} \cdots \xleftarrow{b_{n}} E_{n} \]
	is called a geometrically controlled chain complex over $X$ if each $E_p$ is a geometrically controlled $X$-module, and  each $b_p$  is a geometrically controlled linear map. 
\end{definition}

\begin{definition}
	Let $f_1, f_2\colon (E, b) \to (E', b')$ be geometrically controlled chain maps between two geometrically controlled chain complexes $(E, b)$ and $(E', b')$. We say $f_1$ and $f_2$ are geometrically controlled homotopic to each other, if there exists a geometrically controlled linear map $h\colon (E_\ast, b) \to (E'_{\ast +1}, b')$ such that 
	\[  f_1 - f_2 = b'h + hb. \]
	In this case, $h$ is called a geometrically controlled chain homotopy between $f_1$ and $f_2$. 
\end{definition}

Now we give the definition of geometrically controlled Poincar\'{e} complexes. 

\begin{definition}\label{def:gcp}
	An $n$-dimensional geometrically controlled Poincar\'{e} complex (with control respect to $X$) is a complex of geometrically controlled $X$-modules 
	\[  E_0 \xleftarrow{b_1} E_{1} \xleftarrow{b_{2}} \cdots \xleftarrow{b_{n}} E_{n} \]
	together with geometrically controlled linear maps $T\colon E_p \to E_{n-p}$ and  $b\colon E_p \to E_{p-1}  $ such that 
	\begin{enumerate}[(1)]
		\item if $v\in E_{p}$, then $T^\ast v = (-1)^{(n-p)p} Tv$;
		\item if $v\in E_{p}$, then $Tb^\ast (v) + (-1)^{p} bT(v) = 0$;
		\item $T$ is a geometrically controlled chain homotopy equivalence  from  the dual complex 
		\[  E_n \xleftarrow{b_n^\ast} E_{n-1} \xleftarrow{b_{n-1}^\ast} \cdots \xleftarrow{b_{1}^\ast} E_{0} \]
		to the complex $(E, b)$. Here we have identified the finitely supported dual $E^\ast_f$ with $E$. 
	\end{enumerate}
\end{definition}

\begin{example}\label{ex:geoPoincare}
	Our typical example of a geometrically controlled Poincar\'{e} complex comes from a triangulation of a closed smooth manifold (more generally a triangulation of a complete Riemannian manifold without boundary, e.g the manifold $CM$ from above), cf. \cite[Section 3 \& 4]{MR2220523}. 
\end{example}

We introduce the following notion of geometrically controlled homotopy equivalences of geometrically controlled Poincar\'{e} complexes. 

\begin{definition}
	Given two $n$-dimensional geometrically controlled Poincar\'{e} complexes  $(E, b, T)$ and $(E', b', T')$, a geometrically controlled homotopy equivalence between them consists of two geometrically controlled chain maps \[ f\colon (E, b)\to (E', b') \textup{ and } g\colon (E', b')\to (E, b)\]  such that 
	\begin{enumerate}[(1)]
		\item  $g\circ f$ and $f\circ g$ are geometrically controlled homotopic to the identity; 
		\item $fTf^\ast$ is  geometrically  controlled homotopic to $T'$,  where $f^\ast$ is the dual of $f$:
		\[	\xymatrix{ E'_n  \ar[d]^{f^\ast}  &  E'_{n-1} \ar[l]_{b^\ast_n} \ar[d]^{f^\ast} & \cdots \ar[l] & E'_0 \ar[d]^{f^\ast} \ar[l]_{b^\ast_1} \\
			E_n  &  E_{n-1} \ar[l]_{b^\ast_n} & \cdots \ar[l] & E_0 \ar[l]_{b^\ast_1}. }  
		\] 
	\end{enumerate} 
\end{definition}

\begin{remark}
	In the above definition,  it is automatic that $gT'g^\ast $ is also geometrically  controlled homotopic to $T$. Indeed, we have
	\[ gT'g^\ast \simeq g(fTf^\ast) g^\ast = (gf)T (f^\ast g^\ast) \simeq T. \]
\end{remark}

There is an obvious equivariant theory of geometrically controlled Poincar\'{e} complexes. We shall omit the details, and refer the reader to \cite[Section 3]{MR2220523} for further reading.

\subsection{Analytically controlled Poincar\'{e} complexes }\label{sec:acp}

In this subsection, we  recall the definition of analytically controlled Poincar\'{e} complexes \cite{MR2220523}. In particular, we review a natural way to pass from the geometrically controlled category to the analytically controlled category, cf. \cite[Section 3]{MR2220523}. 

Recall from Section $\ref{sec:prem}$ that an $X$-module is a separable Hilbert space $H$ equipped with a	$\ast$-representation of $C_0(X)$, the algebra of all continuous functions on $X$ which vanish at infinity. To distinguish from geometrically controlled $X$-modules, we call such $H$ an analytically controlled $X$-module from now on.

\begin{definition}
	Let $H_1$ and $H_2$ be two analytically controlled $X$-modules. A linear map  $T\colon H_1\to H_2$ is said to be analytically controlled, if $T$ is the norm limit of locally compact and finite propagation bounded operators. 
\end{definition}

\begin{remark}
	In this paper, we have chosen to work with signature operators arising from triangulations of manifolds. This is the bounded case, where all operators are bounded. If one wants to work with \emph{unbounded} signature operators arising from  $L^2$-de Rham complexes of Riemannian manifolds, then one needs a slightly different notion of analytical controls. See \cite[Section 5]{MR2220522} for more details.
\end{remark}

The notion of geometrically controlled homotopy equivalence of geometrically controlled chain complexes naturally passes to the following notion of analytically controlled homotopy equivalence of analytically controlled chain complexes. 

\begin{definition}
	A chain complex 
	\[  E_0 \xleftarrow{b_1} E_{1} \xleftarrow{b_{2}} \cdots \xleftarrow{b_{n}} E_{n} \]
	is called a analytically controlled chain complex over $X$ if each $E_p$ is an analytically controlled $X$-module, and  each $b_p$  is an analytically controlled morphism. 
\end{definition}

\begin{definition}
	Let $f_1, f_2\colon (E, b) \to (E', b')$ be analytically controlled chain maps between two analytically controlled  chain complexes $(E, b)$ and $(E', b')$. We say $f_1$ and $f_2$ are analytically controlled homotopic to each other, if there exists an analytically controlled linear map $h\colon (E_\ast, b) \to (E'_{\ast +1}, b')$ such that 
	\[  f_1 - f_2 = b'h + hb. \]
\end{definition}

Now we introduce the notion of analytically controlled Poincar\'{e} complexes.  

\begin{definition}\label{def:acp}
	An $n$-dimensional analytically controlled Poincar\'{e} complex (with control respect to $X$) is a complex of analytically controlled $X$-modules 
	\[  E_0 \xleftarrow{b_1} E_{1} \xleftarrow{b_{2}} \cdots \xleftarrow{b_{n}} E_{n} \]
	together with analytically controlled linear maps $T\colon E_p \to E_{n-p}$ and $b\colon E_p \to E_{p-1}$  such that 
	\begin{enumerate}[(1)]
		\item if $v\in E_{p}$, then $T^\ast v = (-1)^{(n-p)p} Tv$;
		\item if $v\in E_{p}$, then $Tb^\ast (v) + (-1)^{p} bT(v) = 0$;
		\item and $T$ is an analytically controlled chain homotopy equivalence  from  the dual complex 
		\[  E_n \xleftarrow{b_n^\ast} E_{n-1} \xleftarrow{b_{n-1}^\ast} \cdots \xleftarrow{b_{1}^\ast} E_{0} \]
		to the complex $(E, b)$.
	\end{enumerate}
\end{definition}

The following theorem is a rephrasing of a theorem of Higson and Roe \cite[Theorem 3.14]{MR2220523}.
\begin{theorem}[{\cite[Theorem 3.14]{MR2220523}}]
	Every geometrically controlled Poincar\'{e} complex naturally defines an analytically controlled Poincar\'{e} complex, by $\ell^2$-completion.  
\end{theorem}

We introduce the following notion of analytically controlled homotopy equivalences of analytically  controlled Poincar\'{e} complexes.

\begin{definition}
	Given two $n$-dimensional analytically controlled Poincar\'{e} complexes  $(E, b, T)$ and $(E', b', T)$, an analytically controlled homotopy equivalence between them consists of two analytically controlled chain maps \[ f\colon (E, b)\to (E', b') \textup{ and } g\colon (E', b')\to (E, b)\] such that 
	\begin{enumerate}[(1)]
		\item  $g\circ f$ and $f\circ g$ are analytically controlled homotopic to the identity, 
		\item and $f Tf^\ast $ is  analytically controlled homotopic to $T'$, where $f^\ast$  is the adjoint of $f$ .
	\end{enumerate} 
\end{definition}


For an analytically controlled Poincar\'{e} complex, its signature index naturally lies in the $K$-theory of the Roe algebra $C^\ast(X)$.

\begin{definition}\label{def:anasig}
	\begin{enumerate}[(i)]
		\item Let $(E, b, T)$ be an odd-dimensional analytically controlled Poincar\'{e} complex. Its signature is the class in $K_1(C^\ast(X))$ of the invertible operator 
		\[ (B+ S)(B-S)^{-1}: E_{ev} \to E_{ev}\]
		where $E_{ev} = \oplus_{p} E_{2p}$. 
		\item If $(E, b, T)$ is an even-dimensional analytically controlled Poincar\'{e} complex, then its signature is the class in $K_0(C^\ast(X))$ determined by the formal difference $[P_+]- [P_-]$ of the positive projections of $B+S$ and $B-S$. 
	\end{enumerate}
	
\end{definition}

The following simpler notion of analytically controlled homotopy equivalence will also be useful later. 

\begin{definition}
	Let $(E, b)$ be an analytically controlled chain over $X$. An \emph{operator homotopy} of analytically controlled Poincar\'{e} duality operators on $(E, b)$ is a norm continuous family of  operators $T_s$, $s\in [0, 1]$, such that each $(E, b, T_s)$ is an analytically controlled Poincar\'{e} complex. 
\end{definition}

\begin{lemma}[{cf. \cite[Lemma 4.6]{MR2220522}}]\label{lm:trivial}
	If a Poincar\'{e} duality operator $T$ on an analytically controlled Poincar\'{e} complex $(E, b)$ is  operator homotopic to $-T$ through a path of analytically controlled duality operator $T_s$,
	then the path 
	\[  (B+ S)(B-S_s)^{-1} \]
	is a norm-continuous path of analytically controlled invertible elements connecting $(B+S)(B-S)^{-1}$ to the identity. 
\end{lemma}

There is an obvious equivariant theory of analytically controlled Poincar\'{e} complexes. We shall omit the details, and refer the reader to \cite[Section 2]{MR2220523} for further reading.

\begin{remark}
	If one prefers the exposition in terms of Hilbert $C^\ast$-modules, there is a natural way to make sense of everything in this subsection by using Roe algebras in Section $\ref{sec:prem}$. More precisely, we fix an ample and nondegenerate analytically controlled $X$-module $H$. Let $C^\ast(X)$ be the norm closure of locally compact and finite propagation bounded linear operators from $H$ to $H$. That is, $C^\ast(X)$ is the Roe algebra of $X$ associated to $H$. Now suppose $H'$ is any other analytically controlled $X$-module. We define $E(H, H')$ to be the norm closure of locally compact and finite propagation bounded linear operators from $H$ to $H'$. Then clearly $E(H, H')$ carries a natural right Hilbert $C^\ast(X)$-module structure. It is not difficult to see that such a language will give an equivalent description of the discussion in this subsection. 
\end{remark}

\subsection{Higher rho invariant}\label{sec:highrho}

In this subsection, we define the higher rho invariant  for elements in our new description of structure group.  By Proposition $\ref{prop:smooth}$, without loss of generality, it suffices to construct the higher rho invariant and prove its additivity for smooth or PL representatives in $\mathcal S_{n}(X)$. So throughout this subsection, we will be working in the PL category, unless otherwise specified.

Let $\theta = (M, \partial M, \varphi, N, \partial N, \psi, f)$ be an element of $\mathcal S_{n}(X)$.   By the discussion of Section $\ref{sec:PL}$, without loss of generality, we can assume $\theta$ consists of the following data: 
\begin{enumerate}[(1)]
	\item two triangulated  PL manifolds with boundary  $(M, \partial M)$ and  $(N, \partial N)$ with $\dim M = \dim N = n$;
	\item a control map $\varphi\colon M \to X$ which is PL transverse to the triangulation of $X$;
	\item a PL homotopy equivalence  
	\[ f\colon (N, \partial N) \to (M, \partial M)\] such that $\varphi \circ f = \psi$. Moreover, on the boundary $f$ restricts to a PL infinitesimally controlled homotopy equivalence $f|_{\partial N} \colon \partial N\to \partial M$ over $X$. See the discussion after Corollary $\ref{cor:gaincontr}$ for more details. 
\end{enumerate}

Let $X\times [1, \infty)$ be equipped with the product metric, where the metric on $[1,\infty)$ is the standard Euclidean metric.  By using the standard subdivision of Section $\ref{sec:subdiv}$, there exists a triangulation $\textup{Tri}_{X\times[1, \infty)}$ of $X\times [1, \infty)$ such that 
\begin{enumerate}[(1)]
	\item 
	$\textup{Tri}_{X\times[1, \infty)}$  has bounded geometry in the sense of Definition $\ref{def:uniformbd}$; 
	\item the sizes\footnote{Here the size of a simplex is measured with respect to the product metric on $X\times [1, \infty)$.} of simplices in $\textup{Tri}_{X\times[1, \infty)}$   uniformly go to zero, as we approach infinity along the cylindrical direction. 
\end{enumerate}

Recall that every locally finite simplicial complex carries a natural path metric, whose restriction to each $n$-simplex is the Riemannian metric obtained by identifying the $n$-simplex with the standard $n$-simplex in the Euclidean space $\mathbb R^n$. Such a metric is called a simplicial metric.  
\begin{definition}\label{def:cone}
	Let $X\times [1, \infty)$ be equipped with the triangulation $\textup{Tri}_{X\times[1, \infty)}$ from above. We define the simplicial metric cone of $X$, denoted by $CX$, to be the space $X\times [1, \infty)$ equipped with the simplicial metric determined by $\textup{Tri}_{X\times[1, \infty)}$. 
\end{definition}

\begin{remark}
	In order to avoid possible confusion between $CX$ and $
	X\times [1, \infty)$, from now on, the notation $X\times [1, \infty)$ will only stand for the space $X\times [1, \infty)$ equipped with the product metric. 
\end{remark}

Recall that  the space of $M$ attached with a cylinder is defined to be
\[ CM = M\cup_{\partial M} (\partial M\times [1, \infty)). \]
Let us fix a triangulation of $CM$ as follows. On $M$, it is the original triangulation of $M$. The triangulation on $\partial M\times [1, \infty)$ is the pullback triangulation of $\textup{Tri}_{X\times[1, \infty)}$ under the map \[ \varphi_{\partial}\times \id\colon \partial M \times [1, \infty) \to CX, \] where $\varphi_{\partial}\colon \partial M \to X$ is the restriction of $\varphi\colon M \to X$ on $\partial M$. More precisely, for every simplex $\Delta^k \subset CX$,  the inverse image $(\varphi_{\partial}\times \id)^{-1}(\Delta^k)$ is a product $K\times \Delta^k$, where $K$ is some triangulated submanifold of $\partial M$.

Let $\Gamma = \pi_1 X$. We denote by $\widetilde{CM}$ (resp. $\widetilde{CN}$)  the corresponding $\Gamma$-cover of $CM$ (resp. $CN$) induced by $\Phi\colon CM \to X\times [1, \infty)$ (resp. $\Psi\colon CN \to X\times [1, \infty)$). Here we have borrowed the same notation from Definition $\ref{def:infcon}$. 

Note that the simplicial decomposition of $CM$ (resp. $CN$) naturally lifts to a $\Gamma$-equivariant simplicial decomposition of $\widetilde{CM}$ (resp. $\widetilde{CN}$). 
Consider the $\Gamma$-equivariant geometrically controlled Poincar\'{e} complex 
\[  E_0(\widetilde{CM}) \xleftarrow{b_1} E_1(\widetilde{CM}) \xleftarrow{b_2} \cdots \xleftarrow{b_n} E_n(\widetilde{CM})\]
associated to the above $\Gamma$-equivariant simplicial decomposition of $\widetilde{CM}$, where 
\begin{enumerate}[(1)]
	\item $E_i(\widetilde{CM})$ is a geometrically controlled $(\widetilde{CX}, \Gamma)$-module, 
	\item $b_i$ is a geometrically controlled morphism,
	\item and the Poincar\'e duality map $T$ is given by the usual cap product with the $\Gamma$-equivariant fundamental class of $\widetilde{CM}$. 
\end{enumerate}
The $\ell^2$-completion of this $\Gamma$-equivariant geometrically controlled Poincar\'{e} complex gives rises to a $\Gamma$-equivariant analytically controlled Poincar\'{e} complex, still denoted  by $(E(\widetilde{CM}), b, T)$. We summarize this in the following lemma. 

\begin{lemma}
	$(E(\widetilde{CM}), b, T)$ is a $\Gamma$-equivariant  analytically controlled Poincar\'{e} complex. 
\end{lemma}

Similarly, we have the $\Gamma$-equivariant analytically controlled Poincar\'{e} complex  $(E(\widetilde{CN}), b', T')$ associated to $\widetilde{CN}$: 
\[  E_0(\widetilde{CN}) \xleftarrow{b'_1} E_1(\widetilde{CN}) \xleftarrow{b'_2} \cdots \xleftarrow{b'_n} E_n(\widetilde{CN}).\]

Now let us proceed to define the higher rho invariant for elements in $\mathcal S_{n}(X)$. We will only give the details for the odd dimensional case, that is, the case where $n$ is odd. The even dimensional case is completely similar.

In the following, all controls are measured with respect to the control maps 
\[ \Phi\colon CM\to X\times [1,  \infty) \textup{ and } \Psi\colon CN \to X\times [1, \infty).\] 
For notational simplicity,  we shall drop the term ``$\Gamma$-equivariant'' in the construction below, with the understanding that all steps below are done $\Gamma$-equivariantly. Also,  we write $E = E(\widetilde{CM})$ and $E' = E(\widetilde{CN})$. 

Let us consider  the $\Gamma$-equivariant analytically controlled Poincar\'{e} complex 
\[  (\mathcal E, \mathfrak b, \mathcal  T) = (E \oplus E', b\oplus b', T\oplus -T') \]
Let $\mathcal B = B \oplus B'$ and $\mathcal S = S \oplus -S'$ (cf. Definition $\ref{def:sop}$). The signature index  of $ (\mathcal E, \mathfrak b, \mathcal  T)$ is defined to be the class of $(\mathcal B + \mathcal S)(\mathcal B-\mathcal S)^{-1}$ in $K_1(C^\ast(CX))$. Clearly, the  map 
\begin{equation}\label{eq:propmap}
\tau \colon CX \to X\times [1, \infty)
\end{equation} 
by $\tau(x, t) = (x, t)$ is a proper continuous map that induces a $C^\ast$-algebra homomorphism
\[  \tau_\ast\colon C^\ast(CX) \to C_c^\ast(X\times [1, \infty)).\]
Similarly, we have 
\[  \tau_\ast\colon C_{L, 0}^\ast(CX) \to C_{L, 0, c}^\ast(X\times [1, \infty)).\]
There are also obvious $\Gamma$-equivariant versions. In the following, unless otherwise specified, all elements below are to be thought of as their corresponding images under the map $\tau_\ast$.  

Following Higson and Roe \cite[Section 4]{MR2220522}, we shall first build an explicit path of invertible elements connecting 
\[  (\mathcal B + \mathcal S)(\mathcal B-\mathcal S)^{-1} \]
to the identity element, within the $C^\ast$-algebra  $C^\ast_{c}(\widetilde X\times [1, \infty))^\Gamma$.

Let $F\colon E' \to E$ and $G\colon E \to E'$ be the chain maps induced by $F\colon CN \to CM$ and $G\colon CM \to CN$. 
\begin{lemma}\label{lm:extra}
	With the same notation above, $F\colon E' \to E$ and $G\colon E \to E'$ satisfy the following conditions: 
	\begin{enumerate}[$(1)$]
		\item the chain maps $F\colon (E', b')\to (E, b)$ and $G\colon (E, b)\to (E', b')$ are analytically controlled; 
		\item  $GF$ and $FG$ are analytically controlled homotopic to the identity; 
		\item $GTG^\ast$ is  analytically controlled homotopic to $T'$. 
	\end{enumerate} 
	Moreover, for any $\varepsilon > 0$, there exists a positive number $k$ such that the chain maps $F$, $G$ and the various homotopies  have propagation $< \varepsilon$ away from $N\cup (\partial N\times [1, k])$ and $M\cup (\partial M\times [1, k])$. 
\end{lemma}

In particular, we see that the operator 
\[ \begin{pmatrix}
T & 0 \\ 0 & (s-1)T'-s G TG^\ast
\end{pmatrix}  \]
implements a Poincar\'{e} duality operator for the complex $(\mathcal E, \mathfrak b)$,  for each $s\in [0, 1]$. This path connects the  duality operator $T\oplus -T'$ to $T \oplus -G TG^\ast$. 

Now consider the operator 
\[  \widehat T(s)= \begin{pmatrix}
\cos(s) T & \sin(s) TG^\ast \\ \sin(s)G T & -\cos(s) G TG^\ast
\end{pmatrix}. \]

\begin{lemma}\label{lm:extra2}
	The operator $\widehat{T}(s)$ implements a Poincar\'{e} duality operator for the complex $(\mathcal E, \mathfrak b)$, for each $s\in [0, \pi/2]$. 
\end{lemma}
\begin{proof}
	An analytically controlled homotopy inverse of  $\widehat T(s)$ is given by 
	\[ \begin{pmatrix}
	\cos(s) \alpha & \sin(s) \alpha F  \\ \sin(s)F^\ast\alpha & -\cos(s) F^\ast\alpha F  
	\end{pmatrix} \] 
	where $\alpha$ is an analytically controlled homotopy inverse of $T$. 
\end{proof}	

Now concatenate the two paths above and denote the resulting path by 
\[   (\mathcal E, \mathfrak b, \mathscr T_t) \]
with $t\in [0, 1]$.  Let $(\mathcal B +\mathcal S_t)(\mathcal B - \mathcal S_t)^{-1}$ be the invertible operator  representing the corresponding signature of $(\mathcal E, \mathfrak b, \mathscr T_t)$ (cf. Definition $\ref{def:sig}$).

Note that  the last duality operator 
\[  \mathscr T_1 =\begin{pmatrix}
0 & TG^\ast \\ GT & 0 
\end{pmatrix} \]
is analytically controlled homotopic to its addtive inverse along the path 
\begin{equation}\label{eq:htpy}
\begin{pmatrix}
0 & \exp(is) TG^\ast \\ \exp(-is) G T & 0 
\end{pmatrix}, 
\end{equation}
with $s\in [0, \pi]$. By Lemma $\ref{lm:trivial}$, we see that $(\mathcal B + \mathcal S_1)(\mathcal B -\mathcal S_1)^{-1}$ is connected to the identity operator through a path of analytically controlled invertible elements.

To summarize, we have constructed a path of analytically controlled invertible elements in $C_c^\ast(\widetilde X\times [1, \infty))^\Gamma$ connecting  $(\mathcal B + \mathcal S)(\mathcal B -\mathcal S)^{-1}$ to the identity element. Let us re-parametrize the time variable, and denote this path by 
\[ V_s =  (\mathcal B + \mathcal S_s)(\mathcal B -\mathcal S_s)^{-1}  \]
with $ V_0 = I$ and $ V_1 =  (\mathcal B + \mathcal S)(\mathcal B -\mathcal S)^{-1} $.

Now we shall extend this path to obtain an element in $C^\ast_{L, 0, c}(\widetilde X\times [1, \infty))^\Gamma$. In other words, we will construct  an element  \[ W \in C^\ast_{L, 0, c}(\widetilde X\times [1, \infty))^\Gamma\] so that $W_s = V_s\oplus I$ for all $s\in [0, 1]$, where $I$ is the identity operator. 

In fact, the construction of $W_s$, starting at $s\geq 1$, coincides with the construction of the $K$-homology class of the signature operator on $CM \bigsqcup -CN$ with controls with respect to $CX$.  To be more precise, there are in fact two equivalent ways of constructing the path $W_s$. 
\begin{enumerate}[(i)]
	\item One directly works with the geometrically controlled Poincar\'e complex and its refinements associated to  $CM\bigsqcup -CN$. In particular, everything is controlled over $CX$. This is what we have chosen to do for the construction of $V_s$ above. Note that, although $CM\bigsqcup -CN$ is not a closed PL manifold, it is a complete manifold without boundary. It is easy to see that the construction in Appendix $\ref{app:special}$ (\emph{not} Appendix $\ref{app:gen}$)  applies verbatim to the space $CM\bigsqcup -CN$ controlled over $CX$. As a result, we obtain a $K$-theory class $(W_s)_{0\leq s <\infty}$ in  $K_n( C^\ast_{L, 0, c}(\widetilde X\times [1, \infty) )^\Gamma).$ 
	\item Alternatively, we consider the Poincar\'e space  $M\cup_f(-N)$. Although the space $M\cup_f (-N)$ is not a manifold in general, it is still a space equipped with a Poincar\'e duality. In fact, since $ f\colon\partial N \to \partial M$ is a PL infinitesimally controlled homotopy equivalence,  we can still make sense of the $K$-homology class of its ``signature operator", as in Appendix $\ref{app:gen}$. Let us denote this $K$-homology class by a path of invertible elements $(U_s)_{0\leq s <\infty}$.  Moreover, in the current situation, we also have that 
	$ f\colon N\to M$ is a homotopy equivalence. Similar to the discussion following Lemma $\ref{lm:extra}$, the homotopy equivalence $f$ can be used to connect $U_0$ to the identity operator through a path of invertible elements. Re-parametrize the resulting new path, and define it to be the higher rho invariant of $\theta$ in  $K_n(C_{L, 0}^\ast(\widetilde X)^\Gamma)$.  
\end{enumerate} 
It is not difficult to see that these two constructions define the same $K$-theory class in $K_n(C_{L, 0}^\ast(\widetilde X)^\Gamma) \cong K_n( C^\ast_{L, 0, c}(\widetilde X\times [1, \infty) )^\Gamma). $

To summarize, we have constructed a path of invertible elements $W(\theta)$ for each element $\theta = (M, \partial M, \varphi, N, \partial N, \psi, f) $ in $\mathcal S_n(X)$. 

\begin{proposition}\label{prop:highrho}
	For every element \[ \theta = (M, \partial M, \varphi, N, \partial N, \psi, f) \in \mathcal S_n(X), \] we have $W(\theta)\in C^\ast_{L, 0, c}(\widetilde X\times [1, \infty))^\Gamma$.  
\end{proposition}
\begin{proof}
	Note that the simplicial chain complex associated to the triangulation on $\widetilde{CM}$ (resp. $\widetilde{CN}$)  is a $\Gamma$-equivariant geometrically controlled module over $\widetilde{CX}$. Since the map $f\colon \partial M \to \partial N$ is PL infinitesimally controlled over $X$,  it follows that all maps $F$, $G$, $H$ and $H'$ as in Definition $\ref{def:infcon}$ are $\Gamma$-equivariantly geometrically controlled over $\widetilde{CX}$. Therefore, our construction  produces an element in  $C^\ast_{L, 0}(\widetilde {CX})^\Gamma$, whose image under the map $\tau$ from line $\eqref{eq:propmap}$ is precisely the element 
	\[ W(\theta)\in C^\ast_{L, 0, c}(\widetilde X\times [1, \infty))^\Gamma. \]
\end{proof}

\begin{definition}\label{def:highrho}
	When $n$ is odd,	for each element 
	\[\theta = (M, \partial M, \varphi, N, \partial N, \psi, f) \in \mathcal S_n(X),\]
	we define the higher rho invariant of $\theta$ to be 
	\[  \rho(\theta) = [W(\theta)]\in K_{n}(C^\ast_{L, 0, c}(\widetilde X\times [1, \infty))^\Gamma) \cong K_n(C_{L, 0}^\ast(\widetilde X)^\Gamma).\]
\end{definition}

The definition of the higher rho invariant for the even dimensional case (i.e. for $\mathcal S_n(X)$
when $n$ even) is completely similar. We omit the details.

\begin{remark}
	We point out that, in the odd dimensional case,  the higher rho invariant for signature operators in the literature (cf. \cite[Section 3]{MR2220524} \cite[Remark 4.6]{MR3514938} \cite{MR3622237}) is twice of the  higher rho invariant of this paper, cf. Remark $\ref{rmk:rhocompare}$ and Theorem $\ref{thm:tworho}$ below. 
\end{remark}

We have the following main theorem of our paper. 	
\begin{theorem}\label{thm:main}
	The map 
	\[ \rho\colon \mathcal S_n(X) \to  K_n(C_{L, 0}^\ast(\widetilde X)^\Gamma)  \] 
	is a group homomorphism. 
\end{theorem}
\begin{proof}
	The well-definedness of the map 
	\[ \rho\colon \mathcal S_n(X) \to  K_n(C_{L, 0}^\ast(\widetilde X)^\Gamma)\] will be proved in Theorem $\ref{thm:bordrho}$. Now the group structure on $\mathcal S_n(X)$ is given by disjoint union, and $\rho$ is obviously additive on disjoint unions. This finishes the proof. 
\end{proof}

If $X$ is a closed oriented connected topological manifold of dimension $\geq 5$, then $\mathcal S^\topo(X)$ is naturally identified with  $\mathcal S_n(X)$. Hence we have the following immediate corollary. 

\begin{corollary}
	If $X$ is a closed oriented topological manifold of dimension $ n \geq 5$ with $\pi_1 X = \Gamma$, then the higher rho invariant map 
	\[\rho\colon \mathcal S^\topo(X) \to K_n(C_{L,0}^\ast(\widetilde X)^\Gamma) \]
	is a group homomorphism. 
\end{corollary}

\begin{remark}
	By the discussion following Theorem $\ref{thm:struciso}$ and the discussion following Proposition $\ref{prop:plplus}$, we see that, if $\dim X = n \geq 6$,    the higher rho invariant map in fact defines a group homomorphism from the homology manifold structure group $\mathcal S^{HTOP}(X)$ to $K_n(C_{L,0}^\ast(\widetilde X)^\Gamma)$. More generally, following Remark $\ref{rk:hommfld}$, the higher rho invariant map can also be defined for the homology manifold structure group of a closed oriented connected homology manifold of dimension $\geq 6$. 
\end{remark}

\begin{remark}
	Although we have chosen to work with the reduced version of various $C^\ast$-algebras, we point out that the exact same proofs work equally well for the maximal version of these $C^\ast$-algebras. In particular, we also have a well-defined group homomorphism:
	\[ \rho\colon \mathcal S_n(X) \to  K_n(C_{L, 0}^\ast(\widetilde X)^\Gamma_{\max}).  \] 
\end{remark}

\section{Well-definedness of the higher rho invariant map}
\label{sec:bord}
In this section, we prove that the higher rho invariant map 
\[ \rho\colon \mathcal S_n(X) \to  K_n(C_{L, 0}^\ast(\widetilde X)^\Gamma)\] is well-defined. Our method is modeled upon Higson and Roe's proof for the bordism invariance of higher  signature index \cite[Section 7]{MR2220522}. 

The following definitions are geometrically controlled analogues of the corresponding definitions in    \cite[Section 7]{MR2220522}. We refer the reader to \cite[Section 7]{MR2220522} for more details. 

\begin{definition}
	A complemented subcomplex of the geometrically controlled complex $(E, b)$ is a family of complemented geometrically controlled submodules $E'_p \subset E_p$ such that $b$ maps $E_p'$ to $E'_{p-1}$, for all $p$. 
\end{definition}

For each complemented subcomplex $(PE, Pb)$ of $(E, b)$, there is a corresponding geometrically controlled complement complex $(P^{\perp}E, P^{\perp }b)$. The inclusion $PE\subset E$ is a chain map from $(PE, Pb)$ into $(E, b)$, whereas the orthogonal projection $E\to P^{\perp}E$ gives a chain map from $(E, b)$ onto $(P^\perp E, P^\perp b)$. Note that \[ P^\perp b = P^\perp b P^\perp. \]

\begin{definition}\label{def:ppair}
	An $(n+1)$-dimensional geometrically controlled Poincar\'{e} pair is a geometrically controlled complex 
	\[  E_0 \xleftarrow{b_1} E_{1} \xleftarrow{b_{2}} \cdots \xleftarrow{b_{n}} E_{n} \]
	together with a family of geometrically controlled operators $T: E_p \to E_{n+1-p}$ and a family of geometrically controlled orthogonal projections $P\colon E_p\to E_p$  such that 
	\begin{enumerate}[(1)]
		\item the orthogonal projections $P$ determines a subcomplex of $(E, b)$; that is, $PbP = bP$;
		\item the range of the operator  $T b^\ast  + (-1)^{p} bT \colon E_p \to E_{n-p}$ is contained within the range of $P\colon E_{n-p}\to E_{n-p}$;
		\item $T^\ast  = (-1)^{(n+1-p)p} T\colon E_p \to E_{n+1-p}$;
		\item $P^\perp T$ is a geometrically controlled chain homotopy equivalence  from  the dual complex $(E, b^\ast)$ to $(P^\perp E, P^\perp b)$.
	\end{enumerate}
\end{definition}

\begin{example}
	A typical example of geometrically controlled Poincar\'{e} complexes comes from a triangulation of a smooth manifold with boundary \cite[Section 4.2]{MR2220523}.  
\end{example}

The following lemma is a geometrically controlled analogue of \cite[Lemma 7.4]{MR2220522}.

\begin{lemma}[{\cite[Lemma 7.4]{MR2220522}}]\label{lm:pbd}
	Let $(E, b, T, P)$ be an $(n+1)$-dimensional geometrically controlled Poincar\'{e} pair. The operators 
	\[ T_0 = Tb^\ast + (-1)^p bT\colon E_p \to E_{n-p} \]
	satisfy the following relations: 
	\begin{enumerate}[(1)]
		\item $T_0^\ast = (-1)^{(n-p)p}T_0\colon E_p \to E_{n-p}$;
		\item $T_0 = PT_0 = T_0 P$;
		\item $T_0 b^\ast + (-1)^p bT_0 = 0\colon PE_p \to PE_p$;
		
		\item $T_0 =  Tb^\ast + (-1)^p bT $ induces a geometrically controlled chain homotopy from $(PE, Pb^\ast)$ to $(PE, Pb)$.  
	\end{enumerate}
\end{lemma}
\begin{proof}
	The proof is a combination of the proof of \cite[Lemma 7.4]{MR2220522} together with \cite[Lemma 4.2]{MR2220523}. We leave out the details. 
\end{proof}

The above lemma asserts  $(PE, Pb, T_0)$ is an $n$-dimensional geometrically controlled Poincar\'{e} complex.

\begin{definition}
	The geometrically controlled Poincar\'{e} complex 
	\[ (PE, Pb, T_0)\] is called the boundary of the geometrically controlled Poincar\'{e} pair $(E, b, T, P)$. 
\end{definition}

Note that there is an obvious analogous theory in the analytically controlled category. Moreover, there are obvious equivariant theories for both geometrically controlled Poincar\'{e} pairs and analytically controlled Poincar\'{e} pairs respectively. 

The following theorem is a rephrasing of a theorem of Higson and Roe \cite[Theorem 3.18]{MR2220523}.
\begin{theorem}[{\cite[Theorem 3.18]{MR2220523}}]
	Every geometrically controlled Poincar\'{e} pair naturally defines an analytically controlled Poincar\'{e} pair, by $\ell^2$-completion.  
\end{theorem}

Before we prove the well-definedness of the higher rho invariant map, let us give a proof of the bordism invariance of the $K$-homology class of signature operators (compare \cite[Theorem 2]{MR2170494}). Our proof below is modeled upon Higson and Roe's proof for the bordism invariance of higher signature index \cite[Theorem 7.6]{MR2220522}.  Note that, in the theorem below, we do \emph{not} invert $2$.  

\begin{theorem}[Bordism invariance of $K$-homology signature]\label{thm:bordK}
	Let $V$ be an $(n+1)$-dimensional oriented PL manifold with boundary $\partial V$,  equipped with a continuous map $\psi\colon V\to X$, where $X$ is a proper metric space. Then 
	\[  \ind_{L}(\partial V) = 0 \in K_n(C^\ast_{L}(\widetilde X)^\Gamma),\]
	where $\widetilde X$ is the universal cover of $X$ with $\Gamma = \pi_1 X$.
\end{theorem}
\begin{proof}
	Fix a triangulation of $V$, together with a sequence of successive refinements $\sub^n(V)$  as in Section $\ref{sec:subdiv}$. Note that, for the geometrically controlled Poincar\'{e} pair associated to the triangulation  $\sub^n(V)$, all maps appearing in Definition  $\ref{def:ppair}$ are geometrically controlled, with their propagations go to zero as $n\to \infty$.
	
	Let us denote the geometrically controlled Poincar\'{e} pair associated to the triangulation $\sub^n(V)$ by $(E^{(n)}, b^{(n)}, T^{(n)}, P^{(n)} )$.  Since our construction below works for  these refinements simultaneously, we shall omit the superscript $(n)$ from now on. Equivalently, one consider the direct sum 
	\[   (E, b, T, P) = \bigoplus_{n=1}^\infty (E^{(n)}, b^{(n)}, T^{(n)}, P^{(n)} ).\]
	In particular, by the construction in Appendix $\ref{app:special}$, the geometrically controlled Poincar\'{e} complex $ (PE, Pb, T_0)$ produces a specific representative of the local index $\ind_L(\partial V, \psi) \in K_{n}(C_L^\ast(X))$ of the signature operator of $\partial V$ (cf. Definition $\ref{def:local}$). 
	
	Let $\lambda$ be a real number and define a complex $(\tilde E, \tilde b_\lambda)$ by 
	\[ \tilde E_p = E_p \oplus P^\perp E_{p+1} \textup{ and }  \tilde b_\lambda = \begin{pmatrix}
	b & 0  \\ \lambda P^\perp & -P^\perp b 
	\end{pmatrix}\]
	This is the mapping cone complex for the chain map \[ \lambda P^\perp\colon (E, b) \to (P^\perp E, P^\perp b).\] Together with the operators 
	\[ \tilde T = \begin{pmatrix}
	0 & TP^\perp \\ (-1)^p P^\perp T & 0 
	\end{pmatrix}\colon \tilde E_p \to \tilde E_{n-p}, \]
	the triple $(\tilde E, \tilde b_\lambda, \tilde T)$ is an $n$-dimensional geometrically controlled Poincar\'{e} complex for any $\lambda$ (including $\lambda =0$). Of course, we need to check that $\tilde T$ is indeed a geometrically controlled homotopy equivalence. This can be verified by applying \cite[Lemma 4.2]{MR2220523} to the following commutative diagram\footnote{One needs to take into account of the sign convention when verifying various identities. For example,  the map $(-1)^pP^\perp T$ carries the sign $(-1)^p$ when it maps from $E_p$ to $P^\perp E_{n-p+1}$.} :
	\[ \xymatrixcolsep{1.5pc}\xymatrix{ 0 \ar[r] & (E, b^\ast) \ar[r] \ar[d]^{(-1)^pP^\perp T} & (\tilde E, \tilde b^\ast_\lambda ) \ar[r] \ar[d]^{\tilde T} & (P^\perp E, -b^\ast P^\perp) \ar[d]^{TP^\perp} \ar[r] & 0  \\
		0 \ar[r] & (P^\perp E, -P^\perp b) \ar[r] &  (\tilde E, \tilde b_{\lambda}) \ar[r]  & (E, b) \ar[r] & 0.  } \]
	
	Note that, when $\lambda =-1$, the map  $A(v) = v\oplus 0 \in E_p \oplus P^\perp E_{p+1}  $ defines a geometrically controlled chain homotopy equivalence of geometrically controlled Poincar\'e complexes
	\[  A\colon (PE, Pb, T_0) \to (\tilde E, \tilde b_{-1}, \tilde T). \]
	Indeed, we apply  \cite[Lemma 4.2]{MR2220523} to the following commutative diagram: 
	\[ \xymatrix{ 0 \ar[r] & (PE, Pb) \ar[r]^{=} \ar[d]^{=} & (PE, Pb) \ar[r] \ar[d]^{A} & 0 \ar[d] \ar[r] & 0  \\
		0 \ar[r] & (PE, Pb) \ar[r]^{A} &  (\tilde E, \tilde b_{-1}) \ar[r]^{Q}  & (E', b') \ar[r] & 0  } \]  
	where $E'_p = P^\perp E_{p} \oplus P^\perp E_{p+1}$ with $b' =\begin{psmallmatrix}
	P^\perp b & 0  \\ -1 & -P^\perp b 
	\end{psmallmatrix} $ and $Q$ is the obvious orthogonal projection.  It is easy to see that $(E', b')$ is geometrically controlled chain homotopy equivalent to the trivial chain $0$. Moreover, we have\footnote{Note that the appearance of $(-1)^{p}$ is due to our sign convention. } 
	\[  AT_0A^\ast - \tilde T = h_{p+1}\circ \tilde b^\ast_{-1} + (-1)^p \tilde b_{-1}\circ h_p\colon \tilde E_{p} \to \tilde E_{n-p},    \] 
	where $h_p = \begin{psmallmatrix}
	T & 0 \\ 0 & 0 
	\end{psmallmatrix}\colon E_{p} \oplus P^\perp E_{p+1} \to E_{n-p+1} \oplus P^\perp E_{n-p+2}.$ This shows that $AT_0A^\ast$  and  $\tilde T$ are geometrically controlled homotopic to each other.

	We abuse our notation and denote by 
	\[   (B + S)(B - S)^{-1} \in (C^\ast_L(\widetilde X)^\Gamma)^+\] the explicit representative for the local index $\ind_L(\partial V, \psi)$ constructed by using  $(PE, Pb, T_0)$ as in Appendix $\ref{app:khom}$ (cf. Definition $\ref{def:local}$). 
	Then the same argument from Section $\ref{sec:highrho}$ produces a continuous path of invertible elements in $(C^\ast_L(\widetilde X)^\Gamma)^+$  connecting $(B + S)(B - S)^{-1}$ to   \[ (B_{-1} + S_{-1})(B_{-1}-S_{-1})^{-1} \in (C^\ast_L(\widetilde X)^\Gamma)^+,\] 
	where $(B_{-1} + S_{-1})(B_{-1}-S_{-1})^{-1} $ stands for the representative of the local index constructed out of  $(\tilde E, \tilde b_{-1}, \tilde T)$  (cf. Definition $\ref{def:local}$).  To be precise, we in fact need to stabilize $(B + S)(B - S)^{-1}$ by the identity operator,  and consider \[ (B + S)(B - S)^{-1} \oplus I\] instead. For notational simplicity, we will omit these stabilizing steps. 
	
	On the other hand, there is a continuous path of invertible elements 
	\[ (B_t + S_t)(B_t-S_t)^{-1} \in (C^\ast_L(\widetilde X)^\Gamma)^+ \] 
	representing the local index class constructed out of $(\tilde E, \tilde b_{t}, \tilde T)$ for $t\in [-1, 0]$  (cf. Definition $\ref{def:local}$). Of course, it is important to appropriately control  the propagations of various terms. This can be achieved by Proposition $\ref{prop:uniformbd}$ in Appendix $\ref{app:inv}$.  Note that, for $(\tilde E, \tilde b_{0}, \tilde T)$, the duality operator $\tilde T$ is operator homotopic to its additive inverse along the path 
	\[ \tilde T = \begin{pmatrix}
	0 & \exp(is)TP^\perp \\ (-1)^p \exp(is)P^\perp T & 0 
	\end{pmatrix} \]
	with $s\in [0, \pi]$. Now the same argument from Section $\ref{sec:highrho}$ again shows that 
	\[ (B_0 + S_0)(B_0-S_0)^{-1}\] is connected to the identity by a path of invertible elements in  $(C^\ast_L(\widetilde X)^\Gamma)^+ $. This finishes the proof. 
\end{proof}

\begin{theorem}\label{thm:bordrho}
	The higher rho invariant map 
	\[ \rho\colon \mathcal S_n(X) \to  K_n(C_{L, 0}^\ast(\widetilde X)^\Gamma)\] is well-defined.   	
\end{theorem}
\begin{proof}
	Let $\theta	= (M, \partial M, \varphi, N, \partial N, \psi, f)$ be an element in $\mathcal S_n(X)$. Suppose $\theta$ is cobordant to zero in $\mathcal S_n(X)$. Let 
	\[ (W, \partial W, \Phi, V, \partial V, \Psi, F)\] be a cobordism between $\theta$ and $0$ (cf. Definition $\ref{def:strequiv}$). Note that $F|_{\partial_2 V} \colon \partial_2 V\to \partial_2 W$ is an infinitesimally controlled homotopy equivalence over $X$, thus $\partial_2 V$ and $\partial_2 W$ will not contribute to the higher rho invariant of $F|_{\partial V}\colon \partial V \to \partial W$. More precisely, \[ F|_{\partial_2 V}\colon \partial_2 V \to \partial_2 W\] induces an infinitesimally controlled chain homotopy equivalence between the Poincar\'e pair associated to $\partial_2 V$ and the Poincar\'e pair associated to $\partial_2 W$. It follows that the geometrically controlled  Poincar\'e complex associated to $M\cup_f (-N)$ and its refinements are geometrically  controlled equivalent to the geometrically controlled Poincar\'e complex associated to $\partial V\bigsqcup (-\partial W)$ and its refinements. See Appendix $\ref{app:gen}$ for a related discussion. To summarize,  we have 
	\[  \rho(\theta) = \rho( F|_{\partial V}\colon \partial V \to \partial W).\]
	Therefore, it suffices to show that 
	\[ \rho( F|_{\partial V}\colon \partial V \to \partial W) = 0.\]

	In the following, we use the reversed orientation of $W$. Since no confusion will arise, let us still write $W$ to denote $-W$ for the rest of the proof.    Fix a triangulation of $V$ and of $W$, together with a sequence of successive refinements $\sub^n(V)$ and $\sub^n(W)$ as before. Note that, for the geometrically controlled Poincar\'{e} pair associated to the triangulation  $\sub^n(V)$ (resp. $\sub^n(W)$), all maps appearing in Definition  $\ref{def:ppair}$ are geometrically controlled, with their propagations go to zero as $n\to \infty$.
	
	Now the theorem follows from a combination of the proof of Theorem $\ref{thm:bordK}$ above with the construction of the higher rho invariant in Section $\ref{sec:highrho}$. Indeed, let 
	\[  (B_{\partial V} + S_{\partial V})(B_{\partial V} - S_{\partial V})^{-1} \in (C^\ast_L(\widetilde X)^\Gamma)^+  \]
	and 
	\[  (B_{\partial W} + S_{\partial W})(B_{\partial W} - S_{\partial W})^{-1} \in (C^\ast_L(\widetilde X)^\Gamma)^+ \] 
	be the representatives of the local indices for the signature operators of $\partial V$ and $\partial W$ respectively. By the proof of Theorem $\ref{thm:bordK}$, we have an explicit continuous path of invertible elements $\{\mathcal V_s\}_{0\leq s\leq 1}$ in $(C^\ast_L(\widetilde X)^\Gamma)^+ $ connecting 
	\[  \mathcal V_0 = (B_{\partial V} + S_{\partial V})(B_{\partial V} - S_{\partial V})^{-1}  \]
	to the identity operator $\mathcal V_1 = I$. Similarly, there is an explicit continuous path of invertible elements $\{\mathcal W_s\}_{0\leq s\leq 1}$ in $(C^\ast_L(\widetilde X)^\Gamma)^+ $ connecting 
	\[  \mathcal W_0 = (B_{\partial W} + S_{\partial W})(B_{\partial W} - S_{\partial W})^{-1}  \]
	to the identity operator $\mathcal W_1 = I$.

	Consider the following elements  at time $t=0$:
	\[  \mathcal V_s(0) \textup{ and }  \mathcal W_s(0)  \in C^\ast(\widetilde X)^\Gamma. \]
	Let $(E, b, T)_{V, \partial V}$ and $(E', b', T')_{W, \partial W}$ be the geometrically controlled Poincar\'{e} pairs associated to the triangulations of  $V$ and $W$ respectively. Note that we are not taking subdivisions at the moment. Then the homotopy equivalence $F\colon V \to W$ induces a geometrically controlled chain homotopy equivalence between the geometrically controlled  Poincar\'{e} pairs 
	\[ (E, b, T)_{V, \partial V} \textup{ and } (E', b', T')_{W, \partial W}. \]
	In fact,  the homotopy equivalence $F\colon V\to W$ also induces  corresponding geometrically controlled chain homotopy equivalences between various  Poincar\'{e} complexes, such as $(\tilde E, \tilde b_{\lambda}, \tilde T)_{V, \partial V}$ and $(\tilde E, \tilde b_{\lambda}, \tilde T)_{W, \partial W}$,  that appear in the proof of Theorem $\ref{thm:bordK}$. Consequently, the construction in Section $\ref{sec:highrho}$ simultaneously produces continuous paths $\{ \mathcal U_s(t)\}_{-1\leq t\leq 0}$ of invertible elements connecting 
	\[  \mathcal V_s(0) \oplus \mathcal W_s(0) \]
	to the identity operator, for $s\in [0, 1]$. 
	
	For each $s\in [0, 1]$, we concatenate the path $\{\mathcal U_s(t)\}_{-1\leq t\leq 0}$ with the path $\{\mathcal V_s(t)\oplus \mathcal W_s(t)\}_{0\leq t < \infty}$. This produces an element, denoted by $\rho_s $,  in $(C^\ast_{L, 0}(\widetilde X)^\Gamma)^+$ for each $s\in [0, 1]$. Since $\{\rho_s\}_{0\leq s\leq 1}$ is a norm continuous path of invertible elements in $(C^\ast_{L, 0}(\widetilde X)^\Gamma)^+$, it is clear that 
	\[  [\rho_0] = [\rho_1] \in K_1(C^\ast_{L, 0}(\widetilde X)^\Gamma).   \]
	On the other hand, $\rho_0$ is precisely the definition of the higher rho invariant of $F_{\partial}\colon \partial V \to \partial W$, while $\rho_1 \equiv I$ is the constant map with value the identity operator. Therefore, $\rho(\theta) = [\rho_0] = 0.$ This finishes the proof.

\end{proof}

\section{Mapping surgery to analysis}\label{sec:toana}

In this section, for each closed oriented topological manifold $X$ of dimension $\geq 5$, we prove the commutativity of the following diagram of abelian groups:
\begin{equation}\label{diag:surgery}
\begin{split} 
\scalebox{1}{\xymatrixcolsep{1.5pc}\xymatrix{  \mathcal N_{n+1}(X)  \ar[r]^{i_\ast} \ar[d]_{\ind_L}  & L_{n+1}(\Gamma ) \ar[r]^-{j_\ast} \ar[d]_{\ind}  &  \mathcal S_n(X) \ar[r]  \ar[d]_{k_n\cdot \rho} &  \mathcal N_{n}(X )  \ar[d]_{k_n\cdot  \ind_L} \\
		K_{n+1}(C_{L}^\ast(\widetilde X)^\Gamma) \ar[r]^-{\mu_\ast} & K_{n+1}(C_r^\ast(\Gamma)) \ar[r] & K_n(C_{L, 0}^\ast(\widetilde X)^\Gamma)  \ar[r]  &  K_n(C_{L}^\ast(\widetilde X)^\Gamma) 	} }
\end{split}
\end{equation}
where the maps $\ind$ and $\ind_L$ will be defined below,  $\Gamma = \pi_1 X$ and  
\[ k_n  = \begin{cases*}
1 & \textup{ if $n$ is even, }\\
2 & \textup{ if $n$ is odd.}
\end{cases*} \]

In the case of smooth manifolds, a similar commutative diagram   was proved by Higson and Roe  \cite{MR2220523}. Since the structure set of a smooth manifold does not carry a group structure, the commutative diagram of Higson and Roe is a commutative diagram of sets in an appropriate sense \cite[Section 5]{MR2220523}. Piazza and Schick 
gave a different proof of Higson and Roe's commutative diagram for smooth manifolds \cite{MR3514938}. Zenobi proved a similar commutative diagram for topological manifolds, but only treating $\mathcal S_n(X)$ as a set \cite{MR3622237}.

The local index map 
\[ \ind_L\colon  \mathcal N_{n}(X) \to  K_{n}(C^\ast_L(\widetilde X)^\Gamma) \]
is defined by assigning each element in $\mathcal N_{n}(X)$ the $K$-homology class of its signature operator. See Appendix $\ref{app:khom}$ for more details. The well-definedness of the map $\ind_L$ follows from the bordism invariance of the $K$-homology class of signature operators. See Theorem $\ref{thm:bordK}$ above. 

\begin{remark}
	Note that the well-definedness of the map  
	\[ \ind_L\colon  \mathcal N_{n}(X) \to  K_{n}(C^\ast_L(\widetilde X)^\Gamma)\] implies  Novikov's theorem on the topological invariance of the rational Pontrjagin classes \cite{MR0193644}. Also see \cite{MR1388315}.  
\end{remark}

The index map 
\[ \ind\colon L_{n+1}(\Gamma) \to K_{n+1}(C_r^\ast(\Gamma)) \] is defined as follows. 
Suppose we have  an element 
\[ \theta = (M, \partial M, \varphi, N, \partial N, \psi, f) \in L_{n+1}(\Gamma) \] 
satisfying the conditions in Definition $\ref{def:Lgrp}$. Let  $M \cup_{f} (-N) $  be the space obtained by gluing $-N$ with $M$ along the boundary by the map $f$, where $-N$ is the manifold $N$ with the reversed orientation.  Although $M\cup_f (-N)$ is \emph{not} a manifold in general, it is still a space equipped with Poincar\'{e} duality. In particular, the higher signature index of $M\cup_{f} (-N)$ makes sense. 
\begin{definition}
	For each element \[ \theta = (M, \partial M, \varphi, N, \partial N, \psi, f) \in L_{n+1}(\pi_1 X), \]  we define $\ind(\theta)$ to be the higher signature index of $M\cup_{f} (-N)$.  
\end{definition}

\begin{proposition}
	The map $\ind\colon L_{n+1}(\Gamma) \to K_{n+1}(C_r^\ast(\Gamma))$ is a well-defined group homomorphism. 	
\end{proposition}
\begin{proof}
	The well-definedness follows immediately from the bordism invariance of the higher signature index. Moreover, the higher signature index is clearly additive on disjoint unions, hence the index map is a group homomorphism. This finishes the proof.
\end{proof}

We need some preparation before we prove the commutativity of diagram $\eqref{diag:surgery}$ above.
Recall that, by Proposition $\ref{prop:isorel}$,  there is a natural isomorphism  
\[   \mathcal S_n(X)\cong L_{n+1}(\pi_1 X, X). \]
In the following, first we shall give another definition,  denoted by $\widehat\rho$,  of higher rho invariant by using the description of $L_{n+1}(\pi_1 X, X)$. Then we will prove that $\widehat \rho = k_n \rho$, where $k_n = 1$ if $n$ is even and $2$ if $n$ is odd. 

Recall from Definition $\ref{def:relL}$ that, for each element \[ \theta = (M, \partial_\pm M, \varphi, N, \partial N_\pm, \psi, f)\in L_{n+1}(\pi_1 X, X), \] the map $f|_{\partial_+ N} \colon \partial_+ N \to \partial_+ M$ is a homotopy equivalence, and $f_{\partial_+ N}$ restricts to a PL infinitesimally controlled homotopy equivalence 
\[ f|_{\partial(\partial_\pm N)} \colon \partial (\partial_\pm N) \to \partial(\partial_\pm M). \]  
Let $Z = M\cup_{f_+} (-N)$ be the space obtained by gluing $M$ and $N$ along the boundary $\partial_+ N$ to $\partial_+ M$ through the homotopy equivalence $f|_{\partial_+ N}$. Though $Z$ is not a manifold, it is a space equipped with Poincar\'{e} duality. Note that the ``boundary" of $Z$ is the space $ \partial_-M\cup_{\partial f_{+}} \partial_-(-N)$, where the latter is obtained by gluing $\partial_- M$ and $\partial_- N$ along the boundary $\partial(\partial_- N)$ to $\partial(\partial_- M)$ through the infinitesimally controlled homotopy equivalence $f|_{\partial(\partial_- N)}$. Let us write 
$ \partial Z \coloneqq   \partial_-M\cup_{\partial f_{+}} \partial_-(-N).$ 
\begin{figure}[H]
	\centering
	\begin{tikzpicture}[scale=1, every node/.style={transform shape}]
	\clip (-4.5,1) rectangle (6.5,6); 	
	\coordinate (center) at (0, 0);
	\draw[blue, thick] ($(center) +(-1,5)$) .. controls ($(center) + (2,4.9)$) and ($(center) +(2, 4.1)$) .. ($(center) +(-0.25, 4)$);

	\begin{scope}
	\clip	($(center) +(-0.55, 1.85)$) rectangle ($(center) +(-1.1, 5)$);
	\draw[ultra  thick, dashed] ($(center) +(-1,5)$) .. controls ($(center) +(-0.875, 1)$) and ($(center) +(-0.3, 1)$).. ($(center) +(-0.25, 4)$);
	\end{scope}
	
	\begin{scope}
	\clip	($(center) +(-0.55, 1.85)$) rectangle ($(center) +(-0.15, 4)$);
	\draw[ultra  thick] ($(center) +(-1,5)$) .. controls ($(center) +(-0.875, 1)$) and ($(center) +(-0.3, 1)$).. ($(center) +(-0.25, 4)$);
	\end{scope}

	\draw[blue, thick] ($(center) +(-1,5)$) .. controls ($(center) + (-3.5,4.9)$) and ($(center) +(-3.5, 4.1)$) .. ($(center) +(-0.25, 4)$);

	\filldraw [red] ($(center) + (-1,5)$) circle (1.5pt); 			
	\filldraw [red] ($(center) +(-0.25,4)$) circle (1.5pt); 
	
	\draw[blue, thick] ($(center) +(-0.55, 1.85)$) .. controls   ($(center) +(0.5, 3)$) and ($(center) +(2.8, 0)$) ..  ($(center) +(1.35, 4.5)$); 
	
	\draw[blue, thick] ($(center) + (-0.55, 1.87)$) .. controls ($(center) +(-2.7, 2.5)$) ..  ($(center) +(-2.8, 4.54)$); 
	
	\draw[blue, thick] ($(center) +(0, 3)$) to [out=70, in=190] ($(center) +(0.8,3.6)$); 
	\draw[blue, thick] ($(center) +(0.1, 3.2)$) to [out=-20, in=-100] ($(center) +(0.7,3.6)$); 		
	\draw[blue, thick] ($(center) +(0.7, 2.5)$) to [out=80, in=160] ($(center) +(1.5, 3.1)$); 
	\draw[blue, thick] ($(center) +(0.78, 2.7)$) to [out=-20, in=-100] ($(center) +(1.3,3.1)$);
	
	\begin{scope}[shift={(-5, 3.8)}]
	\coordinate (center) at (0, 0);
	\begin{scope}[rotate = 270]
	\draw[blue,thick] ($(center) +(0.2, 2.8)$) to [out=70, in=190] ($(center) +(1,3.4)$); 
	\draw[blue,thick] ($(center) +(0.3, 3)$) to [out=-20, in=-100] ($(center) +(0.9,3.4)$); 	
	\end{scope}
	\end{scope}

	
	\node at (0, 5.2) {$\partial_- M$};
	\node at (-1.8, 5.2) {$\partial_- N$};
	\node at (2.3, 2.5) {$M$};
	\node at (-3.2, 2.5) {$N$};
	\node at (-0.5, 1.2) {$Z = M\cup_{f_+} (-N)$}; 
	\node at (6, 3) {$X$}; 	
	\draw[->, thick] (2.7, 3) -- (5.7, 3);
	\node at (4.3, 3.3) {$\Phi = \varphi\cup_{f_+} \psi$};

	\end{tikzpicture}
	\caption{Picture for $Z= M\cup_{f_+}(-N)$.}
	
\end{figure}
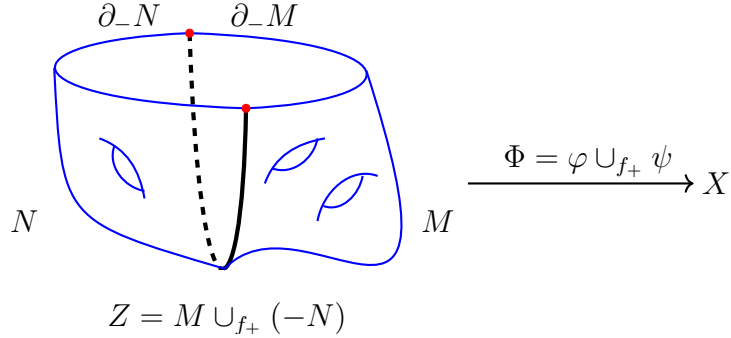
Recall that  the space obtained from  $Z$ by attaching a cylinder is denoted by 
\[ CZ = Z\cup_{\partial Z} (\partial Z\times [1, \infty)). \]
Let us fix a triangulation of $CZ$ as follows. On $Z$, it is the original triangulation of $Z$. The triangulation on $\partial Z\times [1, \infty)$ is the pullback triangulation of $\textup{Tri}_{X\times[1, \infty)}$ under the map $\Phi_\partial\times \id\colon \partial Z \times [1, \infty) \to CX,$ where $\Phi_\partial$ is the restriction of $\Phi = \varphi\cup_{f_+} \psi$ on $\partial Z$. That is,  for every simplex $\Delta^k \subset CX$,  the inverse image $(\Phi_{\partial}\times \id)^{-1}(\Delta^k)$ is a product $K\times \Delta^k$, where $K$ is some triangulated submanifold of $\partial Z$.

\begin{remark}
	To be precise, we should be using a sequence of spaces $\{Z_i\}_{i\geq 1}$, where each $Z_i = M\cup_{(f_i)_+} (-N)$ be the space obtained by gluing $M$ and $N$ along the boundary $\partial_+ N$ to $\partial_+ M$ through the homotopy equivalence $(f_i)|_{\partial_+ N}$. Here $f_i$ is the map $f$ with the additional condition that the homotopy equivalence 
	\[  f|_{\partial(\partial_\pm N)} \colon \partial (\partial_\pm N) \to \partial(\partial_\pm M) \] has control $\leq \frac{1}{i}$. Since no confusion is likely to arise, we shall abuse our notation and continue as if we are working with a single space.  
\end{remark} 

Now the geometrically controlled Poincar\'e complex 
associated to $CZ$ defines a higher signature index in $K_{n+1}( C^\ast_c(\widetilde X \times [1, \infty) )^\Gamma) $. 
\begin{definition}
	With the same notation as above, for each element
	\[\theta = (M, \partial_\pm M, \varphi, N, \partial N_\pm, \psi, f)\in L_{n+1}(\pi_1 X, X), \] 
	We define $\widehat\rho (\theta)$ to be the higher signature index of $CZ$, which is an element of $K_{n+1}( C^\ast_c(\widetilde X \times [1, \infty))^\Gamma )$.  
\end{definition}

The following lemma is an immediate consequence of bordism invariance of higher signature index. 
\begin{lemma} The map $\widehat \rho$ is a group homomorphism 
	\[ \widehat \rho \colon   L_{n+1}(\pi_1 X, X) \to  K_{n+1}( C^\ast_c(\widetilde X \times [1, \infty))^\Gamma ).  \]
\end{lemma}

Recall that we have the following natural isomorphism
\[  c_\ast\colon   \mathcal S_{n}(X) \to L_{n+1}(\pi_1 X, X) \] 	
by taking the product with the unit interval (see Section $\ref{sec:iden}$): 
\[\theta = \{M, \partial M, \varphi, N, \partial N, \psi, f\} \mapsto  \theta\times I.  \] 
Moreover, by Proposition $\ref{prop:coninf}$ and  Corollary $\ref{cor:hybrid}$, we have 
\begin{align*}
K_{n+1}(C^\ast_{c}(\widetilde X\times [1, \infty))^\Gamma  ) & \cong  K_{n}(C^\ast_{L,  0, c}(\widetilde X\times [1, \infty))^\Gamma ) \\
& \cong K_n(C^\ast_{L,0}(\widetilde X)^\Gamma).
\end{align*} 
It follows that the map $\widehat \rho$ can be viewed as a group homomorphism:  
\[ \widehat \rho \colon   \mathcal S_{n}(X) \to  K_n(C^\ast_{L,0}(\widetilde X)^\Gamma).  \]

\begin{remark}\label{rmk:rhocompare}
	In the case of smooth structure sets, it is easy to see that the definition $\widehat \rho$ above agrees with the \emph{structure invariant} of Higson and Roe \cite[Section 3]{MR2220524}.
\end{remark}

In Theorem $\ref{thm:tworho}$ below, we will prove  that   $\widehat \rho$ is equal to $k_n\cdot \rho$, where  $\rho$ is the higher rho invariant from  Definition $\ref{def:highrho}$ and $k_n = 1$ if $n$ is even and $2$ if $n$ is odd. Before doing this, let us first prove a product formula for the higher rho invariant $\rho$, which will be useful for the proof of Theorem $\ref{thm:tworho}$.

Given an element 
\[ \theta = (M, \partial M, \varphi, N, \partial N, \psi, f) \in \mathcal S_n(X), \] let $\theta\times \mathbb R \in \mathcal S_{n+1}(X\times \mathbb R)$ be the product of $\theta$ and $\mathbb R$.  Here various undefined terms take the obvious meanings (see Section $\ref{sec:iden}$ for the definition of $\theta\times I$ for example). Note that the construction in Section $\ref{sec:highrho}$ also applies to $\theta\times \mathbb R$ and defines its higher rho invariant $\rho(\theta\times \mathbb R) \in K_{n+1}(C^\ast_{ L, 0}(\widetilde X\times \mathbb R)^\Gamma)$. Also there is a natural homomorphism    
\[ \alpha \colon C^\ast_{L, 0}(\widetilde X)^\Gamma \otimes C^\ast_L(\mathbb R)  \to C_{L, 0}^\ast (\widetilde X\times \mathbb R)^\Gamma, \] which induces an isomorphism on $K$-theory. 

\begin{theorem}\label{thm:prod}  With the same notation as above, we have 
	\[  k_n\cdot \alpha_\ast\big( \rho(\theta) \otimes \ind_L(\mathbb R) \big) =  \rho(\theta\times \mathbb R)\]
	in $K_{n+1}(C^\ast_{ L, 0}(\widetilde X\times \mathbb R)^\Gamma)$, 
	where $\ind_L(\mathbb R)$ is the $K$-homology class of the signature operator on $\mathbb R$, and $k_n = 1$ if $n$ is even and $2$ if $n$ is odd. 
\end{theorem}
\begin{proof}
	The proof is elementary and will be given in Appendix $\ref{app:prod}$. 
\end{proof}

To prepare for the proof of Theorem $\ref{thm:tworho}$,  let us introduce some notation.  Consider the $ C^\ast$-algebra $ \mathscr A = C^\ast_{ L, 0}(\widetilde X\times \mathbb R)^\Gamma$. Using the notation from Definition $\ref{def:localg}$, we define 
\[ \mathscr A_{-} = \bigcup_{n\in \mathbb N} C^\ast_{L, 0}(\widetilde X \times (-\infty, n]; \widetilde X\times \mathbb R)^\Gamma,  \]
\[\mathscr A_{+} = \bigcup_{n\in \mathbb N} C^\ast_{L, 0}(\widetilde X \times [-n, \infty); \widetilde X\times \mathbb R)^\Gamma \] and 
\[  \mathscr A_{\cap} =  \bigcup_{n\in \mathbb N} C^\ast_{L, 0}(\widetilde X \times [-n, n]; \widetilde X\times \mathbb R)^\Gamma. \] 
It is clear that $\mathscr A_\pm$ and $\mathscr A_\cap$  are closed two-sided ideals of $\mathscr A$. Moreover, we have $\mathscr A_+ + \mathscr A_- = \mathscr A$ and $\mathscr A_+ \cap \mathscr A_- = \mathscr A_\cap$, which gives rise to the following   Mayer-Vietoris sequence in $K$-theory: 
\[  \xymatrix{ K_0(\mathscr A_\cap) \ar[r]  &  K_0(\mathscr A_+) \oplus K_0(\mathscr A_-) \ar[r]  & K_0(\mathscr A) \ar[d]^{\partial_{MV}} \\
	K_1(\mathscr A)   \ar[u]^{\partial_{MV}} & K_1(\mathscr A_+) \oplus K_1(\mathscr A_-) \ar[l]	 & K_1(\mathscr A_\cap). \ar[l]		  } \] 
Similarly, consider the following ideals of  $C^\ast$-algebra $\mathscr B = C^\ast_L(\mathbb R)$:  
\[ \mathscr B_- = \bigcup_{n\in \mathbb N} C^\ast_{L}((-\infty, n]; \mathbb R), \]
\[    \mathscr B_+ = \bigcup_{n\in \mathbb N} C^\ast_{L}([-n, \infty); \mathbb R)\]
and \[ \mathscr B_\cap =  \bigcup_{n\in \mathbb N} C^\ast_{L}([-n, n]; \mathbb R). \]
These $C^\ast$-algebras  give rise to the following   Mayer-Vietoris sequence in $K$-theory: 
\[  \xymatrix{ K_0(\mathscr B_\cap) \ar[r]  &  K_0(\mathscr B_+) \oplus K_0(\mathscr B_-) \ar[r]  & K_0(\mathscr B) \ar[d]^{\partial_{MV}} \\
	K_1(\mathscr B)   \ar[u]^{\partial_{MV}} & K_1(\mathscr B_+) \oplus K_1(\mathscr B_-) \ar[l]	 & K_1(\mathscr B_\cap). \ar[l]		  } \] 
Note that there is a natural homomorphism 
\[ \alpha\colon C^\ast_{L, 0}(\widetilde X)^\Gamma \otimes \mathscr B \to \mathscr A,\]  which restricts to homomorphisms 
\[  \alpha\colon C^\ast_{L, 0}(\widetilde X)^\Gamma \otimes \mathscr B_\pm  \to \mathscr A_\pm \textup{ and } \alpha\colon C^\ast_{L, 0}(\widetilde X)^\Gamma \otimes \mathscr B_\cap \to \mathscr A_\cap \]
such that the following diagram commutes:
\begin{equation}\label{diag:mv}
\begin{split}
\scalebox{1}{\xymatrixcolsep{1pc}\xymatrix{ K_{n}(C^\ast_{L, 0}(\widetilde X)^\Gamma) \otimes K_1(\mathscr B) \ar[r] \ar[d]_{1\otimes \partial_{MV}} &  K_{n+1}(C^\ast_{L, 0}(\widetilde X)^\Gamma \otimes \mathscr B) \ar[r]^-{\alpha_{\ast}}_-{\cong} &  K_{n+1}(C^\ast_{ L, 0}(\widetilde X\times \mathbb R)^\Gamma) \ar[d]^{\partial_{MV}}\\ 
		K_{n}(C^\ast_{L, 0}(\widetilde X)^\Gamma) \otimes K_0(\mathscr B_\cap) \ar[r] &  K_{n}(C^\ast_{L, 0}(\widetilde X)^\Gamma \otimes \mathscr B_\cap) \ar[r]^-{\alpha_{\ast}}_-{\cong} & K_n(\mathscr A_\cap) = K_{n}(C^\ast_{ L, 0}(\widetilde X)^\Gamma).
}  }
\end{split}
\end{equation}

\begin{theorem}\label{thm:tworho}
	The following diagram commutes: 
	\[ \scalebox{1}{\xymatrixcolsep{2pc} \xymatrix{   L_{n+1}(\pi_1 X, X)  \ar[r]^-{\widehat{\rho}}  & K_{n+1}( C^\ast_c(\widetilde X \times [1, \infty))^\Gamma)  \ar[d]^{\partial_\ast} \\
			\mathcal S_n(X) \ar[u]^{c_\ast} \ar[r]^-{k_n\cdot \rho} &  K_{n}(C^\ast_{L,0, c}(\widetilde X\times [1, \infty))^\Gamma) = K_{n}(C^\ast_{ L, 0}(\widetilde X)^\Gamma), } }
	\]
	where $\partial_\ast$	is the connecting map in the $K$-theory long exact sequence associated to    
	\[\scalebox{1}{$0 \to C^\ast_{L,0, c}(\widetilde X\times [1, \infty))^\Gamma\to  C^\ast_{L, c}(\widetilde X\times [1, \infty))^\Gamma 
		\to C^\ast_{ c}(\widetilde X\times [1, \infty))^\Gamma \to 0,$} \]
	and $k_n = 1$ if $n$ is even and $2$ if $n$ is odd. 
\end{theorem}

\begin{proof}
	Recall that a standard way to construct the connecting map $\partial_\ast$ is by lifting a projection (resp. invertible) in  $C^\ast_{ c}(\widetilde X\times [1, \infty))^\Gamma$ to an element in  $C^\ast_{L, c}(\widetilde X\times [1, \infty))^\Gamma$. For an element 
	\[ \theta = \{M, \partial_\pm M, \varphi, N, \partial_\pm N, \psi, f\} \in L_{n+1}(\pi_1 X, X), \]  there exists a lifting  $a_\theta  \in  C^\ast_{L, c}(\widetilde X\times [1, \infty))^\Gamma $ of the element  $\widehat\rho(\theta) \in C^\ast_{ c}(\widetilde X\times [1, \infty))^\Gamma $ as follows. Let 
	\[ a_\theta(n) =  \chi_{n} \widehat \rho(\theta)   \chi_n, \]
	where  $\chi_n$ is the characteristic function on $\widetilde X\times [n, \infty)$. We define \[ a_\theta(t) = (n+1 -t)a_\theta(n) + (t-n)a_\theta(n+1) \]
	for all $ n \leq t \leq n+1$. It is clear that $a_\theta$ lies in  $ C^\ast_{L, c}(\widetilde X\times [1, \infty))^\Gamma $  and is a lift of $\widehat{\rho}(\theta)$.  
	
	On the other hand, by the discussion  before Proposition $\ref{prop:highrho}$, one sees that the same $a_\theta$ above is also a lift of $\rho(\theta\times \mathbb R) $, for the construction of the connecting map 
	\[ \partial_{MV}\colon K_{n+1}(C_{L,0}^\ast(\widetilde X\times \mathbb R)^\Gamma) \to K_{n}(C^\ast_{L,0}(\widetilde X)^\Gamma)  \] 
	in diagram $\eqref{diag:mv}$. In particular, we see that 
	\[ \partial_\ast\circ \widehat\rho \circ c_\ast (\theta) = \partial_{MV} \left[\rho(\theta\times \mathbb R)\right]. \]  
	Now by Theorem $\ref{thm:prod}$ and the commutative diagram $\eqref{diag:mv}$, it follows that 
	\[ \partial_\ast\circ \widehat\rho \circ c_\ast (\theta) = k_n\cdot \rho(\theta) \otimes \partial_{MV}[\ind_L(\mathbb R)] = k_n\cdot \rho(\theta). \]
	This finishes the proof.

\end{proof}
Combining the above discussion, we have the following main result of this section. 
\begin{theorem}\label{thm:stoa}
	We have the  following commutative diagram:
	\[ 
	\scalebox{1}{\xymatrixcolsep{1.5pc}\xymatrix{  \mathcal N_{n+1}(X)  \ar[r]^{i_\ast} \ar[d]_{\ind_L}  & L_{n+1}(\Gamma ) \ar[r]^-{j_\ast} \ar[d]_{\ind}  &  \mathcal S_n(X) \ar[r]  \ar[d]_{k_n\cdot \rho} &  \mathcal N_{n}(X )  \ar[d]_{k_n\cdot  \ind_L} \\
			K_{n+1}(C_{L}^\ast(\widetilde X)^\Gamma) \ar[r]^-{\mu_\ast} & K_{n+1}(C_r^\ast(\Gamma)) \ar[r] & K_n(C_{L, 0}^\ast(\widetilde X)^\Gamma)  \ar[r]  &  K_n(C_{L}^\ast(\widetilde X)^\Gamma). 	} }\]
\end{theorem}
\begin{proof}
	The commutativity of the right square and the left square follows immediately from definition. 
	
	The commutativity of the middle square is an immediate consequence of Theorem $\ref{thm:tworho}$ above. Indeed, by Theorem $\ref{thm:tworho}$, we have  the following commutative diagram:
	\[ \scalebox{1}{\xymatrix{   &   	\mathcal S_n(X) \ar[ddr]^{k_n\cdot \rho} \ar[d]_{c_\ast}   \\
			L_n(\pi_1 X)  \ar[r] \ar[ur]^{j_\ast} \ar[d]^{\ind} &  L_{n+1}(\pi_1 X, X)  \ar[d]_{\widehat \rho}   \\
			K_{n}(C^\ast(\widetilde X)^\Gamma) \ar[r] &  K_{n+1} (C_c^\ast(\widetilde X\times [1, \infty))^\Gamma) \ar[r]^-{\partial_\ast}& K_n(C^\ast_{L,0}(\widetilde X)^\Gamma), }} \]
	Here the commutativity of the lower left square follows from the definition of the index map and the map $\widehat \rho$. This finishes the proof. 
\end{proof}

\begin{remark}
	The exact same proof also applies to the maximal version, and we have the following commutative diagram:
	\[ 
	\begin{split}
	\scalebox{1}{	\xymatrixcolsep{1.5pc}\xymatrix{  \mathcal N_{n+1}(X)  \ar[r]^{i_\ast} \ar[d]_{\ind_L}  & L_{n+1}(\Gamma ) \ar[r]^-{j_\ast} \ar[d]_{\ind}  &  \mathcal S_n(X)   \ar[d]_{k_n\cdot \rho} \ar[r] & \mathcal N_{n}(X )  \ar[d]_{k_n\cdot  \ind_L} \\
			K_{n+1}(C_{L}^\ast(\widetilde X)^\Gamma_{\max}) \ar[r]^-{\mu_\ast} & K_{n+1}(C_{\max}^\ast(\Gamma)) \ar[r] & K_n(C_{L, 0}^\ast(\widetilde X)^\Gamma_{\max}) \ar[r] & K_n(C_{L}^\ast(\widetilde X)^\Gamma_{\max}).  
	} 	}  
	\end{split}
	\]
\end{remark}

Moreover, the same method can also be applied to the homology manifold surgery exact sequence of a closed oriented connected ANR homology manifold of dimension $\geq 6$. 

\begin{proposition}\label{prop:homocomm}
	Let $X$ be a closed oriented connected ANR homology manifold of dimension $\geq 6$. Suppose $\pi_1 X = \Gamma$. Then we have the following commutative diagram: 
	\[ 
	\scalebox{1}{
		\xymatrixcolsep{1.5pc}	\xymatrix{  \mathfrak N_{n+1}(X)  \ar[r]^{i_\ast} \ar[d]_{\ind_L}  & \mathfrak L_{n+1}(\Gamma) \ar[r]^-{j_\ast} \ar[d]_{\ind}  &  \mathfrak S_n(X) \ar[r]  \ar[d]_{k_n\cdot \rho} &  \mathfrak N_{n}(X )  \ar[d]_{k_n\cdot  \ind_L} \\
			K_{n+1}(C_{L}^\ast(\widetilde X)^\Gamma) \ar[r]^-{\mu_\ast} & K_{n+1}(C_r^\ast(\Gamma)) \ar[r] & K_n(C_{L, 0}^\ast(\widetilde X)^\Gamma)  \ar[r]  &  K_n(C_{L}^\ast(\widetilde X)^\Gamma), 	}  }
	\] 
	where the upper exact sequence is the $4$-periodic exact sequence from line $\eqref{perilong}$ in Theorem $\ref{thm:les}$ \textup{(}cf. Remark $\ref{rk:hommfld}$\textup{)}.  
\end{proposition}

\section{Novikov rho invariant and strong Novikov conjecture} \label{sec:novrho}

In this section, we define a homological version of the higher rho invariant from Section $\ref{sec:highrho}$.  This homological higher rho invariant will be called \emph{Novikov rho invariant} for reasons that will be explained in later part of this section. One importance of the Novikov rho invariant is that it can be used to detect nontrivial elements in the structure group of a  closed oriented topological manifold, even when the fundamental group  of the manifold is torsion free. In particular, we apply the Novikov  rho invariant to show that the structure group is not finitely generated for a class of manifolds.  Throughout the section, we assume $n \geq 5$. 

Let $X$ be a proper metric space with $\pi_1 X = \Gamma$. Suppose $\widetilde X$ is the universal cover of $X$. We have the following commutative diagram: 
\begin{equation}
\begin{split}
\scalebox{1}{\xymatrixcolsep{1pc}\xymatrix{  K^\Gamma_{n+2}(\underline E\Gamma, \widetilde X)  \ar[r] \ar[d]_{\Lambda}  & K_{n+1}^\Gamma(\widetilde X) \ar[r] \ar[d]^{\cong}  &  K_{n+1}^\Gamma(\underline E\Gamma) \ar[r] \ar[d]^{\mu_\ast} & K^\Gamma_{n+1}(\underline E\Gamma, \widetilde X)  \ar[d]_{\Lambda} \\
		K_{n+1}(C_{L, 0}^\ast(\widetilde X)^\Gamma) \ar[r]  & K_{n+1}(C_{L}^\ast(\widetilde X)^\Gamma) \ar[r] & K_{n+1}(C^\ast_r(\Gamma))  \ar[r]^-\partial   &  K_{n}(C_{L, 0}^\ast(\widetilde X)^\Gamma)  
} }
\end{split}
\end{equation}
where $K^\Gamma_{n}(\underline E\Gamma, \widetilde X) $ is the  $\Gamma$-equivariant relative $K$-homology group for the pair of spaces $(\underline E\Gamma, \widetilde X) $. For example, $K^\Gamma_{n}(\underline E\Gamma, \widetilde X)$ is the $\Gamma$-equivariant $K$-homology group of the mapping cone of the map $\widetilde X \to \underline E\Gamma$. Also, $K^\Gamma_{n}(\underline E\Gamma, \widetilde X)$ is naturally isomorphic to the $K$-theory group of the $C^\ast$-algebra mapping cone associated to the natural map $C_{L}^\ast(\widetilde X)^\Gamma \to C_{L}^\ast(\underline E\Gamma)^\Gamma$. Furthermore,  the $K$-theory  of $C_{L, 0}^\ast(\widetilde X)^\Gamma  $ is naturally isomorphic to  the $K$-theory  of the $C^\ast$-algebra mapping cone associated to the evaluation map $C_{L}^\ast(\widetilde X)^\Gamma \to C_r^\ast(\Gamma)$,
cf. \cite{MR4045309}. In view of  this mapping cone picture, the commutativity of the above diagram is clear.

We would like to see in what circumstances there exists a natural homomorphism 	$\beta \colon K_{n}(C_{L, 0}^\ast(\widetilde X)^\Gamma) \to K^\Gamma_{n+1}(\underline E\Gamma, \widetilde X)$ such that the following diagram remains commutative. 
\begin{equation}\label{diag:tophr}
\begin{split}
\scalebox{1}{\xymatrixcolsep{1pc}\xymatrix{  K^\Gamma_{n+2}(\underline E\Gamma, \widetilde X)  \ar[r] \ar[d]_{\Lambda}  & K_{n+1}^\Gamma(\widetilde X) \ar[r] \ar[d]^{\cong}  &  K_{n+1}^\Gamma(\underline E\Gamma) \ar[r] \ar[d]^{\mu_\ast} & K^\Gamma_{n+1}(\underline E\Gamma, \widetilde X)  \ar[d]_{\Lambda} \\
		K_{n+1}(C_{L, 0}^\ast(\widetilde X)^\Gamma) \ar[r] \ar@{-->}@/_/[u]_{\beta} & K_{n+1}(C_{L}^\ast(\widetilde X)^\Gamma) \ar[r] & K_{n+1}(C^\ast_r(\Gamma))  \ar[r]^-\partial   &  K_{n}(C_{L, 0}^\ast(\widetilde X)^\Gamma) \ar@{-->}@/_/[u]_{\beta} \\
		\mathcal S_{n+1}(X)  \ar[r] \ar[u]_{ \rho }  & \mathcal N_{n+1}(X ) \ar[r] \ar[u]_{\ind_L}  &  L_{n+1}(\Gamma ) \ar[r] \ar[u]_{\ind } &  \mathcal S_{n}(X )  \ar[u]_{k_{n}\cdot \rho}	}  }
\end{split}
\end{equation}

Now suppose that the strong Novikov conjecture holds for $\Gamma$, that is, the Baum-Connes assembly map $ \mu_\ast\colon K_{n+1}^\Gamma(\underline{E}\Gamma) \to K_{n+1}(C^\ast_r(\Gamma))$ is injective. In fact, let us assume a slightly stronger condition: 
\[ \mu_\ast\colon K_{n+1}^\Gamma(\underline{E}\Gamma) \to K_{n+1}(C^\ast_r(\Gamma)) \]
is a split injection. So far, in all known cases where the strong Novikov conjecture holds, the split injectivity of the Baum-Connes assembly map is known to be true as well, cf. \cite{CM90, MR2217050, MR1779613, MR1821144,  GK88, MR2980001, MR1905840, GY98, GY00}.

In this case, let us denote the splitting map by \[ \alpha\colon K_{n+1}(C^\ast_r(\Gamma))\to K_{n+1}^\Gamma(\underline{E}\Gamma),\] which induces a direct sum decomposition:
\[  K_{n+1}(C^\ast_r(\Gamma))  \cong  K_{n+1}^\Gamma(\underline{E}\Gamma) \oplus \mathscr E. \]  Then a routine diagram chase shows that
\begin{enumerate}[(1)]
	\item the homomorphism $\Lambda\colon K^\Gamma_{n+1}(\underline E\Gamma, \widetilde X) \to   K_{n}(C_{L, 0}^\ast(\widetilde X)^\Gamma) $
	is also an injection;
	\item and $\partial(\mathscr E)\cap \partial(K_{n+1}^\Gamma(\underline{E}\Gamma) ) = 0 $.   
\end{enumerate} 
It follows that  we have the following commutative diagram: 
\begin{equation}\label{diag:tophr2}
\begin{split}
\scalebox{1}{\xymatrixcolsep{0.8pc}\xymatrix{    K_{n+1}^\Gamma(\widetilde X) \ar[r] \ar[d]^{\cong}  &  K_{n+1}^\Gamma(\underline E\Gamma) \ar[r] \ar[d]_{\mu_\ast} & K^\Gamma_{n+1}(\underline E\Gamma, \widetilde X)  \ar[d]^\Lambda \ar[r] & K_{n}^\Gamma (\widetilde X)  \ar[d]^{\cong} \\
		K_{n+1}(C_{L}^\ast(\widetilde X)^\Gamma) \ar[r] \ar[d]^{=} & K_{n+1}^\Gamma(\underline{E}\Gamma) \oplus \mathscr E  \ar[r]^-\partial  \ar[d]^\alpha &  K_{n}(C_{L, 0}^\ast(\widetilde X)^\Gamma) \ar[d]^q \ar[r]	& K_{n}(C_{L}^\ast(\widetilde X)^\Gamma) \ar[d]^{=}  \\
		K_{n+1}(C_{L}^\ast(\widetilde X)^\Gamma) \ar[r] & K_{n+1}^\Gamma(\underline{E}\Gamma) \ar[r]^-\partial   &  K_{n}(C_{L, 0}^\ast(\widetilde X)^\Gamma)/\partial(\mathscr E) \ar[r] & 	K_{n}(C_{L}^\ast(\widetilde X)^\Gamma),  	} 
}
\end{split}
\end{equation}
where $q$ is the quotient map 
\[ q\colon  K_{n}(C_{L, 0}^\ast(\widetilde X)^\Gamma)\to  K_{n}(C_{L, 0}^\ast(\widetilde X)^\Gamma)/\partial(\mathscr E).  \]
Note that the last row in diagram $\eqref{diag:tophr2}$ is also a long exact sequence. 
By the five lemma, it follows that the composition \[ q\circ \Lambda \colon K^\Gamma_{n+1}(\underline E\Gamma, \widetilde X) \xrightarrow{\ \cong \ } K_{n}(C_{L, 0}^\ast(\widetilde X)^\Gamma)/\partial(\mathscr E)\] is an isomorphism.

Now we define 
\[ \beta \coloneqq  (q\circ \Lambda)^{-1} \circ q \colon K_{n}(C_{L, 0}^\ast(\widetilde X)^\Gamma) \to   K^\Gamma_{n+1}(\underline E\Gamma, \widetilde X).  \]
By definition, $\beta$ makes diagram $\eqref{diag:tophr}$ commute. 

\begin{definition}
	We define the Novikov rho invariant map  $\rho^{\textup{Nov}}$ to be the composition 
	\[  \rho^{\textup{Nov}} = \beta \circ (k_{n}\cdot \rho) \colon \mathcal S_{n}(X) \to K^\Gamma_{n+1}(\underline E\Gamma, \widetilde X),  \]
	where we have
	\[ k_{n}  = \begin{cases*}
	1 & \textup{ if $n$ is even, }\\
	2 & \textup{ if $n$ is odd.}
	\end{cases*} \]
\end{definition} 

\begin{remark}
	Note that our definition of the invariant $\rho^{\textup{Nov}}$ only works when the strong Novikov conjecture holds. This is the reason that we name this homological higher rho invariant after Novikov. To be more precise, we have assumed a slightly stronger condition that the Baum-Connes assembly map is split injective. As noted before,  in all known cases where the strong Novikov conjecture holds, the split injectivity of the Baum-Connes assembly map is known to true as well, cf. \cite{CM90, MR2217050, MR1779613, MR1821144,  GK88, MR2980001, MR1905840, GY98, GY00}. 
\end{remark}

\begin{remark}
	There is also a maximal version of the Novikov rho invariant defined above. In this case,  we assume that the Baum-Connes assembly map is split injective for the maximal group $C^\ast$-algebra $C_{\max}^\ast(\pi_1 X)$. The Novikov rho invariant is defined similarly. This split injectivity assumption for maximal group $C^\ast$-algebras is weaker than the split injectivity assumption for reduced group $C^\ast$-algebras.
\end{remark}

We invert $2$ for the rest of this section. With some minor modifications, all discussions in this paper work equally well for the real case.  Roughly speaking, whenever the imaginary number $i = \sqrt{-1}$ appears in a formula, we replace it by the matrix $\begin{psmallmatrix}
0 & 1 \\ -1 & 0
\end{psmallmatrix}$. More precisely, for a geometrically controlled Poincar\'{e} complex $(E, b, T)$, we consider the direct sum 
\[  (E, b, T)\oplus (E, b, T). \] For the operator $S$ in Definition $\ref{def:sop}$,   we define   
\[ S = \begin{psmallmatrix}
0 & 1 \\ -1 & 0
\end{psmallmatrix}^{p(p-1)+ \ell} T. \]
The same remark applies to various other formulas, such as the formula in line $\eqref{eq:htpy}$, where complex numbers are used. 
In the case of dimension $n= 0$ or $1$ mod $4$,  this  gives rise to a signature operator that is twice of the actual signature operator with real coefficients. Now taking product with $\mathbb R^2$ takes care of the case where the dimension $n=2$ or $3$ mod $4$. The analogue of the diagram $\eqref{diag:surgery}$ for the real case involves extra powers of $2$ in front of various maps.  We will leave out the details. In any case, since we have already inverted $2$, we do not lose any information by introducing these extra powers of $2$.

Recall that, after inverting $2$, the maps $\tau_1$ and $\tau_2$ in the following commutative diagram are split injective: 
\[ \xymatrix{ KO_i^\Gamma(\underline{E}\Gamma)[{\scriptstyle\frac{1}{2}}]  \ar[r]^{\mu_{\mathbb R}} \ar[d]_{\tau_1} & K_i(C_r^\ast(\Gamma, \mathbb R))[{\scriptstyle\frac{1}{2}}]  \ar[d]^{\tau_2} \\
	K_i^\Gamma(\underline{E}\Gamma)[{\scriptstyle\frac{1}{2}}] \ar[r]^\mu & K_i(C_r^\ast(\Gamma))[{\scriptstyle\frac{1}{2}}]  } \] 
where the notation $[{\scriptstyle\frac{1}{2}}]$ means $\otimes \mathbb Z[{\scriptstyle\frac{1}{2}}]$, $\mu_{\mathbb R}$ and $\mu$ are Baum-Connes assembly maps \cite{PBAC88}\cite{BCH94}, and $\tau_1$ and $\tau_2$ are induced by changing the scalars from $\mathbb R$ to $\mathbb C$. 
In particular, if we assume the split injectivity of the Baum-Connes assembly map, then we also have the Novikov rho invariant map in the real case:
\[  \rho^{\textup{Nov}} \colon \mathcal S_{n}(X) \to KO^\Gamma_{n+1}(\underline E\Gamma, \widetilde X)[{\scriptstyle\frac{1}{2}}]. \]
Moreover, we shall see that $\rho^{\textup{Nov}}$ is surjective in this case. 

Recall that, after inverting $2$, the split injectivity of Baum-Connes assembly map  implies that the split injectivity of Farrell-Jones assembly map (cf. \cite[Proposition 95, page 758]{MR2181833}), that is, 
\[ A\colon H_i^\Gamma(\underline E\Gamma; \mathbb L_\bullet)[{\scriptstyle\frac{1}{2}}] \to L_i(\Gamma)[{\scriptstyle\frac{1}{2}}] = L_i(\mathbb Z\Gamma)[{\scriptstyle \frac{1}{2}}]. \] 
is a split injection. Denote the splitting map by 
\[ \gamma\colon L_i(\Gamma)[{\scriptstyle \frac{1}{2}}] \to H_i^\Gamma(\underline E\Gamma; \mathbb L_\bullet)[{\scriptstyle \frac{1}{2}}].   \] Also, note that the natural map, which takes a $KO$-class to its associated (local) Poincar\'e complex, 
induces an isomorphism 
\[ KO_i(Y) [{\scriptstyle \frac{1}{2}}] \xrightarrow{\cong }
H_i(Y; \mathbb L_\bullet) [{\scriptstyle \frac{1}{2}}] \] for any $Y$. Indeed, this natural map induces a map between two homology theories. Hence it suffices to verify that it is an isomorphism for $Y = \{pt\}$, which follows from a straightforward calculation (cf. \cite{MR1388305}).  We still write $\gamma$ for the splitting map  
\[ L_i(\Gamma)[{\scriptstyle \frac{1}{2}}] \to H_i^\Gamma(\underline E\Gamma; \mathbb L_\bullet)[{\scriptstyle \frac{1}{2}}] \cong KO_i(\underline{E}\Gamma) [{\scriptstyle \frac{1}{2}}].\]  It follows that we have the following commutative diagram:
\begin{equation}\label{diag:toprho3}
\begin{split}
\scalebox{0.95}{\xymatrixcolsep{0.7pc}
	\xymatrix{  KO^\Gamma_{n+2}(\underline E\Gamma, \widetilde X)[{\scriptstyle \frac{1}{2}}]  \ar[r]   & KO_{n+1}^\Gamma(\widetilde X)[{\scriptstyle \frac{1}{2}}] \ar[r]   &  KO_{n+1}^\Gamma(\underline E\Gamma)  [{\scriptstyle \frac{1}{2}}] \ar[r]  & KO^\Gamma_{n+1}(\underline E\Gamma, \widetilde X)  [{\scriptstyle \frac{1}{2}}]   \\
		\mathcal S_{n+1}(X) [{\scriptstyle \frac{1}{2}}] \ar[r] \ar[u]_{ \rho^{\textup{Nov}}}  & H_{n+1}^\Gamma(\widetilde X; \mathbb L_\bullet) [{\scriptstyle \frac{1}{2}}] \ar[r] \ar[u]_{\cong}  &  L_{n+1}(\Gamma )  [{\scriptstyle \frac{1}{2}}] \ar[r]^\partial  \ar[u]_{\gamma } &  \mathcal S_{n}(X )  [{\scriptstyle \frac{1}{2}}] \ar[u]_{\rho^{\textup{Nov} }}	
} }
\end{split}
\end{equation} 
By the splitting map $\gamma\colon L_{n+1}(\Gamma) [{\scriptstyle \frac{1}{2}}]\to H_{n+1}^\Gamma(\underline E\Gamma; \mathbb L_\bullet)[{\scriptstyle \frac{1}{2}}]$, we have the following direct sum decomposition 
\[ L_{n+1}(\Gamma) [{\scriptstyle \frac{1}{2}}] \cong  KO_{n+1}^\Gamma(\underline E\Gamma)  [{\scriptstyle \frac{1}{2}}] \oplus \mathscr F \]
such that $\partial(KO_{n+1}^\Gamma(\underline E\Gamma)  [{\scriptstyle \frac{1}{2}}] ) \cap \partial (\mathscr F) = 0$. It follows that commutative diagram $\eqref{diag:toprho3}$ above descends to the following diagram: 
\begin{equation}
\begin{split}
\scalebox{0.95}{\xymatrixcolsep{0.7pc}\xymatrix{   KO^\Gamma_{n+2}(\underline E\Gamma, \widetilde X)[{\scriptstyle \frac{1}{2}}]  \ar[r]   & KO_{n+1}^\Gamma(\widetilde X)[{\scriptstyle \frac{1}{2}}] \ar[r]   &  KO_{n+1}^\Gamma(\underline E\Gamma)  [{\scriptstyle \frac{1}{2}}] \ar[r]  & KO^\Gamma_{n+1}(\underline E\Gamma, \widetilde X)  [{\scriptstyle \frac{1}{2}}]    \\
		\mathcal S_{n+1}(X) [{\scriptstyle \frac{1}{2}}]/\partial{\mathscr F} \ar[r] \ar[u]_{\widetilde\rho^{\textup{Nov}}}  & H_{n+1}^\Gamma(\widetilde X; \mathbb L_\bullet) [{\scriptstyle \frac{1}{2}}] \ar[r] \ar[u]_{\cong}  &  L_{n+1}(\Gamma )  [{\scriptstyle \frac{1}{2}}]/\mathscr F \ar[r]^\partial  \ar[u]_{\cong} &  \mathcal S_{n}(X )[{\scriptstyle \frac{1}{2}}] /\partial(\mathscr F)  \ar[u]_{\widetilde\rho^{\textup{Nov} }}	} 
}
\end{split}
\end{equation} 
By the five lemma, it follows that  
\[  \widetilde \rho^{\textup{Nov}} \colon \mathcal S_{n}(X)[{\scriptstyle \frac{1}{2}}]/\partial(\mathscr F) \xrightarrow{\ \cong \ } KO^\Gamma_{n+1}(\underline E\Gamma, \widetilde X)[{\scriptstyle \frac{1}{2}}] \]
is an isomorphism.  In particular, it implies that 
\[  \rho^{\textup{Nov}} \colon \mathcal S_{n}(X)[{\scriptstyle \frac{1}{2}}] \twoheadrightarrow KO^\Gamma_{n+1}(\underline E\Gamma, \widetilde X)[{\scriptstyle \frac{1}{2}}] \]
is surjective. 

This surjection can be used to detect many elements in $\mathcal S_{n}(X)$. 
For example, if  $KO_{n+1}^\Gamma(\underline E\Gamma)[{\scriptstyle \frac{1}{2}}]$ is not finitely generated as a module over $\mathbb Z[{\scriptstyle \frac{1}{2}}]$, then $ KO^\Gamma_{n+1}(\underline E\Gamma, \widetilde X)[{\scriptstyle \frac{1}{2}}]$ is  not finitely generated either, for any closed oriented manifold $X$. Indeed, if $X$ is a closed manifold, then $KO_{i}^\Gamma(\widetilde X)[{\scriptstyle \frac{1}{2}}]$ is a finitely generated $\mathbb Z [{\scriptstyle \frac{1}{2}}]$-module for all $i$. By the following exact sequence:
\[   \scalebox{1}{$\to KO_{n+1}^\Gamma(\widetilde X)[{\scriptstyle \frac{1}{2}}] \to  KO_{n+1}^\Gamma(\underline E\Gamma)  [{\scriptstyle \frac{1}{2}}] \to  KO^\Gamma_{n+1}(\underline E\Gamma, \widetilde X)  [{\scriptstyle \frac{1}{2}}]  \to KO_{n}^\Gamma(\widetilde X)[{\scriptstyle \frac{1}{2}}]\to ,$ } \]
we see that if $KO_{n+1}^\Gamma(\underline E\Gamma)[{\scriptstyle \frac{1}{2}}]$ is not finitely generated, then  $KO^\Gamma_{n+1}(\underline E\Gamma, \widetilde X)[{\scriptstyle \frac{1}{2}}]$ is not finitely generated either. In this case, it follows that  $\mathcal S_{n}(X)[{\scriptstyle \frac{1}{2}}]$ is not finitely generated. 

Given a discrete group $\Gamma$, let $\mathscr F\Gamma$ be the $\mathbb C$-vector space of finitely supported functions on the set of finite order elements of $\Gamma$. Define $ \mathscr F^0\Gamma$ be the subspace of $\mathscr F\Gamma$ consisting of elements $f$ such that $f(\gamma) =  f(\gamma^{-1})$. Similarly, define $\mathscr F^1\Gamma$ be the subspace of $\mathscr F\Gamma$ consisting of elements $f$ such that $f(\gamma) =  -f(\gamma^{-1})$. Then we have  (cf. \cite[Section 2]{MR3830202})
\begin{equation}\label{eq:KO}
KO_{n}^\Gamma(\underline{E}\Gamma)\otimes \mathbb C  \cong \bigoplus_{k\in \mathbb Z} H_{n+4k}(\Gamma; \mathscr F^0\Gamma) \oplus H_{n+2+4k}(\Gamma; \mathscr F^1\Gamma). 
\end{equation}
In particular, it follows that if a group contains infinitely many distinct conjugacy classes of finite order elements, then $KO_{0}^\Gamma(\underline{E}\Gamma)\otimes \mathbb C$ is not finitely generated. For example, Grigorchuk produced a family of amenable groups which contains infinitely many distinct conjugacy classes of finite order elements \cite{MR1616436}. Since each of these groups is amenable, its associated Baum-Connes assembly map is actually an isomorphism, hence in particular a split injection.  

Furthermore, the isomorphism $\eqref{eq:KO}$ also shows that  
$ KO_{n}^\Gamma(\underline{E}\Gamma)\otimes \mathbb C $ contains the group homology $\bigoplus_{k\in \mathbb Z} H_{n+4k}(\Gamma; \mathbb C)$ as a direct summand. Therefore, to find examples of groups $\Gamma$ such that  $KO_{n}^\Gamma(\underline E\Gamma)\otimes \mathbb C$ is not finitely generated, it suffices to find groups whose group homology with complex coefficients has infinite rank. Here are some examples.
\begin{enumerate}[(1)]
	\item Stalling gave a group in \cite{MR0158917}, whose $3$-dimensional homology group is not finitely generated. Since this group is obtained from free groups of two generators by using amalgamated products and HNN extensions, it follows that its associated Baum-Connes assembly map is an isomorphism \cite{MR860685}. 
	
	\item Generalizing ideas of Stalling \cite{MR0158917} and Bieri \cite{MR715779}, Bridson produces the following class of groups whose group homology is not finitely generated.  Let $F_k$ be the free group of $k$-generators with $k\geq 2$, and $\phi\colon F_k\to \mathbb Z$ be any surjective homomorphism. Denote  the direct product of $n$ copies of $F_k$ by $F_k^{(n)}$. Then $\phi$ induces a homomorphism $\phi_n \colon F_k^{(n)}\to  \mathbb Z$ which coincides with $\phi$ on each component. Let $K_n$ be the kernel of this map $\phi_n$. Then $H_n(K_n; \mathbb C)$ has infinite rank\footnote{In his paper \cite[Theorem B]{MR1680601}, Bridson proves that the integral homology $H_n(K_n; \mathbb Z)$ is not finitely generated. The same proof  shows that $H_n(K_n; \mathbb C)$ has infinite rank. }, cf. \cite[Theorem B]{MR1680601}. Also, the Baum-Connes assembly map for $K_n$ is an isomorphism. This for example follows from \cite{MR860685}.   
	\item For Thompson's group $F$, we have $H_n(F; \mathbb C) = \mathbb C\oplus \mathbb C$ for all $n\geq 1$, \cite[Theorem 7.1]{MR752825}. Recall that  Thompson's group $F$ is a-T-menable \cite{MR2006480}, hence its Baum-Connes assembly map is an isomorphism \cite{MR1821144}. 
\end{enumerate} 

Note that the examples in $(1)$, $(2)$ and $(3)$ above are torsion-free. We see that the Novikov rho invariant can detect elements in $\mathcal S_n(X)$ even when $\pi_1 X$ is torsion-free.  Let us summarize the above discussion as the following theorem. 

\begin{theorem}
	Let $X$ be a closed oriented topological manifold of dimension $n\geq 5$, and $\Gamma$ its fundamental group. Suppose the Baum-Connes assembly map for $\Gamma$ is split injective.  If   $KO_{n+1}^\Gamma(\underline E\Gamma)[{\scriptstyle \frac{1}{2}}]$ is not finitely generated as  $\mathbb Z[{\scriptstyle \frac{1}{2}}]$-module, then $\mathcal S^\topo(X)$ is not finitely generated. Consequently, we have 
	\begin{enumerate}[$(1)$]
		\item if $\bigoplus_{k\in \mathbb Z} H_{n+1+4k}(\Gamma; \mathbb C)$ is not finitely generated, then  $\mathcal S^\topo(X)$ is not finitely generated;
		\item if $n \equiv 3 \pmod 4$ and $\Gamma$ has infinitely many distinct conjugacy classes of finite order elements, then $\mathcal S^\topo(X)$ is not finitely generated.
	\end{enumerate} 
\end{theorem}

\begin{remark}
	By the work of the first and third authors, part $(2)$ of the above theorem is actually known to hold for a large class of groups, for some of which the strong Novikov conjecture has not yet been verified. In particular, part $(2)$ holds for groups that are finitely embeddable into Hilbert space \cite[Theorem 1.5 and Corollary 1.6]{MR3416114}. See Definition $\ref{def:finembed}$ below for the definition of groups that are finitely embeddable into Hilbert space.  
\end{remark}

\section{Non-rigidity of topological manifolds}\label{sec:nonrig}
In this section, we apply our main theorem (Theorem $\ref{thm:main}$) to give a lower bound of the free rank of reduced structure groups of closed oriented topological manifolds. There are in fact  two different versions of reduced structure groups, $\widetilde{\mathcal S}^\topo_{alg}(X)$ and $\widetilde {\mathcal S}^\topo_{geom}(X)$, whose precise definitions will be given below.  The group $\widetilde{\mathcal S}^\topo_{alg}(X)$ is functorial and fits well with the surgery long exact sequence. On the other hand, the group $\widetilde {\mathcal S}^\topo_{geom}(X)$ measures  the size of the collection of closed manifolds homotopic equivalent but not homeomorphic to $X$.   

Since we will be using the maximal version of various $C^\ast$-algebras throughout this section, we will omit the subscript ``max'' for notational simplicity. 

Let $X$ be an $n$-dimensional oriented closed topological manifold. Denote the monoid of orientation-preserving self homotopy equivalences of $X$ by $\homt(X)$. There are two different actions of $\homt(X)$ on $\mathcal S_n(X)$, which induce two different versions of reduced structure groups as follows,  cf. \cite{MR2508900} for the essentially same discussion in the context of algebraic surgery exact sequence. 

On one hand,  $\homt(X)$ acts naturally on ${\mathcal S}_n(X)$ by 
\[   \alpha_u(\theta) = (M, \partial M, u\circ \varphi, N, \partial N, u\circ \psi, f)   \] 
for all $u\in \homt(X)$ and all 
\[ \theta = (M, \partial M, \varphi, N, \partial N, \psi, f)\in {\mathcal S}_n(X).\]
Recall that, the natural isomorphism $\mathcal S^\topo(X) \cong  \mathcal S_n(X)$ maps  an element $ \theta  = (f, M) \in \mathcal S^\topo(X)$ to 
\[ \begin{minipage}[c]{0.38\textwidth}  \xymatrix{ M \ar[rr]^{f} \ar[dr]_f &   & X \ar[dl]^{\id}\\
	& X & 	} 
\end{minipage}
\hspace{0.03\textwidth}
\begin{minipage}[c]{0.05\textwidth}
\[ \in \mathcal S_n(X). \]
\end{minipage}
\]
In this case, $\alpha_u$  maps   
\[  \begin{minipage}[c]{0.38\textwidth} 
\xymatrix{ M \ar[rr]^{f} \ar[dr]_f &   & X \ar[dl]^{\id}\\
& X & 	} 
\end{minipage} \hspace{0.03\textwidth}
\begin{minipage}[c]{0.05\textwidth}
~~~~ to 
\end{minipage}
\hspace{0.08\textwidth}
\begin{minipage}[c]{0.38\textwidth} 
\xymatrix{ M \ar[rr]^{f} \ar[dr]_{u\circ f} &   & X \ar[dl]^{u}\\
& X & 	} 
\end{minipage}. \]
Clearly, $\alpha_u$ is a group homomorphism from ${\mathcal S}_n(X)$ to ${\mathcal S}_n(X)$. Note that this action $\alpha$ is compatible with the actions of $\homt(X)$ on other terms in the topological surgery  exact sequence.

On the other hand,  $\homt(X)$ also naturally acts on ${\mathcal S}^\topo(X)$ by compositions of homotopy equivalences, that is,    
\[   \beta_u(\theta) = (u\circ f, M)    \] 
for all $u\in \homt(X)$ and all $\theta =  (f, M) \in {\mathcal S}^\topo(X)$.  In comparison with the action $\alpha_u$ above,  the action $\beta_u$ on $ \mathcal S^\topo(X)$ maps   
\[  \begin{minipage}[c]{0.38\textwidth} 
\xymatrix{ M \ar[rr]^{f} \ar[dr]_f &   & X \ar[dl]^{\id}\\
& X & 	} 
\end{minipage} \hspace{0.03\textwidth}
\begin{minipage}[c]{0.05\textwidth}
~~~~ to 
\end{minipage}
\hspace{0.08\textwidth}
\begin{minipage}[c]{0.38\textwidth} 
\xymatrix{ M \ar[rr]^{u\circ f} \ar[dr]_{u\circ f} &   & X \ar[dl]^{\id}\\
& X & 	} 
\end{minipage}  \]
Note that the map 
\[ \beta_u \colon {\mathcal S}^\topo(X) \to {\mathcal S}^\topo(X) \] only defines a bijection of sets, and is not a group homomorphism in general. 
\begin{definition}\label{def:rdstr} With the same notation as above, we define the following reduced structure groups. 
\begin{enumerate}[(1)]
	\item 	Define 	$\widetilde{\mathcal S}_{alg}^\topo(X)$  to be the quotient group of ${\mathcal S}^\topo(X)$ by the subgroup generated by elements of the form $\theta - \alpha_u (\theta)$ for all $\theta\in \mathcal S^\topo(X)$ and all $u\in \homt(X)$. 
	\item 	Define 	$\widetilde{\mathcal S}^\topo_{geom}(X)$ to be the quotient group of ${\mathcal S}^\topo(X)$  by the subgroup generated by elements of the form $\theta - \beta_u(\theta)$ for all $\theta\in \mathcal S^\topo(X)$ and all $u\in \homt(X)$. 
\end{enumerate}

\end{definition}

Recall the following definitions and theorems  from \cite{MR3416114} and \cite{MR3590536}. Let $G$ be a countable group. An element $g\in G$ is said to have order $d$ if $ d$ is the smallest positive integer such that $g^d=e$, where $e$ is the identity element of $G$. If no such positive integer exists, we say that the order of $g$ is $\infty$.

Let us recall the notion of finite embeddability for groups in the following. We shall first recall the notion of coarse embeddability due to Gromov. 

\begin{definition}[Gromov]
A countable discrete group $\Gamma$ is said to be coarsely embeddable into Hilbert space $H$ if there exists a map $f:\Gamma \to H$ such that 
\begin{enumerate}[(1)]
	\item for any finite subset $F\subseteq \Gamma$, there exists $R >0$ such that if $\gamma^{-1}\beta\in F$, then $\|f(\gamma)- f(\beta)\|\leq R$;
	\item for any $S>0$, there exists a finite subset $F\subseteq \Gamma$ such that if $\gamma^{-1}\beta\in \Gamma-F$, then $\|f(\gamma) - f(\beta)\| \geq S$.
\end{enumerate}
\end{definition}

The notion of finite embeddability for groups, introduced by the first and third authors, is more flexible than the notion of coarse embeddability. 

\begin{definition}\label{def:finembed}
A countable discrete group $\Gamma$ is said to be finitely embeddable into Hilbert space $H$ if for any finite subset $F\subseteq \Gamma$, there exist a group $\Gamma'$ that is coarsely embeddable into $H$ and  a map $\phi: F\to \Gamma'$ such that 
\begin{enumerate}[(1)]
	\item if $\gamma, \beta$ and $\gamma\beta$ are all in $F$, then $\phi(\gamma\beta) = \phi(\gamma) \phi(\beta)$;
	\item if $\gamma$ is a finite order element in $F$, then $\ord(\phi(\gamma)) = \ord(\gamma)$. Here $\ord(\gamma)$ is the order of $\gamma$.  
\end{enumerate}
\end{definition}

The class of groups with finite embeddability into Hilbert space is quite large, including all residually finite groups, amenable groups, Gromov's monster groups, virtually torsion free groups (for example, $Out(F_n)$), and any group of analytic diffeomorphisms of a connected analytic manifold fixing a given point \cite{MR3416114}.

Let $G$ be a countable group. If $g\in G$ has  finite order $d$, then we can define an idempotent in the group algebra ${\mathbb  Q} G$ by:
$$p_g= \frac{1}{d}(\sum_{k=1}^{d} g^k).$$

For the rest of this paper, we denote the maximal group $C^*$-algebra of $G$  by $C^\ast(G)$.
\begin{definition}
We define the finite part  $ K_0 ^{\fin}( C^*(G)) $ of $K_0 ( C^*(G))$  to be the abelian subgroup of $K_0 ( C^*(G))$ generated by $[p_g]$ for all elements $g\neq e$ in $G$ with finite order.
\end{definition} 

We remark that rationally all representations of finite groups are induced from finite cyclic groups \cite{MR0450380}. This explains that the finite part of K-theory, rationally, contains all K-theory elements which can be constructed using finite subgroups, despite being constructed using only cyclic subgroups.

\begin{theorem}[{\cite[Theorem 1.4]{MR3416114}}] \label{thm:finemb}
Suppose $\Gamma$ is finitely embeddable into Hilbert space. 	If $\{g_1, \cdots, g_m\}$ is a collection of elements in $\Gamma$ with distinct finite orders  such that $g_i\neq e$ for all $1\leq i\leq m$,
then the following holds: 
\begin{enumerate}[(1)]
	\item  $\{[p_{g_1}], \cdots, [p_{g_m}]\}$ generates an abelian subgroup of $ K_0 ^{\fin}( C^\ast(\Gamma)) $ of rank $m$;
	\item  any nonzero element in the abelian subgroup of $ K_0^{\fin}( C^\ast(\Gamma)) $  generated by the elements $\{[p_{g_1}], \cdots, [p_{g_m}]\}$  is not in the image of the assembly map
	\[ \mu_\ast\colon  K_0(B\Gamma) \cong K_0^\Gamma (E\Gamma) \rightarrow K_0 (C^\ast(\Gamma)), \] where $E\Gamma$ is the universal space for proper and free $\Gamma$-actions.
\end{enumerate}
\end{theorem}

Before we go into the main result of this section, let us  recall the following key step of constructing elements in the structure group by the finite part of $K$-theory \cite[Theorem 3.4]{MR3416114}. 

\begin{example}
Let $M$ be a $(4k-1)$-dimensional closed oriented connected topological manifold with $\pi_1 M = \Gamma$. Suppose 
\[ \{g_1, \cdots, g_m\}\] is a collection of elements in $\Gamma$ with distinct finite orders  such that $g_i\neq e$ for all $1\leq i\leq m$.  Recall  the topological surgery exact sequence:
\[  \rightarrow H_{4k}(M, \mathbb{L}_\bullet) \rightarrow L_{4k}(\Gamma) \xrightarrow{\mathscr S} \mathcal S^\topo(M) \rightarrow H_{4k-1}(M, \mathbb{L}_\bullet )\rightarrow. \]
For each finite subgroup $H$ of $\Gamma$, we have the following commutative diagram:
\[ 
\xymatrix{ H_{4k}^{H}(\underline{E}H , {\mathbb L_\bullet}) \ar[r]^-A  \ar[d] &  L_{4k}(H) \ar[d] \\
	H_{4k}^{G}(\underline{E}\Gamma , {\mathbb L}_\bullet) \ar[r]^-A   &  L_{4k}(\Gamma),	}
\]
where the vertical maps are induced by the inclusion homomorphism from $H$ to $\Gamma$.
For each element $g$ in $H$ with finite order $d$, the idempotent  $p_g = \frac{1}{d}(\sum_{k=1}^{d} g^k)$ produces a class in $L_0 ( \mathbb{Q} H)$, where $L_0 ( \mathbb{Q} H)$ is the algebraic definition of $L$-groups using quadratic forms and formations with coefficients in $\mathbb Q$.  Let
$[q_g]$ be the corresponding element in $L_{4k} ( \mathbb{Q} H)$ given by periodicity.
Recall that $$L_{4k} (H) \otimes \mathbb{Q} \simeq L_{4k} (\mathbb{Q} H) \otimes \mathbb{Q}.$$
For each element $g$ in $H$ with finite order, we use the same notation $[q_g]$ to denote  the element in $L_{4k}(H)\otimes \mathbb{Q}$ corresponding to   $[q_g]\in L_{4k} ( \mathbb{Q} H)$ under the above isomorphism.

We also have the following commutative diagram:
\[ 
\xymatrix{ H_{4k}^{\Gamma}(E\Gamma , {\mathbb L_\bullet}) \otimes \mathbb Q \ar[r]^-A  \ar[d] &  L_{4k}(\Gamma) \otimes \mathbb Q \ar[d] \\
	K_0^\Gamma(E\Gamma)\otimes \mathbb Q \ar[r]^-{\mu_\ast}   &  K_{0}( C^\ast(\Gamma))\otimes \mathbb Q,	}
\]
where the left vertical map is induced by a map at the spectra level and the right vertical map is induced by the inclusion map:
$$ L_{4k}(\Gamma)\rightarrow L_{4k}( C^\ast (\Gamma) ) \cong  K_0( C^\ast (\Gamma))$$ (see \cite{MR1388305} for the last identification).

Now if $\Gamma$ is finitely embeddable into Hilbert space, then the abelian subgroup of $K_0(C^\ast(\Gamma))$ generated by $\{[p_{g_1}], \cdots, [p_{g_m}]\}$ is not in the image of
of the map 
\[ \mu_\ast:
K_0^\Gamma(E\Gamma)\to K_0(C^\ast(\Gamma)). \] It follows that 
\begin{enumerate}[(1)]
	\item any nonzero element in the abelian subgroup of $L_{4k} (  \Gamma)\otimes \mathbb{Q}$ generated by the elements $\{[q_{g_1}], \cdots, [q_{g_m}]\}$
	is not in the image of the rational assembly map \[ A: H_{4k}^{\Gamma}(E\Gamma , \mathbb L_\bullet) \otimes \mathbb{Q} \to L_{4k} ( \Gamma)\otimes \mathbb{Q}; \]
	\item the abelian subgroup of $L_{4k} (\Gamma)\otimes \mathbb{Q}$ generated by $\{[q_{g_1}], \cdots, [q_{g_m}]\}$ has rank $m$. 
\end{enumerate}

By the exactness of the surgery sequence,
we know that the map   
\begin{equation}\label{eq:struc}
\mathscr S\colon L_{4k} ( \Gamma)\otimes \mathbb{Q}\to \mathcal S^{\topo}(M)\otimes \mathbb{Q},
\end{equation} 
is injective on the abelian subgroup of  $L_{4k} (\Gamma)\otimes \mathbb{Q}$ generated by $\{[q_{g_1}], \cdots, [q_{g_n}]\}$.
\end{example}

In fact, to prove the main result (Theorem $\ref{thm:structure}$) of this section, we need to apply the above argument not only to $\Gamma$,  but also to certain semi-direct products of $\Gamma$ with  free groups of finitely many generators.

Let $\Gamma$ be a countable discrete group. Note that any set of $n$ automorphisms of $\Gamma$, say, $\psi_1, \cdots, \psi_n  \in \textup{Aut}(\Gamma) $, induces a natural action of $F_n$ the free group of $n$ generators on $\Gamma$. More precisely, if we denote the set of generators of $F_n$ by $\{s_1, \cdots, s_n\}$, then we have a homomorphism $F_n\to \textup{Aut}(\Gamma)$ by $s_i\mapsto \psi_i$. This homomorphism induces an action of $F_n$ on $\Gamma$. We denote by $\Gamma\rtimes_{\{\psi_1, \cdots, \psi_n\}} F_n$ the semi-direct product of $\Gamma$ and $F_n$ with respect to this action. If no confusion arises, we shall write $\Gamma\rtimes F_n$ instead of $\Gamma\rtimes_{\{\psi_1, \cdots, \psi_n\}} F_n$.

\begin{definition}
A countable discrete group $\Gamma$ is said to be strongly finitely embeddable into Hilbert space $H$, if $\Gamma\rtimes_{\{\psi_1, \cdots, \psi_n\}} F_n$ is finitely embeddable into Hilbert space $H$ for  all $\psi_1, \cdots, \psi_n  \in \textup{Aut}(\Gamma) $ and all $n\in \mathbb N$.
\end{definition}

We remark that all coarsely embeddable groups are strongly finitely embeddable. Indeed, if a group $\Gamma$ is coarsely embeddable into Hilbert space, then the group  $\Gamma\rtimes_{\{\psi_1, \cdots, \psi_n\}} F_n$ is also coarsely embeddable (hence finitely embeddable) into Hilbert space for $\psi_1, \cdots, \psi_n  \in \textup{Aut}(\Gamma)$ and  all $n\in \mathbb N$. Moreover, if a group $\Gamma$ has a torsion free normal subgroup $\Gamma^\prime$ such that $\Gamma/\Gamma^\prime$ is residually finite, then $\Gamma$ is strongly finitely embeddable into Hilbert space, cf. \cite[Section 4]{MR3590536}.  In particular, all residually finite groups are strongly finitely embeddable into Hilbert space.

We denote by $N_\fin(\Gamma)$ the the cardinality of the following collection of positive integers:
\[ \{ d\in \mathbb N_+\mid \exists \gamma\in \Gamma \ \textup{such that} \ \gamma\neq e \ \textup{and}\ \ord(\gamma) = d \}. \]
Then we have the following main theorem of this section. At the moment, we are only able to  prove the theorem for $\widetilde{\mathcal S}_{alg}^\topo(M)$. We will give a brief discussion after the theorem to indicate the difficulties in  proving the version  $\widetilde{\mathcal S}^\topo_{geom}(M)$.
\begin{theorem}\label{thm:structure}
Let $M$ be a closed oriented topological manifold with dimension $ n = 4k-1$ \textup{($k>1$)} and $\pi_1 M = \Gamma$. If $\Gamma$ is strongly finitely embeddable into Hilbert space, then the free rank of  $\widetilde{\mathcal S}_{alg}^\topo(M)$ is  $\geq N_\fin(\Gamma)$. 
\end{theorem}

\begin{proof}
A key point of the argument below is to use a semi-direct product $\Gamma\rtimes F_m$ to turn certain outer automorphisms of $\Gamma$ into inner automorphisms of $\Gamma\rtimes F_m$.

Note that  every self-homotopy equivalence $\psi \in \homt(M)$ induces a homomorphism\footnote{The homomorphism $\widetilde \psi_\ast: K_{1}(C_{L,0}^\ast(\widetilde M)^\Gamma)\to K_{1}(C_{L,0}^\ast(\widetilde M)^\Gamma)$ is defined as follows. The map $\psi\colon M\to M$ lifts to a map $\widetilde \psi\colon \widetilde M \to \widetilde M$. However, to view  $\widetilde \psi$ as a $\Gamma$-equivariant map, we need to use two different actions of $\Gamma$ on $\widetilde M$. Let $\tau$ be a right action of $\Gamma$ on $\widetilde M$ through deck transformations. Then we define a new action $\tau'$ of $\Gamma$ on $\widetilde M$ by $\tau'_g = \tau_{\psi_\ast(g)}$, where $\psi_\ast\colon \Gamma \to \Gamma$ is the automorphism induced by $\psi$. It is easy to see that $\widetilde \psi\colon \widetilde M \to \widetilde M$ is $\Gamma$-equivariant, when $\Gamma$ acts on the first copy of $\widetilde M$ by $\tau$ and the second copy of $\widetilde M$ by $\tau'$. Let us denote the corresponding $C^\ast$-algebras by $C_{L,0}^\ast(\widetilde M)_{\tau}^\Gamma$ and $C_{L,0}^\ast(\widetilde M)_{\tau'}^\Gamma$. Observe that,  despite  the two different actions of $\Gamma$ on $\widetilde M$,  the two $C^\ast$-algebras $C_{L,0}^\ast(\widetilde M)_{\tau}^\Gamma$ and $ C_{L,0}^\ast(\widetilde M)_{\tau'}^\Gamma$ are canonically identical, since an operator is invariant under the action $\tau$ if and only if it is invariant under the action $\tau'$. 	 }
\[\widetilde \psi_\ast: K_{1}(C_{L,0}^\ast(\widetilde M)^\Gamma)\to K_{1}(C_{L,0}^\ast(\widetilde M)^\Gamma).\]
Let $\mathcal I_1(C_{L,0}^\ast(\widetilde M)^\Gamma)$ be the subgroup of $K_{1}(C_{L,0}^\ast(\widetilde M)^\Gamma)$  generated by elements of the form $[x] - \widetilde \psi_\ast[x]$ for all $[x]\in K_{1}(C_{L,0}^\ast(\widetilde M)^\Gamma)$ and all $\psi\in \homt(M)$. By the definition of the higher rho invariant map (cf. Definition $\ref{def:highrho}$):
\[  \rho: \mathcal S^\topo(M) \to  K_{1}(C_{L,0}^\ast(\widetilde M)^\Gamma), \] we have
\[  \rho(\alpha_\psi(\theta)) = \widetilde \psi_\ast(\rho(\theta))  \in K_1(C_{L,0}^\ast(\widetilde M)^\Gamma) \]
for all $\theta\in \mathcal S^\topo(M)$ and $\psi\in \homt(M)$. It follows that  $\rho$ descends to a group homomorphism
\[ \widetilde{\mathcal S}^\topo_{alg}(M) \to K_{1}(C_{L,0}^\ast(\widetilde M)^\Gamma)\big/\mathcal I_1(C_{L,0}^\ast(\widetilde M)^\Gamma). \]

Now for a collection of elements $\{\gamma_1, \cdots, \gamma_\ell \mid \gamma_i\neq e\}$ with distinct finite orders, let  $\mathscr S(p_{\gamma_1}), \cdots, \mathscr S(p_{\gamma_\ell}) \in \mathcal S^\topo(M)$ be the corresponding elements from line $\eqref{eq:struc}$. To be precise, the elements $\mathscr S(p_{\gamma_1}), \cdots, \mathscr S(p_{\gamma_\ell})$ actually lie in  $\mathcal S^\topo(M) \otimes \mathbb Q$. Consequently, all abelian groups in the following need to be tensored by the rationals $\mathbb Q$. For simplicity, we shall omit $\otimes \mathbb Q$ from our notation, with the understanding that the abelian groups below are to be viewed as tensored with $\mathbb Q$.  Also, let us write 
\[  \rho (\gamma_i)  \coloneqq  \rho (\mathscr S(p_{\gamma_i}))  \in  K_{1}(C_{L,0}^\ast(\widetilde M)^\Gamma). \]
To prove the theorem, it suffices to show that the elements
\[  \rho(\gamma_1), \cdots, \rho(\gamma_\ell)\]  
are linearly independent in $K_{1}(C_{L,0}^\ast(\widetilde M)^\Gamma)\big/\mathcal I_1(C_{L,0}^\ast(\widetilde M)^\Gamma)$,  for any collection of elements 
$ \{\gamma_1, \cdots, \gamma_\ell\mid \gamma_i\neq e\}$ with distinct finite orders.

Assume to the contrary that there exist  
\[ [x_1], \cdots,  [x_m] \in K_{1}(C_{L,0}^\ast(\widetilde M)^\Gamma) \textup{ and } \psi_1, \cdots, \psi_m\in \homt(M)\] such that
\begin{equation}\label{eq:van}
\sum_{i=1}^\ell c_i \rho(\gamma_i) = \sum_{j=1}^{m} \big([x_j] - (\widetilde \psi_{j})_\ast[x_j]\big),
\end{equation}
where $c_1, \cdots, c_\ell\in \mathbb Z$ with at least one $c_i\neq 0$. In fact, we shall investigate Equation $\eqref{eq:van}$ in the group $K_1(C_{L,0}^\ast(E(\Gamma\rtimes F_m))^{\Gamma\rtimes F_m})$ and arrive at a contradiction, where $\Gamma\rtimes F_m$ is a certain semi-direct product of $\Gamma$ with the free group of $m$ generators $F_m$ and $E(\Gamma\rtimes F_m)$ is the universal space for free and proper $\Gamma\rtimes F_m$-actions. 

Let us fix a map $\sigma\colon M \to B\Gamma$ that induces an isomorphism of their fundamental groups, where $B\Gamma$ is the classifying space of $\Gamma$.  Suppose $\varphi \colon M \to M$ is an orientation preserving self homotopy equivalence of $M$. Then $\varphi$ induces an automorphism of $\Gamma$\footnote{Precisely speaking, $\varphi$ only defines an outer automorphism of $\Gamma$, and one needs to make a specific choice of a representative in $ \textup{Aut}(\Gamma)$. Any such  choice will work for our proof.}, also denoted by  $\varphi\in \textup{Aut}(\Gamma)$. Now consider the semi-direct product $\Gamma\rtimes_{\varphi} \mathbb Z$ and its associated classifying space $B(\Gamma\rtimes_{\varphi} \mathbb Z)$. Let  $
\hat \varphi$ be the element in $\Gamma\rtimes_{\varphi} \mathbb Z$ that corresponds to the generator $ 1\in \mathbb Z$. We write   \[ \Phi\colon B(\Gamma\rtimes_{\varphi} \mathbb Z)\to B(\Gamma\rtimes_{\varphi} \mathbb Z) \] for the map induced by the automorphism 
\[ \Gamma\rtimes_\varphi \mathbb Z \to \Gamma\rtimes_\varphi \mathbb Z \textup{ given by } a \mapsto \hat \varphi a \hat \varphi^{-1}.\] 
Suppose $\iota\colon B\Gamma \to B(\Gamma\rtimes_{\varphi}\mathbb Z)$ is the map induced by the inclusion $\Gamma \hookrightarrow \Gamma\rtimes_{\varphi}\mathbb Z$.  Then the map  
\[ \iota \circ  \sigma\circ  \varphi \colon M \xrightarrow{\varphi} M \xrightarrow{\sigma} B\Gamma \xrightarrow{\iota} B(\Gamma\rtimes_{\varphi}\mathbb Z)\]  is homotopy equivalent to the map 
\[  \Phi\circ \iota \circ \sigma\colon M \xrightarrow{\sigma} B\Gamma \xrightarrow{\iota} B(\Gamma\rtimes_{\varphi}\mathbb Z)\xrightarrow{\Phi} B(\Gamma\rtimes_{\varphi}\mathbb Z),\] since they induce the same map on fundamental groups. 

Let $ \widetilde\sigma\colon \widetilde M \to E\Gamma$ be the lift of the map $\sigma\colon M \to B\Gamma$. Similarly, $\widetilde \varphi\colon \widetilde M\to \widetilde M$ is the lift of $\varphi\colon M \to M$, and 
\[ \widetilde \Phi\colon E(\Gamma\rtimes_{\varphi}\mathbb Z)\to E(\Gamma\rtimes_{\varphi}\mathbb Z)\] is the lift of $\Phi\colon B(\Gamma\rtimes_{\varphi}\mathbb Z)\to B(\Gamma\rtimes_{\varphi}\mathbb Z).$

Since $ \Phi\colon B(\Gamma\rtimes_{\varphi} \mathbb Z)\to B(\Gamma\rtimes_{\varphi} \mathbb Z)$ is induced by an inner conjugation morphism on $\Gamma\rtimes_{\varphi} \mathbb Z$, the map\footnote{The $C^\ast$-algebra $C_{L,0}^\ast(E\Gamma)^\Gamma$ is the  inductive limit of $C_{L,0}^\ast(Y)^\Gamma$,  where $Y$ ranges over all $\Gamma$-cocompact subspaces of $E\Gamma$. } 
\[ \widetilde \Phi_{\ast}\colon K_1(C_{L,0}^\ast(E\Gamma)^\Gamma) \to K_1(C_{L,0}^\ast(E\Gamma)^\Gamma)\] 
is the identity map.   It follows that for each $[x]\in K_1(C_{L,0}^\ast(\widetilde M)^\Gamma)$, we have 
\begin{align*}
\widetilde \iota_\ast \widetilde \sigma_\ast ({\widetilde\varphi}_\ast[x])  
= {\widetilde \Phi}_{\ast} \widetilde \iota_\ast \widetilde \sigma_\ast ([x])  =  \widetilde \iota_\ast \widetilde \sigma_\ast ([x]) 
\end{align*}   
in $ K_1(C_{L,0}^\ast(E(\Gamma\rtimes_{\varphi}\mathbb Z))^{\Gamma\rtimes_{\varphi}\mathbb Z})$, where $\widetilde \iota_\ast \widetilde \sigma_\ast$ is the composition 
\[ \scalebox{0.95}{$K_1(C_{L,0}^\ast(\widetilde M)^\Gamma) \xrightarrow{\widetilde \sigma_\ast} K_1(C_{L,0}^\ast(E\Gamma)^\Gamma) \xrightarrow{\widetilde \iota_\ast} K_1(C_{L,0}^\ast(E(\Gamma\rtimes_{\varphi}\mathbb Z))^{\Gamma\rtimes_{\varphi}\mathbb Z}).$ } \]	   The same argument also works simultaneously for an arbitrary finite number of orientation preserving self homotopy equivalences 
\[ \psi_1, \cdots, \psi_m \in \homt(M), \] in which case we have 
\[    \widetilde \iota_\ast \widetilde \sigma_\ast (({\widetilde\psi}_i)_\ast[x])  
=  \widetilde \iota_\ast \widetilde \sigma_\ast ([x])  \textup{ in }
K_1(C_{L,0}^\ast(E(\Gamma\rtimes_{\{\psi_1, \cdots, \psi_m\}} F_m))^{\Gamma\rtimes_{\{\psi_1, \cdots, \psi_m\}} F_m}),\]
for all $[x]\in K_1(C_{L,0}^\ast(\widetilde M)^\Gamma)$.  In other words, $({\widetilde\psi}_i)_\ast[x] $ and $[x]$ have the same image in \[ K_1(C_{L,0}^\ast(E(\Gamma\rtimes_{\{\psi_1, \cdots, \psi_m\}} F_m))^{\Gamma\rtimes_{\{\psi_1, \cdots, \psi_m\}} F_m}). \] For notational simplicity, let us write $\Gamma\rtimes F_m$ for $\Gamma\rtimes_{\{\psi_1, \cdots, \psi_m\}} F_m$. If no confusion is likely to arise, we shall still denote $\widetilde \iota_\ast \widetilde \sigma_\ast ([x]) $ in $K_1(C_{L,0}^\ast(E(\Gamma\rtimes F_m))^{\Gamma\rtimes F_m})$ by $[x]$.

If we pass Equation $\eqref{eq:van}$ to $K_1(C_{L,0}^\ast(E(\Gamma\rtimes F_m))^{\Gamma\rtimes F_m})$ under the map \[ \scalebox{1}{ $K_1(C_{L,0}^\ast(\widetilde M)^\Gamma) \xrightarrow{\widetilde \sigma_\ast} K_1(C_{L,0}^\ast(E\Gamma)^\Gamma) \xrightarrow{\widetilde \iota_\ast} K_1(C_{L,0}^\ast(E(\Gamma\rtimes F_m))^{\Gamma\rtimes F_m}),$} \] then it follows from the above discussion that 
\[  \sum_{k=1}^\ell c_k \rho(\gamma_k) = 0 \textup{ in $K_1(C_{L,0}^\ast(E(\Gamma\rtimes F_m))^{\Gamma\rtimes F_m})$}  , \]
where at least one $c_k\neq 0$. By the commutative diagram $\eqref{diag:surgery}$, we have 
\begin{equation}\label{eq:vanish}
\partial_{\Gamma\rtimes F_m} \Big( \sum_{k=1}^\ell c_k [p_{\gamma_k}] \Big) = 2 \cdot \big( \sum_{k=1}^\ell c_k \rho(\gamma_k) \big) = 0,
\end{equation}
where $ \partial_{\Gamma\rtimes F_m} $ is the connecting map in the following long exact sequence: 
\begin{equation}\label{cd:bc}
\begin{gathered}
\scalebox{1}{\xymatrixcolsep{1pc}\xymatrix{K_0(C_{L,0}^\ast(E(\Gamma\rtimes F_m))^{\Gamma\rtimes F_m})\ar[r] &  K_0^{\Gamma\rtimes F_m}(E(\Gamma\rtimes F_m))  \ar[r]^-{\mu} &  K_0(C^\ast(\Gamma\rtimes F_m)) \ar[d]^{\partial_{\Gamma\rtimes F_m}} \\
		K_1(C^\ast(\Gamma\rtimes F_m)) \ar[u]& K_{1}^{\Gamma\rtimes F_m}(E(\Gamma\rtimes F_m)) \ar[l] &  K_{1}(C_{L,0}^\ast(E(\Gamma\rtimes F_m))^{\Gamma\rtimes F_m}). \ar[l]} }
\end{gathered}
\end{equation}		

Now by assumption $\Gamma$ is strongly finitely embeddable into Hilbert space. Hence $\Gamma\rtimes F_m$ is finitely embeddable into Hilbert space. By Theorem $\ref{thm:finemb}$, we have the following. 
\begin{enumerate}[(i)]
	\item $\{[p_{\gamma_1}], \cdots, [p_{\gamma_\ell}]\}$ generates an  abelian subgroup  of $K_0^\fin(C^\ast(\Gamma\rtimes F_m))$ of rank $\ell$, since $\gamma_1, \cdots, \gamma_\ell$ have distinct finite orders. In other words, 
	\[ \sum_{k=1}^n c_k [p_{\gamma_k}] \neq 0  \in K_0^\fin(C^\ast(\Gamma\rtimes F_m))\]
	if at least one $c_k\neq 0$.
	\item Every nonzero element in $K_0^\fin(C^\ast(\Gamma\rtimes F_m))$ is not in the image of the assembly map 
	\[ \mu\colon  K_0^{\Gamma\rtimes F_m}(E(\Gamma\rtimes F_m)) \to K_0(C^\ast(\Gamma\rtimes F_m)). \]
	In particular, it follows that the map 
	\[  \partial_{\Gamma\rtimes F_m}\colon  K^\fin_0(C^\ast(\Gamma\rtimes F_m)) \to K_{1}(C_{L,0}^\ast(\widetilde X)^{\Gamma\rtimes F_m})\]
	is injective.  
\end{enumerate} 
Therefore, we have  $\partial_{\Gamma\rtimes F_m} \Big( \sum_{k=1}^\ell c_k [p_{\gamma_k}] \Big) \neq 0$, which contradicts Equation $\eqref{eq:vanish}$.
This finishes the proof.
\end{proof}

\begin{remark}
An obvious analogue of Theorem $\ref{thm:structure}$ holds for homology manifold structure groups. See Remark $\ref{rk:hommfld}$ for a discussion on homology manifold structure groups. 
\end{remark}

It is tempting to use a similar argument to prove an analogue of Theorem $\ref{thm:structure}$ above for  $\widetilde{\mathcal S}^\topo_{geom}(M)$. However, there are some subtleties. 
First, note that (cf. \cite{MR2508900})  
\[  \alpha_\varphi(\theta) +  [\varphi] =  \beta_\varphi (\theta)    \]	
for all $\theta = (f, N) \in \mathcal S^\topo(M)$ and all $\varphi\in \homt(M)$, where $[\varphi] = (\varphi, M)$ is the element given by $\varphi\colon M \to M$ in $ \mathcal S^\topo(M). $
It follows that 
\[  \rho(\beta_\varphi(\theta))  = \rho(\alpha_\varphi(\theta)) + \rho([\varphi]) = \varphi_\ast(\rho(\theta)) + \rho([\varphi]). \]
In other words, in general,  $\rho(\beta_\varphi (\theta)) \neq \varphi_\ast(\rho(\theta))$, and consequently the homomorphism \[ \rho\colon {\mathcal S}^\topo(M) \to  K_1(C_{L,0}^\ast(\widetilde M)^\Gamma) \]
does \emph{not} descend to a homomorphism from the group   $\widetilde{\mathcal S}^\topo_{geom}(M)$ to  the quotient group $ K_{1}(C_{L,0}^\ast(\widetilde M)^\Gamma)\big/\mathcal I_1(C_{L,0}^\ast(\widetilde M)^\Gamma)$. 
New ingredients are needed to take care of this issue.  On the other hand, there is strong evidence which suggests an analogue of Theorem $\ref{thm:structure}$ for  $\widetilde{\mathcal S}^\topo_{geom}(M)$. For example, this has been verified by the first and third authors  for residually finite groups \cite[Theorem 3.9]{MR3416114}. Also, Chang and the first author gave a different lower bound of  $\widetilde{\mathcal S}^\topo_{geom}(M)$ that works for all non-torsion-free groups, although the lower bound is weaker \cite[Theorem 1]{MR1988288}. In any case, we hope to deal with this question in a separate paper.

We close this section by proving the following theorem, which is an analogue of the theorem of Chang and the first author cited above \cite[Theorem 1]{MR1988288}.

\begin{theorem}
Let $X$ be a closed oriented topological manifold with dimension $ n = 4k-1$ \textup{($k>1$)} and $\pi_1 X = \Gamma$. If $\Gamma$ is not torsion free, then the free rank of $\widetilde{\mathcal S}^\topo_{alg}(X)$  is $\geq 1$. 
\end{theorem}
\begin{proof}
Recall that for any non-torsion-free countable discrete group $G$, if $\gamma \neq e$ is a finite order element of $G$, then $[p_\gamma]$ generates a subgroup of rank one in $K_0(C^\ast(G))$ and any nonzero multiple of $[p_\gamma]$ is not in the image of the assembly map $\mu\colon K_0^\Gamma(EG)\to K_0(C^\ast(G))$ \cite[Theorem 2.3]{MR3416114}. Using this fact, the statement follows from the same proof as in Theorem $\ref{thm:structure}$.  
\end{proof}

\section{Lipschitz structures and Siebenmann periodicity map}\label{sec:lip}

In this section, we show how our approach can be adapted to deal with signature operators arising from  Lipschitz structures on topological manifolds. We also show that the higher rho invariant map defined using Lipschitz structures is  compatible with the Siebenmann periodicity map.   Throughout this section, all manifolds are assumed to have dimension $\geq 5$.

As we have seen, for our main theorem (Theorem $\ref{thm:main}$) and our main application $\ref{thm:structure}$, it suffices to work with the smooth or PL representatives, that is, the groups $\mathcal N^{C^\infty}_n(X; w)$, $L^{C^\infty}_n(\pi_1 X; w)$ and $ \mathcal S_{n}^{C^\infty}(X, w)$, or $\mathcal N^{PL}_n(X; w)$, $L^{PL}_n(\pi_1 X; w)$ and $ \mathcal S_{n}^{PL}(X, w)$, cf. Section $\ref{sec:smstr}$. On the other hand, our approach to the higher rho invariant given in Section $\ref{sec:highrho}$ applies essentially verbatim to signature operators associated to Lipschitz structures on topological manifolds. In particular, with some minor modifications given below, we can directly deal with signature operators arising from Lipschitz structures as well. 

There are two modifications that are needed for the construction of the higher rho invariant in this case. 
\begin{enumerate}[(i)]
\item We use the unbounded theory (cf. \cite[Section 5]{MR2220522}) instead of the bounded theory that is used in Section $\ref{sec:highrho}$.  For various properties of the signature operator associated to a Lipschitz structure, we refer the reader to \cite{MR799572, MR1050489} \cite[Section 3]{MR1142484}. The higher rho invariant map can be defined by the same formula as in Definition  $\ref{def:highrho}$, and the proofs are  identical.

\item To prove the well-definedness of the higher rho invariant map (for the Lipschitz case), the techniques in Section $\ref{sec:bord}$ do not quite apply to the unbounded theory. Recall that, for an even dimensional manifold $Y$ with boundary $\partial Y$,  the restriction of the signature operator $D_Y$ of $Y$ is 2 times the signature operator  $D_{\partial Y}$ on $\partial Y$. In order to take care this factor of $2$, we use the results of Stern on topological vector fields (\cite[Corollary 1.5]{MR0394659}) and techniques developed by Pedersen, Roe and the first author in \cite[Section 4]{MR1388315}.  In particular, these results allow us to use a vector field to split the signature operator\footnote{Technically speaking, we may need to punch out a disc in $Y$, replace it with an infinite cylinder, and control this cylinder appropriately over the reference control space.} $D_Y$ into two halves. The rest of the proof is similar to the proof for the commutativity of the middle square in Theorem $\ref{thm:stoa}$. 
\end{enumerate}

Let $X$ be a connected oriented closed topological manifold of dimension $n \geq 5$ with $\pi_1 X = \Gamma$. Recall that an element of $\mathcal S^\topo(X)$ is an orientation preserving homotopy equivalence $f\colon M \to X$, and an element of $\mathcal S^\topo_\partial (X\times D^4)$ is an orientation preserving homotopy equivalence $g\colon (N, \partial N) \to (X\times D^4, X\times S^3)$ such that $g$ restricts to a homeomorphism on the boundary.  As pointed out in item (i) above, by using Lipschitz structures, our construction of higher rho invariant from Definition $\ref{def:highrho}$ can be directly applied to elements in both $\mathcal S^\topo(X)$ and $\mathcal S^\topo_\partial(X\times D^4)$.   Let us denote the corresponding higher rho invariant map by 
\[\rho^{Lip}\colon \mathcal S^\topo(X) \to K_n(C_{L,0}^\ast(\widetilde X)^\Gamma) \]
and 
\[ \rho^{Lip}_s\colon \mathcal S^\topo_\partial(X\times D^4) \to K_n(C_{L,0}^\ast(\widetilde X)^\Gamma).\]
The following proposition is a consequence of Theorem $\ref{thm:main}$ and Theorem $\ref{thm:struciso}$. 

\begin{proposition}\label{prop:comp}
The higher rho invariant map is compatible with the Siebenmann periodicity map. More precisely, we have 
\[   \rho^{Lip} = \rho^{Lip}_s \circ \GP \colon \mathcal S^\topo(X) \to K_n(C_{L,0}^\ast(\widetilde X)^\Gamma), \]
where $\GP$ is the geometric periodicity map\footnote{See the discussion before Proposition $\ref{prop:strhom}$ and Theorem $\ref{thm:gp=sp}$.} due to Cappell and the first author. 
\end{proposition}
\begin{proof}
Consider the following commutative diagram from Proposition $\ref{prop:strhom}$: 
\[ 
\xymatrixcolsep{4pc}\xymatrix{ \mathcal S^\topo (X) \ar@{^{(}->}[r]^-{\GP} \ar[d]_{\iota_\ast} &  \mathcal S^\topo_\partial (X\times D^4) \ar[d]^{\beta_\ast} \\
	\mathcal S_{n}(X)  \ar@{^{(}->}[r]^-{\times \mathbb{CP}^2}  &  \mathcal S_{n+4 }(X). 
}
\]
Since every element in $\mathcal S_n(X)$ is equivalent to a smooth representative, it follows that the higher rho invariant map  $\rho^{Lip}$ for $\mathcal S^\topo(X)$ (resp. $\mathcal S^\topo_\partial (X\times D^4)$)  defined using Lipschitz structures above coincides with our  definition of the higher rho invariant for $\mathcal S_n(X)$ in Definition $\ref{def:highrho}$. Similarly, the higher rho invariant map $ \rho_s^{Lip}$ for $\mathcal S^\topo_\partial (X\times D^4)$ defined using Lipschitz structures coincides with our  definition of the higher rho invariant for $\mathcal S_{n+4}(X)$. On the other hand, by the product formula for higher rho invariants (cf. the discussion before Remark $\ref{rmk:sigprod}$), the higher rho invariant remains unchanged under the map 
$\times \mathbb{CP}^2$. By the commutative diagram above, this finishes the proof. 
\end{proof}

\begin{remark}
We point out that, if one is only interested in the well-definedness  of the higher rho invariant map  after inverting $2$, that is, if one only wants to  prove the map 
\[ \rho\colon \mathcal S_{n}(X) \to K_n(C_{L,0}^\ast(\widetilde X)^\Gamma)\otimes \mathbb Z[1/2] \]
is well-defined, then there is  a simpler argument than the one outlined in item (ii) above. Indeed, in this case,  an argument similar to the proof for the commutativity of the middle square in Theorem $\ref{thm:stoa}$ suffices. 
\end{remark}



\appendices

\section{Uniform control and uniform invertibility}\label{app:inv}

In this part of the appendix, we show that a uniform family of geometrically controlled Poincar\'{e} complexes gives rise to a uniform family of analytically controlled Poincar\'{e} complexes. 

First, let us introduce the notion of  uniform families of geometrically controlled Poincar\'{e} complexes.

\begin{definition} \label{def:uniformgc}
A uniform family of  geometrically controlled Poincar\'{e} complexes over $X$ is a family of geometrically controlled Poincar\'{e} complexes over $X$
\[ \{ (E_\lambda, b_\lambda, T_\lambda)\}_{\lambda \in \Lambda}\]
such that the following conditions are satisfied: 
\begin{enumerate}[(1)]
	\item the propagations of $b_\lambda$ and $T_\lambda$ are uniformly bounded;
	\item the propagations of the chain homotopy inverses $T'_\lambda$ of $T_\lambda$ are uniformly bounded;
	\item the propagations of the chain homotopies between $T'_\lambda\circ  T_\lambda$ and $1$ are uniformly bounded,   and the propagations of the chain homotopies between $T_\lambda\circ  T'_\lambda$ and $1$ are uniformly bounded;
	\item the matrix coefficients of all maps above (including the chain homotopies) are uniformly bounded.
\end{enumerate}
\end{definition}

There is a natural counterpart of the above notion of uniform families in the analytically controlled category. Recall from Definition $\ref{def:sop}$ that,  for an analytically controlled Poincar\'{e} complex $(E, b, T)$, we have 
\[ B = b + b^\ast \textup{ and } S = i^{p(p-1)+l} T. \] 

\begin{definition} 
A uniform family of analytically controlled Poincar\'{e} complexes over $X$ is a family of analytically controlled Poincar\'{e} complexes over $X$
\[ \{ (E_\lambda, b_\lambda, T_\lambda)\}_{\lambda \in \Lambda}\]
such that the following conditions are satisfied: 
\begin{enumerate}[(1)]
	\item the norms of $b_\lambda$ and $T_\lambda$ are uniformly bounded;
	\item the norms of the chain homotopy inverses $T'_\lambda$ of $T_\lambda$ are uniformly bounded;
	\item there exist $\varepsilon > 0 , C>0$ such that  
	\[ \varepsilon <  \| B_\lambda \pm S_\lambda \| < C. \] 
\end{enumerate}
\end{definition}

\begin{proposition}\label{prop:uniformbd}
Suppose  $\{(E_\lambda, b_\lambda, T_\lambda)\}_{\lambda \in \Lambda}$ is a uniform family of  geometrically controlled Poincar\'{e} complexes over $X$. Then their $\ell^2$-completions give rise to a uniform family of analytically controlled Poincar\'{e} complexes over $X$. In particular, there exists $\varepsilon> 0, C>0$ such that 
\[   \varepsilon < \|B_\lambda \pm S_\lambda \| < C \]
for all $\lambda \in \Lambda$. 
\end{proposition}
\begin{proof}
Note that $\bigoplus_\lambda E_\lambda$ is a geometrically controlled Pincar\'{e} complex over $\bigsqcup_\lambda X $. Let $\mathcal E$ be the $\ell^2$ completion of $\bigoplus_\lambda E_\lambda$. Then by the discussion in Section $\ref{sec:geomhp}$ and $\ref{sec:acp}$, the operators 
\[   \bigoplus_{\lambda  \in \Lambda }(B_\lambda  + S_\lambda ) \textup{ and } \bigoplus_{\lambda \Lambda}  (B_\lambda  - S_\lambda )\]
are bounded and invertible \cite[Lemma 3.5]{MR2220522}. In particular, there exist $\varepsilon > 0 , C>0$  such that 
\[  \varepsilon < \big\| \bigoplus_{\lambda \in\Lambda}  (B_\lambda  + S_\lambda )\big\|_{\mathcal E} < C  \textup{ and }  \varepsilon < \big\| \bigoplus_{\lambda \in \Lambda} (B_\lambda  - S_\lambda )\big\|_{\mathcal E} < C.  \]
It follows that $ \varepsilon < \|B_\lambda\pm S_\lambda\|_{\mathcal E_\lambda} < C $ for all $\lambda \in \Lambda $, where $\mathcal E_j$ is the $\ell^2$-completion of $E_\lambda$.  
\end{proof}

\section{K-homology class of signature operator }\label{app:khom}

In this section of the appendix, we give a detailed construction of the $K$-homology classes of  signature operators on PL manifolds. The material of this section is very much inspired by \cite{higsonxie-witt}.  We will only give the details for the odd case. The even case is similar.

\subsection{Special case: closed PL manifolds}\label{app:special}
In this subsection, we construct the $K$-homology classes of signature operators on closed PL manifolds. The construction for the more general case of elements in $\mathcal N_m(X)$ will be considered in the next subsection. 

Let $M$ be a closed oriented PL manifold of dimension $m$. Assume that $M$ is equipped with a triangulation that has bounded geometry. Suppose there is a control map $\varphi\colon M \to X$. Let  $\sub(M)$ be the  subdivision from Section $\ref{sec:subdiv}$. Here is an outline of the construction of the $K$-homology class of the signature operator on $M$.
\begin{enumerate}[(1)]
\item Consider the vector space $  \bigoplus_{k} C_k^{\Delta}(M)$, where $C_k^\Delta(M)$ is the complex vector space of simplicial $k$-chains in the triangulation of $M$.  It  carries a natural geometrically controlled $X$-module structure, where the basis of $H$ consists of simplices of $\sub(M)$ and each basis element is labeled by the image  of its barycenterunder the control map $\varphi$. Similarly, we have geometrically controlled $X$-modules 
$\bigoplus_{k} C_k^\Delta(\sub^j(M)). $
\item Based on the geometrically controlled $X$-module $E^{(j)}$, we obtain a geometrically controlled Poincar\'e complex of $M$ for each $j\in \mathbb N$ (cf. Example $\ref{ex:geoPoincare}$), whose propagation approaches zero, as $j$ goes to $\infty$. 
\item We connect these geometrically controlled Poincar\'e complexes in a canonical way to form a continuous family of geometrically controlled Poincar\'e complexes parametrized by $t\in [0, \infty)$ such that the propagation of each Poincar\'e complex goes to zero, as $t$ goes to $\infty$. The $K$-homology signature class of $M$ will be an element of $K_m(C_L^\ast(X))$ that is naturally determined by this family. 
\end{enumerate}

Let us first  carry out the details for connecting the geometrically controlled Poincar\'e complexes on  
\[ Q_0 = \bigoplus_{k} C_k^{\Delta}(M)  \textup{ and  } Q_2 = \bigoplus_{k} C_k^{\Delta}(\sub(M)). \]  

Note that $Q_0$ is naturally a vector subspace of $Q_2$, but $Q_0$ is \emph{not} a geometrically controlled $X$-submodule of $Q_2$, since the natural basis of $Q_0$ is not a subbasis of the natural basis of $Q_2$ and furthermore their labelings are not compatible either. In order to fix this issue, we shall introduce an auxiliary  geometrically controlled $X$-module structure on $\bigoplus_{k} C_k^{\Delta}(\sub(M))$ as follows. 
\begin{enumerate}
\item[($\hat a$)] For a simplex $\sigma = \Delta^k\in M$, suppose $\sub(\sigma)$ is the union of distinct $k$-simplices $\{ \omega_i\}$. Instead of the usual basis $\{\omega_i\}$ for the vector space $V_\sigma = C_k^{\Delta}(\sub(\sigma))$, we shall construct a new basis that  contains  $\sigma = \sum_{i}\omega_i $ as one of the basis elements. For example, start with the vector $\sum_{i}\omega_i \in V_\sigma$, and linear-independently choose any other  vectors to  form a basis of $V_\sigma$.  
\item[($\hat b$)] All basis elements of $V_\sigma$ from part (a) are labeled by the image of the barycenter of $\sigma$ under the control map $\varphi$. 
\item[($\hat c$)] We apply the same construction to each simplex of $M$, and obtain a basis of $\bigoplus_{k} C_k^{\Delta}(\sub(M))$ with the corresponding labeling. 
\end{enumerate} 
Let us write $\widehat Q_1$ for $\bigoplus_{k} C_k^{\Delta}(\sub(M))$ with this new geometrically controlled $X$-module structure.  In fact, to make our exposition more transparent, we shall introduce yet another geometrically controlled $X$-module structure  on $\bigoplus_{k} C_k^{\Delta}(\sub(M))$ as follows. 
\begin{enumerate}[(a)]
\item For a simplex $\sigma = \Delta^k\in M$, suppose $\sub(\sigma)$ is the union of some distinct $k$-simplices $\{ \omega_i\}$, which in particular forms a basis for the vector space $V_\sigma = C_k^{\Delta}(\sub(\sigma))$. 
\item All basis elements $\omega_i$ of $V_\sigma$  are labeled by the image of the barycenter of $\sigma$ under the control map $\varphi$. 
\item We apply the same construction to each simplex of $M$, and obtain a basis of $\bigoplus_{k} C_k^{\Delta}(\sub(M))$ with the corresponding labeling. 
\end{enumerate} 
We denote the resulting geometrically controlled $X$-module by $Q_1$. 

Clearly, $Q_0$ is a geometrically controlled $X$-submodule of $\widehat Q_1$. Moreover, $\widehat Q_1$ and $Q_1$ are isomorphic as geometrically controlled $X$-modules through an isomorphism which has zero propagation. 

Let $b$ and $S$ be the differential and duality operator of the simplicial chain complex (with its $X$-module structure $\widehat Q_1$)
\[ \cdots \xrightarrow{b} C_{k+1}^\Delta(\sub(M))\xrightarrow{b} C_k^\Delta(\sub(M)) \xrightarrow{b} C_{k-1}^\Delta(\sub(M)) \xrightarrow{b} \cdots.  \]
which contains   \[ \cdots 
\to  C_{k+1}^\Delta(M)\to C_k^\Delta(M) \to C_{k-1}^\Delta(M) \to  \cdots  \]
as a chain subcomplex, where the  $X$-module structure of the latter is  $Q_0$. 

We endow $ \widehat Q_1$ (resp. $Q_0$, $Q_1$ and $Q_2$)  with the inner product for which the basis elements are orthonormal. From now on,  we shall identify the dual space of $ \widehat Q_1$ (resp. $Q_0$, $Q_1$ and $Q_2$) with $\widehat  Q_1$ (resp. $Q_0$, $ Q_1$ and $Q_2$) by this inner product.    With respect to the orthogonal decomposition of $\widehat  Q_1 = Q_0 \oplus Q_0^\perp$, we have 
\[  b = \begin{pmatrix}
b_{11} & b_{12} \\ 0 & b_{22}
\end{pmatrix} \textup{ and } S = \begin{pmatrix}
S_{11} & S_{12} \\
S_{21} & S_{22}
\end{pmatrix}, \] 
where $b_{11}$  and $S_{11}$ are the differential and duality operator of the chain complex 
\[ \cdots \to C_{k+1}^\Delta(M)\to C_k^\Delta(M) \to C_{k-1}^\Delta(M) \to  \cdots. \] 
In particular, $b_{11} S_{11} = S_{11} b_{11}^\ast$ and  $bS = Sb^\ast$. If no confusion is likely to arise, we shall denote these geometrically controlled Poincar\'e complexes by $(Q_0, b_{11}, S_{11})$ and $(\widehat  Q_1, b, S)$. 

Now let us define
\[  b_t = \begin{pmatrix}
b_{11} & t b_{12} \\ 0 & b_{22}
\end{pmatrix}  \textup{ and }  S_t = \begin{pmatrix}
S_{11} & tS_{12} \\
tS_{21} & S_{22}
\end{pmatrix} \]
where $t\in [0, 1]$. A straightforward calculation shows that \[ b_tS_t = S_t b_t^\ast. \] Consider the short exact sequence 
\[ 0 \to   (Q_0, b_{11}) \xrightarrow{\iota} (\widehat Q_1, b) \to (Q_0^\perp, b_{22}) \to 0. \]
Since $\iota$ is a quasi-isomorphism (i.e. an isomorphism on homology), it follows that $(Q_0^\perp, b_{22})$ is acyclic (i.e. with zero homology). Hence the short exact sequence 
\[ 0 \to   (Q_0, b_{11}) \xrightarrow{\iota} (\widehat  Q_1, b_t) \to (Q_0^\perp, b_{22}) \to 0\]
implies that $(Q_0, b_{11}) \xrightarrow{\iota} (\widehat Q_1, b_t) $ is a quasi-isomorphism for all $t \in [0, 1]$. Therefore, in the following commutative diagram, the vertical maps are quasi-isomorphisms. 
\[ \xymatrix{ (Q_0, b_{11}^\ast) \ar[r]^{S_{11}}  & (Q_0, b_{11}) \ar[d]^-\iota  \\
(\widehat Q_1, b^\ast_t) \ar[u] \ar[r]^{S_t} & (\widehat Q_1, b_t) 
} \]
It follows that $S_t$ is a quasi-isomorphism for all $t\in [0, 1]$, since $S_{11}$ is a duality operator and thus a quasi-isomorphism. In fact, by localizing the calculation at stars of simplices and using a Mayer-Vietoris argument as in \cite[Section 4]{MR2220523}, it is not difficult to show that the maps $\iota\colon (Q_0, b_{11}) \to (\widehat Q_1, b_t) $ and   $S_t$ are geometrically controlled chain equivalences, and $\{(\widehat Q_1, b_t, S_t)\}_{t\in [0, 1]}$ form a uniform family of geometrically controlled Poincar\'e complexes in the sense of Definition $\ref{def:uniformgc}$. 

Now let us turn to the geometrically controlled Poincar\'e complexes $( Q_1, b, S)$ and $(Q_2, b, S)$. Here $Q_1$ and $Q_2$ carry the same differential $b$ and duality operator $S$. The only difference between the two complexes is their geometrically controlled $X$-module structures. We shall introduce a new family of geometrically controlled $X$-module structures on the underlying space $\bigoplus_{k} C_k^{\Delta}(\sub(M))$ which will connect $(Q_1, b, S)$ and $(Q_2, b, S)$. For a simplex $\sigma = \Delta^k$ in $M$, suppose $\sub(\sigma)$ is the union of distinct $k$-simplices $\{ \omega_i\}$. Let $\gamma_i\colon [1, 2] \to M$ be the unique linear path that starts with the barycenter of $\sigma$ and ends with the barycenter of  $\omega_i$. Now for each $t\in [1, 2]$, we introduce the following geometrically controlled $X$-module on $\bigoplus_{k} C_k^{\Delta}(\sub(M))$. 
\begin{enumerate}[(I\textsubscript{t})]
\item For a simplex $\sigma = \Delta^k\in M$, suppose $\sub(\sigma)$ is the union of some distinct $k$-simplices $\{ \omega_i\}$, which forms a basis for the vector space $V_\sigma = C_k^{\Delta}(\sub(\sigma))$. 
\item All basis elements $\omega_i$ of $V_\sigma$  are labeled by the image of $\gamma_i(t)$ under the control map $\varphi\colon M \to X$. 
\item We apply the same construction to each simplex of $M$, and obtain a basis of $\bigoplus_{k} C_k^{\Delta}(\sub(M))$ with the corresponding labeling. 
\end{enumerate} 
We denote the resulting geometrically controlled $X$-module by ${Q}_t$. Note that, when $t=1$ or $2$, the definition of $Q_t$ is consistent with the definition of $Q_1$ or $Q_2$ above. Again, it is not difficult to see that $\{ (Q_t, b, S)\}_{t\in [1, 2]}$ form a uniform family of geometrically controlled Poincar\'e complexes in the sense of Definition $\ref{def:uniformgc}$.

To summarize, we have constructed a uniform family of geometrically controlled Poincar\'e complexes $\{ (Q_t, b_t, S_t)\}_{t\in [0, 2]}$, where 
\begin{equation}\label{eq:unifam}
(Q_t, b_t, S_t) = \begin{cases}
(\widehat Q_1, b_t, S_t) & \textup{ if } 0\leq t \leq 1, \\
(Q_t, b, S) & \textup{ if } 1\leq t \leq 2.
\end{cases}
\end{equation}
We shall explain our notation. On one hand, when $t=1$,  we have identified $(\widehat Q_1, b, S)$ with $(Q_1, b, S)$ by an isomorphism with zero propagation. On the other hand, when $t =0$, there appears to be a conflict of notation. Indeed,  $(\widehat Q_1, b_0, S_0)$ is \emph{not} a chain complex modeled on the vector space $Q_0$, but rather a chain complex modeled on $\widehat Q_1$ with 
\[  b_0 = \begin{pmatrix}
b_{11} & 0  \\ 0 & b_{22}
\end{pmatrix}  \textup{ and }  S_0 = \begin{pmatrix}
S_{11} & 0 \\
0 & S_{22}
\end{pmatrix}. \]
However, $(\widehat Q_1, b_0, S_0)$ is geometrically controlled chain equivalent to $(Q_0, b_{11}, S_{11})$.  Indeed, we have that  
\[ (\widehat Q_1, b_0, S_0) = (Q_0, b_{11}, S_{11}) \oplus (Q_0^\perp, b_{22}, S_{22}), \] and $(Q_0^\perp, b_{22}, S_{22})$ is acyclic.
In other words, in terms of $K$-theory, $(\widehat Q_1, b_0, S_0) $ is equivalent to $(Q_0, b_{11}, S_{11})$; and $(\widehat Q_1, b_0, S_0) $ can be viewed as a stabilization of $(Q_0, b_{11}, S_{11})$. More precisely, let $B_0 = b_0 + b_0^\ast$, $ B_{11} = b_{11} + b_{11}^\ast$ and $ B_{22} = b_{22} + b_{22}^\ast$. Then we have
\[   B_0 \pm S_0  = (B_{11} \pm S_{11}) \oplus (B_{22} \pm S_{22}). \]
Since $(Q_0^\perp, b_{22}, S_{22})$ is acyclic,  we see that  $ uS_{22}$ is a duality operator for $(Q_0^\perp, b_{22})$ for all $u\in [0, 1]$. It follows that the invertible element (cf. Definition \ref{def:sig})
\begin{align*}
& (B_0 + S_0)(B_0 - S_0)^{-1}  =  \begin{pmatrix}
\frac{B_{11} + S_{11}}{B_{11} - S_{11}}  & \\
& \frac{B_{22} + S_{22}}{B_{22} - S_{22}}
\end{pmatrix}
\end{align*} 
is connected to the invertible element 
\[  \begin{pmatrix}
\frac{B_{11} + S_{11}}{B_{11} - S_{11}} & \\
& 1
\end{pmatrix} \] 
through the following path of invertible elements 
\[  \begin{pmatrix}
\frac{B_{11} + S_{11}}{B_{11} - S_{11}} & \\
& \frac{B_{22} + uS_{22}}{B_{22} - uS_{22}}
\end{pmatrix}. \] 
Therefore, $(\widehat Q_1, b_0, S_0) $ is just a $K$-theoretic stabilization of the geometrically controlled Poincar\'e complex  $(Q_0, b_{11}, S_{11})$. If no confusion is likely to arise, we shall continue this slight abuse of notation and say that the uniform family of geometrically controlled Poincar\'e complexes $\{ (Q_t, b_t, S_t)\}_{t\in [0, 2]}$ connects $(Q_0, b_{11}, S_{11})$ and $(Q_2, b, S)$. We shall implicitly assume this type of stabilization throughout the following discussion.

Let us write $Q_{2j} = \bigoplus_{k} C_k^\Delta(\sub^j(M))$. We can apply the same construction above to form a uniform family of geometrically controlled Poincar\'e complexes connecting $Q_{2j}$ to  $Q_{2(j+1)}$. Denote this uniform family of geometrically controlled Poincar\'e complexes
by $\{(Q_t, b_t, S_t)\}_{t\in [2j, 2j+2]}$. In fact, it is not difficult to see that the union of these families 
\[ \{(Q_t, b_t, S_t)\}_{t\in [0, \infty)}\]
is a uniform family of geometrically controlled Poincar\'e complexes. By Proposition $\ref{prop:uniformbd}$, there exist $\varepsilon >0$ and $C>0$ such that 
\[ \varepsilon < \|B_t \pm S_t\| < C   \]  
for all $t\in [0, \infty)$, where $B_t = b_t + b_t^\ast$.

Let $p(x)$ be a polynomial on  $[\varepsilon, C]\cup [-C, -\varepsilon]$ such that  
\[ \sup_{x\in [\varepsilon, C]} | p(x) - x ^{-1}| < \frac{1}{C}.\]
Then $\| p(B_t - S_t) - (B_t - S_t)^{-1}\| < \frac{1}{\|B_t - S_t\|}$, which implies that $p(B_t - S_t)$ is invertible. Moreover, the element 
\[ U_t \coloneqq ( B_t + S_t)\cdot p(B_t-S_t)  \]
has finite propagation. Since the propagations of $B_t$ and $S_t$ go to zero as $t$ goes to $\infty$, it follows that the propagation of $( B_t + S_t)\cdot p(B_t-S_t)$ goes to zero, as $t$ goes to infinity.

To summarize, we  obtain a norm-bounded and  uniformly continuous path of invertible elements
\[  U: [0, \infty) \to C^\ast(X)^+ \] 
such that the propagation of $U_t$ goes to zero, as $t\to \infty$. Here $C^\ast(X)^+$ is the unitization of $C^\ast(X)$. 

\begin{definition}\label{def:local}
The local index  $\ind_L(M, \varphi)$ of the signature operator of $M$ under the control $\varphi\colon M \to X$ is defined to be the $K$-theory class of the path $U$ in $K_1(C^\ast_L(X))$.
\end{definition}

The even dimensional case is similar. We omit the details.

\subsection{General case: elements in $\mathcal N_m(X)$}
\label{app:gen}

In this subsection, we construct the $K$-homology classes of signature operators for elements in $\mathcal N_m(X)$. See Definition \ref{def:norm} for a description of  $\mathcal N_m(X)$.

Let $\xi = (M, \partial M, \varphi, N, \partial, N, f) \in \mathcal N_m(X)$ (cf. Definition $\ref{def:norm}$). 
In this case, we consider the space  $M\cup_f (-N)$ obtained by gluing $M$ and $-N$ along the boundary by the map $f\colon \partial N \to \partial M$. Although $M\cup_f (-N)$ is not a manifold, it is still a space equipped with a Poincar\'e duality. In fact, since $f\colon\partial N \to \partial M$ is a PL infinitesimally controlled homotopy equivalence,  we can still make sense of the $K$-homology class of its ``signature operator". 

More precisely, let 
\[ \big(E_M^{(n)}, b_M^{(n)}, T_M^{(n)}, P_M^{(n)} \big),  \textup{ resp. $\big(E_N^{(n)}, b_N^{(n)}, T_N^{(n)}, P_N^{(n)} \big)$ }\] be the geometrically controlled Poincar\'{e} pair associated to the triangulation $\sub^n(M)$, resp. $\sub^n(N)$. Define  $\big(E^{(n)}, b^{(n)}\big)$ to be the quotient complex of the inclusion
\[ \big(P_N^{(n)}E_N^{(n)}, P_N^{(n)} b_N^{(n)}\big) \xrightarrow{ f^{(n)}P_N^{(n)} \oplus P_N^{(n)}} \big(E_M^{(n)}, b_M^{(n)}\big) \oplus \big(E_N^{(n)}, b_N^{(n)}\big) \]
where $f^{(n)}$ is a chain homotopy equivalence
\[ f^{(n)}\colon \big( P_N^{(n)} E_N^{(n)}, P_N^{(n)}b_N^{(n)}, T^{(n)}_{0, N}\big) \to  \big( P_M^{(n)} E_M^{(n)}, P_M^{(n)}b_M^{(n)}, T^{(n)}_{0, M}\big) \] induced by the PL infinitesimally controlled homotopy equivalence $f\colon \partial N \to \partial M$. Here $T^{(n)}_{0, N}$ and $T^{(n)}_{0, M}$ are the Poincar\'e duality operators on the boundary as defined in Lemma $\ref{lm:pbd}$.   Note that 
\begin{enumerate}[(1)]
\item the natural inclusion 
\[\hspace{1cm}\iota_n\colon (E_M^{(n)}, b_M^{(n)}, T_M^{(n)}, P_M^{(n)} )\to  (E_M^{(n+1)}, b_M^{(n+1)}, T_M^{(n+1)}, P_M^{(n+1)} )\] is a geometrically controlled homotopy equivalence of Poincar\'e pairs; 
\item since $f\colon \partial N \to \partial M$ is a PL infinitesimally controlled homotopy equivalence, $f^{(n)}$ can be chosen so that the propagation of $f^{(n)}$ goes to $0$, as $n$ approaches infinity; 
\item under the above inclusion $\iota_n$, the map  $f^{(n)}$ is geometrically controlled chain homotopic to $f^{(n+1)}$; moreover, the propagation of the homotopy goes to $0$, as $n$ goes to infinity; 
\end{enumerate}  

In particular, it follows that the maps $T_M^{(n)}$ and $T_N^{(n)}$ induce a natural Poincar\'{e} duality operator $T^{(n)}$ on  $(E^{(n)}, b^{(n)})$. To summarize, $ (E^{(n)}, b^{(n)}, T^{(n)})$ is a geometrically controlled Poincar\'{e} complex such that the propagations of all relevant maps go to $0$, as $n$ goes to infinity. 

Now essentially the same construction from Section $\ref{app:special}$ above can be applied to $\{ (E^{(n)}, b^{(n)}, T^{(n)} ) \}_{n\geq 1}$ and ``connect them  together" to produce a uniform family of geometrically controlled Poincar\'e complexes. Technically speaking, some extra care is needed for the chain equivalences $\{f^{(n)}\}$. We need to specify how to connect $f^{(n)}$ to $f^{(n+1)}$. Let us define $g^{(n)} = \iota_n^\ast \circ f^{(n+1)} \circ \iota_n $,  which fits into the following commutative diagram: 
\[ \scalebox{1}{\xymatrixcolsep{2pc}\xymatrix{ \big( P_N^{(n)} E_N^{(n)}, P_N^{(n)}b_N^{(n)}, T^{(n)}_{0, N}\big)  \ar[r]^-{\iota_n} \ar@{-->}[d]^-{g^{(n)}} &   \big( P_N^{(n+1)} E_N^{(n+1)}, P_N^{(n+1)}b_N^{(n+1)}, T^{(n+1)}_{0, N}\big) \ar[d]^-{f^{(n+1)}} \\
	\big( P_M^{(n)} E_M^{(n)}, P_M^{(n)}b_M^{(n)}, T^{(n)}_{0, M}\big) 	 & \big( P_M^{(n+1)} E_M^{(n+1)}, P_M^{(n+1)}b_M^{(n+1)}, T^{(n+1)}_{0, M}\big)\ar[l]_-{\iota^\ast_n}
}}\]
where $\iota^\ast_n$ is the adjoint of $\iota_n$. We observe that for each $t\in [0, 1]$, the map 
\begin{align*}
h_t  \coloneqq  (1-t)f^{(n)} + t g^{(n)}\colon 
\big( P_N^{(n)} E_N^{(n)}, P_N^{(n)}b_N^{(n)}, T^{(n)}_{0, N}\big)\to \big( P_M^{(n)} E_M^{(n)}, P_M^{(n)}b_M^{(n)}, T^{(n)}_{0, M}\big) 
\end{align*}   is a geometrically controlled chain equivalence of Poincar\'e complexes. The family $\{h_t\}_{0\leq t\leq 1}$ provides a natural way to deform $f^{(n)}$ into $g^{(n)}$. 
Now  consider the orthogonal decompositions  
\[ P_N^{(n+1)} E_N^{(n+1)} =  P_N^{(n)} E_N^{(n)} \oplus (P_N^{(n)} E_N^{(n)})^\perp \] and  
\[ P_M^{(n+1)} E_M^{(n+1)} =  P_M^{(n)} E_M^{(n)} \oplus (P_M^{(n)} E_M^{(n)})^\perp. \] If we write $f^{(n+1)}$ as a matrix with respect to these decompositions, then   the map 
\[ g^{(n)} \colon \big( P_N^{(n)} E_N^{(n)}, P_N^{(n)}b_N^{(n)}, T^{(n)}_{0, N}\big)\to \big( P_M^{(n)} E_M^{(n)}, P_M^{(n)}b_M^{(n)}, T^{(n)}_{0, M}\big)  \]
becomes a matrix entry of $f^{(n+1)}$. Now the construction in Section $\ref{app:special}$ above goes through in a straightforward way. 

Consequently, for each $\xi\in \mathcal N_m(X)$, we obtain a $K$-theory class  in $K_m(C^\ast_{L}(X))$. We denote this class by $\ind_L(\xi)$, and call it the $K$-homology class of signature operator associated to $\xi$, or simply the local index of $\xi$.

\section{Tensor products of Poincar\'e complexes}\label{app:tensor}
In this section, we briefly review tensor products of Poincar\'e complexes. The discussion below works simultaneously for geometrically controlled or analytically controlled Poincar\'e complexes. We will not specify which category, and simply call them Poincar\'e complexes. 

Let $(E, d, T)$ and $(F, b, R)$  be two  Poincar\'e complexes of dimension $n$ and $m$ respectively. Recall that the tensor product of two chain complexes is naturally a double complex. The total complex $(E\otimes F, \partial)$ of this double complex can be described as follows: 
\begin{enumerate}[(1)]
\item the $k$-th term of the total chain complex is 
\[ (E\otimes F)_{k} = \bigoplus_{k= p+q} E_{p}\otimes F_q;   \]
\item the differential is defined as 
\[  \partial (x\otimes y) =  dx \otimes y  + (-1)^{|x|} x\otimes b y \]
in $(E_{p-1}\otimes  F_q) \oplus (E_{p}\otimes F_{q-1})$, 
for all $x\otimes y \in E_p\otimes F_q$, 	where $|x| = p$ if $x\in E_p$. 
\end{enumerate}
Roughly speaking,  $\partial  = d\hat\otimes 1 + 1\hat\otimes b$, where $\hat \otimes $ stands for graded tensor products. And the sign convention is that a sign $(-1)^{|\alpha|\cdot |\beta|}$ is introduced whenever a symbol $\alpha$ (a chain element or a map) passes over another symbol $\beta$  (a chain element or a map). Now it is  easy to verify that  
\[  \partial^\ast (x\otimes y) = d^\ast x \otimes y  + (-1)^{|x|} x\otimes b^\ast y \]
in $ (E_{p+1}\otimes  F_q) \oplus (E_{p}\otimes F_{q+1})$, 
for all $x\otimes y \in E_p\otimes F_q$. 

The Poincar\'e duality operator $T$ and $R$ also naturally induce a Poincar\'e duality operator $T\hat\otimes R$ on $(E\otimes F, \partial)$ as follows.  
\begin{definition}
We define 
\[ (T\hat\otimes R) (x\otimes y) \coloneqq  (-1)^{(n-|x|)\cdot|y|} Tx\otimes Ry. \] 
\end{definition}
The following lemma shows that $T\hat\otimes R$ satisfies the conditions in Definition $\ref{def:hilpoin}$, hence implements a Poincar\'e duality operator for $(E\otimes F, \partial)$. 
\begin{lemma} We have 
\begin{enumerate}[$(1)$]
	\item 	$ (T\hat\otimes R )\partial^\ast v + (-1)^{k}\partial (T\hat\otimes R) v = 0$ for  all $v\in (E\otimes F)_k$;
	\item 	$ (T\hat\otimes R )^\ast v = (-1)^{(n+m-k)k}(T\hat\otimes R) v$ for all $v\in (E\otimes F)_k$. 
\end{enumerate}

\end{lemma}
\begin{proof}
Let $x\otimes y \in E_p\otimes F_q$ with $p+q = k$. Then we have 
\begin{align*}
\partial (T\hat\otimes R)(x\otimes y) = &  (-1)^{(n-|x|)\cdot|y|}\partial(Tx\otimes Ry) \\
=&  (-1)^{(n-|x|)\cdot|y|} \big( dTx\otimes Ry + (-1)^{|Tx|} Tx\otimes bRy\big)
\end{align*} 
and 
\begin{align*}
(T\hat\otimes R)\partial^\ast(x\otimes y) 	= &  (T\hat\otimes R) \big(d^\ast x\otimes y + (-1)^{|x|}x\otimes b^\ast y\big)\\
= & (-1)^{(n-|x|-1)\cdot |y|} Td^\ast x\otimes Ry + \\  & \hspace{2cm} (-1)^{|x|}(-1)^{(n-|x|)\cdot (|y|+1)} Tx\otimes Rb^\ast y \\
= &  (-1)^{(n-|x|-1)\cdot |y|} Td^\ast x\otimes Ry + \\
& \hspace{2cm} (-1)^{n}(-1)^{(n-|x|)\cdot |y|} Tx\otimes Rb^\ast y.
\end{align*} 
It follows that 
\[ (T\hat\otimes R)\partial^\ast (x\otimes y)  +  (-1)^{k}\partial (T\hat\otimes R) (x\otimes y) = 0,  \]
since $Td^\ast x + (-1)^{|x|} dT x = 0$ and $Rb^\ast y + (-1)^{|y|} bR y = 0$. This prove part $(1)$. 

The calculation for Part $(2)$ is similar. We leave out the details.

\end{proof}

\section{Product formula for higher rho invariants}\label{app:prod}
In the section, we prove Theorem $\ref{thm:prod}$, which gives  a product formula for higher rho invariants.

Given \[ \theta = (M, \partial M, \varphi, N, \partial N, \psi, f) \in \mathcal S_n(X), \] let $\theta\times \mathbb R \in \mathcal S_{n+1}(X\times \mathbb R)$ be the product of $\theta$ and $\mathbb R$.  Here various undefined terms take the obvious meanings (see Section $\ref{sec:iden}$ for the definition of $\theta\times I$ for example). Note that the construction in Section $\ref{sec:highrho}$ also applies to $\theta\times \mathbb R$ and defines its higher rho invariant $\rho(\theta\times \mathbb R) \in K_{n+1}(C^\ast_{ L, 0}(\widetilde X\times \mathbb R)^\Gamma)$. Also there is a natural homomorphism    \[ \alpha \colon C^\ast_{L, 0}(\widetilde X)^\Gamma \otimes C^\ast_L(\mathbb R)  \to C_{L, 0}^\ast (\widetilde X\times \mathbb R)^\Gamma, \] which induces an isomorphismon of  $K$-theory. 

\begin{theorem}  With the same notation as above, we have 
\[  k_n\cdot \alpha_\ast\big( \rho(\theta) \otimes \ind_L(\mathbb R) \big) =  \rho(\theta\times \mathbb R)\]
in $K_{n+1}(C^\ast_{ L, 0}(\widetilde X\times \mathbb R)^\Gamma)$, 
where $\ind_L(\mathbb R)$ is the $K$-homology class of the signature operator on $\mathbb R$, and $k_n = 1$ if $n$ is even and $2$ if $n$ is odd. 
\end{theorem}
\begin{proof}
We will prove the even case  and the odd case separately. 

\textbf{Even case.} Let us first consider the case where $n$ is even.  We use the unbounded theory to define the $K$-homology class of the signature operator on $\mathbb R$. The Hilbert-Poincar\'e complex $(F, d, R)$ associated to the signature operator on $\mathbb R$ is 
\[ \Omega^0_{L^2}(\mathbb R) \xleftarrow{d} \Omega^1_{L^2}(\mathbb R) \]
where $d$ is the adjoint of the de Rham differential map and the duality operator $R$ is the Hodge star operator. See \cite[Section 5]{MR2220522} and \cite[Section 5]{MR2220523}.

Let $(E, b, T)$ be the analytically controlled Poincar\'e complex associated to  the space $M\cup_{f}(-N)$ as in the definition of $\rho(\theta)$. See Section $\ref{sec:highrho}$ and Appendix $\ref{app:khom}$. Then the tensor product 
\[ (E\otimes F, \partial, T\hat\otimes R) \]
gives rise to a specific representative of the  Hilbert-Pioncar\'e complex associated to  $\theta\times \mathbb R$. See Appendix $\ref{app:tensor}$ for more details on tensor products of Poincar\'e complexes.   It is straightforward to verify that the self-adjoint duality operator $S_{T\hat{\otimes} R}$ (as in Definition $\ref{def:sop}$) associated to $T\hat\otimes R$ is precisely $S_T\otimes R$. For notational simplicity,  let us write $S = S_T$.  Note that we have 
\[ \partial = b \otimes 1 + 1\otimes d \colon E_{even} \otimes F_1 \to E_{odd}\otimes F_1 \oplus E_{even}\otimes F_0.    \]
\[ \partial^\ast = b^\ast \otimes 1 - 1\otimes d^\ast \colon E_{odd} \otimes F_0 \to E_{even}\otimes F_0 \oplus E_{odd}\otimes F_1.    \]
Let us identify $F_1 = \Omega^1_{L^2}(\mathbb R)$ with $F_0 = \Omega^0_{L^2}(\mathbb R)$ by  $ h\textup{dt} \mapsto h$.  With this identification,  $d$ is skew-adjoint, that is,  $d^\ast = -d$. Moreover, we have the following: 
\begin{enumerate}[(i)]
	\item \begin{align*}
	(E\otimes F)_{odd} & = (E_{even}\otimes F_1) \oplus (E_{odd}\otimes F_0) \\
	& \cong (E_{even}\oplus E_{odd})\otimes F_0  = E\otimes F_0;
	\end{align*}
	\item  \begin{align*}
	(E\otimes F)_{even} & = (E_{even}\otimes F_0) \oplus (E_{odd}\otimes F_1) \\
	& \cong (E_{even}\oplus E_{odd})\otimes F_0 = E\otimes F_0;
	\end{align*}
	\item  
	\begin{align*}
	\partial + \partial^\ast \pm S\otimes R  = B\otimes 1 - 1\otimes i D \pm   S\otimes 1 \colon 
	E\otimes F_0 \to E\otimes F_0, 
	\end{align*} 
	where  $B = b + b^\ast$ and $D = i\cdot  d$ with $i = \sqrt{-1}$.
\end{enumerate} 

\begin{claim}\label{cl:inv}
	$\rho(\theta\times \mathbb R)$ is represented by a path of invertibles: 
	\begin{equation}\label{eq:invert}
	V_t = 	\frac{B\otimes 1 - 1\otimes iD_t + S\otimes 1}{B\otimes 1 - 1\otimes iD_t - S\otimes 1} \colon 
	E\otimes F_0 \to E\otimes F_0,
	\end{equation} 
	where $D_t = (1+t)^{-1}D $ is the operator $D$ rescaled by $(1+t)^{-1}$, for $0\leq t < \infty$.
\end{claim}

\begin{proof}[Proof of Claim $\ref{cl:inv}$]
	We point out that, when defining $V_t$,  we should in fact use refinements $B_j$ of $B$ and $S_j$ of $S$ respectively such as in Appendix $\ref{app:khom}$. For notational simplicity, we will leave out the details, and continue writing $B$ and $S$ instead.  Note that there exists $\delta > 0 $ such that  $(B_j\pm S_j)^2 > \delta$ for all $j$. 	
	Furthermore, we have 
	\begin{equation}\label{eq:terms}
	\begin{aligned}
	& \frac{B\otimes 1 - 1\otimes iD_t + S\otimes 1} {B\otimes 1 - 1\otimes iD_t - S\otimes 1}	  \\
	= & \frac{\big((B+S)\otimes 1 - 1\otimes iD_t\big)  \big( (B-S)\otimes 1 + 1\otimes iD_t\big)}{ (B-S)^2\otimes 1 + 1\otimes D_t^2} \\
	= & 1+  \frac{\big( 2S(B-S)\otimes 1 +2S\otimes iD_t\big)}{(B-S)^2\otimes 1 + 1\otimes D_t^2}. 
	\end{aligned}
	\end{equation}
	Clearly, we have 
	\[ 
	(B-S)^2\otimes 1 + 1\otimes D_t^2 =  ((B-S)^2 -\delta)\otimes 1 +   1\otimes (\delta + D_t^2),\]
	where both $(B-S)^2 -\delta$ and $(\delta +  D_t^2)$ are invertible. 
	It follows that 
	\begin{align*}
	& \left( (B-S)^2\otimes 1 + 1\otimes D_t^2\right)^{-1} \\
	= & \big( 1 \otimes (\delta +  D_t^2)^{-1}\big)  \big[ \big( (B-S)^2 -\delta\big)\otimes (\delta + D_t^2)^{-1} +  1 \otimes 1\big]^{-1}. 
	\end{align*} 
	Now one can use an argument similar to the discussion in \cite[Page 841-843]{Xie2014823} to show that all terms in line $\eqref{eq:terms}$ can be approximated  arbitrarily well operator-norm-wise by elements with arbitrary small propagations. Alternatively, we can argue as follows. 
	Recall the following functional calculus by using Fourier transform and wave operator: 
	\[ f(D_t) = \frac{1}{2\pi}\int_{-\infty}^{\infty} e^{i\xi D_t} \hat f(\xi) d\xi, \]
	where $\hat f$ is the Fourier transform of $f$. Now by the propagation estimate of the wave operator $e^{i\xi D_t}$,  there exists $G_t$ such that  
	\[ \|(\delta + D_t^2)^{-1} -G_t\| \] is sufficiently small for all $t$ and the propagation of $G_t$ goes to $0$, as $t$ goes to infinity.  A similar statement holds for $D_t(\delta + D_t^2)^{-1}$. 
	
	Now approximate $\big[ \big( (B-S)^2 -\delta\big)\otimes (\delta + D_t^2)^{-1} +  1 \otimes 1 \big]^{-1} $ by 
	\[ \big[ \big( (B-S)^2 -\delta\big)\otimes G_t +  1 \otimes 1\big]^{-1}. \]
	Note that 	there exist $\varepsilon_0>0$ and $C>0$ such that 
	\[ \varepsilon_0 < \|\big( (B-S)^2 -\delta\big)\otimes G_t +  1 \otimes 1\| < C\]
	uniformly for all $t$. Now use a polynomial to approximate the function $h(x) = x^{-1}$ on the interval $[\varepsilon_0, C]$. We see that   there exists $K_t$ that approximates 
	\[ \big[ \big( (B-S)^2 -\delta\big)\otimes (\delta + D_t^2)^{-1} +  1 \otimes 1 \big]^{-1} \] sufficiently well, and the propagation of $K_t$ goes to $0$, as $t$ goes to infinity. 
	
	To summarize, for any $\varepsilon >0$, there exists a path of invertible elements $( V^\varepsilon_t)_{0\leq t < \infty}$  such that
	\begin{enumerate}[(1)]
		\item $V_0^\varepsilon = 1$,
		\item  $\|V_t - V^\varepsilon_t\| <\epsilon$ for all $t\in [0, \infty)$, 
		\item and  the propagation of $ V^\varepsilon_t$ goes to $0$, as $t$ goes to infinity.
	\end{enumerate}  
	It follows that the path $(V_t)_{0 \leq t< \infty}$ is a representative of the $K$-theory class \[ \rho(\theta\times \mathbb R) \in K_{n+1}(C_{L, 0}^\ast(\widetilde X\times \mathbb R)^\Gamma). \] 	This finishes the proof. 
\end{proof}	

Now we return to the proof of the theorem. 	Recall that  $B\pm S$ is a self-adjoint invertible operator. Therefore,  $B\pm S$ is homotopic to $P_\pm  - Q_\pm$ through a path of invertible elements, where $P_\pm$ is the positive projection of $B\pm S$ and $Q_\pm$ is the negative projection of $B\pm S$. Note that $P_\pm + Q_\pm = 1$. 
We see that the path
\[ B\otimes 1 - 1\otimes iD_t \pm S\otimes 1 =  (B\pm S)\otimes 1 - 1\otimes iD_t\] is homotopic to the path 
\begin{align*}
& (P_\pm -Q_\pm)\otimes 1 - (P_\pm +Q_\pm)\otimes iD_t \\
= & P_\pm\otimes (1-iD_t) + Q_\pm\otimes (-iD_t-1).
\end{align*} 
To be precise,  again we need to approximate $P_\pm$ and $Q_\pm$ by elements with appropriate propagations, and use these approximations instead of $P_\pm$ and $Q_\pm$. This is straightforward. In particular, the calculation below can be easily modified to work for these approximations of $P_\pm$ and $Q_\pm$ as well. For notational simplicity, we will leave out the details, and continue using $P_\pm$ and $Q_\pm$.

A routine calculation shows that 
\begin{align*}
& \big(P_\pm\otimes (-iD_t+1) + Q_\pm\otimes (-iD_t-1)\big)^{-1} \\
=&  P_\pm\otimes (-iD_t+1)^{-1} + Q_\pm\otimes (-iD_t-1)^{-1}.
\end{align*} 
It follows that, at the $K$-theory level, the path $(V_t)_{0\leq t <\infty}$  is equivalent to  
\begin{align*}
& \big(P_+\otimes (-iD_t+1) + Q_+\otimes (-iD_t-1)\big)  \\
& \hspace{2cm}\cdot\big(P_-\otimes (-iD_t+1)^{-1} + Q_-\otimes (-iD_t-1)^{-1}\big)  \\
= 	&  P_+P_-\otimes 1 + P_+Q_-\otimes (iD_t-1)(iD_t+1)^{-1}  \\
& \hspace{2cm}+ Q_+P_-\otimes (iD_t+1)(iD_t-1)^{-1}  + Q_+Q_-\otimes 1\\
=  &  \big( P_+\otimes (iD_t-1)(iD_t+1)^{-1} + (1-P_+)\otimes 1\big)  \\
& \hspace{2cm}\cdot \big( P_-\otimes (iD_t+1)(iD_t-1)^{-1} + (1-P_-)\otimes 1\big) \\
= & ([P_+]-[P_-])\otimes\big( (D_t+i)(D_t-i)^{-1}\big) 
\end{align*}
where the last term is precisely $\rho(\theta)\otimes \ind_L(\mathbb R)$. 
To summarize, when $n$ is even,  we have proved that 
\[  \alpha_\ast\big( \rho(\theta) \otimes \ind_L(\mathbb R) \big) =  \rho(\theta\times \mathbb R).\]

\textbf{Odd case.} Now we consider the case where $n$ is odd. Note that we have the following  commutative diagram:
\begin{equation}\label{diag:prod}
\begin{split}
\scalebox{1}{\xymatrixcolsep{1.7pc}\xymatrix{ K_{n}(C^\ast_{L, 0}(\widetilde X)^\Gamma) \otimes K_1(C_{L}^\ast(\mathbb R)) 
		\otimes K_1(C_{L}^\ast(\mathbb R)) \ar[r]^-{\alpha_\ast\otimes 1}_-{\cong} \ar[d]_-{1\otimes \beta_\ast}^-{\cong} & K_{n+1}(C^\ast_{ L, 0}(\widetilde X\times \mathbb R)^\Gamma) \otimes K_1(C_{L}^\ast(\mathbb R)) \ar[d]_{\cong}^{\gamma_\ast}\\ 
		K_{n}(C^\ast_{L, 0}(\widetilde X)^\Gamma) \otimes K_0(C_{L}^\ast(\mathbb R^2)) \ar[r]^-{\tau_\ast}_-\cong   & K_{n}(C^\ast_{ L, 0}(\widetilde X\times \mathbb R^2)^\Gamma).
} }
\end{split}
\end{equation} 
where various isomorphisms are induced by  the following natural homomorphisms:
\[ \alpha \colon C^\ast_{L, 0}(\widetilde X)^\Gamma \otimes C_{L}^\ast(\mathbb R) \to C^\ast_{ L, 0}(\widetilde X\times \mathbb R)^\Gamma,\]
\[  \beta \colon C_{L}^\ast(\mathbb R)
\otimes C_{L}^\ast(\mathbb R) \to C_{L}^\ast(\mathbb R^2),  \]
\[ \gamma \colon C^\ast_{L, 0}(\widetilde X\times \mathbb R)^\Gamma \otimes C_{L}^\ast(\mathbb R) \to C^\ast_{ L, 0}(\widetilde X\times \mathbb R^2)^\Gamma \]
and 
\[ \tau \colon C_{L, 0}^\ast(\widetilde X)^\Gamma
\otimes C_{L}^\ast(\mathbb R^2) \to C_{L, 0}^\ast(\widetilde X\times \mathbb R^2)^\Gamma.\]

In Proposition $\ref{prop:odd}$ below, we will show that  \[  \tau_\ast\big( \rho(\theta)\otimes \ind_L(\mathbb R^2) \big) =  \rho(\theta\times \mathbb R^2), \]
where $\ind_L(\mathbb R^2)$ is the $K$-homology class of the signature operator on $\mathbb R^2$. Assuming this for the moment, by the commutativity of diagram $\eqref{diag:prod}$, it follows that 
\begin{align*}
\gamma_\ast \left[\rho(\theta\times \mathbb R) \otimes \ind_L(\mathbb R)  \right] 
= &  \rho(\theta\times \mathbb R^2)  \\
= &   2\cdot  \tau_\ast \left[\rho(\theta)\otimes \beta_\ast \big(\ind_{L}(\mathbb R)\otimes \ind_{L}(\mathbb R) \big)\right] \\
= & 2\cdot \gamma_\ast \left[ \alpha_\ast\big(\rho(\theta)\otimes \ind_L(\mathbb R)\big) \otimes \ind_L(\mathbb R)  \right], 
\end{align*}
where the first equality is a consequence of the even case. Here we have used the fact that 
\[  \ind_L(\mathbb R^2) = 2\cdot  \beta_\ast\big(\ind_L(\mathbb R)\otimes \ind_L (\mathbb R)\big). \]  Therefore, we have 
\[ \rho(\theta\times \mathbb R) \otimes \ind_L(\mathbb R)  = 2\cdot\alpha_\ast\big(\rho(\theta) \otimes \ind_L(\mathbb R)\big) \otimes \ind_L(\mathbb R),    \]
which implies that   $\rho(\theta\times \mathbb R) =2\cdot\alpha_\ast\big(\rho(\theta) \otimes \ind_L(\mathbb R)\big).$ This finishes the proof.  
\end{proof}

\begin{proposition} \label{prop:odd} We have $ \tau_\ast\big( \rho(\theta)\otimes \ind_L(\mathbb R^2) \big) =  \rho(\theta\times \mathbb R^2), $
where $\ind_L(\mathbb R^2)$ is the $K$-homology class of the signature operator on $\mathbb R^2$.
\end{proposition}		
\begin{proof}
The proof is similar to the even case above, but the details are more involved. Here again the precise details of the proof would involve a discussion about approximations by finite propagation elements. Since this is very similar to the even case above, we will leave out the details.   

Let $(F, d, R)$ be the Hilbert-Poincar\'e complex associated to $\mathbb R^2$:
\[ \Omega^0_{L^2}(\mathbb R^2) \xleftarrow{d} \Omega^1_{L^2}(\mathbb R^2) \xleftarrow{d} \Omega^2_{L^2}(\mathbb R^2)\]
where $d$ is the dual of the de Rham differential and $R$ is the Hodge star operator. Let us write 
\[ F_{even} = \Omega^0_{L^2}(\mathbb R^2) \oplus \Omega^2_{L^2}(\mathbb R^2) \textup{ and } F_{odd} = \Omega^1_{L^2}(\mathbb R^2).  \] 

Let $(E, b, T)$ be the analytically controlled Poincar\'e complex associated to  the space $M\cup_{f}(-N)$ as in the definition of $\rho(\theta)$. See Section $\ref{sec:highrho}$ and Appendix $\ref{app:khom}$. Then the tensor product 
\[ (E\otimes F, \partial, T\hat\otimes R) \]
gives rise to a specific representative of the  Hilbert-Pioncar\'e complex associated to  $\theta\times \mathbb R^2$. Let  $S_{T\hat{\otimes}R}$, $S_T$ and $S_R$ be the self-adjoint operators (as in Definition $\ref{def:sop}$) associated to $T\hat{\otimes}R$, $T$ and $R$ respectively. Now a straightforward calculation shows that 
\begin{enumerate}[(1)]
	\item $S_{T\hat{\otimes}R} = S_T\otimes S_R$ on $E\otimes F_{even}$, and $S_{T\hat{\otimes}R} =  - S_T\otimes S_R$ on $E\otimes F_{odd}$;
	\item $\partial= b\otimes 1 + 1\otimes d $ on $E_{even}\otimes F$ and $\partial = b\otimes 1 - 1\otimes d $ on $E_{odd}\otimes F$. 
\end{enumerate}
It follows that 
\[  \partial +\partial^\ast\pm S_{T\hat{\otimes}R} =  B \otimes 1  - 1\otimes D \pm  S_T\otimes S_R, \]
on $E_{odd}\otimes F_{even}$
and 
\[  \partial +\partial^\ast  \pm S_{T\hat{\otimes}R} =  B \otimes 1  + 1\otimes D\mp  S_T\otimes S_R, \]
on $E_{even}\otimes F_{odd}$, 
where $B = b + b^\ast$ and $D= d + d^\ast$. Let $F^{\pm}$ be the eigenspace of $S_R$ belonging to the  eigenvalue $\pm 1$.  We make the following identifications: 
\[  \scalebox{1}{$E_{odd}\otimes F =  \substack{ \displaystyle E_{odd}\otimes F_{odd}^+ \\ \oplus \\ \displaystyle E_{odd}\otimes F_{odd}^- \\ \oplus \\ \displaystyle E_{odd}\otimes F_{even}^+ \\ \oplus \\ \displaystyle E_{odd}\otimes F_{even}^-  }  \xrightarrow[ \quad \substack{ \\ \oplus \\ 1\otimes 1 \\ \oplus \\ 1\otimes 1 } \quad ]{ \quad \substack{ - (B+S_T)\otimes 1 \\ \oplus \\ (B-S_T)\otimes 1 \\ \\ } \quad } \substack{ \displaystyle E_{even}\otimes F_{odd}^+ \\ \oplus \\ \displaystyle E_{even}\otimes F_{odd}^- \\ \oplus \\ \displaystyle E_{odd}\otimes F_{even}^+ \\ \oplus \\ \displaystyle E_{odd}\otimes F_{even}^-  }= (E\otimes F)_{odd},$ } \]
and 
\[  \scalebox{1}{$(E\otimes F)_{even} =  \substack{ \displaystyle E_{even}\otimes F_{even}^+ \\ \oplus \\ \displaystyle E_{even}\otimes F_{even}^- \\ \oplus \\ \displaystyle E_{odd}\otimes F_{odd}^+ \\ \oplus \\ \displaystyle E_{odd}\otimes F_{odd}^-  }  \xrightarrow[ \quad \substack{ \\ \oplus \\ - (B-S_T)\otimes 1 \\ \oplus \\ -(B+S_T)\otimes 1 } \quad  ]{ \quad \substack{   1\otimes 1 \\ \oplus \\ -1\otimes 1 } \quad } \substack{ \displaystyle E_{even}\otimes F_{even}^+ \\ \oplus \\ \displaystyle E_{even}\otimes F_{even}^- \\ \oplus \\ \displaystyle E_{even}\otimes F_{odd}^+ \\ \oplus \\ \displaystyle E_{even}\otimes F_{odd}^-  }  = E_{even}\otimes F.$} \]
With these identifications, we have 
\begin{align*}
& \partial +\partial^\ast + S_{T\hat{\otimes}R} =\scalebox{1}{$\begin{cases}
	\vspace{2mm}
	(B + S_T)\otimes 1 + (B+S_T)\otimes D &  \textup{ on }  E_{odd}\otimes F^+_{even}, \\
	\vspace{2mm}
	(B - S_T)^2(B+S_T)\otimes 1  + (B+S_T)\otimes D  &  \textup{ on }  E_{odd}\otimes F^+_{odd}, \\
	\vspace{2mm}
	- (B - S_T)\otimes 1 + (B-S_T)\otimes D  &  \textup{ on }  E_{odd}\otimes F^-_{even}, \\
	- (B + S_T)^2(B-S_T)\otimes 1 +(B-S_T)\otimes D  &  \textup{ on }  E_{odd}\otimes F^-_{odd}.
	\end{cases} $}
\end{align*} 
Note that 
\[ (B\pm S_T)^2\colon E_{even} \xrightarrow{B\pm S_T} E_{odd} \xrightarrow{B\pm S_T} E_{even}\] are positive invertible operators. It follows that the invertible element \[ \partial + \partial^\ast + S_{T\hat{\otimes}R}\] is homotopic to 
\begin{equation}\label{eq:simple}
\begin{aligned}
\begin{psmallmatrix}
B+S_T & \\  & B-S_T
\end{psmallmatrix}S_R + \begin{psmallmatrix}  B-S_T & \\ &  B + S_T
\end{psmallmatrix}D 
\colon E_{odd}\otimes F \to E_{even}\otimes F
\end{aligned}
\end{equation} 
Here the matrix form is written with respect to  the decomposition $F = F^+ \oplus F^-$. Note that $D$ is off-diagonal in this case.  Now the invertible element from line $\eqref{eq:simple}$ in turn is homotopic to the invertible element  
\[ V =  \begin{psmallmatrix}
B+S_T & \\  & B-S_T
\end{psmallmatrix}S_Rf(D) + \begin{psmallmatrix}  B-S_T & \\  &  B + S_T
\end{psmallmatrix}g(D)\]
where\footnote{In fact, any normalizing function $g$ and   $f(x) = \sqrt{1-g^2(x)}$ will work.} $g(x)  = x(1+x^2)^{-1/2}$ and 
\[ f(x) = \sqrt{1-g^2(x)} = (1+x^2)^{-1/2}. \]  
Similarly,  $\partial + \partial^\ast - S_{T\hat{\otimes}R}$ is homotopic to
\[  U = \begin{psmallmatrix}
B-S_T & \\  & B+S_T
\end{psmallmatrix}S_Rf(D) + \begin{psmallmatrix} B-S_T & \\ &  B + S_T
\end{psmallmatrix}g(D)\]
Note that we have 
\[  U^{-1} = \big(S_Rf(D) + g(D)\big) \begin{psmallmatrix}
(B-S_T)^{-1} & \\  & (B+S_T)^{-1}
\end{psmallmatrix}  \]
since $\big(S_Rf(D) + g(D)\big)^2 = 1$. 

Similarly, if we replace $D$ by $D_t = (1+t)^{-1}D$, we have 
\[ V_t =  \begin{psmallmatrix}
B+S_T & \\  & B-S_T
\end{psmallmatrix}S_Rf(D_t) + \begin{psmallmatrix}  B-S_T & \\  &  B + S_T
\end{psmallmatrix}g(D_t)\]
and 
\[  U_t = \begin{psmallmatrix}
B-S_T & \\  & B+S_T
\end{psmallmatrix}S_Rf(D_t) + \begin{psmallmatrix} B-S_T & \\ &  B + S_T
\end{psmallmatrix}g(D_t).\]
It follows that the path of invertibles $(V_tU_t^{-1})_{0\leq t <\infty}$ is a representative of the $K$-theory class  $\rho(\theta\times \mathbb R^2)$. 
Now note that 
\begin{align*}
& V_tU_t^{-1} =    \left[B\otimes 1  + S_T\otimes \big(S_R(f^2(D_t) - g^2(D_t)) + 2f(D_t)g(D_t)\big)\right]  \\ 
& \hspace{4cm}\cdot \begin{psmallmatrix}
(B-S_T)^{-1} & \\  & (B+S_T)^{-1} 
\end{psmallmatrix}.
\end{align*}
Following the notation of \cite[Section 5.2.1]{MR2220523}, let us denote 
\[ S_1 = S_R \textup{ and } (S_2)_t =  g(D_t) +S_Rf(D_t). \]  We immediately see that 
\[ S_R(f^2(D_t) - g^2(D_t)) + 2f(D_t)g(D_t) = (S_2)_t S_1 (S_2)_t. \] 
Let us denote the latter by $\mathscr S_t \coloneqq  (S_2)_tS_1(S_2)_t$. Note that $\mathscr S_t$ is a symmetry for each $t$. We define a projection $P_t \coloneqq  (\mathscr S_t + 1)/2$. To summarize,  $\rho(\theta\times \mathbb R^2)$ is represented by  the path of invertibles
\begin{align*}
& (B\otimes 1 + S_T \otimes \mathscr S_t) \big[\begin{psmallmatrix}
( B-S_T)  & 0 \\ 0 & 0  
\end{psmallmatrix} +  \begin{psmallmatrix}
0  & 0 \\ 0 & ( B+S_T)  
\end{psmallmatrix}\big]^{-1} \\
= &  \big[ (B +S_T)\otimes P_t + (B -S_T) \otimes (1-P_t)\big] \\
& \hspace{3cm}\cdot \big[\begin{psmallmatrix}
( B-S_T)^{-1}  & 0 \\ 0 & 0  
\end{psmallmatrix} +  \begin{psmallmatrix}
0  & 0 \\ 0 & ( B+S_T)^{-1}
\end{psmallmatrix}\big]\\
= &  \begin{psmallmatrix}
\frac{B+S_T}{B-S_T}  & 0 \\ 0 & 1  
\end{psmallmatrix}\otimes P_t + \begin{psmallmatrix}
1  & 0 \\ 0 & 	\frac{B-S_T}{B+S_T}  
\end{psmallmatrix}\otimes (1-P_t)\\
= &  \Big[ \begin{psmallmatrix}
\frac{B+S_T}{B-S_T}  & 0 \\ 0 & \frac{B+S_T}{ B-S_T}
\end{psmallmatrix}\otimes P_t + \begin{psmallmatrix}
1 & 0 \\ 0 & 1
\end{psmallmatrix} \otimes (1-P_t)\Big] \cdot \begin{psmallmatrix}
1  & 0 \\ 0 & 	\frac{B-S_T}{B+S_T}  
\end{psmallmatrix}\\
= &  [ (B+S_T)(B-S_T)^{-1}] \times \big( [P_t] -  [\begin{psmallmatrix}
0  & 0 \\ 0 & 1  
\end{psmallmatrix}]\big)
\end{align*}
which is precisely\footnote{There appears to be a switch of sign convention in \cite[Section 5.2.1]{MR2220523}.  Our sign convention is consistent with the usual sign convention in the literature.   } $\tau_\ast( \rho(\theta) \otimes \ind_L(\mathbb R^2))$. This finishes the proof.   
\end{proof}

\nocite{AC94, PBAC88, MR1988288, GY98, DFW2016, MR2220524, MR1388305, MR1145337, MR866491, MR1204787, MR2947546, MR1076524, Xie2014823, Chen:2019aa, Xie:2017aa, Xie:2018aa}

\bibliographystyle{plain}


\end{document}